\documentclass{article}
\usepackage[top = 2.5cm, bottom = 2.5cm, left = 2.75cm, right=2.75cm]{geometry}
 
\usepackage{amsmath, amssymb, amsthm}
\usepackage{arydshln}
\usepackage{enumerate}
\usepackage{hyperref}
\hypersetup{
  colorlinks=true,
  citecolor=black,
  linkcolor=black}
\usepackage{booktabs}
\usepackage{pifont}
\newcommand{\cmark}{\ding{51}}%
\newcommand{\xmark}{\ding{55}}%
\usepackage{graphicx}
\graphicspath{{../figs/}}
\usepackage{natbib}
\usepackage{tocloft}
\makeatletter
\renewcommand*{\eqref}[1]{%
  \hyperref[{#1}]{\textup{\tagform@{\ref*{#1}}}}%
}
\makeatother
\newcommand{\one}{\textbf{1}}
\newcommand{\E}{\mathbb{E}}
\newcommand{\event}{\mathcal{V}}

\newcommand{\R}{\mathbb{R}}
\newcommand{\A}{\mathcal{A}}
\newcommand{\Ap}{\td{A}}
\newcommand{\Ep}{\td{E}}
\newcommand{\bp}{\td{b}}
\newcommand{\etap}{\td{\eta}}
\newcommand{\sigmap}{\td{\sigma}}
\newcommand{\D}{\mathcal{D}}

\newcommand{\ER}{Erd\"{o}s-R\'{e}nyi }
\renewcommand{\O}{\mathcal{O}}
\renewcommand{\P}{\mathbb{P}}
\newcommand{\M}{\mathcal{M}}
\newcommand{\C}{\mathcal{C}}
\newcommand{\N}{\mathcal{N}}

\newcommand{\V}{V}
\renewcommand{\i}{\sqrt{-1}}
\renewcommand{\L}{\mathcal{L}}
\newcommand{\Lp}{\td{\mathcal{L}}}
\newcommand{\eig}{\mathrm{eig}}
\renewcommand{\S}{\mathcal{S}}
\newcommand{\lb}{\left(}
\newcommand{\rb}{\right)}
\newcommand{\td}{\tilde}
\newcommand{\rank}{\mathrm{rank}}

\newcommand{\sign}{\mathrm{sign}}
\newcommand{\diag}{\mathrm{diag}}
\newcommand{\eps}{\epsilon}
\newcommand{\mnorm}[1]{\|#1\|_{2\rightarrow\infty}}
\newcommand{\op}{\mathrm{op}}
\newcommand{\sep}{\mathrm{sep}}
\newcommand{\gap}{\Delta}
\newcommand{\tdgap}{\Gamma}
\newcommand{\ttinf}{2\rightarrow\infty}

\newcommand{\maxnorm}[1]{\|#1\|_{\max}}

\newtheorem{theorem}{Theorem}[section]
\newtheorem{lemma}[theorem]{Lemma}
\newtheorem{proposition}[theorem]{Proposition}
\newtheorem{corollary}[theorem]{Corollary}
\newtheorem{definition}{Definition}[section]
\newtheorem{remark}{Remark}[section]
\newtheorem{example}{Example}[section]
\allowdisplaybreaks
\DeclareMathOperator*{\argmin}{arg\,min}

\DeclareMathOperator*{\Var}{Var}
\DeclareMathOperator*{\tr}{tr}
\title{Unified $\ell_{\ttinf}$ Eigenspace Perturbation Theory for Symmetric Random Matrices}
\author{Lihua Lei \thanks{lihualei@stanford.edu}\\
Department of Statistics, Stanford University}
\begin{document}
\maketitle

\begin{abstract}
Modern applications in statistics, computer science and network science have seen tremendous values of finer matrix spectral perturbation theory. In this paper, we derive a generic $\ell_{\ttinf}$ eigenspace perturbation bound for symmetric random matrices, with independent or dependent entries and fairly flexible entry distributions. In particular, we apply our generic bound to binary random matrices with independent entries or with certain dependency structures, including the unnormalized Laplacian of inhomogenous random graphs and $m$-dependent matrices. Through a detailed comparison, we found that for binary random matrices with independent entries, our $\ell_{\ttinf}$ bound is tighter than all existing bounds that we are aware of, while our eigen-gap condition is weaker than all but one of them in a less common regime. We employ our perturbation bounds in three problems and improve the state of the art: concentration of the spectral norm of sparse random graphs, exact recovery of communities in stochastic block models and partial consistency of divisive hierarchical clustering. Finally we discuss the extensions of our theory to random matrices with more complex dependency structures and non-binary entries, asymmetric rectangular matrices and induced perturbation theory in other metrics.
\end{abstract}

\setcounter{tocdepth}{2}
\tableofcontents
\section{Introduction}
Matrix spectral perturbation theory is one of the most fundamental and powerful tool in various areas including statistics, network sciences and computer sciences. In many problems, it is important to understand how eigenvalues and eigenvectors change when the underlying matrix $A^{*}$ is perturbed into $A$. Focusing on symmetric matrices, Weyl's inequality provides a simple bound for eigenvalues \citep{weyl1912asymptotische} and Davis-Kahan Theorem provides a surprisingly clean bound for eigenspaces in terms of any unitarily invariant norm \citep{davis1970rotation}. We refer to interested readers to \cite{stewart1990matrix}, \cite{kato2013perturbation} and the appendix of \cite{bai2010spectral} for fruitful results over the past century. 

Modern applications are posing new challenges for this long-standing area. In these problems, the unitarily invariant norm, such as Frobenius norm or operator norm, in the classical eigenvector perturbation theory may be too coarse to achieve the goal. It is then crucial to derive eigenvector/eigenspace perturbation bounds in terms of finer norms that are not unitarily invariant. Among others, one important norm is the $\ell_{\ttinf}$ norm, which yields the row-wise perturbation bound on the eigenvector matrix. To be specific, let $A$ and $A^{*}$ be two symmetric matrices with
\begin{equation}\label{eq:E}
  E = A - A^{*}.
\end{equation}
 Let $\lambda_{1}\ge \lambda_{2}\ge \ldots \ge \lambda_{n}$ and $\lambda_{1}^{*}\ge \lambda_{2}^{*}\ge \ldots \ge \lambda_{n}^{*}$ be the eigenvalues of $A$ and $A^{*}$, respectively. Given positive integers $s$ and $r$, let
\begin{equation}\label{eq:LambdaLambda*}
\Lambda = \diag(\lambda_{s + 1}, \lambda_{s + 2}, \ldots, \lambda_{s + r}), \quad \Lambda^{*} = \diag(\lambda_{s + 1}^{*}, \lambda_{s + 2}^{*}, \ldots, \lambda_{s + r}^{*}).
\end{equation}
Let $U, U^{*}\in \R^{n\times r}$ be a matrix of eigenvectors such that 
\begin{equation}\label{eq:UU*}
  AU = U\Lambda, \quad A^{*}U^{*} = U^{*}\Lambda^{*}.
\end{equation}
The $\ell_{\ttinf}$ perturbation theory is seeking for bounds on the $\ell_{\ttinf}$ distance between $U$ and $U^{*}$, defined as 
\begin{equation*}
  d_{2\rightarrow \infty}(U, U^{*})\triangleq \inf_{O\in \R^{r\times r}, O^{T}O = I}\mnorm{UO - U^{*}}.
\end{equation*}
where $\mnorm{V} = \max_{i\in [n]}\|V_{i}\|_{2}$ and $V_{i}$ is the $i$-th row of $V$. The $\ell_{\ttinf}$ norm is not invariant to left unitary transformation and thus not supported by classical Davis-Kahan Theorem. When $r = 1$, $d_{2\rightarrow \infty}(U, U^{*})$ reduces to the entrywise perturbation of the eigenvector. Compared to perturbation bounds in operator norm, it provides much finer information. 

\subsection{Existing works}
Early efforts on $\ell_{\ttinf}$ perturbation was motivated by the stability of Markov chain \citep[e.g.][]{o1993entrywise, ipsen1994uniform}. The focus was on the stationary distribution, which is the first eigenvector of the transition matrix of a finite state Markov chain. The investigation in random graph theory can be dated back to \cite{mitra2009entrywise} which studied the entrywise perturbation for the leading eigenvector of the adjacency matrix of an \ER graph $\mathbb{G}(n, p)$. In particular, \cite{mitra2009entrywise} considered the sparse graph regime allowing the parameter $p$ to decay with $n$ as long as $p \ge (\log n)^{6} / n$. He further studied the entrywise perturbation of the second eigenvector for dense planted partition models where the within-block average degree is $\sqrt{n}$. This is the first work proving that spectral clustering can achieve exact recovery, namely zero mis-classification error, with high probability. These two lines of works typify two perspectives: the former assumes a non-random $A$ and studies the deterministic perturbation bound, as in Davis-Kahan Theorem, which yields the worst-case guarantee; the latter assumes a stochastic model for $A$ and studies the high probability perturbation bound that is average-case in nature. The deterministic perturbation theory can be applied to more settings, including the stochastic settings, though it is typically much weaker than those stochastic bounds which are tailored to particular random structure. 

Recent years has seen a surge of interest in $\ell_{\ttinf}$ perturbation theory especially for random matrices. This includes robust covariance estimation and robust principal component analysis for heavy-tailed data \citep{fan2018eigenvector}, community detection \citep{balakrishnan2011noise, eldridge2017unperturbed, mao2017estimating, abbe2017entrywise, cape2019signal}, multiple graph inference \citep{cape2019two}, phase synchronization \citep{zhong2018near, abbe2017entrywise}, matrix completion \citep{abbe2017entrywise}. It also serves as a powerful tool to push forward other theoretical works such as random graph theory \citep[e.g.][]{lugosi2018concentration}.

\begin{table}
  \centering
  \begin{tabular}{llllllll}
    \toprule
    & full & top-$r$ & partial & $np^{*}$ & $r$ & $\kappa^{*}$ & $\mnorm{U^{*}}$ \\
    \midrule
    \cite{abbe2017entrywise} & \cmark & \cmark & \cmark & $\succeq \log n / \log \log n$ & No & $\preceq \log (np^{*})$ & No \\
    \cite{eldridge2017unperturbed} & \cmark & \cmark & \cmark & $\succeq (\log n)^{2 + \eps}$ & $ = 1$ & $ = 1$ & $\preceq 1 / \sqrt{n}$ \\
    \cite{mao2017estimating} & \cmark & \xmark & \xmark & $\succeq (\log n)^{2 + \eps}$ & No & No & $\preceq \sqrt{rp^{*}}$ \\
    \cite{fan2018eigenvector} & \cmark & \cmark* & \xmark & $\succeq \log n $ & $\preceq 1$ & No & $\preceq 1/ \sqrt{n}$ \\
    \cite{cape2019signal} & \cmark & \xmark & \xmark & $\succeq (\log n)^{2 + \eps}$ & $\preceq (\log n)^{2 + \eps}$ & $\preceq 1$ & No \\
    \cite{cape2019two} & \cmark & \cmark* & \xmark & $\succeq \log n$ & No & No & No\\
    This paper & \cmark & \cmark & \cmark & $\succeq \log n / \log \log n$ & No & No & No\\
    \bottomrule
  \end{tabular}
  \caption{The regime that each method works on for binary random matrices with independent entries. The \cmark* symbol in the third column means that the bound is derived for both full and top-$r$ eigenspace recovery but only the former is available for binary random matrices. $p^{*}$ denotes the maximum entry of $A^{*}$ and $\kappa^{*}$ denotes the condition number of $\Lambda^{*}$. The constraints may not be explicitly mentioned in those papers but can be derived from their conditions. They are only necessary conditions and the real constraint may be more stringent. See Appendix \ref{app:comparison} for details.}\label{tab:comparison}
\end{table}

Despite the great success in different applications, the existing $\ell_{\ttinf}$ perturbation theory is not satisfactory in that the bounds work under rather different regimes. To better describe each work, we classify the applicability into three categories:
\begin{itemize}
\item Full eigenspace recovery: $s = 0$ and $\lambda_{r + 1}^{*} = \ldots = \lambda_{n}^{*} = 0$. In this case, $A^{*}$ has to be low rank.
\item Top-$r$ eigenspace recovery: $s = 0$ and there is no restriction on other eigenvalues. 
\item Partial eigenspace recovery: There is no restriction on $s$ or other eigenvalues. 
\end{itemize}
We consider binary random matrices with independent entries for illustration. Table \ref{tab:comparison} summarizes the necessary conditions for each bound to work, where $p^{*}$ denotes the maximum entry of $A^{*}$, $\kappa^{*} = \lambda_{s+1}^{*} / \lambda_{s+r}^{*}$ denotes the condition number of $\Lambda^{*}$, and $\preceq (\succeq)$ denotes smaller (larger) in order. See Appendix \ref{app:comparison} for the derivation of these claims. As with Davis-Kahan theorem, $\ell_{\ttinf}$ perturbation theory also requires sufficient eigen-gap. Since these bounds work under different regimes, we consider the intersection of them for comparison, namely the regime $np^{*}\succeq (\log n)^{2 + \eps}, r\preceq 1, \kappa^{*}\preceq 1, \mnorm{U^{*}}\preceq 1 / \sqrt{n}$. To simplify we also assume that all entries of $A^{*}$ are in the same order of $p^{*}$. Table \ref{tab:condition_bound} summarizes the conditions on eigen-gaps as well as the bounds for each work in this special case. 

From Table \ref{tab:comparison}, we can see that the bounds of \cite{eldridge2017unperturbed, mao2017estimating, cape2019signal} require the maximum degree $np^{*}$ to grow faster than $(\log n)^{2 + \eps}$. However, the critical regime of $np^{*}$ in random graph theory is typically $\log n$, under which phase transition occurs. Therefore, although $(\log n)^{2 + \eps}$ appears to be close to the critical regime, the extra $(\log n)^{1 + \eps}$ term is too artificial to explain interesting phenomena. \cite{abbe2017entrywise}'s work is the first to remove this extra logarithmic terms using an ingenious leave-one-out argument. Nonetheless, it imposes a stringent assumption on the condition number: in the critical regime it only allows the condition number to grow as $\log \log n$. The bound of \cite{fan2018eigenvector} also removes the artificial log-factors but it requires the number of eigenvectors to be bounded and the eigenvectors have low coherence. From Table \ref{tab:condition_bound}, we can see that all works require a stringent assumption on the eigen-gap except \cite{mao2017estimating}. Indeed, it is easy to show that the eigen-gap is always upper bounded by $np^{*}$. As a consequence, the four works except \cite{mao2017estimating} require the eigen-gap to be nearly the largest achievable value even in this special regime. By contrast, in Davis-Kahan Theorem, the lower bound on the eigen-gap is simply the operator norm of the perturbation, which is of order $\sqrt{np^{*}}$ in this case. Although \cite{mao2017estimating} gets rid of the $np^{*}$ lower bound, it still involves extra log-factors. Finally, in terms of the perturbation bound, we see that the bounds of \cite{abbe2017entrywise}, \cite{fan2018eigenvector} and \cite{cape2019two} stay the same as $np^{*}$ increases. They are clearly inferior to the others that are decaying with $np^{*}$. 

The real problems may not lie in the nice intersection regime as above. The existing bounds work in different regimes and there is no one dominating all others in terms of either applicability or tightness. This creates hurdles to choose which bound to use. Furthermore, it provides evidence that none of the existing bounds is tight and there is still much room for improvement and unification.

\begin{table}
  \centering
  \begin{tabular}{lll}
    \toprule
    & condition on the eigen-gap & $\sqrt{n}d_{\ttinf}(U, U^{*})$\\
    \midrule
    \cite{abbe2017entrywise} & $\succeq np^{*} / \log (np^{*})$ & $\preceq 1$\\
    \cite{eldridge2017unperturbed} & $\succeq np^{*}$ & $\preceq \sqrt{(\log n)^{2 + \eps} / np^{*}}$\\
    \cite{mao2017estimating} & $\succeq \sqrt{np^{*}}(\log n)^{1 + \eps / 2}$ & $\preceq \sqrt{(\log n)^{2 + \eps} / np^{*}}$\\
    \cite{fan2018eigenvector} & $\succeq np^{*}$ & $\preceq 1$\\
    \cite{cape2019signal} & $\succeq np^{*}$ & $\preceq \sqrt{(\log n)^{2 + \eps} / np^{*}}$\\
    \cite{cape2019two} & $\succeq np^{*}$ & $\preceq 1$ \\
    This paper & $\succeq \sqrt{np^{*}}$ & $\preceq \sqrt{\log n / np^{*}}$\\
    \bottomrule
  \end{tabular}
  \caption{Comparison of the conditions on the eigen-gap and the bounds in each work under the regime $np^{*}\succeq (\log n)^{2 + \eps}, r\preceq 1, \kappa^{*}\preceq 1, \mnorm{U^{*}}\preceq 1 / \sqrt{n}, A_{ij}^{*}\sim p^{*}$.}\label{tab:condition_bound}
\end{table}

\subsection{This paper}
In this paper, we derive generic $d_{\ttinf}$ bounds that work for random matrices with independent entries or with dependent entries with certain dependency structures (Theorem \ref{thm:generic_bound} - Theorem \ref{thm:generic_bound2_trick}). As with \cite{abbe2017entrywise} and \cite{cape2019signal}, our theory not only covers $d_{\ttinf}(U, U^{*})$ but also covers $d_{\ttinf}(U, U^{*} + V)$ for some $V\in \R^{n\times r}$ that yields a better approximation of $U$. In our theory, the entry distribution can be arbitrary provided that some characteristics of the perturbation $E$ can be controlled, e.g. operator norm and linear contrasts of rows. This includes but not limited to Bernoulli, Gaussian, sub-Gaussian and sub-exponential distributions. For the sake of length, we only discuss binary random matrices, because it is arguably the most challenging yet most common case calling for $\ell_{\ttinf}$ perturbation theory, and provide a brief discussion in Section \ref{subsec:otherdist} for other entry distributions.

For binary random matrices, we derive the $d_{\ttinf}$ bounds for the case with independent entries, or equivalently the adjacency matrix of random graphs (Theorem \ref{thm:generic_binary}). In this case, our bound works in a broader regime than all of existing bounds as shown in the last row of Table \ref{tab:comparison}. Under the special regime for Table \ref{tab:condition_bound}, our result has the least stringent condition on the eigen-gap that matches Davis-Kahan Theorem while our bound is strictly sharper than all others. Through a more detailed comparison in Appendix \ref{app:comparison}, we found that our bound is tighter than all existing bounds that we are aware of in all regimes, except in some corner cases where our bound is equivalent to some of the others. In addition, our eigen-gap condition is weaker than all others except in the case of full eigenspace recovery with $np^{*}\succeq (\log n)^{2 + \eps}$, $\min\{\kappa^{*}, r\}\succeq (\log n)^{1 + \eps / 2}$ for which \cite{mao2017estimating}'s condition is weaker than ours. 

In addition, we derive the inequality for unnormalized Laplacian of random graphs (Theorem \ref{thm:generic_binary_laplacian}) and discuss the case for binary matrices with certain $m$-dependence structure (Section \ref{subsec:mdependence}). The former is particularly useful for community detection problems in network science. Both cases are straightforward applications of our generic bounds. 

Our bounds can be applied to the problems considered in the works listed in Table \ref{tab:comparison}. In this paper, we apply our bounds to three other problems. The first one is to bound the variance and derive the concentration of the spectral norm of sparse random graphs. This is a fundamental and long-standing problem in random graph theory. Classical theory shows that, for \ER graphs, the variance is bounded by a universal constant regardless of the graph sparsity and the spectral norm is sub-gaussian with an $O(1)$ parameter. The recent work by \cite{lugosi2018concentration} drastically improves the bound of variance and sub-gaussian parameters to $O(p^{*})$ for \ER graphs using the $\ell_{\ttinf}$ perturbation theory. However, they only prove the result for $p^{*}\succeq (\log n)^{3} / n$ and conjecture that it carries over to the critical regime $p^{*}\succeq \log n / n$. We prove this conjecture using our new $\ell_{\ttinf}$ perturbation theory. Moreover, we extend the result to general inhomogeneous random graphs. 

The second problem is the strong consistency of spectral clustering for community detection in sparse stochastic block models (SBM). Despite the substantial literature on this topic, existing algorithms are rarely as easy-to-implement and computationally efficient as the standard spectral clustering algorithm \citep[e.g.][]{von2007tutorial}, which is simply a singular value decomposition plus a $K$-means algorithm. However, perhaps surprisingly, the strong consistency of spectral clustering in sparse graphs is not established until recently \citep{abbe2017entrywise, su2019strong}.
In this paper, we apply our $\ell_{\ttinf}$ bounds to prove the strong consistency of spectral clustering algorithms, using the adjacency matrix or the unnormalized Laplacian, for general SBMs in the critical regime with average degree $O(\log n)$. We also study the SBMs with growing number of communities and obtain the best available dependence on it for spectral algorithms.

The third problem is the partial consistency of divisive hierarchical clustering on binary tree stochastic block models (BTSBM), proposed by \cite{li2018hierarchical}. BTSBMs are special SBMs that embed communities into a tree-like hierarchy in order to bring interpretability. It is also a useful framework to analyze divisive hierarchical clustering algorithms such as iterative spectral bi-partitioning \citep[e.g.][]{spielman1996spectral, balakrishnan2011noise}. \cite{li2018hierarchical} derived sufficient conditions that certain divisive hierarchical clustering algorithms are able to exactly recover the whole or a part of the hierarchy in the regime with the average degree $O((\log n)^{2 + \eps})$. In this paper, we extend the result to the critical regime where each connection can be written as $a_{j}\log n / n$ and establish the sufficient condition to recover any part of the hierarchy in terms of the coefficients $a_{j}$'s. We found an unexpected connection between BTSBMs and mis-specified SBMs. In addition we accurately quantify an observation in \cite{li2018hierarchical} that certain partial structure may still be recovered even if the communities are information theoretically unrecoverable. 

The rest of the article is organized as follows: Section \ref{sec:main} presents the generic bounds and Section \ref{sec:binarybound} presents the bounds for binary random matrices. Three examples are collected in Section \ref{sec:spectral_norm} - Section \ref{sec:hierarchical}. Section \ref{sec:extension} discusses several extensions, including binary random matrices with $m$-dependence structure, random matrices with non-binary entry distribution, singular space perturbation for asymmetric matrices and perturbation bounds in other metrics. Most technical proofs are relegated into Appendix. Appendix \ref{app:generic_bound} establishes the proof of our main generic bound. The proof is quite involved so we parse it into six steps. All other proofs related to the generic bounds are presented in Appendix \ref{app:main}. Appendix \ref{app:binarybound} contains all technical proofs for binary random matrices. The miscellaneous proofs in Section \ref{sec:spectral_norm} - Section \ref{sec:hierarchical} are collected in Appendix \ref{app:other}. Appendix \ref{app:comparison} provides a detailed comparison between our bounds and all existing ones that we are aware of. This includes the justification of Table \ref{tab:comparison} and Table \ref{tab:condition_bound}. Finally, Appendix \ref{app:concentration} presents useful concentration inequalities for binary random variables which are useful in Section \ref{sec:binarybound}.

\section{An Generic $\ell_{\ttinf}$ Bound for Symmetric Random Matrices}\label{sec:main}
\subsection{Notations and assumptions}
Throughout the paper we consider the setup \eqref{eq:E} - \eqref{eq:UU*}. We denote by $[n]$ the set $\{1, \ldots, n\}$ and by $\one_{n}$ the $n$-dimensional vector with all entries $1$. For any vector $x$, let $\|x\|_{p}$ denotes its $p$-norm. For any matrix $M$, let $M_{k}^{T}$ denote the $m$-th row of $M$, $\|M\|_{\op}$ denote its operator norm and $\|M\|_{\mathrm{F}}$ denote its Frobenius norm. When $M$ is a square matrix, we define the matrix sign as
\[\sign(M) = UV^{T}.\]
By definition, $\sign(M)$ is orthogonal. When $n = 1$, $M$ is a scalar and $\sign(M)$ reduces to the classical sign of scalars. Further we denote by $\lambda_{\max}(M)$ (resp. $\lambda_{\min}(M)$) the largest (resp. the smallest) eigenvalue of $M$ in absolute values, by $\kappa(M)$ the condition number $\lambda_{\max}(M) / \lambda_{\min}(M)$. In particular we write $\lambda_{\min}^{*}(\Lambda^{*})$ as $\lambda_{\min}^{*}$ for short.

To state our generic bound, we need to define the following quantities. 
\begin{itemize}
\item \emph{Effective eigen-gap} $\gap^{*}$: 
  \begin{equation}
    \label{eq:gap}
    \gap^{*} \triangleq \min\{\sep_{s+1, s+r}(A^{*}), \lambda_{\min}^{*}\},
  \end{equation}
where $\sep_{s+1, s+r}(A^{*}) = \min\{\lambda_{s}^{*} - \lambda_{s + 1}^{*}, \lambda_{s + r}^{*} - \lambda_{s + r + 1}^{*}\}$ with $\lambda_{0}^{*} = \infty$ and $\lambda_{n + 1}^{*} = -\infty$. Note that $\gap^{*} = \sep_{s+1, s+r}(A^{*})$ except when $\Lambda^{*}$ includes both positive and negative eigenvalues but $0$ is not an eigenvalue of $\Lambda^{*}$. Therefore $\gap^{*}$ is essentially the eigen-gap in the conventional sense.
\item \emph{Effective condition number} $\bar{\kappa}^{*}$:
  \begin{equation}
    \label{eq:barkappa}
    \bar{\kappa}^{*}\triangleq \min\{\kappa(\Lambda^{*}), 2r\},
  \end{equation}
Note that the effective condition number is never larger $2r$ however ill-conditioned the problem is. On the other hand, if $\Lambda^{*}$ is well-conditioned but $r$ is large, $\bar{\kappa}^{*}$ can also be small.
\item full eigenspace $\bar{U}^{*}$ of $A^{*}$, i.e. 
  \begin{equation}
    \label{eq:barU}
    A^{*}\bar{U}^{*} = \bar{U}^{*}\bar{\Lambda^{*}},
  \end{equation}
where $\bar{\Lambda}^{*}$ includes all non-zero eigenvalues of $A^{*}$. Note that the number of columns $\bar{U}^{*}$ may significantly differ from that of $U^{*}$;
\end{itemize}

Our generic $\ell_{\ttinf}$ bound requires the following four assumptions.
\begin{enumerate}[\textbf{A}1]\label{enumi:assumptions}
\item For any $\delta \in (0, 1)$, there exists a random matrix $A^{(k)}\in \R^{n\times n}$ such that
  \begin{equation}
    \label{eq:TV_cond}
    d_{TV}(\P_{(A_{k}, A^{(k)})}, \P_{A_{k}}\times \P_{A^{(k)}})\le \delta / n,
  \end{equation}
where $d_{TV}$ denotes the total variation distance and it holds simultaneously for all $k$ and all contiguous subsets $S\subset [r]$ that
\[\|A^{(k)} - A\|_{\op}\le L_{1}(\delta), \quad \frac{\|(A^{(k)} - A)U_{S}\|_{\op}}{\lambda_{\min}(\Lambda_{S}^{*})}\le \lb \kappa(\Lambda_{S}^{*})L_{2}(\delta) + L_{3}(\delta)\rb \mnorm{U_{S}},\]
with probability at least $1 - \delta$ for some deterministic functions $L_{1}(\delta), L_{2}(\delta), L_{3}(\delta)$, where $U_{S}\in \R^{n\times |S|}$ denotes the matrix formed by columns of $U$ with indices in $S$.
\item There exists deterministic functions $\lambda_{-}(\delta), E_{+}(\delta), \bar{E}_{+}(\delta), E_{\infty}(\delta)$, such that for any $\delta \in (0, 1)$, the following event occurs with probability at least $1 - \delta$:
\[\maxnorm{\Lambda - \Lambda^{*}}\le \lambda_{-}(\delta), \quad \|EU^{*}\|_{\op}\le E_{+}(\delta), \quad \|E\bar{U}^{*}\|_{\op}\le \bar{E}_{+}(\delta), \quad \mnorm{E}\le E_{\infty}(\delta).\]
\item There exist deterministic functions $b_{\infty}(\delta), b_{2}(\delta) > 0$, such that for any $\delta \in (0, 1)$, $k\in [n]$, $r'\le r$ and fixed matrix $W\in \R^{n\times r'}$,
  \begin{align*}
    \|E_{k}^{T}W\|_{2} &\le b_{\infty}(\delta)\mnorm{W} + b_{2}(\delta)\|W\|_{\op}, \,\, \mbox{with probability at least } 1 - \delta / n,
  \end{align*}
\item $\gap^{*}\ge 4\lb\sigma(\delta) + L_{1}(\delta) + \lambda_{-}(\delta)\rb$ where
  \begin{equation}
    \label{eq:eta}
\eta(\delta)  = E_{\infty}(\delta) + b_{\infty}(\delta) + b_{2}(\delta), \quad \sigma(\delta) = \{\bar{\kappa}^{*}L_{2}(\delta) + L_{3}(\delta) + 1\}\eta(\delta) + E_{+}(\delta),
  \end{equation}
\end{enumerate}
Assumption \textbf{A}1 is worth some discussion. Roughly speaking, \textbf{A}1 controls the amount and the structure of entry dependence. The following proposition gives three cases where $L_{1}(\delta), L_{2}(\delta)$ and $L_{3}(\delta)$ can be exactly characterized. The proof is relegated to Appendix \ref{app:main}.
\begin{proposition}\label{prop:A1}
  Assume that 
  \begin{enumerate}[(a)]
  \item If $A_{ij}$'s are independent random variables, then there exists $A^{(k)}$ such that \textbf{A}1 is satisfied with
\[L_{1}(\delta) = \sqrt{2}(\mnorm{A^{*}} + E_{\infty}(\delta)), \quad L_{2}(\delta) = 1, \quad L_{3}(\delta) = \frac{E_{\infty}(\delta) + \lambda_{-}(\delta) + \mnorm{A^{*}}}{\lambda_{\min}^{*}}.\]
  \item Assume that for any $k$, there exists a subset $\N_{k}\subset [n]$, such that $A_{k}$ is independent of $\{A_{i}: i\not\in \N_{k}\}$. Let $m = \max_{k}|\N_{k}|$, then there exists $A^{(k)}$ such that \textbf{A}1 is satisfied with
\[L_{1}(\delta) = \sqrt{2m}(\mnorm{A^{*}} + E_{\infty}(\delta)), \quad L_{2}(\delta) = m, \quad L_{3}(\delta) = \frac{m(E_{\infty}(\delta) + \lambda_{-}(\delta) + \mnorm{A^{*}})}{\lambda_{\min}^{*}}.\]
  \end{enumerate}
\end{proposition}

In literature \citep[e.g.][]{abbe2017entrywise}, it is typically assumed that 
\begin{equation}
  \label{eq:E2delta}
  \|E\|_{\op}\le E_{2}(\delta) \quad \mbox{with probability at least }1 - \delta.
\end{equation}
Assumption \textbf{A}2 is satisfied under \eqref{eq:E2delta} if 
\begin{equation}\label{eq:cond_E2}
  \lambda_{-}(\delta) = E_{+}(\delta) = \bar{E}_{+}(\delta) = E_{\infty}(\delta) = E_{2}(\delta).
\end{equation}
This is because $\maxnorm{\Lambda - \Lambda^{*}}\le \|E\|_{\op}$ by Weyl's inequality, $\|EU^{*}\|_{\op}\le \|E\|_{\op}, \|E\bar{U}^{*}\|_{\op}\le \|E\|_{\op}$ and $\mnorm{E}\le \|E\|_{\op}$ by definition. In general, \textbf{A}2 can be strictly weaker than \eqref{eq:E2delta}. 

Assumption \textbf{A}3 requires a concentration inequality on linear transforms of rows of $E$. A similar version is consisdered in \cite{abbe2017entrywise} except that the operator norm is replaced by the Frobenius norm. We emphasize that our assumption \textbf{A}3 can yield tighter result when $r > 1$. Using a standard $\eps$-net argument, \textbf{A}3 can be verified by considering vectors $W$ only. The following proposition summarizes the result with the proof relegated to Appendix \ref{app:binarybound}.
\begin{proposition}\label{prop:vector2matrix}
  Suppose that for any $\delta \in (0, 1)$ and vector $w\in \R^{n}$, there exists $a_{\infty}(\delta), a_{2}(\delta) > 0$ such that for each $k$
  \begin{align*}
    E_{k}^{T}w &\le a_{\infty}(\delta)\|w\|_{\infty} + a_{2}(\delta)\|w\|_{2},
  \end{align*}
with probability at least $1 - \delta$. Then assumption \textbf{A}3 holds with 
\[b_{\infty}(\delta) = 2a_{\infty}\lb\frac{\delta}{5^{r}n}\rb, \quad b_{2}(\delta) = 2a_{2}\lb\frac{\delta}{5^{r}n}\rb.\]
\end{proposition}

Assumption \textbf{A}4 guarantees sufficient eigen-gap. In classical perturbation theory \citep[e.g.][]{davis1970rotation} based on operator norm or Frobenius norm, it is necessary to assume $\gap^{*}\succeq \lambda_{-}(\delta)$. Nonetheless, we will show that \textbf{A}4 is equivalent to $\gap^{*}\succeq \lambda_{-}(\delta)$ in many applications. 

\subsection{Main results}\label{subsec:main}
Based on our assumptions, we first derive an bound for $d_{\ttinf}(U, AU^{*}(\Lambda^{*})^{-1})$. As shown in \cite{abbe2017entrywise} and \cite{cape2019two}, $AU^{*}(\Lambda^{*})^{-1}$ is a better approximation of $U$ than $U^{*}$. The proof is quite involved and thus presented step by step in Appendix \ref{app:generic_bound}.
\begin{theorem}\label{thm:generic_bound}
 Given any $\delta\in (0, 1)$. Let $\gap^{*}, \bar{\kappa}^{*}$ and $\bar{U}^{*}$ be defined in \eqref{eq:gap} - \eqref{eq:barU}. Then under assumptions \textbf{A}1-\textbf{A}4, 
  \begin{align}
    d_{\ttinf}(U, AU^{*}(\Lambda^{*})^{-1})&\le \frac{C}{\gap^{*}}\bigg\{\sigma(\delta)\lb\mnorm{U^{*}} + \frac{\mnorm{EU^{*}}}{\lambda_{\min}^{*}}\rb  + \frac{E_{+}(\delta)b_{2}(\delta)}{\lambda_{\min}^{*}}\nonumber\\
& \qquad\quad + \min\left\{E_{+}(\delta)\xi_{1}, \bar{E}_{+}(\delta)\sqrt{\bar{\kappa}^{*}}\xi_{2}, \bar{E}_{+}(\delta)\bar{\kappa}^{*}\xi_{3}\right\}\bigg\},\nonumber
  \end{align}
with probability at least $1 - B(r)\delta$, where $C$ is a universal constant (that can be chosen as $72$), 
\begin{equation}
  \label{eq:Br}
  B(r) = 10\min\{r, 1 + \log_{2}\kappa^{*}\},
\end{equation}
and
\begin{align}
\xi_{1}  = \frac{\mnorm{A^{*}}}{\lambda_{\min}^{*}}, \quad \xi_{2} = \frac{\sqrt{\maxnorm{A^{*}}}}{\sqrt{\lambda_{\min}^{*}} I(A^{*}\mbox{ is psd})}, \quad \xi_{3} = \mnorm{\bar{U}^{*}}\label{eq:xi}.
\end{align}
\end{theorem}
\begin{remark}
The last term is the minimum of three fundamentally different bounds, which are obtained from Kato's integral \citep[e.g.][]{kato1949convergence}; see Appendix \ref{subapp:step2} for details. The second term kicks in only when $A^{*}$ is postive semi-definite. This is true in many cases such as spiked wigner ensemble and phase synchronization \citep[e.g.][]{abbe2017entrywise}. The third term becomes useful for  full eigenspace recovery. This is common in the cases of community detection (Section \ref{sec:exact_recovery}).
\end{remark}

\begin{remark}\label{rem:EUA3}
All terms in the bound are deterministic except for $\mnorm{EU^{*}}$. Although assumption \textbf{A}3 directly yields a bound as follows:
\[\mnorm{EU^{*}}\le b_{\infty}(\delta) \mnorm{U^{*}} + b_{2}(\delta)\|U^{*}\|_{\op} = b_{\infty}(\delta) \mnorm{U^{*}} + b_{2}(\delta),\]
we found it can be sharpened in some applications; see Section \ref{sec:binarybound} for instance. For this reason we keep this term in the bound and derive its upper tail case by case. 
\end{remark}

\begin{remark}\label{rem:orthogonal}
  Recall that $d_{\ttinf}(U, AU^{*}(\Lambda^{*})^{-1}) = \inf_{O\in \R^{r\times r}, O^{T}O = I}\mnorm{UO - AU^{*}(\Lambda^{*})^{-1}}$. In our proof, we find a particular orthogonal matrix $\td{O}$ such that the proved bound holds for $\mnorm{U\td{O} - AU^{*}(\Lambda^{*})^{-1}}$. Indeed, this orthogonal matrix is in the form of
  \[\td{O} = \diag(\sign(U_{1}^{T}U_{1}^{*}), \ldots, \sign(U_{J}^{T}U_{J}^{*}))\]
  where $U_{1}, \ldots, U_{J}$ (resp. $U_{1}^{*}, \ldots, U_{J}^{*}$) form a partition of $U$ (resp. $U^{*}$) and the eigenvalues corresponding to each $U_{j}$ (resp. $U_{j}^{*}$) consist of consecutive eigenvalues in $\Lambda$ (resp. $\Lambda^{*}$). Furthermore, as shown in Step I-IV in Appendix \ref{app:generic_bound}, Theorem \ref{thm:generic_bound} holds for $\mnorm{U\sign(U^{T}U^{*}) - AU^{*}(\Lambda^{*})^{-1}}$ if $\bar{\kappa}^{*}$ is replaced by $\kappa^{*}$, in both the assumptions and the bound. This also applies to Theorem \ref{thm:generic_bound2} - \ref{thm:generic_bound2_trick}.
\end{remark}

By the triangle inequality, we can directly obtain the following perturbation bound for $d_{\ttinf}(U, U^{*})$ by $d_{\ttinf}(U, AU^{*}(\Lambda^{*})^{-1}) + d_{\ttinf}(AU^{*}(\Lambda^{*})^{-1}, U^{*})$. We leave the proof in Appendix \ref{app:main}. In many cases, the first term dominates the second term, in which cases Theorem \ref{thm:generic_bound2} essentially implies that $d_{\ttinf}(U, U^{*}) \asymp \mnorm{EU^{*}} / \lambda_{\min}^{*}$.
\begin{theorem}\label{thm:generic_bound2}
Under the same setting of Theorem \ref{thm:generic_bound},
  \begin{align}
    \lefteqn{d_{\ttinf}(U, U^{*})\le \frac{C\mnorm{EU^{*}}}{\lambda_{\min}^{*}}}\nonumber\\
& + \frac{C}{\gap^{*}}\bigg\{\sigma(\delta)\mnorm{U^{*}} + \frac{E_{+}(\delta)b_{2}(\delta)}{\lambda_{\min}^{*}} + \min\left\{E_{+}(\delta)\xi_{1}, \bar{E}_{+}(\delta)\sqrt{\bar{\kappa}^{*}}\xi_{2}, \bar{E}_{+}(\delta)\bar{\kappa}^{*}\xi_{3}\right\}\bigg\},\nonumber
  \end{align}
with probability at least $1 - B(r)\delta$, where $C$ is a universal constant (that can be chosen as $72$).
\end{theorem}

\subsection{Sharpening the bound via diagonal surgery}\label{subsec:generic_laplacian}
In this subsection we discuss a trick, referred to as ``diagonal surgery", to further improve the $\ell_{\ttinf}$ bound. The motivation is from matrices with large diagonal elements. If $A^{*}$ has high rank or even full rank, the last term in Theorem \ref{thm:generic_bound} and \ref{thm:generic_bound2} reduce to $\min\{\xi_{1}, \sqrt{\bar{\kappa}^{*}}\xi_{2}\}$. However, both $\xi_{1}$ and $\xi_{2}$ will be large when $A^{*}$ has large diagonal elements. For instance, when $A$ is the unnormalized graph Laplacian, defined in Section \ref{subsec:binary_laplacian} later, of an Erd\"{o}s-Renyi graph with edge connection probability $p$, then $\mnorm{A^{*}} \ge \maxnorm{A^{*}} = (n - 1)p$. By contrast, when $A$ is the adjacency matrix, $\mnorm{A^{*}} = p\sqrt{n - 1}, \maxnorm{A^{*}} = p$. If we had used Theorem \ref{thm:generic_bound} or \ref{thm:generic_bound2}, the bound applied to graph Laplacian would be significantly worse than that applied to the adjacency matrix. To overcome this shortcoming, we found a modification of the proof that allows us to replace $A$ by $A - \Sigma$ where $\Sigma$ is a possibly random diagonal matrix, provided that all diagonal elements of $\Sigma$ are well separated from $\Lambda^{*}$. This trick was implicitly used in \cite{balakrishnan2011noise}. Specifically, we need the following assumption:
\begin{enumerate}[\textbf{A}'0]
\item For any $\delta \in (0, 1)$, 
\[\frac{\min_{j\in [s + 1, s + r]}|\Lambda_{jj}^{*}|}{\min_{j\in [s + 1, s + r], k\in [n]}|\Lambda_{jj}^{*} - \Sigma_{kk}|}\le \Theta(\delta),\]
with probability at least $1 - \delta$ for some deterministic function $\Theta(\delta) > 0$.
\end{enumerate}
Let 
\begin{equation}
  \label{eq:EpAp}
  \Ap = A - \Sigma, \quad \Ap^{*} = \E \Ap, \quad \Ep = \Ap - \Ap^{*}.
\end{equation}
All other assumptions need to be slightly modified.
\begin{enumerate}[\textbf{A}'1]\label{enumi:assumptions}
\item The same as \textbf{A}1 except that \eqref{eq:TV_cond} is replaced by 
\[d_{TV}\lb\P_{(\Ep_{k}, A^{(k)})}, \P_{\Ep_{k}}\times \P_{A^{(k)}}\rb\le \delta / n.\]
\item There exists deterministic functions $\lambda_{-}(\delta), E_{+}(\delta), \Ep_{\infty}(\delta)$, such that for any $\delta \in (0, 1)$, the following event holds with probability at least $1 - \delta$:
\[\maxnorm{\Lambda - \Lambda^{*}}\le \lambda_{-}(\delta), \quad \|E U^{*}\|_{\op}\le E_{+}(\delta), \quad \mnorm{\Ep}\le \Ep_{\infty}(\delta).\]
\item There exists deterministic functions $\bp_{\infty}(\delta), \bp_{2}(\delta) > 0$, such that for any $\delta \in (0, 1)$, $k\in [n]$, $r'\le r$ and fixed matrix $W\in \R^{n\times r'}$,
  \begin{align*}
    \|\Ep_{k}^{T}W\|_{2} &\le \bp_{\infty}(\delta)\mnorm{W} + \bp_{2}(\delta)\|W\|_{\op}, \,\, \mbox{with probability at least } 1 - \delta / n.
  \end{align*}
\item $\gap^{*}\ge 4\lb\Theta(\delta)\sigmap(\delta) + L_{1}(\delta) + \lambda_{-}(\delta) + E_{+}(\delta)\rb$ where
  \begin{equation}
    \label{eq:etap}
\etap(\delta)  = \Ep_{\infty}(\delta) + \bp_{\infty}(\delta) + \bp_{2}(\delta), \quad \sigmap(\delta) = \{\bar{\kappa}^{*}L_{2}(\delta) + L_{3}(\delta) + 1\}\etap(\delta) + E_{+}(\delta).
  \end{equation}
\end{enumerate}

Unlike Theorem \ref{thm:generic_bound}, $AU^{*}(\Lambda^{*})^{-1}$ is no longer an approximation of $U$ after ``diagonal surgery''. In fact, the approximation becomes $U^{*} + \V$ where 
\begin{equation}
  \V_{k}^{T} = E_{k}^{T}U^{*}(\Lambda^{*} - \Sigma_{kk}I)^{-1}. 
\end{equation}
Note that when $\Sigma_{kk}\equiv 0$, 
\[U^{*} + \V = U^{*} + E U^{*}(\Lambda^{*})^{-1} = AU^{*}(\Lambda^{*})^{-1},\]
which recovers the case in Section \ref{subsec:main}.

\begin{theorem}\label{thm:generic_bound_trick}
Given any $\delta\in (0, 1)$. Let $\gap^{*}$ be defined in \eqref{eq:gap} and $\bar{\kappa}^{*}$ be defined in \eqref{eq:barkappa}. Then under assumptions \textbf{A}'0 - \textbf{A}'4,
\begin{align}
d_{\ttinf}(U, U^{*} + \V)&\le C\bigg\{\lb \frac{E_{+}^{2}(\delta)}{(\gap^{*})^{2}} + \frac{\Theta(\delta)\sigmap(\delta)}{\gap^{*}}\rb\lb\mnorm{U^{*}} + \frac{\Theta(\delta)\mnorm{E U^{*}}}{\lambda_{\min}^{*}}\rb\nonumber\\
& \qquad \quad  + \frac{\Theta(\delta)(\bp_{2}(\delta) + \mnorm{\Ap^{*}})E_{+}(\delta)}{\lambda_{\min}^{*}\gap^{*}}\bigg\},\nonumber
\end{align}
with probability at least $1 - B(r)\delta$, where $B(r)$ is defined in \eqref{eq:Br} and  $C$ is a universal constant (that can be chosen as $136$).
\end{theorem}

Similar to Theorem \ref{thm:generic_bound2}, we can derive a bound for $d_{\ttinf}(U, U^{*})$ using the triangle inequality. 
\begin{theorem}\label{thm:generic_bound2_trick}
Under the same setting of Theorem \ref{thm:generic_bound_trick}, 
\begin{align}
 d_{\ttinf}(U, U^{*})&\le C\bigg\{\frac{\Theta(\delta)}{\lambda_{\min}^{*}}\mnorm{E U^{*}} + \lb \frac{E_{+}^{2}(\delta)}{(\gap^{*})^{2}} + \frac{\Theta(\delta)\sigmap(\delta)}{\gap^{*}}\rb \mnorm{U^{*}}\nonumber\\
& \qquad\quad + \frac{\Theta(\delta)(\bp_{2}(\delta) + \mnorm{\Ap^{*}})E_{+}(\delta)}{\lambda_{\min}^{*}\gap^{*}}\bigg\},\nonumber
\end{align}
with probability at least $1 - B(r)\delta$, where $C$ is a universal constant (that can be chosen as $136$).
\end{theorem}

\section{$\ell_{\ttinf}$ Perturbation Theory for Binary Random Matrices}\label{sec:binarybound}

Throughout this section we will ignore the universal constant terms for notational convenience. In particular, we use the symbol $\preceq$ (resp. $\succeq$) to hide universal constants. We say a constant is \emph{universal} if it does not depend on any quantity in the problem (e.g. $A^{*}$, $n$, $\delta$, etc.). Specifically, we say $A \preceq B$ (resp. $A \succeq B$) with probability $1 - \delta$ for two variables $A$ and $B$, stochastic or deterministic, if and only if $A \le CB$ (resp. $A\ge CB$) with probability $1 - \delta$ for some universal constant $C$. All proofs in this section are relegated to Appendix \ref{app:binarybound}.

\subsection{Binary random matrices with independent entries}\label{subsec:binary_wigner}
In this subsection, we consider $A$ as a binary random matrix with independent entries with $A^{*}$ being its expectation, i.e.
\begin{equation}
  \label{eq:binaryA}
  A_{ij}^{*} = A_{ji}^{*} = p_{ij}, \quad A_{ij} = A_{ji}\sim \mathrm{Ber}(p_{ij}), \quad (A_{ij})_{1\le i \le j\le n}\mbox{ are independent}.
\end{equation}
Note that we allow $p_{ii} > 0$. Let
\begin{equation}
  \label{eq:pstar_Rdelta}
  p^{*} = \max_{ij}p_{ij}, \quad \bar{p}^{*} = \max_{i}\frac{1}{n}\sum_{j=1}^{n}p_{ij}, \quad R(\delta) = \log (n / \delta) + r.
\end{equation}

\begin{lemma}\label{lem:A3_binary}
  Under the setting \eqref{eq:binaryA}, given any $\alpha > 0$, assumption \textbf{A}3 is satisfied with
\[b_{\infty}(\delta) \preceq \frac{R(\delta)}{\alpha \log R(\delta)}, \quad b_{2}(\delta) \preceq \frac{\sqrt{p^{*}}R(\delta)^{(1 + \alpha) / 2}}{\alpha \log R(\delta)}.\]
\end{lemma}

\begin{lemma}\label{lem:Eop_binary}
  Under the setting \eqref{eq:binaryA}, assumption \textbf{A}2 is satisfied with
\[\lambda_{-}(\delta), E_{+}(\delta), \bar{E}_{+}(\delta), E_{\infty}(\delta)\preceq  \sqrt{n\bar{p}^{*}} + \sqrt{\log (n / \delta)},\]
\end{lemma}

\begin{lemma}\label{lem:mnormEU_binary}
Under the setting \eqref{eq:binaryA}, with probability $1 - \delta$,
\[\mnorm{EU^{*}}\preceq R(\delta)\mnorm{U^{*}} + \sqrt{R(\delta)p^{*}}.\]
\end{lemma}

It is then easy to derive the $\ell_{\ttinf}$ perturbation theory for binary random matrices with independent entries from Theorem \ref{thm:generic_bound} and \ref{thm:generic_bound2} based on Lemma \ref{lem:A3_binary} - \ref{lem:mnormEU_binary} and Proposition \ref{prop:A1} on $L_{1}(\delta), L_{2}(\delta), L_{3}(\delta)$. 

\begin{theorem}\label{thm:generic_binary}
Fix any $\delta \in (0, 1), \alpha \in (0, 1)$. Let 
\begin{equation}
  \label{eq:gdelta}
  g(\delta) = \sqrt{n\bar{p}^{*}} + \frac{R(\delta)}{\alpha \log R(\delta)},
\end{equation}
and assume that
\begin{equation}
  \label{eq:A4_binary}
  \gap^{*}\ge C\bar{\kappa}^{*}g(\delta),
\end{equation}
for some universal constant $C$ that is large enough. Then with probability at least $1 - (B(r) + 1)\delta$, 
  \begin{align*}
    &d_{\ttinf}(U, AU^{*}(\Lambda^{*})^{-1})\preceq \frac{1}{\gap^{*}}\left\{\bar{\kappa}^{*}g(\delta)\lb 1 + \frac{R(\delta)}{\lambda_{\min}^{*}}\rb\mnorm{U^{*}} + \frac{\sqrt{R(\delta)p^{*}}}{\lambda_{\min}^{*}}\lb \bar{\kappa}^{*}g(\delta) + \frac{\sqrt{n\bar{p}^{*}R(\delta)^{\alpha}}}{\alpha \log R(\delta)}\rb\right\}\\
& \,\, + \frac{\sqrt{n\bar{p}^{*}} + \sqrt{\log (n / \delta)}}{\gap^{*}}\min\left\{\frac{\mnorm{A^{*}}}{\lambda_{\min}^{*}}, \frac{\sqrt{\bar{\kappa}^{*}p^{*}}}{\sqrt{\lambda_{\min}^{*}}I(A^{*}\mbox{ is psd})}, \bar{\kappa}^{*}\mnorm{\bar{U}^{*}}\right\},
  \end{align*}
and 
  \begin{align*}
    &d_{\ttinf}(U, U^{*})\preceq \lb\frac{\bar{\kappa}^{*}g(\delta)}{\gap^{*}} + \frac{R(\delta)}{\lambda_{\min}^{*}}\rb\mnorm{U^{*}} + \frac{\sqrt{R(\delta)p^{*}}}{\lambda_{\min}^{*}}\lb 1 + \frac{\sqrt{n\bar{p}^{*}R(\delta)^{\alpha}}}{\alpha \gap^{*}\log R(\delta)}\rb\\
& \,\, + \frac{\sqrt{n\bar{p}^{*}} + \sqrt{\log (n / \delta)}}{\gap^{*}}\min\left\{\frac{\mnorm{A^{*}}}{\lambda_{\min}^{*}}, \frac{\sqrt{\bar{\kappa}^{*}p^{*}}}{\sqrt{\lambda_{\min}^{*}}I(A^{*}\mbox{ is psd})}, \bar{\kappa}^{*}\mnorm{\bar{U}^{*}}\right\},
  \end{align*}
\end{theorem}

We consider two special but realistic cases where the bounds in Theorem \ref{thm:generic_binary} can be significantly simplified. The first case corresponds to the full eigenspace recovery. This is a widely studied problem in various areas such as community detection. We will discuss a few applications in Section \ref{sec:exact_recovery}. The following Corollary considers a slightly more general case where $\bar{U}^{*}$ can differ from $U^{*}$. In this case the last terms in both inequalities can be removed.

\begin{corollary}\label{cor:full_recovery_binary}
Under the settings of Theorem \ref{thm:generic_binary}, if 
\[\mnorm{\bar{U}^{*}}\preceq \mnorm{U^{*}},\]
then with probability at least $1 - (B(r) + 1)\delta$, 
  \begin{align*}
    d_{\ttinf}(U, AU^{*}(\Lambda^{*})^{-1})&\preceq \frac{1}{\gap^{*}}\bigg\{\bar{\kappa}^{*}g(\delta)\lb 1 + \frac{R(\delta)}{\lambda_{\min}^{*}}\rb\mnorm{U^{*}} + \frac{\sqrt{R(\delta)p^{*}}}{\lambda_{\min}^{*}}\lb \bar{\kappa}^{*}g(\delta) + \frac{\sqrt{n\bar{p}^{*}R(\delta)^{\alpha}}}{\alpha \log R(\delta)}\rb\bigg\},
  \end{align*}
  \begin{align*}
    d_{\ttinf}(U, U^{*})&\preceq \lb\frac{\bar{\kappa}^{*}g(\delta)}{\gap^{*}} + \frac{R(\delta)}{\lambda_{\min}^{*}}\rb\mnorm{U^{*}} + \frac{\sqrt{R(\delta)p^{*}}}{\lambda_{\min}^{*}}\lb 1 + \frac{\sqrt{n\bar{p}^{*}R(\delta)^{\alpha}}}{\alpha \gap^{*}\log R(\delta)}\rb.
  \end{align*}
\end{corollary}

The second case we consider imposes a lower bound for $\lambda_{\min}^{*}$. 

\begin{corollary}\label{cor:typical_binary}
  Under the settings of Theorem \ref{thm:generic_binary}, if 
  \begin{equation}
    \label{eq:typical_minLambda_binary}
  \lambda_{\min}^{*}\succeq \frac{np^{*}}{\sqrt{n}\mnorm{U^{*}}},
  \end{equation}
then with probability at least $1 - (B(r) + 1)\delta$, 
  \begin{align*}
    d_{\ttinf}(U, AU^{*}(\Lambda^{*})^{-1})&\preceq \frac{\bar{\kappa}^{*}g(\delta)}{\gap^{*}}\lb 1 + \frac{R(\delta)}{\lambda_{\min}^{*}}\rb\mnorm{U^{*}},\\
    d_{\ttinf}(U, U^{*})&\preceq \lb\frac{\bar{\kappa}^{*}g(\delta)}{\gap^{*}} + \frac{R(\delta)}{\lambda_{\min}^{*}}\rb\mnorm{U^{*}} + \frac{\sqrt{R(\delta)p^{*}}}{\lambda_{\min}^{*}}.
  \end{align*}
\end{corollary}

\subsection{Unnormalized graph laplacian}\label{subsec:binary_laplacian}
Let $\A$ be a binary matrix with independent entries as defined in \eqref{eq:binaryA} and $\L$ be the unnormalized graph Laplacian of $\A$,
\begin{equation}
  \label{eq:laplacian}
  \L = \D - \A, \quad \mbox{where }\D = \diag(D_{11}, \ldots, D_{nn}), \quad \D_{ii} = \sum_{j=1}^{n}\A_{ij}.
\end{equation}
We use the notation $\L, \D, \A$ by convention. We assume $\A_{ii} = 0$ without loss of generality (because $\L$ does not depend on $\A_{ii}$). Throughout this section we will treat $\L$ as $A$ and $\L^{*}$ as $A^{*}$. Similar to binary matrices with independent entries, we derive the bound for quantities involved in Theorem \ref{thm:generic_bound}. In particular, we apply the diagonal surgery technique with $\Sigma = \diag(\L_{11}, \L_{22}, \ldots, \L_{nn})$. Let 
\[\Lp = \L - \Sigma, \quad E = \L - \E \L, \quad \Ep = \Lp - \E \Lp.\]
Apart from the notation in \eqref{eq:pstar_Rdelta}, we also define the following quantity
\begin{equation}
  \label{eq:Mdelta}
  M(\delta) = \sqrt{n\bar{p}^{*}\log (n / \delta)} + \log (n / \delta).
\end{equation}

\begin{lemma}\label{lem:A3_laplacian}
  Under the setting \eqref{eq:laplacian}, given any $\alpha > 0$, assumption \textbf{A}'3 is satisfied with
\[\bp_{\infty}(\delta) \preceq \frac{R(\delta)}{\alpha \log R(\delta)}, \quad \bp_{2}(\delta) \preceq \frac{\sqrt{p^{*}}R(\delta)^{(1 + \alpha) / 2}}{\alpha \log R(\delta)}.\]
\end{lemma}

\begin{lemma}\label{lem:Eop_laplacian}
  Under the setting \eqref{eq:laplacian}, assumption \textbf{A}'2 is satisfied with
\[\Ep_{\infty}(\delta)\preceq  \sqrt{n\bar{p}^{*}} + \sqrt{\log (n / \delta)},\quad E_{+}(\delta), \lambda_{-}(\delta)\preceq M(\delta).\]
\end{lemma}

\begin{lemma}\label{lem:mnormEU_laplacian}
Under the setting \eqref{eq:laplacian}, with probability $1 - \delta$,
\[\mnorm{E U^{*}}\preceq (M(\delta) + r)\mnorm{U^{*}} + \sqrt{R(\delta)p^{*}}.\]
\end{lemma}

\begin{lemma}\label{lem:L1L2L3_laplacian}
Let $\L$ be the unnormalized Laplacian of $\A$ where $\A$ is a binary random matrix that satisfies the condition in part (b) of Proposition \ref{prop:A1}. Then there exists $\L^{(1)}, \ldots, \L^{(n)}$ satisfying \textbf{A}'1 with
\[L_{1}(\delta) \preceq \sqrt{m}M(\delta) + m(\sqrt{n\bar{p}^{*}} + \sqrt{\log (n / \delta)}), \quad L_{2}(\delta) = m, \quad L_{3}(\delta) \preceq \frac{m\lb n\bar{p}^{*} + \log (n / \delta)\rb}{\lambda_{\min}^{*}}.\]
\end{lemma}
In particular, the setting considered in this subsection is a special case with $m = 1$. Putting the pieces together we deduce the following theorem.
\begin{theorem}\label{thm:generic_binary_laplacian}
Fix any $\delta \in (0, 1)$ and $\alpha > 0$. Let $R(\delta)$, $g(\delta)$ and $M(\delta)$ be defined as in \eqref{eq:pstar_Rdelta}, \eqref{eq:gdelta} and \eqref{eq:Mdelta}, respectively. Further let 
\[\bar{\kappa}' = \bar{\kappa}^{*} + n\bar{p}^{*} / \lambda_{\min}^{*},\]
and assume that
\begin{equation}
  \label{eq:A4_binary_laplacian}
  \gap^{*}\ge C\{\Theta(\delta)\bar{\kappa}' g(\delta) + (\Theta(\delta) + 1)M(\delta)\},
\end{equation}
for some universal constant $C$ that is large enough. Then with probability at least $1 - (B(r) + 1)\delta$, 
  \begin{align*}
    d_{\ttinf}(U, U^{*} + V)&\preceq \lb \frac{M(\delta)^{2}}{(\gap^{*})^{2}} + \frac{\Theta(\delta)(\bar{\kappa}'g(\delta) + M(\delta))}{\gap^{*}}\rb\left\{\lb 1 + \frac{\Theta(\delta) r}{\lambda_{\min}^{*}}\rb\mnorm{U^{*}} + \frac{\Theta(\delta)\sqrt{R(\delta)p^{*}}}{\lambda_{\min}^{*}}\right\}\\
& \quad + \frac{\Theta(\delta) M(\delta)\sqrt{p^{*}}}{\gap^{*}\lambda_{\min}^{*}}\lb \sqrt{n\bar{p}^{*}} + \frac{\sqrt{R(\delta)^{1 + \alpha}}}{\alpha \log R(\delta)}\rb,
\end{align*}
and 
  \begin{align*}
    d_{\ttinf}(U, U^{*})& \preceq \lb \frac{M(\delta)^{2}}{(\gap^{*})^{2}} + \frac{\Theta(\delta)(\bar{\kappa}'g(\delta) + M(\delta))}{\gap^{*}} + \frac{\Theta(\delta) r}{\lambda_{\min}^{*}}\rb\mnorm{U^{*}} + \frac{\Theta(\delta)\sqrt{Rp^{*}}}{\lambda_{\min}^{*}}\\
& \quad + \frac{\Theta(\delta) M(\delta)\sqrt{p^{*}}}{\gap^{*}\lambda_{\min}^{*}}\lb \sqrt{n\bar{p}^{*}} + \frac{\sqrt{R(\delta)^{1 + \alpha}}}{\alpha \log R(\delta)}\rb.
  \end{align*}
\end{theorem}
Note that 
\[\frac{|\Lambda^{*}_{jj}|}{|\Lambda^{*}_{jj} - \L_{kk}|}\le \frac{\Lambda^{*}_{jj}}{\max\{0, |\Lambda^{*}_{jj} - \L^{*}_{kk}| - |\L_{kk} - \L^{*}_{kk}|\}}.\]
It is easy to derive a concentration inequality for $\max_{k}|\L_{kk} - \L^{*}_{kk}|$. This suggests the following bound for $\Theta(\delta)$. 
\begin{lemma}\label{lem:Theta1}
  Let 
  \begin{equation}
    \label{eq:Thetastar}
    \Theta^{*} = \frac{\min_{j\in [s+1, s+r]}|\Lambda^{*}_{jj}|}{\min_{j\in [s+1, s+r], k\in [n]}|\Lambda^{*}_{jj} - \L^{*}_{kk}|}.
  \end{equation}
Then $\Theta(\delta)\le 5\Theta^{*}$ if  
\[\min_{j\in [s+1, s+r], k\in [n]}|\Lambda^{*}_{jj} - \L^{*}_{kk}|\ge 5M(\delta).\]
\end{lemma}



\section{Concentration of The Spectral Norm of Random Graphs}\label{sec:spectral_norm}
\subsection{Background}
\ER graph is the most fundamental object in random graph theory. The probability connection matrix $A^{*}$ of an \ER graph has 
\[A_{ii}^{*} = 0, \quad A_{ij}^{*} = p,\quad \forall i\not = j.\]
Concentration of the spectral norm or the extreme eigenvalues of \ER graphs has received considerable attention. Adapted from the proof of \cite{boucheron2013concentration} (Example 3.14), based on Efron-Stein inequality, we can show that
\begin{equation}\label{eq:var2}
  \Var(\|A\|_{\op})\le 2.
\end{equation}
\cite{alon2002concentration} proved the sub-gaussian behavior of $\|A\|_{\op}$ in the sense that 
\begin{equation}
  \label{eq:subgaussian32}
  \P(|\|A\|_{\op} - \E \|A\|_{\op}| > t)\le 2e^{-t^{2} / 32}, \quad \forall t > 0.
\end{equation}
See also \cite[][Example 8.7]{boucheron2013concentration} for a simplified proof. The key argument underlying the above results is the Efron-Stein-type inequalities that involve the leave-one-out behavior of $\|A\|_{\op}$ as a function of independent random variables $(A_{ij})_{i < j}$. Specifically, write $\|A\|_{\op}$ as $Z$ for convenience and denote by $Z_{ij}$ the operator norm of matrix $A^{(ij)}$, which equals to $A$ except that the $(i,j)$-th entry is replaced by an independent copy $A_{ij}'$ drawn from a Bernoulli distribution with parameter $p$. Then Efron-Stein inequality implies that
\begin{equation}
  \label{eq:EfronStein}
  \Var(Z)\le \E V_{+}, \quad \mbox{where }V_{+} = \E'\sum_{i < j}(Z - Z_{ij})_{+}^{2}.
\end{equation}
Here $(x)_{+}$ denotes $\max\{x, 0\}$ by convention and $\E'$ denotes the expectation over $(A_{ij}')_{i < j}$ (while conditioning on $(A_{ij})_{i < j}$). Using the variational representation of operator norm, we can rewrite $Z - Z_{ij}$ as follows:
\[Z - Z_{ij} = \sup_{u: \|u\| = 1}|u^{T}Au| - \sup_{u: \|u\| = 1}|u^{T}A^{(ij)}u|.\]
Let $u_{1}$ be the eigenvector of $A$ corresponding to its largest eigenvalue in absolute values, then 
\[\sup_{u: \|u\| = 1}|u^{T}Au| = \|A\|_{\op}, \quad \sup_{u: \|u\| = 1}|u^{T}A^{(ij)}u|\ge |u_{1}^{T}A^{(ij)}u_{1}| \ge \pm u_{1}^{T}A^{(ij)}u_{1}.\]
If $u_{1}^{T}Au_{1} \ge 0$, then 
\[Z - Z_{ij}\le u_{1}^{T}(A - A^{(ij)})u_{1}\le |u_{1}^{T}(A - A^{(ij)})u_{1}|;\]
if $u_{1}^{T}Au_{1} < 0$, then 
\[Z - Z_{ij}\le -u_{1}^{T}(A - A^{(ij)})u_{1}\le |u_{1}^{T}(A - A^{(ij)})u_{1}|.\]
Putting two pieces together, we conclude that 
\begin{equation}\label{eq:ZZij}
  (Z - Z_{ij})_{+} \le |u_{1}^{T}(A - A^{(ij)})u_{1}| = 2|u_{1i}||u_{1j}||A_{ij} - A_{ij}'|.
\end{equation}
As a result,
\begin{equation}
  \label{eq:V+bound}
  V_{+}\le 4\sum_{i < j}u_{1i}^{2}u_{1j}^{2}\E'(A_{ij} - A_{ij}')^{2} = 4\sum_{i < j}u_{1i}^{2}u_{1j}^{2}(p + (1 - 2p)A_{ij}).
\end{equation}
To prove \eqref{eq:var2}, one can simply use the naive bound that 
\[p + (1 - 2p)A_{ij}\le p + 1 - 2p \le 1.\]
Therefore, we obtain an almost sure bound for $V_{+}$:
\begin{equation}\label{eq:V+2}
  V_{+} \le 4\sum_{i < j}u_{1i}^{2}u_{1j}^{2} = 2\sum_{i\not = j}u_{1i}^{2}u_{1j}^{2}\le 2\lb\sum_{i=1}^{n}u_{1i}^{2}\rb^{2} = 2.
\end{equation}
Then Efron-Stein inequality implies that
\[\Var(\|A\|_{\op})\le \E V_{+}\le 2.\]

However, this bound is loose for small $p$ in which case the term $(p + (1 - 2p)A_{ij})$ is most likely equal to $p$ instead of the conservative bound $1$. However, the complicated dependence between $u_{1}$ and $A$ makes it hard to operationalize this intuition. Recently, an interesting work by \cite{lugosi2018concentration} showed that 
\begin{equation}
  \label{eq:lugosi_variance}
  \Var(\|A\|_{\op})\le c_{1}p, \quad \mbox{if }p\ge c_{2}(\log n)^{3} / n.
\end{equation}
for some universal constant $c_{1}, c_{2} > 0$. Furthermore, they proved the partial sub-gaussian behavior of $\|A\|_{\op}$ in the sense that
\begin{equation}
  \label{eq:lugosi_subgaussian}
    \P(|\|A\|_{\op} - \E \|A\|_{\op}| > \sqrt{p}t)\le c_{3}e^{-t^{2} / c_{4}}, \quad \forall t \le c_{5}\sqrt{np}\log (np) / (\log n\log (1/p)),
\end{equation}
for some universal constants $c_{3}, c_{4}, c_{5} > 0$ under the condition that $p\ge c_{2}(\log n)^{3} / n$. Under this regime, the upper bound for $t$ in \eqref{eq:lugosi_subgaussian} is diverging, implying that the tail probability holds for values much larger than $\sqrt{p}$. 

The idea is based on the phenomenon called "eigenvector delocalization" that all entries of $u_{1}$ are small \citep[e.g.][]{mitra2009entrywise, erdos2013spectral, lugosi2018concentration}. Intuitively, $u_{1}$ should be close to $u_{1}^{*}$, the eigenvector of $A^{*}$ corresponding to the largest eigenvalue in absolute values. A simple calculation shows that $u_{1}^{*} = \one_{n} / \sqrt{n}$. In particular, \cite{lugosi2018concentration} improved the previous results and proved that with high probability, $\|u_{1}\|_{\infty} \preceq 1 / \sqrt{n}$ if $p \succeq (\log n)^{3} / n$. This motivates the following bound for $V_{+}$ in \eqref{eq:V+bound}
\begin{equation}
  \label{eq:V+W}
  V_{+}\le \|u_{1}\|_{\infty}^{4}W, \quad \mbox{where } W = 4 \sum_{i < j}(p + (1 - 2p)A_{ij}).
\end{equation}
Using the concentration that $\|u\|_{\infty}\approx 1 / \sqrt{n}$ and $W \approx \E W = 4n(n - 1)p(1 - p)$ and by Efron-Stein inequality \eqref{eq:EfronStein}, they proved that
\[\Var(\|A\|_{\op})\le \E V_{+}\preceq \frac{n(n - 1)p(1 - p)}{n^{2}}\preceq p.\]
The bound of tail probability in \eqref{eq:lugosi_subgaussian} can be derived similarly using higher moments of $V_{+}$ and the moment concentration inequalities derived by \cite{boucheron2005moment}.

\cite{lugosi2018concentration} conjectured that the requirement $p\succeq (\log n)^{3} / n$ is artifical and the critical regime should be $p\succeq \log n / n$. In this section, we will close this gap by using our $\ell_{2\rightarrow \infty}$ bound in Section \ref{subsec:binary_wigner}. Furthermore, we will discuss the behavior of $\|A\|_{\op}$ when $p \preceq \log n / n$ but $p \succeq \log / (n\log\log n)$. Throughout the rest of this section we assume that
\begin{equation}
  \label{eq:C0}
\frac{1}{2} \ge  p \ge \frac{C_{0}\log n}{n\log\log n}
\end{equation}
for some universal constant $C_{0} > 0$. Similar to Section \ref{sec:binarybound}, we use the notation $\preceq$ and $\succeq$ to hide universal constants.

\subsection{Improved results for \ER graphs}

To start with, we state the moment concentration inequality by \cite{boucheron2005moment}.

\begin{proposition}[Theorem 15.6 and 15.7 of \cite{boucheron2013concentration}]\label{prop:boucheron}
  Let $Z = f(A)$ and $Z_{ij} = f(A^{(ij)})$. Further let $V_{+}$ be defined in \eqref{eq:EfronStein} and 
\[M = \max_{i < j}(Z - Z_{ij})_{+}.\]
Then for any $k \ge 2$,
\begin{align*}
  &\lb\E (Z - \E Z)_{+}^{k}\rb^{1/k}\le \sqrt{3k}\lb\E [V_{+}]^{k / 2}\rb^{1/k}, \\
  & \lb\E (Z - \E Z)_{-}^{k}\rb^{1/k}\le \sqrt{4.16k}\left\{\lb\E [V_{+}]^{k / 2}\rb^{1/k} + \sqrt{k}(\E [M^{k}])^{1/k}\right\}.
\end{align*}
As a result,
\[(\E |Z - \E Z|^{k})^{1/k} \le 4\sqrt{k}\lb\E [V_{+}]^{k / 2}\rb^{1/k} + 4k(\E [M^{k}])^{1/k}.\]
\end{proposition}
It is well-known \citep[e.g.][]{vershynin2010introduction} that $Z$ is sub-gaussian with parameter $\sigma^{2}$ iff $(\E |Z - \E Z|^{k})^{1/k}\le c\sqrt{k}\sigma$ for any $k \ge 2$ with some universal constant $c$. Suppose this is true for $k\le k_{0}$, then we can still derive the sub-gaussian behavior of $Z$. 

\begin{lemma}\label{lem:markov}
If for $2\le k\le k_{0}$,
\[(\E |Z - \E Z|^{k})^{1/k}\le \sqrt{k}\sigma,\]
for some $\sigma > 0$, then for $t\le \sqrt{k_{0}e}\sigma$,
\[\P(|Z - \E Z|\ge t)\le \exp\left\{1-\frac{t^{2}}{2e\sigma^{2}}\right\}.\]
\end{lemma}
Although Lemma \ref{lem:markov} does not hold for all $t$, it is desirable if $k_{0}$ is large because $\sigma$ is the scale of $|Z - \E Z|$. In the following, we will bound the higher order moments of $V_{+}$ and $M$, thereby bounding the tail probability of $\|A\|_{\op}$. Recalling from \eqref{eq:ZZij} and \eqref{eq:V+W} that 
\[M\le 2\|u_{1}\|_{\infty}^{2}, \quad V_{+}\le \|u_{1}\|_{\infty}^{4}W.\]
Both involve the moments of $\|u_{1}\|_{\infty}$. Our theory in Section \ref{sec:binarybound} gives the $(1 - \delta)$ upper tail bound for $\|u_{1} - u_{1}^{*}\|_{\infty}$ in the form of 
\begin{equation}\label{eq:tail_temp}
\P\lb \|u_{1} - u_{1}^{*}\|_{\infty} \ge A_{1} + A_{2}\sqrt{\log \lb\frac{1}{\delta}\rb} + A_{3}\log \frac{1}{\delta}\rb\le \delta.
\end{equation}
If \eqref{eq:tail_temp} holds for all $\delta$ then it directly yields a moment bound using Fubini's formula. However, $\delta$ cannot be arbitrarily small otherwise the condition \eqref{eq:A4_binary} in Theorem \ref{thm:generic_bound} may be violated so that \eqref{eq:tail_temp} may fail. So we first characterize the minimal $\delta$ such that \eqref{eq:tail_temp} remains valid.

\begin{lemma}\label{lem:deltastar}
Under condition \eqref{eq:C0}, \eqref{eq:A4_binary} holds for all $\delta > \delta^{*}$ where
\begin{equation}
  \label{eq:deltastar}
  \delta^{*} = \exp\left\{-\frac{np\log (np)}{2C}\right\},
\end{equation}
and $C$ is the constant in \eqref{eq:A4_binary} in Theorem \ref{thm:generic_binary}, if $C_{0}$ and $n$ are bounded below by a universal constant (which can be chosen as $\max\{64C^{2}, 12C, 16\}$).
\end{lemma}

Based on Lemma \ref{lem:deltastar} and Corollary \ref{cor:typical_binary}, we can derive the moment bound for $V_{+}$ and $M$.
\begin{lemma}\label{lem:V+moment}
Under the assumptions of Lemma \ref{lem:deltastar}, for any $k > 0$, 
\[ \lb\E V_{+}^{k / 2}\rb^{1 / k}\preceq \sqrt{p}\lb 1 + \frac{(k\vee \log n)^{2}}{(np)^{2}}\rb + \exp\left\{-\frac{np\log (np)}{2Ck}\right\},\]
and 
\[ \lb\E M^{k}\rb^{1 / k}\preceq \frac{1}{n}\lb 1 + \frac{(k\vee \log n)^{2}}{(np)^{2}}\rb + \exp\left\{-\frac{np\log (np)}{2Ck}\right\},\]
where $C$ is the constant in \eqref{eq:A4_binary} in Theorem \ref{thm:generic_binary}. 
\end{lemma}
\begin{proof}
In this case, note that $A^{*} = p\one_{n}\one_{n}^{T} - pI_{n\times n}$. Then 
\[p^{*} = p, \quad \lambda_{1}^{*} = (n - 1)p, \quad \lambda_{2}^{*} = \ldots = \lambda_{n}^{*} = -p,\,\, \mbox{ and }u_{1}^{*} = \one_{n} / \sqrt{n}.\] 
Let $\Lambda^{*} = \lambda_{1}^{*}$, then 
\[\gap^{*} = \lambda_{\min}^{*} = (n - 1)p, \quad \mnorm{U^{*}} = \|u_{1}^{*}\|_{\infty} = \frac{1}{\sqrt{n}}, \quad \bar{\kappa}^{*} = 1.\]
Thus, 
\[\lambda_{\min}^{*}\succeq \frac{np^{*}}{\sqrt{n}\mnorm{U^{*}}}.\]
By Lemma \ref{lem:deltastar} below, the condition \eqref{eq:A4_binary} is satisfied for all $\delta\ge \delta^{*}$. Then by Corollary \ref{cor:typical_binary} with $\alpha = 0.5$ and $\delta \ge \delta^{*}$, with probability $1 - \delta$,
\begin{align*}
  d_{\ttinf}(u_{1}, u_{1}^{*})&\preceq \lb\frac{\sqrt{np} + \log (n / \delta) / \log\log (n / \delta)}{np} + \frac{\log (n / \delta)}{np}\rb\frac{1}{\sqrt{n}} + \frac{\sqrt{\log (n / \delta) p}}{np}\\
& \preceq \frac{1}{\sqrt{n}}\lb \frac{\sqrt{np} + \log (n / \delta) / \log\log (n / \delta)}{np} + \frac{\log (n / \delta)}{np} + \sqrt{\frac{\log (n / \delta)}{np}}\rb\\
& \preceq \frac{1}{\sqrt{n}}\lb\frac{\log (n / \delta)}{np} + \sqrt{\frac{\log (n / \delta)}{np}}\rb
\end{align*}
By the triangle inequality, 
\[\|u_{1}\|_{\infty} \le d_{\ttinf}(u_{1}, u_{1}^{*}) + \|u_{1}^{*}\|_{\infty}.\]
Using the fact that $2\sqrt{y}\le y + 1$ and $\sqrt{n}\|u_{1}^{*}\|_{\infty} = 1$, there exists a universal constant $C_{1}$ such that for each $\delta \ge \delta^{*}$, 
\begin{equation}
  \label{eq:u1delta}
  \sqrt{n}\|u_{1}\|_{\infty}\le  C_{1}\lb 1 + \frac{\log (n / \delta)}{np}\rb\le  C_{1}\lb 1 + \frac{\log n}{np} + \frac{\log(1 / \delta)}{np}\rb,
\end{equation}
with probability $1 - \delta$. 
Denote by $B_{u}$ the RHS of \eqref{eq:u1delta} with $\delta = \delta^{*}$ and by $\event_{1}$ the event that $\sqrt{n}\|u_{1}\|_{\infty}\le B_{u}$. Then 
\[\P(\event_{1}) \ge 1- \delta^{*} = 1 - \exp\left\{-\frac{np\log (np)}{2C}\right\}.\]
On the other hand, note that
\[\E W = 4\sum_{i < j}(p + (1 - 2p)p) = 4n(n - 1)p(1 - p)\le 4n^{2}p,\]
and
\[W - \E W = 4(1 - 2p)\sum_{i < j}(A_{ij} - p).\]
By Lemma \ref{lem:bernoulli} with $w = \one_{n(n-1)/2}$ and $\delta = \delta' = \exp\{-e / 2\Omega\} = \exp\{-e n(n-1)p / 4\}$, 
\[W - \E W \le 8\log \lb\frac{1}{\delta}\rb = 2en (n - 1)p\le 6n^{2}p\]
with probability $1 - \delta$. Let $\event_{2}$ denote the event that $W \le 10n^{2}p$, then
\begin{equation}
  \label{eq:event2}
  \P(\event_{2})\ge 1- \exp\{-e n(n-1)p / 4\}\ge 1 - \exp\{-n^{2}p / 3\}.
\end{equation}
Let $\event = \event_{1}\cup \event_{2}$. Then 
\[\P(\event^{c})\le \exp\left\{-\frac{np\log (np)}{2C}\right\} + \exp\left\{-\frac{n^{2}p}{3}\right\}\le C_{2} \exp\left\{-\frac{np\log (np)}{2C}\right\},\]
where $C_{2}$ is a universal constant. By definition, on $\event$,
\[\|u_{1}\|_{\infty}^{4}W \le 10 p (\sqrt{n}\|u_{1}\|_{\infty})^{4}.\]
In addition, by \eqref{eq:V+2} we have
\[V_{+}I_{\event^{c}}\le 2I_{\event^{c}}.\]
Therefore, 
\begin{align}
  \lb\E V_{+}^{k / 2}\rb^{1 / k} &= \lb \E V_{+}^{k / 2}I_{\event} + \E V_{+}^{k / 2}I_{\event^{c}}\rb^{1 / k} \le \lb \E V_{+}^{k / 2}I_{\event}\rb^{1 / k} + \lb \E V_{+}^{k / 2}I_{\event^{c}}\rb^{1 / k}\nonumber\\
& \preceq \sqrt{p}\E \lb\left[(\sqrt{n}\|u_{1}\|_{\infty})^{2k}I_{\event}\right]\rb^{1/k} + \exp\left\{-\frac{np\log (np)}{2Ck}\right\}\label{eq:V+moment1}
\end{align}
By \eqref{eq:u1delta}, 
\[\P\lb Y\le \frac{\log (1 / \delta)}{np} \rb\le \delta,\quad  \forall \delta \ge \delta^{*}, \quad \mbox{where }Y = \left[\frac{\sqrt{n}\|u_{1}\|_{\infty}}{C_{1}} - \lb 1 + \frac{\log n}{np}\rb\right]_{+}.\]
Denote by $B_{Y}$ the upper bound with $\delta = \delta^{*}$. This can be written equivalently as
\[\P(Y \ge y) \le \exp\left\{-npy\right\}, \quad \forall y \le B_{Y}.\]
By definition, $YI_{\event} \le YI(Y \le B_{Y})$. Using Fubini's theorem, for any $k > 0$,
\begin{align*}
  \E Y^{2k}I(Y\le B_{Y}) &= \int_{0}^{B_{Y}}2k y^{2k-1}\P(Y \ge y)dy \le \int_{0}^{B_{Y}}2k y^{2k-1}\exp\{-npy\}dy\\
&  \le \int_{0}^{\infty}2k y^{2k-1}\exp\{-npy\}dy = \frac{\Gamma(2k + 1)}{(np)^{2k}}.
\end{align*}
By Stirling's formula, we have
\[\lb\E Y^{2k}I_{\event}\rb^{1/k}\le \lb\E Y^{2k}I(Y\le B_{Y})\rb^{1/k} \preceq \frac{k^{2}}{(np)^{2}}.\]
As a result,
\begin{align}
  &\lb\E (\sqrt{n}\|u_{1}\|_{\infty})^{2k}I_{\event}\rb^{1/k} \le \lb\E \lb C_{1}Y + 1 + \frac{\log n}{np}\rb^{2k}I_{\event}\rb^{1/k}\nonumber\\
 \le &\lb2^{2k}\E \lb C_{1}Y\rb^{2k}I_{\event} + 2^{2k}\lb 1 + \frac{\log n}{np}\rb^{2k}\rb^{1/k} \le 4C_{1}^{2}\lb\E Y^{2k}I_{\event}\rb^{1/k} + 4\lb 1 + \frac{\log n}{np}\rb^{2}\nonumber\\
\preceq  & \frac{k^{2}}{(np)^{2}} + 1 + \frac{(\log n)^{2}}{(np)^{2}}\preceq 1 + \frac{(k\vee \log n)^{2}}{(np)^{2}}.\label{eq:sqrtnu1infty}
\end{align}
The bound of $\lb\E [V_{+}^{k/2}]\rb^{1/k}$ is then proved by pluggin this into \eqref{eq:V+moment1}. 

~\\
Recalling \eqref{eq:ZZij}, we have
\[M\le 2\|u_{1}\|_{\infty}^{2}.\] 
It is easy to see that $M \le 2$. Similar to \eqref{eq:V+moment1}, 
\begin{align}
  \lb \E M^{k}\rb^{1/k} &\le \lb \E M^{k}I_{\event} + \E M^{k / 2}I_{\event^{c}}\rb^{1 / k} \le \lb \E M^{k}I_{\event}\rb^{1 / k} + \lb \E M^{k}I_{\event^{c}}\rb^{1 / k}  \nonumber\\
& \preceq \frac{1}{n}\E \lb\left[(\sqrt{n}\|u_{1}\|_{\infty})^{2k}I_{\event}\right]\rb^{1/k} + \exp\left\{-\frac{np\log (np)}{2Ck}\right\}\label{eq:V+moment2}.
\end{align}
The bound of $\lb \E M^{k}\rb^{1/k}$ is then completed by \eqref{eq:sqrtnu1infty}.
\end{proof}

\begin{theorem}\label{thm:varop}
Under condition \eqref{eq:C0} with $C_{0}$ sufficiently large, 
\[\Var(\|A\|_{\op})\preceq p\lb 1 + \frac{(\log n)^{4}}{(np)^{4}}\rb.\]  
In particular, if $p\succeq \log n / n$, then 
\[\Var(\|A\|_{\op})\preceq p.\]
\end{theorem}

\begin{remark}
  This closed the gap conjectured by \cite{lugosi2018concentration}. Moreover, our result shows that when $np \sim \log n / \log\log n$, 
\[\Var(\|A\|_{\op})\preceq p (\log\log n)^{4}.\]
\end{remark}
\begin{proof}
  By Efron-Stein inequality \eqref{eq:EfronStein}, 
\[\Var(\|A\|_{\op})\le \E V_{+}.\]
For sufficiently large $n$, by Lemma \ref{lem:V+moment} with $k = 2$, 
\[\E V_{+}\le p\lb 1 + \frac{(\log n)^{4}}{(np)^{4}}\rb + \exp\left\{-\frac{np\log (np)}{4C}\right\}.\]
Since $np\ge C_{0}\log n / \log \log n$, when $C_{0}$ is sufficiently large, 
\[\exp\left\{-\frac{np\log (np)}{4C}\right\}\preceq n^{-2}\preceq p.\]
For small $n$, $\Var(\|A\|)_{\op}\le 2$ as shown in \eqref{eq:var2}. In summary, 
\[\Var(\|A\|_{\op})\preceq p\lb 1 + \frac{(\log n)^{4}}{(np)^{4}}\rb.\]
\end{proof}

\begin{lemma}\label{lem:tailprob}
Under condition \eqref{eq:C0}, 
\[\lb\E |Z - \E Z|^{k}\rb^{1/k}\le C'\sqrt{kp}\lb 1 + \frac{(\log n)^{2}}{(np)^{2}}\rb,\quad \forall k\le k_{0},\]
where $C'$ is a universal constant and 
\begin{equation}
  \label{eq:k0}
  k_{0} = \frac{np}{2C\vee 1}\min\left\{1, \frac{\log (np)}{\log (1/p)}\right\}.
\end{equation}
\end{lemma}
\begin{proof}
Write $\|A\|_{\op}$ as $Z$ when no confusion can arise. By Proposition \ref{prop:boucheron} and Lemma \ref{lem:V+moment}, 
\[(\E |Z - \E Z|^{k})^{1/k}\preceq \lb\sqrt{kp} + \frac{k}{n}\rb \lb 1 + \frac{(k\vee \log n)^{2}}{(np)^{2}}\rb + k\exp\left\{-\frac{np\log (np)}{2Ck}\right\}.\]
For $k\le k_{0}$, then the first term has the order of $\sqrt{k p}$ because under \eqref{eq:C0},
\[\frac{k}{n} = \sqrt{kp}\frac{1}{\sqrt{n^{2}p}}\preceq \sqrt{kp}, \quad 1 + \frac{(k\vee \log n)^{2}}{(np)^{2}} \preceq 1 + \frac{(\log n)^{2}}{(np)^{2}}\]
To bound the second term is dominated by $\sqrt{kp}$ in order, it is left to show that
\begin{equation}\label{eq:tailAop1}
  \sqrt{\frac{k}{p}}\exp\left\{-\frac{np\log (np)}{2Ck}\right\} = \exp\left\{-\frac{np\log (np)}{2Ck} + \frac{1}{2}\log k + \frac{1}{2}\log \lb\frac{1}{p}\rb\right\}\preceq 1.
\end{equation}
Since $k\le k_{0}\le np / (2C \vee 1)$, 
\[k\log k\le \frac{np \log (np)}{2C} \Longrightarrow \log k\le \frac{np\log (np)}{2Ck}.\]
Further, since $k\le k_{0}\le np \log (np) / 2C\log (1 / p)$, 
\[\log \lb\frac{1}{p}\rb\le \frac{np\log (np)}{2Ck}.\]
Thus, 
\[-\frac{np\log (np)}{2Ck} + \frac{1}{2}\log k + \frac{1}{2}\log \lb\frac{1}{p}\rb\le 0,\]
and \eqref{eq:tailAop1} is proved. The proof is then completed. 
\end{proof}

Together with Lemma \ref{lem:markov}, Lemma \ref{lem:tailprob} implies that partial sub-gaussian behavior of $\|A\|_{\op}$. 
\begin{theorem}\label{thm:tailprob}
Under condition \eqref{eq:C0}, there exists universal constants $C_{1}, C_{2} > 0$ such that 
\[\P(|Z - \E Z|\ge t)\le \exp\left\{1 - \frac{t^{2}}{C_{1}\sigma^{2}}\right\}, \quad \forall t\le C_{2}\sigma \sqrt{np} \min\left\{1, \sqrt{\frac{\log (np)}{\log (1 / p)}}\right\},\]
where 
\[\sigma = \sqrt{p}\lb 1 + \frac{(\log n)^{2}}{(np)^{2}}\rb.\]
\end{theorem}
\begin{remark}
  When $p\succeq \log n / n$, $\sigma \preceq \sqrt{p}$. This proves the conjecture of \cite{lugosi2018concentration}. On the other hand, it is worth comparing the range of $t$ with the sub-gaussian behaviors. In \cite{lugosi2018concentration} (equation (2.1)), the multiplicative factor of $\sqrt{p}$ in the upper bound of $t$ is $\sqrt{np}\log (np) / \log n \log (1 / p)$ while that of ours is $np\min\{1, \log (np)/\log(1 / p)\}$, ignoring the constants. When $np = \mathrm{PolyLog}(n)$, ours reduces to $np\log (np) / \log (1/p)$ which is $\sqrt{np}\log n$ larger than the one in \cite{lugosi2018concentration}.
\end{remark}

\subsection{Extension to inhomogeneous graphs}
The results can be directly extended to general inhomogenous graphs because the proof carries over if both $\|u_{1} - u_{1}^{*}\|_{\infty}$ and $\|u_{1}^{*}\|_{\infty}$ have small moments. To be specific we consider a random graph with an adjacency matrix under the setting of Section \ref{subsec:binary_wigner}.  Apart from $p^{*}$ and $\bar{p}^{*}$ defined in \eqref{eq:pstar_Rdelta}, we further defined 
\[\bar{p} = \frac{1}{n(n - 1)}\sum_{i\not =j}p_{ij}.\]
Note that $\bar{p}\le \bar{p}^{*}\le p^{*}$ and in many applications they differ in small multiplicative factors. We distinguish them to cover the cases where many $p_{ij}$'s are tiny. Then we can apply Theorem \ref{thm:generic_binary} to achieve the task. Here for convenience we assume that $|\lambda_{1}^{*}| \succeq np^{*} / \sqrt{n}\|u_{1}^{*}\|_{\infty}$ so that the simplified bound in Corollary \ref{cor:typical_binary} can be applied. Further we consider the case with $p^{*} \le 1 / 2$ for convenience. The proof is more technical but qualitatively the same as the results for \ER graphs so we present it in Appendix \ref{subapp:other_spectral_norm}.
\begin{theorem}\label{thm:inhomogeneous}
 Let $|\lambda_{1}^{*}| \ge |\lambda_{2}^{*}|\ge \ldots \ge |\lambda_{n}^{*}|$ be eigenvalues of $A^{*}$. Suppose $|\lambda_{1}^{*}| > |\lambda_{2}^{*}|$ and let $u_{1}^{*}$ be the eigenvector corresponding to $\lambda_{1}^{*}$. Let 
\[\zeta = \sqrt{n}\|u_{1}^{*}\|_{\infty}.\]
Assume that 
  \begin{equation}
    \label{eq:C0_general}
    |\lambda_{1}^{*}|\ge C_{0}\frac{np^{*}}{\zeta}, \quad \gap^{*} = \min\left\{|\lambda_{1}^{*}|, \min_{j}|\lambda_{1}^{*} - \lambda_{j}^{*}|\right\} \ge C_{0}\lb \sqrt{n\bar{p}^{*}} + \frac{\log n}{\log \log n}\rb, \quad n^{2}\bar{p}\ge C_{0}.
  \end{equation}
Then 
\[\Var(\|A\|_{\op})\preceq \bar{p}\left\{1 + \lb\frac{\sqrt{n\bar{p}^{*}} + \log n}{\gap^{*}}\rb^{4}\right\}\zeta^{4} + \exp\left\{-\frac{\gap^{*}\log \gap^{*}\wedge n^{2}\bar{p}}{2C\vee 3}\right\}\]
where $C$ is the universal constant in \eqref{eq:A4_binary} in Theorem \ref{thm:generic_binary}. In addition, there exists universal constants $C_{1}, C_{2} > 0$ such that 
\[\P(|Z - \E Z|\ge t)\le \exp\left\{1 - \frac{t^{2}}{C_{1}\sigma^{2}}\right\}, \quad \forall t\le C_{2}\sigma \min\left\{\gap^{*}, \frac{\gap^{*}\log \gap^{*}}{\log (1 / \bar{p}\zeta^{2})}, \frac{n^{2}\bar{p}}{\log (n^{2}\bar{p})}, \frac{n^{2}\bar{p}}{\log (1 / \bar{p}\zeta^{2})}\right\}^{1/2},\]
where 
\[\sigma = \sqrt{\bar{p}}\left\{1 + \lb\frac{\sqrt{n\bar{p}^{*}} + \log n}{\gap^{*}}\rb^{2}\right\}\zeta^{2}.\]
\end{theorem}

\section{Exact Recovery of Spectral Clustering}\label{sec:exact_recovery}
\subsection{Background}\label{subsec:exact_recovery_background}
Stochastic block model (SBM) is a popular model to analyze community detection algorithms. The probability matrix of an SBM with $K$ blocks is given by $A_{ij}^{*} = B_{c_{i}c_{j}}$ for some matrix $B\in [0, 1]^{K\times K}$ where $c_{i}$ denotes the cluster label of $i$. Equivalently, 
\[A^{*} = \td{A}^{*} - \diag(\td{A}^{*}), \quad \mbox{where }\td{A}^{*} = ZBZ^{T},\]
where $Z\in \R^{n\times K}$ denotes the membership matrix whose $i$-th row $Z_{i} = e_{c_{i}}$, the $c_{i}$-th canonical basis of $R^{K}$. The goal is to recover the cluster labels $c = (c_{1}, \ldots, c_{n})$, up to label permutation. Let $\hat{c}$ denote the estimated cluster labels via an algorithm. We say that the algorithm achieves exact recovery if 
\[\P\lb\exists \mbox{ permutation }\pi, \,\, \hat{c}_{\pi(i)} = c_{i}\,\, \forall i \in [n]\rb = 1 - o(1).\]
The problem has been widely studied and can be solved using different algorithms \citep[e.g.][]{bui1987graph, boppana1987eigenvalues, dyer1989solution, snijders1997estimation, jerrum1998metropolis, condon2001algorithms, carson2001hill, mcsherry2001spectral, giesen2005reconstructing, shamir2007improved, bickel2009nonparametric, coja2010graph, rohe2011spectral, oymak2011finding, balakrishnan2011noise, choi2012stochastic, mossel2014consistency, ames2014guaranteed, chen2014improved, massoulie2014community, yun2014accurate, abbe2015community, abbe2015exact, chin2015stochastic, hajek2016achieving, guedon2016community, yun2016optimal, agarwal2017multisection, gao2017achieving, amini2018semidefinite, bandeira2018random, chen2018convexified, vu2018simple, fei2018exponential, li2018convex, li2018hierarchical, su2019strong}; see \cite{abbe2017community} for a nice review of this topic.

Spectral clustering algorithms are appealing due to the straightforward implementation and computational efficiency compared to other algorithms. In this section we consider the standard spectral clustering algorithm \citep[e.g.][]{von2007tutorial}, which embeds each observation into the subspace spanned by $K$ eigenvectors of some operators and applies $K$-means or $K$-medians algorithm on embedded vectors. Specifically, we consider the following procedure:
\begin{enumerate}[Step 1]
\item Compute the $K$ eigenvectors of the adjacency matrix corresponding to the $K$ largest eigenvalues in absolute values, or the $K$ eigenvectors of the unnormalized Laplacian corresponding to the $K$ smallest eigenvalues, denoted by $U$;
\item Perform $K$-medians algorithm on $U$ to get the estimates of cluster membership.
\end{enumerate}
The intuition underlying the algorithm is that $U$ should approximate $U^{*}$, the eigenvector matrix of the expectation of the adjacency matrix or the Laplacian, which identifies the cluster labels using $K$-medians algorithm. Consider the adjacency matrix as an example. Let $n_{i}$ denote the number of units in cluster $i$ and without loss of generality assume that 
\[Z =
  \begin{bmatrix}
    \one_{n_{1}} & 0 & \cdots & 0\\
    0 & \one_{n_{2}} & \cdots & 0\\
    \vdots & \vdots & \ddots & \vdots\\
    0 & 0 & \cdots & \one_{n_{K}}
  \end{bmatrix}.
\]
Let $M = \diag(\sqrt{n_{1}}, \ldots, \sqrt{n_{K}})$ and $Q = Z M^{-1}$. Then $Q^{T}Q = I$ and 
\begin{equation}
  \label{eq:AstarQM}
  \td{A}^{*} = Q(M B M)Q^{T}.
\end{equation}
Let $V\Lambda V^{T}$ be the spectral decomposition of $M B M$. Then $Q V B (Q V)^{T}$ is the spectral decomposition of $\td{A}^{*}$ since $QV$ is an orthogonal matrix. As a result, the eigenvector matrix of $\td{A}^{*}$ is $\td{U}^{*} = QV$. By definition,
\[\td{U}^{*} = 
  \begin{bmatrix}
    \frac{\one_{n_{1}}V_{1}^{T}}{\sqrt{n_{1}}}\\
    \frac{\one_{n_{2}}V_{2}^{T}}{\sqrt{n_{2}}}\\
    \vdots\\
    \frac{\one_{n_{K}}V_{K}^{T}}{\sqrt{n_{K}}}
  \end{bmatrix},
\]
where $V_{i}^{T}$ is the $i$-th row of $V$. It is clear that $\td{U}_{i}^{*} = \td{U}_{i'}^{*}$ iff $i$ and $i'$ belong to the sam cluster. Thus $K$-medians can perfectly identify the clusters using $\td{U}^{*}$. Since $A^{*}$ approximates $\td{A}^{*}$, $\td{U}^{*} \approx U^{*} \approx U$. This intuitively justifies the spectral clustering algorithm.

Early works investigated the weak recovery of spectral clustering, meaning that the misclassification error is vanishing with high probability \citep[e.g.][]{rohe2011spectral, lei2015consistency, joseph2016impact}. This is weaker than exact recovery which requires the misclassification error to be zero with high probability. The exact recovery was proved for dense graphs of which the average degree is polynomial in the graph size \citep[e.g.][]{mcsherry2001spectral, balakrishnan2011noise}. Perhaps surprisingly, for sparse graphs of which the average degree is of order $\log n$, the exact recovery was not proved until \cite{abbe2017entrywise}. In particular, they proved that the spectral clustering is information theoretically optimal for two-block SBMs with equal block sizes: with within-block probability $p = a\log / n$ and between-block probability $q = b \log n / n$ where $a > b$ are two constants, spectral clustering on the adjacency matrix achieves the exact recovery iff $\sqrt{a} - \sqrt{b} > \sqrt{2}$. Later \cite{su2019strong} extended the result to spectral clustering on normalized Laplacians for general $K$-block SBMs when the average degree is of order $\log n$. 

In this section, we derive the exact recovery of spectral clustering on adjacency matrices or unnormalized Laplacians for general SBMs with fixed or growing $K$ through a simple analysis based on our $\ell_{\ttinf}$ bounds in Section \ref{subsec:binary_wigner}. The results can be directly extended to SBMs with dependent entries using the results from Section \ref{subsec:mdependence} but we leave them to interested readers as the derivation is almost identical. Before delving into the analysis, we prove a useful lemma that connects the $K$-medians algorithm and the $\ell_{\ttinf}$ perturbation bounds. In particular, the $K$-medians algorithm applied on $U$ will return cluster labels as 
\begin{equation}
  \label{eq:hatc}
  \hat{c}_{i} = \argmin_{r}\|U_{i} - \hat{v}_{r}\|_{2},
\end{equation}
where $U_{i}^{T}$ is the $i$-th row of $U$ and 
\begin{equation}
  \label{eq:hatv}
  (\hat{v}_{1}, \ldots, \hat{v}_{K}) = \argmin_{(v_{1}, \ldots, v_{K})} \frac{1}{n}\sum_{i=1}^{n}\min_{r\in [K]}\|U_{i} - v_{r}\|_{2}
\end{equation}

\begin{lemma}\label{lem:Kmedian}
  Let $U, \td{U}^{*}\in \R^{n\times k}$ be two matrices and $\C_{1}, \ldots, \C_{K}$ be a partition of $[n]$ with $|\C_{s}| = n\pi_{s}$. Assume that 
\[\td{U}_{i}^{*} = v_{s}^{*}, \quad \forall i\in \C_{s},\]
and $v_{s}^{*}\not = v_{s'}^{*}$ for any pair $s \not = s'$. Then the $K$-medians algorithm exactly recovers $\C_{1}, \ldots, \C_{K}$ if 
\[d_{\ttinf}(U, \td{U}^{*})\le \frac{\min \pi_{r}}{6}\min_{s\not=s'}\|v_{s}^{*} - v_{s'}^{*}\|_{2}.\]
\end{lemma}

\subsection{Exact recovery of SBM with fixed $K$}
In this subsection, we consider a standard asymptotic setting where $K$ is held fixed, 
\[B = \rho_{n}B_{0}, \quad n_{r} / n\rightarrow \pi_{r},\]
for some rate function $\rho_{n}$, fixed matrix $B_{0}$ and fixed numbers $\pi_{1}, \ldots, \pi_{K} > 0$ that sum up to $1$. For this model, it is known that the information theoretic lower bound for $\rho_{n}$ is $\log n / n$, in the sense that if $\rho_{n} / (\log n / n)\rightarrow 0$ then no algorithm can achieve exact recovery \citep[e.g.][]{abbe2015exact, abbe2017community}. In this subsection we will show that spectral clustering can achieve exact recovery if $\rho_{n} > c\log n / n$ for a sufficiently large constant $c$ via either the adjacency matrix or the unnormalized Laplacian separately. 

\begin{theorem}[\textbf{exact recovery via adjacency matrix}]\label{thm:adjacency_spectral_clustering}
Assume that $B_{0}$ is full rank. Fix any $q > 0$. Then there exists constants $c$ and $n_{0}$ that only depends on $B_{0}, \pi_{r}$'s and $q$ such that if $n\ge n_{0}$ and
\begin{equation}
  \label{eq:adjacency_cond}
  n\rho_{n}\ge c\log n,
\end{equation}
then spectral clustering using the adjacency matrix achieves exact recovery with probability at least $1 - n^{-q}$.
\end{theorem}
\begin{proof}
 Let $\Lambda^{*}\in \R^{K\times K}$ denotes the diagonal matrix of all non-zero eigenvalues of $A^{*}$ and $U^{*}\in \R^{n\times K}$ be the corresponding eigenvector matrix. Let $R_{0} = \diag(\sqrt{\pi_{1}}, \ldots \sqrt{\pi_{K}})$. Recalling the decomposition \eqref{eq:AstarQM} and let $R = M / \sqrt{n}$, we have
\[\td{A}^{*} = n\rho_{n}Q(RB_{0}R)Q^{T}, \quad \td{U}^{*} = QV\]
where $V\in \R^{K\times  K}$ is the matrix formed by all eigenvectors of $RB_{0}R$. By definition of $Q$, $\td{U}_{i}^{*} = \frac{V_{s}}{\sqrt{n_{s}}}$ for any $i\in \C_{s}$ where $V_{s}^{T}$ is the $s$-th row of $V$. Let $v_{s}^{*} = \td{U}_{i}^{*}$ for $i\in \C_{s}$. Using the fact that $V$ is an orthogonal matrix, we have
\begin{align}
  \|v_{s}^{*} - v_{s'}^{*}\|_{2} &= \sqrt{\|v_{s}^{*}\|_{2}^{*} + \|v_{s'}^{*}\|_{2}^{2} - 2\langle v_{s}^{*}, v_{s'}^{*}\rangle} = \sqrt{\frac{1}{n_{s}} + \frac{1}{n_{s'}}} \ge \frac{1}{\min_{s\in [K]}\sqrt{\pi_{s}}}\frac{1}{\sqrt{n}}, \nonumber
\end{align}
By Lemma \ref{lem:Kmedian}, it is left to prove that 
\begin{equation}
  \label{eq:adjacency_goal1}
  d_{\ttinf}(U, \td{U}^{*}) \le \frac{\min_{s\in [K]}\sqrt{\pi_{s}}}{6\sqrt{n}} \triangleq \frac{2c_{1}}{\sqrt{n}}.
\end{equation}
By the triangle inequality, 
\[d_{\ttinf}(U, \td{U}^{*})\le d_{\ttinf}(U, U^{*}) + d_{\ttinf}(U^{*}, \td{U}^{*}).\]
By definition,
\[d_{\ttinf}(U^{*}, \td{U}^{*}) = \inf_{O\in \O^{K}}\mnorm{U^{*}O - \td{U}^{*}}\le \inf_{O\in \O^{K}}\|U^{*}O - \td{U}^{*}\|_{\op}.\]
By Davis-Kahan Theorem \citep[][Theorem 2]{yu2014useful}, 
\[\inf_{O\in \O^{K}}\|U^{*}O - \td{U}^{*}\|_{\op}\le \frac{\sqrt{8K}\|A^{*} - \td{A}^{*}\|_{\op}}{\td{\lambda}_{K}^{*} - \td{\lambda}_{K+1}^{*}}\le \frac{\sqrt{8K}\rho_{n}}{\td{\lambda}_{K}^{*}},\]
where we use the fact that $\td{A}^{*} - A^{*} = \diag(\td{A}^{*})$ and $\rank(\td{A}^{*}) = K$. Since 
\[\td{\lambda}_{K}^{*} = n\rho_{n}\lambda_{\min}(R B_{0}R)\ge n\rho_{n}\lb\min_{r\in [K]}\sqrt{\frac{n_{r}}{n}}\rb\lambda_{\min}(B_{0}).\]
Thus, when $n$ is sufficiently large, 
\begin{equation}\label{eq:dttinf_U*_tdU*}
  d_{\ttinf}(U^{*}, \td{U}^{*})\le \inf_{O\in \O^{K}}\|U^{*}O - \td{U}^{*}\|_{\op}\le \frac{\sqrt{16K}}{\lambda_{\min}(B_{0})\min_{r\in [K]}\sqrt{\pi_{r}}}\frac{1}{n}\le \frac{c_{1}}{\sqrt{n}}.
\end{equation}
Combined with \eqref{eq:adjacency_goal1}, it is left to prove that
\begin{equation}
  \label{eq:adjacency_goal}
   d_{\ttinf}(U, U^{*}) \le \frac{c_{1}}{\sqrt{n}}.
\end{equation}

Throughout the rest of the proof we treat $B_{0}$, $q$ and $\pi_{r}$'s as constants. Note that $R\rightarrow R_{0}$ and hence
\[RB_{0}R\rightarrow R_{0}B_{0}R_{0}.\]
In addition, since $B_{0}$ is full-rank and $\pi_{r} > 0$, $R_{0}B_{0}R_{0}$ is full rank. Then there exists $n_{0}$ that only depends on $B$ and $R_{0}$ such that whenever $n\ge n_{0}$, 
\begin{equation*}
\lambda_{\min}(\td{\Lambda}^{*}) > \frac{2n\rho_{n}}{3}\lambda_{\min}(R_{0}B_{0}R_{0}).
\end{equation*}
By Weyl's inequality, 
\[|\lambda_{\min}(\td{\Lambda}^{*}) - \lambda_{\min}^{*}|\le \|\td{A}^{*} - A^{*}\|_{\op}\le \rho_{n}.\]
Thus for sufficiently large $n$,
\begin{equation}
  \label{eq:adjacency2}
\lambda_{\min}^{*} > \frac{n\rho_{n}}{2}\lambda_{\min}(R_{0}B_{0}R_{0}).
\end{equation}
Let $\td{\gap}^{*}$ be the counterpart of $\gap^{*}$ for $\td{A}^{*}$. Then by Weyl's inequality
\[|\gap^{*} - \td{\gap}^{*}|\le 2\rho_{n}.\]
Since $\td{\gap}^{*} = \lambda_{\min}(\td{\Lambda}^{*})$, for sufficiently large $n$, 
\begin{equation}
  \label{eq:adjacency3}
  \gap^{*} > \frac{n\rho_{n}}{2}\lambda_{\min}(R_{0}B_{0}R_{0}).
\end{equation}
On the other hand, it is easy to see that
\begin{equation*}
  \mnorm{\td{U}^{*}} = \frac{1}{\min_{s\in [K]}\sqrt{n_{s}}} = \frac{1}{\min_{s\in [K]}\sqrt{\pi_{s}}}\frac{1}{\sqrt{n}}.
\end{equation*}
By \eqref{eq:dttinf_U*_tdU*}, 
\[|\mnorm{\td{U}^{*}} - \mnorm{U^{*}}|\le d_{\ttinf}(U^{*}, \td{U}^{*})\le \frac{c_{1}}{\sqrt{n}}.\]
Thus, 
\begin{equation}
    \label{eq:adjacency1}
    \mnorm{U^{*}} \preceq \frac{1}{\sqrt{n}}.
\end{equation}

Set $\delta = n^{-q}$ and $\alpha = 0.5$ in Corollary \ref{cor:typical_binary}. By \eqref{eq:adjacency2} and \eqref{eq:adjacency1}, 
\[\lambda_{\min}^{*}\succeq \frac{np^{*}}{\sqrt{n}\mnorm{U^{*}}}.\]
Moreover, $\bar{\kappa}^{*}\le 2K\preceq 1$, $p^{*}\preceq \rho_{n}$ and 
\[R(\delta)\preceq \log n, \quad g(\delta)\preceq \sqrt{n\rho_{n}} + \frac{\log n}{\log \log n}\]
By \eqref{eq:adjacency3}, for sufficiently large $n$ and $c$ in the condition \eqref{eq:adjacency_cond}, 
\[\gap^{*} > \frac{n\rho_{n}}{2}\lambda_{\min}(R_{0}B_{0}R_{0}) \ge C\bar{\kappa}^{*}g(\delta)\]
where $C$ is the universal constant in Theorem \ref{thm:generic_binary}. Thus, the conditions of Corollary \ref{cor:typical_binary} are satisfied. As a result, 
\begin{align*}
      d_{\ttinf}(U, U^{*})&\preceq \lb\frac{\sqrt{n\rho_{n}} + \log n / \log \log n}{n\rho_{n}} + \frac{\log n}{n\rho_{n}}\rb\frac{1}{\sqrt{n}} + \frac{\sqrt{(\log n) \rho_{n}}}{n\rho_{n}} \preceq \sqrt{\frac{\log n}{n\rho_{n}}}\frac{1}{\sqrt{n}}.
\end{align*}
Equivalently, there exists a constant $c_{2}$ that only depends on $B_{0}$, $q$ and $\pi_{r}$'s such that
\[d_{\ttinf}(U, U^{*})\le \sqrt{\frac{\log n}{n\rho_{n}}}\frac{c_{2}}{\sqrt{n}}.\]
By condition \eqref{eq:adjacency_cond},
\[d_{\ttinf}(U, U^{*})\le \frac{c_{2}}{\sqrt{c}}\frac{1}{\sqrt{n}}.\]
Therefore, \eqref{eq:adjacency_goal} follows if $c > c_{2}^{2} / c_{1}^{2}$. The proof is then completed.
\end{proof}

\begin{theorem}[\textbf{exact recovery via unnormalized Laplacian}]\label{thm:laplacian_spectral_clustering}
Let $R_{0} = \diag(\sqrt{\pi_{1}}, \ldots \sqrt{\pi_{K}})$, $\td{d}_{0} = B_{0}R_{0}^{2}\one_{K}$ and $\td{\D}_{0} = \diag(\td{d}_{01}, \ldots, \td{d}_{0K})$. Further let 
\[\td{\L}_{0} = \td{\D}_{0} - R_{0}B_{0}R_{0}.\] 
Assume that 
\begin{equation}
  \label{eq:laplacian_cond1}
  \rank(\td{\L}_{0}) = K - 1, \quad \lambda_{\max}(\td{\L}_{0}) < \lambda_{\min}(\td{\D}_{0})
\end{equation}
Fix any $q > 0$. Then there exists constants $c$ and $n_{0}$ that only depends on $B_{0}, \pi_{r}$'s and $q$ such that if $n\ge n_{0}$ and
\begin{equation}
  \label{eq:laplacian_cond}
  n\rho_{n}\ge c\log n.
\end{equation}
then spectral clustering using the unnormalized Laplacian achieves exact recovery with probability at least $1 - n^{-q}$.
\end{theorem}
The condition \eqref{eq:laplacian_cond1} is motivated by the following result on the eigen-structure of population Laplacian $\L^{*}$. The proof is relegated to Appendix \ref{subapp:other_exact_recovery}. 

\begin{lemma}\label{lem:eigen_laplacian}
  Let $R = \diag(\sqrt{n_{1} / n}, \ldots, \sqrt{n_{K} / n})$, $\td{d} = B_{0}R^{2}\one_{K}$ and $\td{\D} = \diag(\td{d}_{1}, \ldots, \td{d}_{K})$. Further let 
\[\td{\L} = \td{\D} - RB_{0}R.\]
Let $\td{L} = V\Sigma V^{T}$ be the spectral decomposition. Then the spectral decomposition of $\L^{*}$ can be written as
\[\L^{*}
  =   \begin{bmatrix}
    U^{*} & \td{U}^{*}
  \end{bmatrix}
  \begin{bmatrix}
    \Lambda^{*} & 0\\
    0 & \td{\Lambda}^{*}
  \end{bmatrix}
  \begin{bmatrix}
    U^{*} & \td{U}^{*}
  \end{bmatrix}^{T}
\]
where 
\[U^{*} = QV, \quad \Lambda^{*} = n\rho_{n}\Sigma , \quad \td{U}^{*} =
  \begin{bmatrix}
    Q_{1} & 0 & \ldots & 0\\
    0 & Q_{2} & \ldots & 0\\
    \vdots & \vdots & \ddots & \vdots\\
    0 & 0 & \ldots & Q_{K} 
  \end{bmatrix},  \quad \td{\Lambda}^{*} = n\rho_{n}
  \begin{bmatrix}
    \td{d}_{1}I_{n_{1} - 1} & 0 & \ldots & 0\\
    0 & \td{d}_{2}I_{n_{2} - 1} & \ldots & 0\\
    \vdots & \vdots & \ddots & \vdots\\
    0 & 0 & \ldots & \td{d}_{K}I_{n_{K} - 1} 
  \end{bmatrix}
\]
and $Q_{j}\in \R^{n_{j}\times (n_{j} - 1)}$ can be any matrix such that $Q_{j}Q_{j}^{T} = I_{n_{j}} - \one_{n_{j}}\one_{n_{j}}^{T} / n_{j}$.
\end{lemma}

Lemma \ref{lem:eigen_laplacian} implies that $U^{*}$ can be used to identify the clusters. Since $R\rightarrow R_{0}$, we have $\td{L}_{0}\rightarrow \td{L}$ and $\td{D}_{0}\rightarrow \td{D}$. Under condition \eqref{eq:laplacian_cond1}, when $n$ is large enough,
\[\lambda_{\min}(\td{\Lambda}^{*}) > \lambda_{\max}^{*}.\]
As a result, $U^{*}$ corresponds to the $K$ smallest eigenvalues of $\L^{*}$. This is almost necessary for the algorithm described in section \ref{subsec:exact_recovery_background} to work. The proof of Theorem \ref{thm:laplacian_spectral_clustering} is relegated to Appendix \ref{subapp:other_exact_recovery} since it is similar to that of Theorem \ref{thm:adjacency_spectral_clustering}.

\subsection{Exact recovery of SBM with growing $K$}\label{subsec:grow_K}

The last subsection discusses the case with fixed $K$ and the analyses hide all constants that depend on $K$. It is also of interest to investigate the tolerance on $K$ in applications where $K$ tends to be large. Most of work tackling with this question focuses on computationally less efficient algorithms such as likelihood method \citep[e.g.][]{choi2012stochastic} or semidefinite programming (SDP) method \citep[e.g.][]{amini2018semidefinite, chen2018convexified, fei2018exponential}. By contrast, the results on the vanilla spectral clustering algorithms, described in Section \ref{subsec:exact_recovery_background}, are rather limited \citep[e.g.][]{rohe2011spectral, su2019strong}. 

Equipped with the $\ell_{\ttinf}$ perturbation theory in this paper, we can easily derive the results for general SBMs with growing $K$. To keep the exposition clear, we focus on the balanced assortative four-parameter model with 
\begin{equation}\label{eq:pqmodel}
B = \rho_{n}B_{0},\quad   B_{0} = (a - b)I + b\one_{K}\one_{K}^{T}, \quad a > b > 0, \quad n_{1} = \ldots = n_{K} = m \triangleq n / K,
\end{equation}
where $a$ and $b$ are two constants. This is a widely studied special SBM in literature. Under this model, the best available dependence on $K$ is given by SDP methods \citep[e.g.][]{fei2018exponential} with
\begin{equation}
  \label{eq:dependence_K_SDP}
  n\rho_{n}\succeq K^{2} + K\log n.
\end{equation}
For spectral clustering \citep[][Example 2.1]{su2019strong}, the dependence becomes much worse
\begin{equation}
  \label{eq:dependence_K_spectral}
  n\rho_{n}\succeq K^{7}\log n.
\end{equation}
In this subsection, we analyze the spectral clustering algorithms, described in Section \ref{subsec:exact_recovery_background}, using the adjacency matrix and the unnormalized Laplacian, respectively.  
For both algorithms, we show that the dependence on $K$ is better than \eqref{eq:dependence_K_spectral}, although it still leaves a gap to \eqref{eq:dependence_K_SDP}. 

\begin{theorem}\label{thm:grow_K}
Fix any $q > 0$  Under the model \eqref{eq:pqmodel}, there exists a constant $c$ which only depends on $(a, b, q)$ such that
  \begin{enumerate}[(1)]
  \item the spectral clustering algorithm using the adjacency matrix achieves exact recovery with probability $1 - n^{-q}$ if $n\rho_{n}\ge c(K^{4} + K^{3}\log n)$;
  \item the spectral clustering algorithm using the unnormalized Laplacian achieves exact recovery with probability $1 - n^{-q}$ if $n\rho_{n}\ge cK^{3}\log n$;
  \end{enumerate}
\end{theorem}

\begin{proof}
For the sake of length, we only present the proof for part (1) and leave part (2) to Appendix \ref{subapp:other_exact_recovery}. 

By \eqref{eq:AstarQM} and \eqref{eq:pqmodel}, it is clear that 
\[A^{*} = \td{A}^{*} - \rho_{n}aI_{n}, \quad \td{A}^{*} = \frac{n\rho_{n}}{K}QB_{0}Q^{T}.\]
Thus, $A^{*}$ and $\td{A}^{*}$ have the same eigenvectors and $U^{*} = \td{U}^{*}$. Let $V\Sigma V^{T}$ be the spectral decomposition of $B_{0}$. Then it is easy to see that
\[\Sigma = \diag\lb a + (K - 1)b, \underbrace{a - b, \ldots, a - b}_{(K - 1)\mbox{ copies}}\rb.\]
Thus, the top-$K$ eigenvalue matrix $\Lambda^{*}$ of $A^{*}$ is $(n\rho_{n} / K)\Sigma - \rho_{n}aI_{n} = (m - a)\rho_{n}\Sigma$ and the corresponding eigenvector matrix $U^{*}$ can be written as $U^{*} = QV$. As a result, 
\begin{equation}
  \label{eq:grow_K1}
  \min_{s\not= s'}\|\nu_{s}^{*} - \nu_{s'}^{*}\|_{2} = \sqrt{\frac{2}{m}}.
\end{equation}
By Lemma \ref{lem:Kmedian}, it is left to prove that 
\begin{equation}
  \label{eq:grow_K_goal}
  d_{\ttinf}(U, U^{*})\le \frac{1}{6K}\min_{s\not= s'}\|\nu_{s}^{*} - \nu_{s'}^{*}\|_{2} = \frac{\sqrt{2}}{6K\sqrt{m}}.
\end{equation}
We split $U^{*}$ into two parts $U_{1}^{*}\in\R^{n\times 1}$ and $U_{2}^{*}\in \R^{n\times (K - 1)}$ where $U_{1}^{*}$ corresponds to the largest eigenvalue of $A^{*}$ while $U_{2}^{*}$ gives other eigenvectors in $U^{*}$. The same split is applied to $U$ which yields $U_{1}\in \R^{n\times 1}$ and $U_{2}\in \R^{n\times (K - 1)}$. We also split $\Lambda$ and $\Lambda^{*}$ similarly. It is easy to see that $d_{\ttinf}(U, U^{*})\le d_{\ttinf}(U_{1}, U_{1}^{*}) + d_{\ttinf}(U_{2}, U_{2}^{*})$. Thus \eqref{eq:grow_K_goal} is true provided that
\begin{equation}
  \label{eq:grow_K_goal2}
  d_{\ttinf}(U_{1}, U_{1}^{*})\le \frac{\sqrt{2}}{12K\sqrt{m}}, \quad d_{\ttinf}(U_{2}, U_{2}^{*})\le \frac{\sqrt{2}}{12K\sqrt{m}}
\end{equation}

By definition, 
\[\lambda_{\min}(\Lambda_{1}^{*}) \ge \gap_{1}^{*} = (m - a)\rho_{n}Kb, \quad \bar{\kappa}_{1}^{*} = \kappa_{1}^{*} = 1,\]
and 
\[\lambda_{\min}(\Lambda_{2}^{*}) \ge \gap_{2}^{*} = (m - a)\rho_{n}\min\{Kb, a - b\}, \quad \bar{\kappa}_{1}^{*} = \kappa_{1}^{*} = 1,\]
Set $\delta = n^{-q}$ and $\alpha = 1 / \log R(\delta)$ in Corollary \ref{cor:full_recovery_binary}. Note that this choice of $\alpha$ implies that
\[\frac{R(\delta)}{\alpha \log R(\delta)} = R(\alpha), \quad R(\delta)^{\alpha} = \exp\{\alpha \log R(\delta)\} = e.\]
Moreover, $p^{*}\preceq \rho_{n}$, and 
\[R(\delta) \preceq \log n + K, \quad g(\delta)\preceq \sqrt{n\rho_{n}} + \log n + K.\]
Since $n\rho_{n} > c(K^{4} + K^{3}\log n)$, for sufficiently large $n$ and $c$,
\[\gap_{1}^{*}\ge \gap_{2}^{*} = (m - a)\rho_{n}\min\{Kb, a - b\} = (n - aK)\rho_{n}\frac{\min\{Kb, a - b\}}{K}\ge C\bar{\kappa}_{1}^{*}g(\delta) = C\bar{\kappa}_{2}^{*}g(\delta),\]
where $C$ is the universal constant in \eqref{eq:A4_binary}. Thus the condition of Corollary \ref{cor:full_recovery_binary} is satisfied for both $(\Lambda_{1}^{*}, U_{1}^{*})$ and $(\Lambda_{2}^{*}, U_{2}^{*})$. Note that $m - a \succeq m$. Thus, $\lambda_{\min}(\Lambda_{1}^{*}), \lambda_{\min}(\Lambda_{2}^{*}), \gap_{1}^{*}, \gap_{2}^{*}\succeq m\rho_{n}$. By Corollary \ref{cor:full_recovery_binary}, for both $j = 1, 2$,
\begin{align*}
  d_{\ttinf}(U_{j}, U_{j}^{*}) & \preceq \lb \frac{\sqrt{n\rho_{n}} + \log n + K}{m\rho_{n}} + \frac{\log n + K}{m\rho_{n}}\rb \frac{1}{\sqrt{m}} + \frac{\sqrt{(\log n + K) \rho_{n}}}{m\rho_{n}}\lb 1 + \frac{\sqrt{n\rho_{n}}}{m\rho_{n}}\rb\\
& \stackrel{(i)}{\preceq} \lb \frac{\sqrt{n\rho_{n}}}{m\rho_{n}} + \frac{\log n + K}{m\rho_{n}} + \sqrt{\frac{\log n + K}{m\rho_{n}}}\rb\frac{1}{\sqrt{m}}\\
& \stackrel{(ii)}{\preceq} \sqrt{\frac{\log n + K}{m\rho_{n}}}\frac{1}{\sqrt{m}},
\end{align*}
where (i) uses the fact that
\[\frac{\sqrt{n\rho_{n}}}{m\rho_{n}} = \frac{K}{\sqrt{n\rho_{n}}}\preceq 1,\]
and (ii) uses the fact that 
\[\frac{\log n + K}{m\rho_{n}} = \sqrt{\frac{K\log n}{n\rho_{n}}}\preceq 1, \quad \frac{\sqrt{n\rho_{n}}}{m\rho_{n}} = \sqrt{\frac{K}{m\rho_{n}}}\preceq \sqrt{\frac{\log n + K}{m\rho_{n}}}.\]
Equivalently, there exists a constant $c_{1}$ that only depends on $a, b$ and $q$ such that
\[d_{\ttinf}(U_{j}, U_{j}^{*})\le \sqrt{\frac{\log n + K}{m\rho_{n}}}\frac{c_{1}}{\sqrt{m}} = \sqrt{\frac{K\log n + K^{2}}{n\rho_{n}}}\frac{c_{1}}{\sqrt{m}}\le \frac{c_{1}}{\sqrt{c}K\sqrt{m}}.\]
As a result, \eqref{eq:grow_K_goal2} is true if $c > 1 / 72c_{1}^{2}$. This completes the proof.
\end{proof}

\section{Partial Consistency of Divisive Hierarchical Clustering}\label{sec:hierarchical}

\subsection{Background}
Hierarchical community detection is widely used in practice. As opposed to agglomerative (bottom-up) hierarchical clustering, divisive (top-down) hierarchical clustering starts from the whole network, tests if there are at least two communities, and then divides it into a few mega-communities if so and stops splitting otherwise. The procedure proceeds recursively for each of mega-community until none of the mega-communities at hand passes the test for more than one community. Unlike agglomerative clustering, divisive clustering is scalable for giant networks with large number of clusters in terms of both computation and storage cost. The idea emerged in machine learning problems such as graph partitioning and image segmentation, referred to as graph bi-partitioning \citep{spielman1996spectral, shi2000normalized, kannan2004clusterings}. Despite the empirical success in various applications, the theoretical analysis is challenging. The existing analyses either require complicated but artificial modification of the algorithms \citep[e.g.][]{dasgupta2006spectral} or only hold for dense networks (with polynomial average degree) \citep[e.g.][]{balakrishnan2011noise}.

\begin{figure}
  \centering
  \includegraphics[width = 0.7\textwidth]{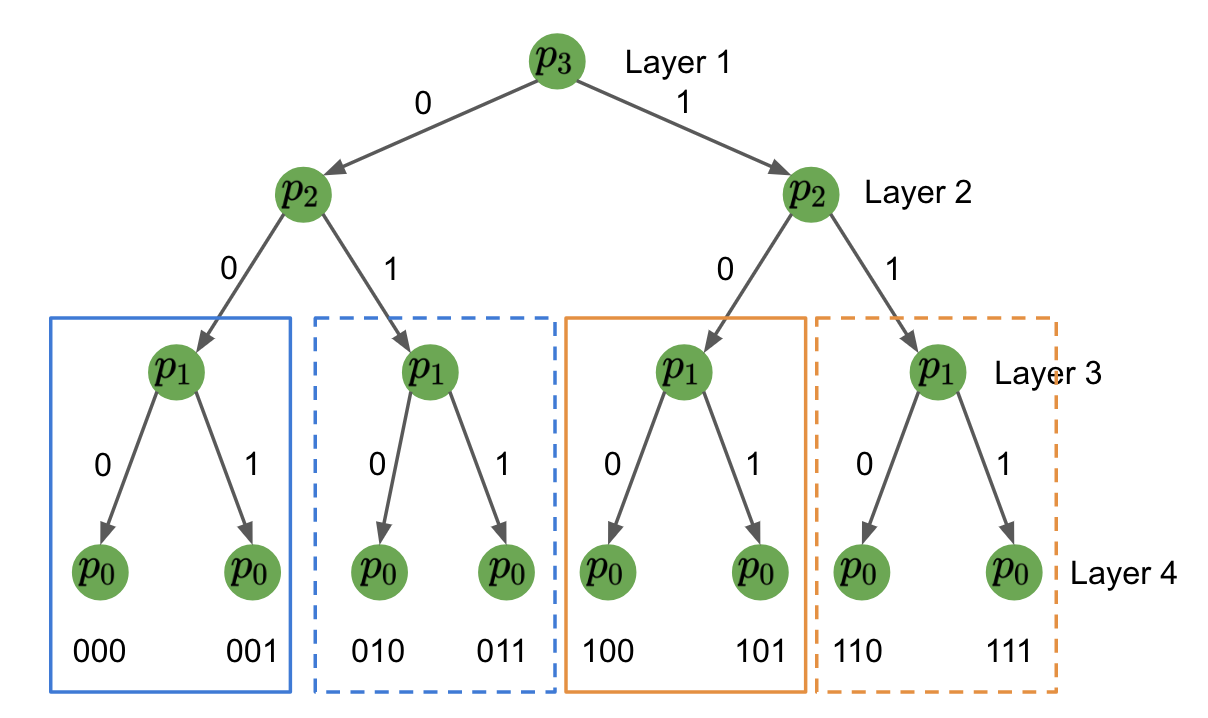}
  \caption{An $8$-cluster Binary Tree SBM.  Rectangles correspond to mega-communities.}\label{fig:BTSBM}
\end{figure}

A recent work by \cite{li2018hierarchical} established a framework to study divisive hierarchical clustering. They proposed the Binary Tree SBM (BTSBM) as the basis for analysis. A BTSBM is an SBM, described in Section \ref{subsec:exact_recovery_background}, with $K = 2^{d}$ clusters, embedded into the leaf nodes of a full binary tree with $d + 1$ layers. The $\ell$-th layer is equipped with a parameter $p_{d - \ell + 1}$ and each cluster is encoded as a length-$d$ binary string. We illustrate it in Figure \ref{fig:BTSBM} with $d = 3$. The connection probability matrix $B$ is then decided as follows: for any two clusters $c$ and $c'$, let $x_{d}\ldots x_{1}$ and $x_{d}'\ldots x_{1}'$ be their binary representation, then 
\[B_{c, c'} = p_{\scriptscriptstyle{D(c, c')}}, \quad \mbox{where }D(c, c') = \max\{i: x_{i}\not= x_{i}'\}I(c \not = c').\]
For instance, for the BTSBM in Figure \ref{fig:BTSBM}, 
\begin{equation}
  \label{eq:BTSBM_B}
  B = \left[
      \begin{array}{cc|cc|cccc}
        p_{0} & p_{1} & p_{2} & p_{2} & p_{3} & p_{3} & p_{3} & p_{3}\\
        p_{1}& p_{0} & p_{2} & p_{2} & p_{3} & p_{3} & p_{3} & p_{3}\\
        \hline
        p_{2} & p_{2} & p_{0} & p_{1} & p_{3} & p_{3} & p_{3} & p_{3}\\
        p_{2} & p_{2} &p_{1}& p_{0} & p_{3} & p_{3} & p_{3} & p_{3}\\
        \hline
        p_{3} & p_{3} & p_{3} & p_{3} & p_{0} & p_{1} & p_{2} & p_{2}\\
        p_{3} & p_{3} & p_{3} & p_{3} &p_{1}& p_{0} & p_{2} & p_{2}\\
        p_{3} & p_{3} & p_{3} & p_{3} & p_{2} & p_{2} & p_{0} & p_{1}\\
        p_{3} & p_{3} & p_{3} & p_{3} & p_{2} & p_{2} &p_{1}& p_{0}
      \end{array}
      \right].
\end{equation}

Assuming equal block sizes for all communities, \cite{li2018hierarchical} analyzed the HCD-Sign algorithm, which splits the network nto two mega-communities according to the sign of the second eigenvector of the adjacency matrix, under both assortative BTSBM ($p_{0} > p_{1} > \ldots > p_{d}$) and dis-assortative BTSBM ($p_{0} < p_{1} < \ldots < p_{d}$). They provided explicit conditions under which all mega-communities in the first $\ell$ layers can be exactly recovered for any $\ell\le d$. Unlike the $K$-way spectral clustering, described in Section \ref{subsec:exact_recovery_background}, which requires knowing $K$ exactly, HCD-Sign can be completely agnostic to $K$ while only requires a ``consistent'' stopping rule, such as the one based on non-backtracking operator \citep{le2015estimating}.

The proof relies on the nice eigen-structure of BTSBM as stated below.
\begin{proposition}\label{prop:eigen_BTSBM}[Theorem 5 of \cite{li2018hierarchical}]
Under either assortative or dis-assorative BTSBM, 
\begin{enumerate}[(1)]
\item the first and the second largest eigenvalue (in absolute value) of $A^{*}$ are both unique, given by
\[\lambda_{1}^{*} = (m - 1)p_{0} + m\sum_{i=1}^{d-1}2^{i-1}p_{i} + m2^{d - 1}p_{d}, \quad \lambda_{2}^{*} = (m - 1)p_{0} + m\sum_{i=1}^{d-1}2^{i-1}p_{i} - m2^{d - 1}p_{d};\]
\item the eigen-gap $\gap^{*}$ between $\lambda_{2}^{*}$ and others is 
\[\gap^{*} = \left\{
    \begin{array}{ll}
     n\min\{p_{d}, |p_{d - 1} - p_{d}| / 2\} & \mbox{(for assortative BTSBM)}\\
     n |p_{d - 1} - p_{d}| / 2 & \mbox{(for dis-assortative BTSBM)}
    \end{array}\right.;
 \] 
\item the eigenvector corresponding to $\lambda_{2}^{*}$ is  
\[u_{2} = \frac{1}{\sqrt{n}}
\begin{bmatrix}
  \one_{n / 2} \\ -\one_{n / 2}
\end{bmatrix};\]
\item $\rank(A^{*}) = K$. $A^{*}$ is psd under assortative BTSBM, but not psd under dis-assortative BTSBM.
\end{enumerate}
\end{proposition}
Proposition \ref{prop:eigen_BTSBM} implies that the second eigenvector of $A^{*}$ perfectly identifies the binary split. They proceed by showing that $\|u_{2} - u_{2}^{*}\|_{\infty} <\!\! < 1 / \sqrt{n}$ thereby proving the exact recovery for each split. The technique underlying their analysis is derived by \cite{eldridge2017unperturbed}. 

Although BTSBM is still restrictive, it is much more general than the typical four-parameter model described in Section \ref{subsec:grow_K}. More importantly, BTSBM captures the multi-scale nature of real-world networks: the clusters are not treated as ``exchangeable'' but encoding different granularity of similarity. Furthermore, divisive hierarchical clustering algorithms are able to achieve partial recovery, i.e. recovering mega-communities up to layer $\ell + 1$ without recovering the communities at the finest level, imply that the stable structure (i.e. mega-communities) may be recovered even if the finest communities cannot, depending on how similar different clusters are. 

With the new $\ell_{\ttinf}$ perturbation bounds in this paper, we can refine \cite{li2018hierarchical}'s results and obtain a more accurate characterization of the partial exact recovery phenomenon. We start from a generic sufficient condition for recovering one split and then apply it to analyze the case with fixed $K$. The case with growing $K$ can be analyzed similarly but we leave it to interested readers for the sake of length.

\subsection{A generic sufficient condition for recovering one split}

The exact recovery for the first split is achieved by HCD-Sign iff there exists $s\in \{-1, +1\}$ such that
\begin{equation}
  \label{eq:equal_sign}
  \sign(u_{2i}^{*}) = \sign(u_{2i})s, \quad \forall i\in [n].
\end{equation}
As observed by \cite{abbe2017entrywise} as well as our theory, $u_{2}$ is closer to $Au_{2}^{*} / \lambda_{2}^{*}$ than to $u_{2}^{*}$. The following lemma provides a sufficient condition for \eqref{eq:equal_sign} based on $Au_{2}^{*} / \lambda_{2}^{*}$.
\begin{lemma}\label{lem:equal_sign}
\eqref{eq:equal_sign} holds if there exists $s\in \{-1, +1\}$ such that for all $i\in [n]$,
  \begin{equation}
    \label{eq:equal_sign_cond1}
    \sign(u_{2i}^{*}) = \sign(A_{i}^{T} (\sqrt{n}u_{2}^{*}))s, 
   \end{equation}
and 
\begin{equation}
  \label{eq:equal_sign_cond2}
|A_{i}^{T} (\sqrt{n}u_{2}^{*})| > \lb \sqrt{n}\bigg\|u_{2} - \frac{Au_{2}^{*}}{\lambda_{2}^{*}}\bigg\|_{\infty}\rb |\lambda_{2}^{*}|.
  \end{equation}
\end{lemma}
\begin{proof}
We only prove the assortative case and the proof for the dissortative case is similar. Assume $s = 1$ without loss of generality. The condition\eqref{eq:equal_sign_cond2} can be rewritten as
\[\bigg|\frac{A_{i}^{T} u_{2}^{*}}{\lambda_{2}^{*}}\bigg| > \bigg\|u_{2} - \frac{Au_{2}^{*}}{\lambda_{2}^{*}}\bigg\|_{\infty}\ge \bigg|u_{2i} - \frac{A_{i}^{T}u_{2}^{*}}{\lambda_{2}^{*}}\bigg|.\]
By Proposition \ref{prop:eigen_BTSBM}, $\lambda_{2}^{*} > 0$ and 
\[\sign(u_{2i}) = \sign\lb \frac{A_{i}^{T}u_{2}^{*}}{\lambda_{2}^{*}}\rb = \sign (A_{i}^{T}(\sqrt{n}u_{2}^{*})).\]
Together with condition \eqref{eq:equal_sign_cond1}, we complete the proof.
\end{proof}
By Proposition \ref{prop:eigen_BTSBM}, 
\begin{equation}\label{eq:Aione}
  A_{i}^{T}(\sqrt{n}u_{2}^{*}) = \sum_{j=1}^{n / 2}A_{ij} - \sum_{j= n / 2 + 1}^{n}A_{ij}.
\end{equation}
Let 
\begin{equation}
  \label{eq:Zi}
  Z_{i} = (-1)^{I(i > n / 2)}A_{i}^{T}(\sqrt{n}u_{2}^{*})
\end{equation}
Under BTSBM, it is not hard to see that $Z_{1}, \ldots, Z_{n}$ are i.i.d.. Note that $\sign(u_{2i}^{*}) = (-1)^{I(i > n / 2)}$. By Lemma \ref{lem:equal_sign}, \eqref{eq:equal_sign_cond1} and \eqref{eq:equal_sign_cond2} are both satisfied if 
\begin{equation}
  \label{eq:minmaxZi}
  \min_{i\in [n]}Z_{i} > \lb \sqrt{n}\bigg\|u_{2} - \frac{Au_{2}^{*}}{\lambda_{2}^{*}}\bigg\|_{\infty}\rb |\lambda_{2}^{*}| \quad \mbox{or}\quad \max_{i\in [n]}Z_{i} < -\lb \sqrt{n}\bigg\|u_{2} - \frac{Au_{2}^{*}}{\lambda_{2}^{*}}\bigg\|_{\infty}\rb |\lambda_{2}^{*}|.
\end{equation}
It is left to show the above event occurs with high probability.

The following lemma provides a tail probability estimate for $Z_{i}$. The proof is relegated to Appendix \ref{subapp:other_hierarchical}.
\begin{lemma}\label{lem:prob_Zi}
~  \begin{enumerate}[(1)]
  \item In the assortative case, for any $t > 0$,
\[\log\P(Z_{i} \le t)\le \frac{t}{2}\log \lb\frac{\lambda_{1}^{*} + \lambda_{2}^{*}}{\lambda_{1}^{*} - \lambda_{2}^{*}}\rb  -\frac{1}{2}\lb\sqrt{\lambda_{1}^{*} + \lambda_{2}^{*}} - \sqrt{\lambda_{1}^{*} - \lambda_{2}^{*}}\rb^{2}.\]
  \item In the dis-assortative case, for any $t > 0$,
\[\log\P(Z_{i} \ge -t) \le \frac{t}{2}\log \lb\frac{\lambda_{1}^{*} - \lambda_{2}^{*}}{\lambda_{1}^{*} + \lambda_{2}^{*}}\rb  -\frac{1}{2}\lb\sqrt{\lambda_{1}^{*} + \lambda_{2}^{*}} - \sqrt{\lambda_{1}^{*} - \lambda_{2}^{*}}\rb^{2}.\]
  \end{enumerate}
\end{lemma}
Combined with the high probability upper bound for $\|u_{2} - Au_{2}^{*} / \lambda_{2}^{*}\|_{\infty}$ obtained by Theorem \ref{thm:generic_binary}, we can derive the following result for recovering the first split exactly. The proof is straightforward so we relegate it into Appendix \ref{subapp:other_hierarchical}. 

\begin{theorem}\label{thm:first_split}
Consider a BTSBM that is either assortative or dis-assortative. Let $\gap^{*}$ be defined as in Proposition \ref{prop:eigen_BTSBM} and $p^{*} = \max\{p_{0}, p_{d}\}$. Suppose there exists $\alpha > 0$ and $q > 0$ such that
  \begin{equation}
    \label{eq:first_split_cond}
   \gap^{*} > (q + 1)C\lb \sqrt{\lambda_{1}^{*}} + \frac{\log n}{\alpha \log \log n}\rb, 
  \end{equation}
where $C$ is the universal constant in \eqref{eq:A4_binary}. Fix any $\delta \ge n^{-q}$. Then the first split can be recovered with probability $1 - 2 \delta$ if 
\begin{equation}
  \label{eq:first_split_regime}
  \frac{1}{2}\lb\sqrt{\lambda_{1}^{*} + \lambda_{2}^{*}} - \sqrt{\lambda_{1}^{*} - \lambda_{2}^{*}}\rb^{2} - \log n - C'\frac{\lambda_{2}^{*}}{\gap^{*}}\bigg|\log \lb\frac{\lambda_{1}^{*} + \lambda_{2}^{*}}{\lambda_{1}^{*} - \lambda_{2}^{*}}\rb\bigg|(\xi_{n1} + \xi_{n2})\ge \log \lb\frac{1}{\delta}\rb,
\end{equation}
where $C'$ is a universal constant and
\[\xi_{n1} = \lb\sqrt{\lambda_{1}^{*}} + \frac{\log n}{\alpha\log \log n}\rb\lb 1 + \frac{\log n}{|\lambda_{2}^{*}|}\rb + \frac{\sqrt{(\log n)np^{*}}}{|\lambda_{2}^{*}|} \frac{\log n + \sqrt{\lambda_{1}^{*}(\log n)^{\alpha}}}{\alpha\log \log n},\]
and
\[\xi_{n2} = \left\{
    \begin{array}{ll}
     \min\left\{\sqrt{np^{*}}, \sqrt{\lambda_{2}^{*}K}\right\} & \mbox{(assortative)}\\
      \min\left\{\sqrt{\lambda_{1}^{*}np^{*} / |\lambda_{2}^{*}|}, \sqrt{\lambda_{2}^{*}K}\right\} & \mbox{(dis-assortative)}\\
    \end{array}\right..\]
\end{theorem}

\subsection{Exact recovery of mega-communities for BTSBMs}

Consider the setting where $K$ is fixed and
\[p_{j} = \rho_{n}a_{j},\]
for a set of constants $(a_{0}, \ldots, a_{d})$. \cite{li2018hierarchical} proves the exact recovery in the regime $n\rho_{n} \succeq (\log n)^{2+\eps}$. On the other hand, if $K$ is known, the information theoretic lower bound for recovering all communities (not including mega-communities) is $\rho_{n} = \log n / n$ \citep[e.g.][]{abbe2015community}. The extra logarithmic factors in \cite{li2018hierarchical} is simply an artifact of using the non-tight $\ell_{\infty}$ perturbation bound by \cite{eldridge2017unperturbed}. With Theorem \ref{thm:first_split} derived from our $\ell_{\infty}$ perturbation theory, we can prove the exact recovery in the regime $\rho_{n} = \log n / n$ and provide precise condition on the constants $(a_{0}, \ldots, a_{d})$.

\begin{theorem}\label{thm:partial_exact_recovery}
Assume that $\rho_{n} = \log n / n$ and either $a_{0} > a_{1} > \ldots > a_{d} > 0$ (assortative) or $0 < a_{0} < a_{1} < \ldots < a_{d}$ (dis-assortative). Fix any $\ell \in [d]$. If further
\begin{equation}
  \label{eq:BTSBM_fix_K_cond}
  |\sqrt{\bar{a}_{r}} - \sqrt{a_{r}}| > \sqrt{2^{d - r + 1}}, \quad r = d, d - 1, \ldots, d - \ell + 1,
\end{equation}
where
\begin{equation}
  \label{eq:bar_aj}
  \bar{a}_{r} = \frac{a_{0} + \sum_{j=1}^{r-1}2^{j-1}a_{j}}{2^{r-1}},
\end{equation}
then all mega-communities up to layer $\ell + 1$ can be exactly recovered with probability $1 - o(1)$ as $n$ tends to infinity.
\end{theorem}
\begin{remark}
  The quantity $\bar{a}_{r}$ is essentially the average conncetion probability in each mega-community at $(d - r + 2)$-th layer. 
\end{remark}

The condition \eqref{eq:BTSBM_fix_K_cond} has an interesting implication. Take $\ell = d$, it is equivalent to
\[|\sqrt{\bar{a}_{d}} - \sqrt{a_{d}}| > \sqrt{2}.\]
Further take $d = 3$ for illustration and recall \eqref{eq:BTSBM_B}. Consider the following hypothetical SBM with connection probability matrix
\begin{equation}\label{eq:BTSBM_B'}
  B' = \rho_{n}\left[
      \begin{array}{cccc|cccc}
        \bar{a}_{3} & \bar{a}_{3} & \bar{a}_{3} & \bar{a}_{3} & a_{3} & a_{3} & a_{3} & a_{3}\\
        \bar{a}_{3}& \bar{a}_{3} & \bar{a}_{3} & \bar{a}_{3} & a_{3} & a_{3} & a_{3} & a_{3}\\
        \bar{a}_{3} & \bar{a}_{3} & \bar{a}_{3} & \bar{a}_{3} & a_{3} & a_{3} & a_{3} & a_{3}\\
        \bar{a}_{3} & \bar{a}_{3} &\bar{a}_{3}& \bar{a}_{3} & a_{3} & a_{3} & a_{3} & a_{3}\\
        \hline
        a_{3} & a_{3} & a_{3} & a_{3} & \bar{a}_{3} & \bar{a}_{3} & \bar{a}_{3} & \bar{a}_{3}\\
        a_{3} & a_{3} & a_{3} & a_{3} &\bar{a}_{3}& \bar{a}_{3} & \bar{a}_{3} & \bar{a}_{3}\\
        a_{3} & a_{3} & a_{3} & a_{3} & \bar{a}_{3} & \bar{a}_{3} & \bar{a}_{3} & \bar{a}_{3}\\
        a_{3} & a_{3} & a_{3} & a_{3} & \bar{a}_{3} & \bar{a}_{3} &\bar{a}_{3}& \bar{a}_{3}
      \end{array}
      \right].
\end{equation}
This is essentially a $2$-block SBM with parameter $\bar{a}_{3}\log n / n$ and $a_{3}\log n / n$ where $\bar{a}_{3}$ is the average probability in the mega-community. It is well-known that \eqref{eq:BTSBM_B'} can be exactly recovered \emph{if and only if}
\[|\sqrt{\bar{a}_{3}} - \sqrt{a_{3}}| > \sqrt{2}.\]
In other words, recovering the first split of the BTSBM with connection probability matrix \eqref{eq:BTSBM_B} is indistinguishable from recovering the blocks if the induced $2$-block model, which replaces the connection probabilities by the within-mega-community average. This is an unexpected robustness result for mis-specified SBM models.

On the other hand, as mentioned earlier, we want to investigate the possibility that the finest communites cannot be recovered but higher-level mega-communities can. To investigate this, we first derive the necessary condition for the exact recovery of $K$ communities in the leaf nodes. This is a simple consequence of the existing results on general SBMs \citep[e.g.][]{abbe2015community}. 
\begin{lemma}\label{lem:necessary_cond_BTSBM}
  No algorithm can recovery all of $K$ communities in the leaf nodes of a BTSBM, that is either assortative or dis-assortative, with high probability if 
  \begin{equation}
    \label{eq:necessary_cond_BTSBM}
    |\sqrt{a_{0}} - \sqrt{a_{1}}| < \sqrt{K}.
  \end{equation}
\end{lemma}
Note that $\bar{a}_{1} = a_{0}$. The necessary condition \eqref{eq:necessary_cond_BTSBM} is essentially the negation of \eqref{eq:BTSBM_fix_K_cond} with $\ell = d$ and $r = 1$. If $|\sqrt{a_{0}} - \sqrt{a_{1}}| < \sqrt{K}$, Lemma \ref{lem:necessary_cond_BTSBM} implies that the finest communities cannot be recovered exactly by any algorithm, including HCD algorithms. However, if \eqref{eq:BTSBM_fix_K_cond} holds for $\ell = d - 1$, we may recover all mega-communities up to the second last layers. This is true, for instance, if $a_{0}\approx a_{1}$ and
\[|\sqrt{\bar{a}_{r}'} - \sqrt{a_{r}}| > \sqrt{2^{d - r + 1}}, \quad r = d, d - 1, \ldots, 2,\]
where 
\[  \bar{a}_{r}' = \frac{a_{1} + \sum_{j=1}^{r-1}2^{j-1}a_{j}}{2^{r-1}}.\]
Therefore, Theorem \ref{thm:partial_exact_recovery} provides a precise characterization of the partial exact recovery phenomenon. 

\section{Extensions}\label{sec:extension}
\subsection{Random matrices with other dependency structure}\label{subsec:mdependence}
In Section \ref{subsec:binary_laplacian} we discuss the unnormalized Laplacian as an example of random matrices with dependent entries. From assumption \textbf{A}1, it is not hard to see that our generic bounds allow much more flexible dependency structure. As shown in part (b) of Proposition \ref{prop:A1}, \textbf{A}1 is satisfied if the rows are $m$-dependent. If we can further derive bounds for $\|E\|_{\op}$ and $E_{k}^{T}W$ as in \textbf{A}2 and \textbf{A}3, Theorem \ref{thm:generic_bound} would yield an $\ell_{\ttinf}$ perturbation bound. 

Concentration inequalities for both quantities have been investigated for various dependency structures. We consider a slightly artificial one, motivated by \cite{paulin2012concentration}, just to illustrate the possibility to handle complex dependency structure. In particular, we assume that $A_{ij}$'s can be partitioned into $M$ subsets such that the entries within the same block are independent while the blocks can be arbitrarily dependent. In this case, $E$ can be decomposed as the sum of $M$ matrices $\{E^{(\ell)}: \ell \in [M]\}$ where $E^{(\ell)}_{ij} = E_{ij}$ if $(i, j)$ belongs to the $\ell$-th block and $E^{(\ell)}_{ij} = 0$ otherwise. By Lemma \ref{lem:Eop_binary} and a union bound, 
\[\lambda_{-}(\delta), E_{+}(\delta), \bar{E}_{+}(\delta), E_{\infty}(\delta)\preceq E_{2}(\delta)\preceq M\lb\sqrt{n\bar{p}^{*}} + \sqrt{\log (nM / \delta)}\rb.\] 
On the other hand, we apply the same decomposition on $E_{k}^{T}W$ and Lemma \ref{lem:A3_binary} implies that 
\[b_{\infty}(\delta) \preceq \frac{MR(\delta / M)}{\alpha \log R(\delta / M)}, \quad b_{2}(\delta)\preceq \frac{M\sqrt{p^{*}}R(\delta / M)^{(1 + \alpha) / 2}}{\alpha \log R(\delta / M)}.\]
Given these bounds, it is a simple exercise to derive the condition on $\gap^{*}$ through \textbf{A}4 as well as the $\ell_{\ttinf}$ bound for $d_{\ttinf}(U, AU^{*}(\Lambda^{*})^{-1})$ and $d_{\ttinf}(U, U^{*})$ by Theorem \ref{thm:generic_bound} and Theorem \ref{thm:generic_bound2}. It is also straightforward to derive the bounds for the unnormalized Laplacian using the results in Section \ref{subsec:generic_laplacian}.

The above case is by no means the end of the story. More complicated dependency structures can be handled similarly using more delicate bounds \citep[e.g.][]{paulin2012concentration}. The punchline is that our theory reduces the less tractable $\ell_{\ttinf}$ perturbation bound to the more tractable concentration bounds on $E$.

\subsection{Non-binary random matrices}\label{subsec:otherdist}
Our result can also be easily extended to non-binary random matrices. For instance, for Gaussian random matrices with $A_{ij}\stackrel{\mathrm{indep.}}{\sim} N(\mu_{ij}, \sigma_{ij}^{2})$, Corollary 3.9 of \cite{bandeira2016sharp} implies that
\[E_{2}(\delta) \preceq \bar{\sigma}^{*}\sqrt{n} + \sigma^{*}\sqrt{\log \lb\frac{n}{\delta}\rb}, \quad \mbox{where }\bar{\sigma}^{*} = \max_{i}\sqrt{\frac{1}{n}\sum_{j=1}^{n}\sigma_{ij}^{2}}, \quad \sigma^{*} = \max_{i}\sigma_{ij}.\]
This establishes the bounds for quantities in assumption \textbf{A}2. By Propostion \ref{prop:A1}, 
\[L_{1}(\delta)\preceq \mnorm{A^{*}} + E_{2}(\delta), \quad L_{2}(\delta)\preceq 1, \quad L_{3}(\delta)\preceq (\mnorm{A^{*}} + E_{2}(\delta)) / \lambda_{\min}^{*}.\]
On the other hand, for any $W\in \R^{n\times r'}$ with $r'\le r$, $E_{k}^{T}W\sim N(0, W^{T}D_{k}W)$ where $D_{k} = \diag(\sigma_{kj}^{2})_{j=1}^{n}$. Thus, $E_{k}^{T}W\stackrel{d}{=}(W^{T}D_{k}W)^{1/2}X$ where $X\sim N(0, I_{r'})$. Since the mapping $f: x\mapsto \|((W^{T}D_{k}W)^{1/2})x\|_{2}$ is $\|W^{T}D_{k}W\|_{\op}$-Lipschitz and $\|W^{T}D_{k}W\|_{\op}\le \sigma^{*}\|W\|_{\op}$, by Gaussian concentration inequality, 
\[\P\lb \|E_{k}^{T}W\|_{2}\ge \E \|E_{k}^{T}W\|_{2} + t\rb\le \exp\left\{-\frac{t^{2}}{\sigma^{*2}\|W\|_{\op}^{2}}\right\}.\]
Taking $t = \sigma^{*}\|W\|_{\op}\sqrt{\log (n / \delta)}$, we know that 
\begin{align*}
  \|E_{k}^{T}W\|_{2}& \le  \E \|E_{k}^{T}W\|_{2} + t\le \sqrt{\E \|E_{k}^{T}W\|_{2}^{2}} + t \le \sqrt{\tr(W^{T}D_{k}W)} + t\\
& \le \|W\|_{\op}\sigma^{*}\lb\sqrt{r} + \sqrt{\log (n / \delta)}\rb\preceq \sigma^{*}R(\delta)\|W\|_{\op}.
\end{align*}
As a result, we have $b_{\infty}(\delta) = 0$ and $b_{2}(\delta) \preceq\sigma^{*}R(\delta)$ in Assumption \textbf{A}3. As commented in Remark \ref{rem:EUA3}, 
\[\mnorm{EU^{*}}\preceq \sigma^{*}R(\delta).\]
Putting pieces together, we can derive the $\ell_{\ttinf}$ bound in this case by Theorem \ref{thm:generic_bound} and Theorem \ref{thm:generic_bound2}. 

Following the same strategy, we can extend the results to other entry distributions. The bound for $\|E\|_{\op}$ can be found in \cite{bandeira2016sharp, latala2018dimension, rebrova2018spectral} for sub-gaussian, sub-exponential, heavy-tailed, symmetric random variables. The row-wise concentration inequality can be obtained from standard moment generating function arguments \citep[e.g.][]{vershynin2010introduction}. 

\subsection{Asymmetric random matrices}
Our perturbation bound can be extended to asymmetric and rectangular matrices using the \emph{Hermitian dilation} trick \citep[e.g.][]{paulsen2002completely}. Given a rectangular matrix $A\in \R^{m\times n}$ with $m\ge n$, the Hermitian dilation of $A$ is defined as 
\[\td{A} =
  \begin{bmatrix}
    0 & A\\
    A^{T} & 0
  \end{bmatrix}, 
\]
Let $\bar{U}\bar{\Sigma} \bar{V}^{T}$ be the singular value decomposition (SVD) of $A$ where $\bar{U} \in \R^{m\times n}, \bar{V} \in \R^{n\times n}$ are two orthogonal matrices and $\bar{\Sigma}\in \R^{n\times n}$ is a diagonal matrix. Then $\td{U}\td{\Sigma}\td{U}^{T}$ is the SVD of $\td{A}$ where
\[\td{U} = \frac{1}{\sqrt{2}}
  \begin{bmatrix}
    \bar{U} & \bar{U}\\
    \bar{V} & -\bar{V}
  \end{bmatrix}, \quad \td{\Sigma} =
  \begin{bmatrix}
    \bar{\Sigma} & 0 \\
    0 & -\bar{\Sigma} \\
  \end{bmatrix}.
\]
Given a pair of asymmetric matrices $A, A^{*}\in \R^{m\times n}$, their left singular spaces $U, U^{*}$ and corresponding right singular spaces $V, V^{*}$, our bound can be applied to their Hermitian dilation to yield a bound for 
\[d_{\ttinf}\lb
  \begin{bmatrix}
    U \\ V
  \end{bmatrix}
,
\begin{bmatrix}
  U^{*} \\ V^{*}
\end{bmatrix}
\rb.\]
This provides an upper bound for both $d_{\ttinf}(U, U^{*})$ and $d_{\ttinf}(V, V^{*})$.

\subsection{Perturbation in other metrics}
Given an $\ell_{\ttinf}$ bound, we can derive the perturbation bound in other metrics. One example is the $\ell_{\ttinf}$ bound for projection matrices, namely $\mnorm{UU^{T} - U^{*}(U^{*})^{T}}$, which is studied in \cite{mao2017estimating}.  Note that for any $O\in \O^{r}$, 
  \[UU^{T} - U^{*}(U^{*})^{T} = U\O (U\O)^{T} - U^{*}(U^{*})^{T} = (U\O - U^{*})(U\O)^{T} + U^{*}(U\O - U^{*})^{T}.\]
Then 
\[\mnorm{(U\O - U^{*})(U\O)^{T}}\le \mnorm{U\O - U^{*}}\|U\O\|_{\op} = \mnorm{U\O - U^{*}},\]
and 
\[\mnorm{U^{*}(U\O - U^{*})^{T}} \le \mnorm{U^{*}}\|U\O - U^{*}\|_{\op}\le \sqrt{n}\mnorm{U^{*}}\mnorm{U\O - U^{*}}.\]
Taking $O$ as the orthogonal matrix that minimizes $\mnorm{U\O - U^{*}}$, we conclude that
\begin{equation*}
  \mnorm{UU^{T} - U^{*}(U^{*})^{T}}\le (\sqrt{n}\mnorm{U^{*}} + 1)\mnorm{U\O - U^{*}}\preceq \sqrt{n}\mnorm{U^{*}}\mnorm{U\O - U^{*}}.
\end{equation*}
Another example is the entrywise bound for $UU^{T} - U^{*}(U^{*})^{T}$. For any $O\in \O^{r}$, we have
\[UU^{T} - U^{*}(U^{*})^{T} = (U\O - U^{*})(U\O - U^{*})^{T} + U^{*}(U\O - U^{*})^{T} + (U\O - U^{*})(U^{*})^{T}.\]
Taking $O\in \O^{r}$ that minimizes $\mnorm{U\O - U^{*}}$ and using the fact that $\maxnorm{AB}\le \mnorm{A}\mnorm{B}$, we have
\[\maxnorm{UU^{T} - U^{*}(U^{*})^{T}}\preceq d_{\ttinf}(U, U^{*})^{2} + \mnorm{U^{*}}d_{\ttinf}(U, U^{*}).\]
Finally, we can derive an entry-wise bound for $U\Lambda U^{T} - U^{*}\Lambda^{*} (U^{*})^{T}$, which is of interest if the goal is to recover the low-rank component. Similar to the derivation for projection matrices, by Weyl's inequality,
\begin{align*}
  \maxnorm{U\Lambda U^{T} - U^{*}\Lambda^{*} (U^{*})^{T}}& \preceq d_{\ttinf}(U, U^{*})\mnorm{U^{*}}\lambda_{\max}^{*} + \mnorm{U^{*}}^{2}\|E\|_{\op}\\
&  + d_{\ttinf}(U, U^{*})^{2} \lambda_{\max}^{*}  + d_{\ttinf}(U, U^{*})^{2}\|E\|_{\op}.
\end{align*}

\section*{Acknowledgment}
The author would like to thank Peter J. Bickel, Elizaveta Levina, Tianxi Li, Sharmodeep Bhattacharyya and Purnamrita Sarkar for helpful discussion and comments.

\bibliography{two_to_infty}

\begin{thebibliography}{77}
\providecommand{\natexlab}[1]{#1}
\providecommand{\url}[1]{\texttt{#1}}
\expandafter\ifx\csname urlstyle\endcsname\relax
  \providecommand{\doi}[1]{doi: #1}\else
  \providecommand{\doi}{doi: \begingroup \urlstyle{rm}\Url}\fi

\bibitem[Abbe(2017)]{abbe2017community}
Emmanuel Abbe.
\newblock Community detection and stochastic block models: recent developments.
\newblock \emph{The Journal of Machine Learning Research}, 18\penalty0
  (1):\penalty0 6446--6531, 2017.

\bibitem[Abbe and Sandon(2015)]{abbe2015community}
Emmanuel Abbe and Colin Sandon.
\newblock Community detection in general stochastic block models: Fundamental
  limits and efficient algorithms for recovery.
\newblock In \emph{2015 IEEE 56th Annual Symposium on Foundations of Computer
  Science}, pages 670--688. IEEE, 2015.

\bibitem[Abbe et~al.(2015)Abbe, Bandeira, and Hall]{abbe2015exact}
Emmanuel Abbe, Afonso~S Bandeira, and Georgina Hall.
\newblock Exact recovery in the stochastic block model.
\newblock \emph{IEEE Transactions on Information Theory}, 62\penalty0
  (1):\penalty0 471--487, 2015.

\bibitem[Abbe et~al.(2017)Abbe, Fan, Wang, and Zhong]{abbe2017entrywise}
Emmanuel Abbe, Jianqing Fan, Kaizheng Wang, and Yiqiao Zhong.
\newblock Entrywise eigenvector analysis of random matrices with low expected
  rank.
\newblock \emph{arXiv preprint arXiv:1709.09565}, 2017.

\bibitem[Agarwal et~al.(2017)Agarwal, Bandeira, Koiliaris, and
  Kolla]{agarwal2017multisection}
Naman Agarwal, Afonso~S Bandeira, Konstantinos Koiliaris, and Alexandra Kolla.
\newblock Multisection in the stochastic block model using semidefinite
  programming.
\newblock In \emph{Compressed Sensing and its Applications}, pages 125--162.
  Springer, 2017.

\bibitem[Alon et~al.(2002)Alon, Krivelevich, and Vu]{alon2002concentration}
Noga Alon, Michael Krivelevich, and Van~H Vu.
\newblock On the concentration of eigenvalues of random symmetric matrices.
\newblock \emph{Israel Journal of Mathematics}, 131\penalty0 (1):\penalty0
  259--267, 2002.

\bibitem[Ames(2014)]{ames2014guaranteed}
Brendan~PW Ames.
\newblock Guaranteed clustering and biclustering via semidefinite programming.
\newblock \emph{Mathematical Programming}, 147\penalty0 (1-2):\penalty0
  429--465, 2014.

\bibitem[Amini and Levina(2018)]{amini2018semidefinite}
Arash~A Amini and Elizaveta Levina.
\newblock On semidefinite relaxations for the block model.
\newblock \emph{The Annals of Statistics}, 46\penalty0 (1):\penalty0 149--179,
  2018.

\bibitem[Bai and Silverstein(2010)]{bai2010spectral}
Zhidong Bai and Jack~W Silverstein.
\newblock \emph{Spectral analysis of large dimensional random matrices},
  volume~20.
\newblock Springer, 2010.

\bibitem[Balakrishnan et~al.(2011)Balakrishnan, Xu, Krishnamurthy, and
  Singh]{balakrishnan2011noise}
Sivaraman Balakrishnan, Min Xu, Akshay Krishnamurthy, and Aarti Singh.
\newblock Noise thresholds for spectral clustering.
\newblock In \emph{Advances in Neural Information Processing Systems}, pages
  954--962, 2011.

\bibitem[Bandeira(2018)]{bandeira2018random}
Afonso~S Bandeira.
\newblock Random laplacian matrices and convex relaxations.
\newblock \emph{Foundations of Computational Mathematics}, 18\penalty0
  (2):\penalty0 345--379, 2018.

\bibitem[Bandeira and Van~Handel(2016)]{bandeira2016sharp}
Afonso~S Bandeira and Ramon Van~Handel.
\newblock Sharp nonasymptotic bounds on the norm of random matrices with
  independent entries.
\newblock \emph{The Annals of Probability}, 44\penalty0 (4):\penalty0
  2479--2506, 2016.

\bibitem[Bickel and Chen(2009)]{bickel2009nonparametric}
Peter~J Bickel and Aiyou Chen.
\newblock A nonparametric view of network models and newman--girvan and other
  modularities.
\newblock \emph{Proceedings of the National Academy of Sciences}, 106\penalty0
  (50):\penalty0 21068--21073, 2009.

\bibitem[Boppana(1987)]{boppana1987eigenvalues}
Ravi~B Boppana.
\newblock Eigenvalues and graph bisection: An average-case analysis.
\newblock In \emph{28th Annual Symposium on Foundations of Computer Science
  (sfcs 1987)}, pages 280--285. IEEE, 1987.

\bibitem[Boucheron et~al.(2005)Boucheron, Bousquet, Lugosi, and
  Massart]{boucheron2005moment}
St{\'e}phane Boucheron, Olivier Bousquet, G{\'a}bor Lugosi, and Pascal Massart.
\newblock Moment inequalities for functions of independent random variables.
\newblock \emph{The Annals of Probability}, 33\penalty0 (2):\penalty0 514--560,
  2005.

\bibitem[Boucheron et~al.(2013)Boucheron, Lugosi, and
  Massart]{boucheron2013concentration}
St{\'e}phane Boucheron, G{\'a}bor Lugosi, and Pascal Massart.
\newblock \emph{Concentration inequalities: A nonasymptotic theory of
  independence}.
\newblock Oxford university press, 2013.

\bibitem[Bui et~al.(1987)Bui, Chaudhuri, Leighton, and Sipser]{bui1987graph}
Thang~Nguyen Bui, Soma Chaudhuri, Frank~Thomson Leighton, and Michael Sipser.
\newblock Graph bisection algorithms with good average case behavior.
\newblock \emph{Combinatorica}, 7\penalty0 (2):\penalty0 171--191, 1987.

\bibitem[Cand{\`e}s and Recht(2009)]{candes2009exact}
Emmanuel~J Cand{\`e}s and Benjamin Recht.
\newblock Exact matrix completion via convex optimization.
\newblock \emph{Foundations of Computational mathematics}, 9\penalty0
  (6):\penalty0 717, 2009.

\bibitem[Cape et~al.(2019{\natexlab{a}})Cape, Tang, and Priebe]{cape2019signal}
Joshua Cape, Minh Tang, and Carey~E Priebe.
\newblock Signal-plus-noise matrix models: eigenvector deviations and
  fluctuations.
\newblock \emph{Biometrika}, 106\penalty0 (1):\penalty0 243--250,
  2019{\natexlab{a}}.

\bibitem[Cape et~al.(2019{\natexlab{b}})Cape, Tang, Priebe,
  et~al.]{cape2019two}
Joshua Cape, Minh Tang, Carey~E Priebe, et~al.
\newblock The two-to-infinity norm and singular subspace geometry with
  applications to high-dimensional statistics.
\newblock \emph{The Annals of Statistics}, 47\penalty0 (5):\penalty0
  2405--2439, 2019{\natexlab{b}}.

\bibitem[Carson and Impagliazzo(2001)]{carson2001hill}
Ted Carson and Russell Impagliazzo.
\newblock Hill-climbing finds random planted bisections.
\newblock In \emph{Proceedings of the twelfth annual ACM-SIAM symposium on
  Discrete algorithms}, pages 903--909. Society for Industrial and Applied
  Mathematics, 2001.

\bibitem[Chen et~al.(2014)Chen, Sanghavi, and Xu]{chen2014improved}
Yudong Chen, Sujay Sanghavi, and Huan Xu.
\newblock Improved graph clustering.
\newblock \emph{IEEE Transactions on Information Theory}, 60\penalty0
  (10):\penalty0 6440--6455, 2014.

\bibitem[Chen et~al.(2018)Chen, Li, Xu, et~al.]{chen2018convexified}
Yudong Chen, Xiaodong Li, Jiaming Xu, et~al.
\newblock Convexified modularity maximization for degree-corrected stochastic
  block models.
\newblock \emph{The Annals of Statistics}, 46\penalty0 (4):\penalty0
  1573--1602, 2018.

\bibitem[Chin et~al.(2015)Chin, Rao, and Vu]{chin2015stochastic}
Peter Chin, Anup Rao, and Van Vu.
\newblock Stochastic block model and community detection in sparse graphs: A
  spectral algorithm with optimal rate of recovery.
\newblock In \emph{Conference on Learning Theory}, pages 391--423, 2015.

\bibitem[Choi et~al.(2012)Choi, Wolfe, and Airoldi]{choi2012stochastic}
David~S Choi, Patrick~J Wolfe, and Edoardo~M Airoldi.
\newblock Stochastic blockmodels with a growing number of classes.
\newblock \emph{Biometrika}, 99\penalty0 (2):\penalty0 273--284, 2012.

\bibitem[Coja-Oghlan(2010)]{coja2010graph}
Amin Coja-Oghlan.
\newblock Graph partitioning via adaptive spectral techniques.
\newblock \emph{Combinatorics, Probability and Computing}, 19\penalty0
  (2):\penalty0 227--284, 2010.

\bibitem[Condon and Karp(2001)]{condon2001algorithms}
Anne Condon and Richard~M Karp.
\newblock Algorithms for graph partitioning on the planted partition model.
\newblock \emph{Random Structures \& Algorithms}, 18\penalty0 (2):\penalty0
  116--140, 2001.

\bibitem[Damle and Sun(2019)]{damle2019uniform}
Anil Damle and Yuekai Sun.
\newblock Uniform bounds for invariant subspace perturbations.
\newblock \emph{arXiv preprint arXiv:1905.07865}, 2019.

\bibitem[Dasgupta et~al.(2006)Dasgupta, Hopcroft, Kannan, and
  Mitra]{dasgupta2006spectral}
Anirban Dasgupta, John Hopcroft, Ravi Kannan, and Pradipta Mitra.
\newblock Spectral clustering by recursive partitioning.
\newblock In \emph{European Symposium on Algorithms}, pages 256--267. Springer,
  2006.

\bibitem[Davis and Kahan(1970)]{davis1970rotation}
Chandler Davis and William~Morton Kahan.
\newblock The rotation of eigenvectors by a perturbation. iii.
\newblock \emph{SIAM Journal on Numerical Analysis}, 7\penalty0 (1):\penalty0
  1--46, 1970.

\bibitem[Dyer and Frieze(1989)]{dyer1989solution}
Martin~E. Dyer and Alan~M. Frieze.
\newblock The solution of some random np-hard problems in polynomial expected
  time.
\newblock \emph{Journal of Algorithms}, 10\penalty0 (4):\penalty0 451--489,
  1989.

\bibitem[Eldridge et~al.(2017)Eldridge, Belkin, and
  Wang]{eldridge2017unperturbed}
Justin Eldridge, Mikhail Belkin, and Yusu Wang.
\newblock Unperturbed: spectral analysis beyond davis-kahan.
\newblock \emph{arXiv preprint arXiv:1706.06516}, 2017.

\bibitem[Erd\"{o}s et~al.(2013)Erd\"{o}s, Knowles, Yau, and
  Yin]{erdos2013spectral}
L\'{a}szl\'{o} Erd\"{o}s, Antti Knowles, Horng-Tzer Yau, and Jun Yin.
\newblock Spectral statistics of erd\"{o}s--r\'{e}nyi graphs i: local
  semicircle law.
\newblock \emph{The Annals of Probability}, 41\penalty0 (3B):\penalty0
  2279--2375, 2013.

\bibitem[Fan et~al.(2018)Fan, Wang, and Zhong]{fan2018eigenvector}
Jianqing Fan, Weichen Wang, and Yiqiao Zhong.
\newblock An $\ell_{\infty}$ eigenvector perturbation bound and its application
  to robust covariance estimation.
\newblock \emph{Journal of Machine Learning Research}, 18\penalty0
  (207):\penalty0 1--42, 2018.

\bibitem[Fei and Chen(2018)]{fei2018exponential}
Yingjie Fei and Yudong Chen.
\newblock Exponential error rates of sdp for block models: Beyond
  grothendieck’s inequality.
\newblock \emph{IEEE Transactions on Information Theory}, 65\penalty0
  (1):\penalty0 551--571, 2018.

\bibitem[Gao et~al.(2017)Gao, Ma, Zhang, and Zhou]{gao2017achieving}
Chao Gao, Zongming Ma, Anderson~Y Zhang, and Harrison~H Zhou.
\newblock Achieving optimal misclassification proportion in stochastic block
  models.
\newblock \emph{The Journal of Machine Learning Research}, 18\penalty0
  (1):\penalty0 1980--2024, 2017.

\bibitem[Giesen and Mitsche(2005)]{giesen2005reconstructing}
Joachim Giesen and Dieter Mitsche.
\newblock Reconstructing many partitions using spectral techniques.
\newblock In \emph{International Symposium on Fundamentals of Computation
  Theory}, pages 433--444. Springer, 2005.

\bibitem[Gu{\'e}don and Vershynin(2016)]{guedon2016community}
Olivier Gu{\'e}don and Roman Vershynin.
\newblock Community detection in sparse networks via grothendieck’s
  inequality.
\newblock \emph{Probability Theory and Related Fields}, 165\penalty0
  (3-4):\penalty0 1025--1049, 2016.

\bibitem[Hajek et~al.(2016)Hajek, Wu, and Xu]{hajek2016achieving}
Bruce Hajek, Yihong Wu, and Jiaming Xu.
\newblock Achieving exact cluster recovery threshold via semidefinite
  programming: Extensions.
\newblock \emph{IEEE Transactions on Information Theory}, 62\penalty0
  (10):\penalty0 5918--5937, 2016.

\bibitem[Ipsen and Meyer(1994)]{ipsen1994uniform}
Ilse~CF Ipsen and Carl~D Meyer.
\newblock Uniform stability of markov chains.
\newblock \emph{SIAM Journal on Matrix Analysis and Applications}, 15\penalty0
  (4):\penalty0 1061--1074, 1994.

\bibitem[Jerrum and Sorkin(1998)]{jerrum1998metropolis}
Mark Jerrum and Gregory~B Sorkin.
\newblock The metropolis algorithm for graph bisection.
\newblock \emph{Discrete Applied Mathematics}, 82\penalty0 (1-3):\penalty0
  155--175, 1998.

\bibitem[Joseph and Yu(2016)]{joseph2016impact}
Antony Joseph and Bin Yu.
\newblock Impact of regularization on spectral clustering.
\newblock \emph{The Annals of Statistics}, 44\penalty0 (4):\penalty0
  1765--1791, 2016.

\bibitem[Kannan et~al.(2004)Kannan, Vempala, and Vetta]{kannan2004clusterings}
Ravi Kannan, Santosh Vempala, and Adrian Vetta.
\newblock On clusterings: Good, bad and spectral.
\newblock \emph{Journal of the ACM (JACM)}, 51\penalty0 (3):\penalty0 497--515,
  2004.

\bibitem[Kato(1949)]{kato1949convergence}
Tosio Kato.
\newblock On the convergence of the perturbation method. i.
\newblock \emph{Progress of Theoretical Physics}, 4:\penalty0 514--523, 1949.

\bibitem[Kato(2013)]{kato2013perturbation}
Tosio Kato.
\newblock \emph{Perturbation theory for linear operators}, volume 132.
\newblock Springer Science \& Business Media, 2013.

\bibitem[Lata{\l}a et~al.(2018)Lata{\l}a, van Handel, and
  Youssef]{latala2018dimension}
Rafa{\l} Lata{\l}a, Ramon van Handel, and Pierre Youssef.
\newblock The dimension-free structure of nonhomogeneous random matrices.
\newblock \emph{Inventiones mathematicae}, 214\penalty0 (3):\penalty0
  1031--1080, 2018.

\bibitem[Le and Levina(2015)]{le2015estimating}
Can~M Le and Elizaveta Levina.
\newblock Estimating the number of communities in networks by spectral methods.
\newblock \emph{arXiv preprint arXiv:1507.00827}, 2015.

\bibitem[Lei and Rinaldo(2015)]{lei2015consistency}
Jing Lei and Alessandro Rinaldo.
\newblock Consistency of spectral clustering in stochastic block models.
\newblock \emph{The Annals of Statistics}, 43\penalty0 (1):\penalty0 215--237,
  2015.

\bibitem[Li et~al.(2018{\natexlab{a}})Li, Lei, Bhattacharyya, Sarkar, Bickel,
  and Levina]{li2018hierarchical}
Tianxi Li, Lihua Lei, Sharmodeep Bhattacharyya, Purnamrita Sarkar, Peter~J
  Bickel, and Elizaveta Levina.
\newblock Hierarchical community detection by recursive bi-partitioning.
\newblock \emph{arXiv preprint arXiv:1810.01509}, 2018{\natexlab{a}}.

\bibitem[Li et~al.(2018{\natexlab{b}})Li, Chen, and Xu]{li2018convex}
Xiaodong Li, Yudong Chen, and Jiaming Xu.
\newblock Convex relaxation methods for community detection.
\newblock \emph{arXiv preprint arXiv:1810.00315}, 2018{\natexlab{b}}.

\bibitem[Lugosi et~al.(2018)Lugosi, Mendelson, and
  Zhivotovskiy]{lugosi2018concentration}
G{\'a}bor Lugosi, Shahar Mendelson, and Nikita Zhivotovskiy.
\newblock Concentration of the spectral norm of erd\"{o}s-r\'{e}nyi random
  graphs.
\newblock \emph{arXiv preprint arXiv:1801.02157}, 2018.

\bibitem[Mao et~al.(2017)Mao, Sarkar, and Chakrabarti]{mao2017estimating}
Xueyu Mao, Purnamrita Sarkar, and Deepayan Chakrabarti.
\newblock Estimating mixed memberships with sharp eigenvector deviations.
\newblock \emph{arXiv preprint arXiv:1709.00407}, 2017.

\bibitem[Massouli{\'e}(2014)]{massoulie2014community}
Laurent Massouli{\'e}.
\newblock Community detection thresholds and the weak ramanujan property.
\newblock In \emph{Proceedings of the forty-sixth annual ACM symposium on
  Theory of computing}, pages 694--703. ACM, 2014.

\bibitem[McSherry(2001)]{mcsherry2001spectral}
Frank McSherry.
\newblock Spectral partitioning of random graphs.
\newblock In \emph{Proceedings 42nd IEEE Symposium on Foundations of Computer
  Science}, pages 529--537. IEEE, 2001.

\bibitem[Mitra(2009)]{mitra2009entrywise}
Pradipta Mitra.
\newblock Entrywise bounds for eigenvectors of random graphs.
\newblock \emph{the electronic journal of combinatorics}, 16\penalty0
  (1):\penalty0 131, 2009.

\bibitem[Mossel et~al.(2014)Mossel, Neeman, and Sly]{mossel2014consistency}
Elchanan Mossel, Joe Neeman, and Allan Sly.
\newblock Consistency thresholds for binary symmetric block models.
\newblock \emph{arXiv preprint arXiv:1407.1591}, 3\penalty0 (5), 2014.

\bibitem[O'cinneide(1993)]{o1993entrywise}
Colm~Art O'cinneide.
\newblock Entrywise perturbation theory and error analysis for markov chains.
\newblock \emph{Numerische Mathematik}, 65\penalty0 (1):\penalty0 109--120,
  1993.

\bibitem[Oliveira(2009)]{oliveira2009concentration}
Roberto~Imbuzeiro Oliveira.
\newblock Concentration of the adjacency matrix and of the laplacian in random
  graphs with independent edges.
\newblock \emph{arXiv preprint arXiv:0911.0600}, 2009.

\bibitem[Oymak and Hassibi(2011)]{oymak2011finding}
Samet Oymak and Babak Hassibi.
\newblock Finding dense clusters via" low rank+ sparse" decomposition.
\newblock \emph{arXiv preprint arXiv:1104.5186}, 2011.

\bibitem[Paulin(2012)]{paulin2012concentration}
Daniel Paulin.
\newblock Concentration inequalities in locally dependent spaces.
\newblock \emph{arXiv preprint arXiv:1212.2013}, 2012.

\bibitem[Paulsen(2002)]{paulsen2002completely}
Vern Paulsen.
\newblock \emph{Completely bounded maps and operator algebras}, volume~78.
\newblock Cambridge University Press, 2002.

\bibitem[Rebrova(201)]{rebrova2018spectral}
Elizaveta Rebrova.
\newblock \emph{Spectral Properties of Heavy-Tailed Random Matrices}.
\newblock PhD thesis, University of Michigan, 201.

\bibitem[Rohe et~al.(2011)Rohe, Chatterjee, Yu, et~al.]{rohe2011spectral}
Karl Rohe, Sourav Chatterjee, Bin Yu, et~al.
\newblock Spectral clustering and the high-dimensional stochastic blockmodel.
\newblock \emph{The Annals of Statistics}, 39\penalty0 (4):\penalty0
  1878--1915, 2011.

\bibitem[Shamir and Tsur(2007)]{shamir2007improved}
Ron Shamir and Dekal Tsur.
\newblock Improved algorithms for the random cluster graph model.
\newblock \emph{Random Structures \& Algorithms}, 31\penalty0 (4):\penalty0
  418--449, 2007.

\bibitem[Shi and Malik(2000)]{shi2000normalized}
Jianbo Shi and Jitendra Malik.
\newblock Normalized cuts and image segmentation.
\newblock \emph{Pattern Analysis and Machine Intelligence, IEEE Transactions
  on}, 22\penalty0 (8):\penalty0 888--905, 2000.

\bibitem[Snijders and Nowicki(1997)]{snijders1997estimation}
Tom~AB Snijders and Krzysztof Nowicki.
\newblock Estimation and prediction for stochastic blockmodels for graphs with
  latent block structure.
\newblock \emph{Journal of classification}, 14\penalty0 (1):\penalty0 75--100,
  1997.

\bibitem[Spielman and Teng(1996)]{spielman1996spectral}
Daniel~A Spielman and Shang-Hua Teng.
\newblock Spectral partitioning works: Planar graphs and finite element meshes.
\newblock In \emph{Foundations of Computer Science, 1996. Proceedings., 37th
  Annual Symposium on}, pages 96--105. IEEE, 1996.

\bibitem[Stewart(1990)]{stewart1990matrix}
Gilbert~W Stewart.
\newblock \emph{Matrix perturbation theory}.
\newblock Citeseer, 1990.

\bibitem[Su et~al.(2019)Su, Wang, and Zhang]{su2019strong}
Liangjun Su, Wuyi Wang, and Yichong Zhang.
\newblock Strong consistency of spectral clustering for stochastic block
  models.
\newblock \emph{arXiv preprint arXiv:1710.06191}, 2019.

\bibitem[Vershynin(2010)]{vershynin2010introduction}
Roman Vershynin.
\newblock Introduction to the non-asymptotic analysis of random matrices.
\newblock \emph{arXiv preprint arXiv:1011.3027}, 2010.

\bibitem[Von~Luxburg(2007)]{von2007tutorial}
Ulrike Von~Luxburg.
\newblock A tutorial on spectral clustering.
\newblock \emph{Statistics and computing}, 17\penalty0 (4):\penalty0 395--416,
  2007.

\bibitem[Vu(2018)]{vu2018simple}
Van Vu.
\newblock A simple svd algorithm for finding hidden partitions.
\newblock \emph{Combinatorics, Probability and Computing}, 27\penalty0
  (1):\penalty0 124--140, 2018.

\bibitem[Weyl(1912)]{weyl1912asymptotische}
Hermann Weyl.
\newblock Das asymptotische verteilungsgesetz der eigenwerte linearer
  partieller differentialgleichungen (mit einer anwendung auf die theorie der
  hohlraumstrahlung).
\newblock \emph{Mathematische Annalen}, 71\penalty0 (4):\penalty0 441--479,
  1912.

\bibitem[Yu et~al.(2014)Yu, Wang, and Samworth]{yu2014useful}
Yi~Yu, Tengyao Wang, and Richard~J Samworth.
\newblock A useful variant of the davis--kahan theorem for statisticians.
\newblock \emph{Biometrika}, 102\penalty0 (2):\penalty0 315--323, 2014.

\bibitem[Yun and Proutiere(2014)]{yun2014accurate}
Se-Young Yun and Alexandre Proutiere.
\newblock Accurate community detection in the stochastic block model via
  spectral algorithms.
\newblock \emph{arXiv preprint arXiv:1412.7335}, 2014.

\bibitem[Yun and Proutiere(2016)]{yun2016optimal}
Se-Young Yun and Alexandre Proutiere.
\newblock Optimal cluster recovery in the labeled stochastic block model.
\newblock In \emph{Advances in Neural Information Processing Systems}, pages
  965--973, 2016.

\bibitem[Zhong and Boumal(2018)]{zhong2018near}
Yiqiao Zhong and Nicolas Boumal.
\newblock Near-optimal bounds for phase synchronization.
\newblock \emph{SIAM Journal on Optimization}, 28\penalty0 (2):\penalty0
  989--1016, 2018.

\end{thebibliography}
\bibliographystyle{plainnat}

\newpage
\appendix
\section{Proof of Theorem \ref{thm:generic_bound}}\label{app:generic_bound}
The proof is very involved, so we split the proof into six steps.

\subsection{Notation}\label{subsec:notation}
Let $\O^{r}$ denote the space of all $r\times r$ orthogonal matrices and $\one_{n}$ denote a $n$-dimensional vector with all entries $1$. 
For any vector $x$, let $\|x\|_{p}$ denotes its $p$-norm. For any matrix $M$, denote by $M_{k}^{T}$ the $m$-th row of $M$, by $\|M\|_{\op}$ its operator norm and by $\|M\|_{\mathrm{F}}$ its Frobenius norm. Moreover, for any $p, q\in [1, \infty]$, let
\[\|M\|_{p\rightarrow q} = \sup_{\omega: \|\omega\|_{p} = 1}\|M\omega\|_{q}.\]
In particular,
\[\mnorm{M} = \max_{k}\|M_{k}\|_{2}.\]
Suppose $U\Sigma V^{T}$ is the singular value decomposition of $M$. When $M$ is a square matrix, we define the matrix sign as
\[\sign(M) = UV^{T}.\]
By definition, $\sign(M)$ is orthogonal. When $n = 1$, $M$ is a scalar and $\sign(M)$ reduces to the classical sign of scalars. Further we denote by $\lambda_{\max}(M)$ (resp. $\lambda_{\min}(M)$) the largest (resp. the smallest) eigenvalue of $M$ in absolute values and by $\kappa(M)$ the condition number $\lambda_{\max}(M) / \lambda_{\min}(M)$. We say a square matrix $M$ positive semi-definite (psd) if all eigenvalues of $M$ are non-negative. In particular, we write $\lambda_{\max}(\Lambda^{*})$ (resp. $\lambda_{\min}(\Lambda^{*})$) as $\lambda_{\min}^{*}$ (resp. $\lambda_{\max}^{*}$) and $\kappa(\Lambda^{*})$ as $\kappa^{*}$ for short.

For any matrices $U, Z\in \R^{n\times r}$ with orthonormal columns, let $\Theta$ denote the principal angle matrix between the two subspaces spanned by $U$ and $Z$, such that $U^{T}Z$ has the singular value decomposition $U^{T}Z = \bar{U}(\cos \Theta) \bar{V}^{T}$ where $\Theta = \diag(\theta_{1}, \ldots, \theta_{r})$ with $\theta_{j}\in [0, \frac{\pi}{2}]$.

For any Hermitian matrices $B_{1}, B_{2}\in \R^{n\times n}$, let $\lambda_{1}(B_{j})\ge \lambda_{2}(B_{j})\ge \ldots \ge \lambda_{n}(B_{j})$ be the eigenvalues of $B_{j} (j = 1, 2)$. Let 
\[\sep_{s+1, s+r}(B_{1}, B_{2}) = \min\{|\lambda_{i}(B_{1}) - \lambda_{j}(B_{2})|: i\not\in \{s + 1, \ldots, s + r\}, j\in \{s + 1, \ldots, s + r\}\}.\]
Note that $\sep_{s+1, s+r}$ is not symmetric in the sense that $\sep_{s+1, s+r}(B_{1}, B_{2})\not = \sep_{s+1, s+r}(B_{2}, B_{1})$. When $B_{1} = B_{2} = B$, we write it as $\sep_{s + 1, s + r}(B)$ for short.

\subsection{Preparation: preliminary properties}\label{subapp:preparation}
When $r = 1$, $U$ and $U^{*}$ are vectors and it is straightforward to show that
\[d_{2\rightarrow \infty}(U, U^{*}) = \mnorm{U\sign(U^{T}U^{*}) - U^{*}}.\]
This motivates us to consider an upper bound of $d_{2\rightarrow \infty}(U, U^{*})$ as
\begin{equation}\label{eq:UUstar}
d_{2\rightarrow \infty}(U, U^{*}) \le \mnorm{U\sign(H) - U^{*}},
\end{equation}
where
\begin{equation}
  \label{eq:H}
  H = U^{T}U^{*}.
\end{equation}
Similarly for distance between $U$ and $AU^{*}(\Lambda^{*})^{-1}$, we consider the upper bound
\begin{equation}\label{eq:UAUstar}
  d_{2\rightarrow \infty}(U, AU^{*}(\Lambda^{*})^{-1}) \le \mnorm{U\sign(H) - AU^{*}(\Lambda^{*})^{-1}}.
\end{equation}
This was also considered in \cite{abbe2017entrywise}. Our goal is to derive upper bounds for \eqref{eq:UAUstar} and \eqref{eq:UUstar}.


Finally, let $A^{(1)}, \ldots, A^{(n)}$ be $n$ auxiliary matrices that satisfy the following condition, as the deterministic analogue of assumption \textbf{A}1 with $S = [r]$. 
\begin{enumerate}[\textbf{C}0]
\item There exists $L_{1}, L_{2}, L_{3}$ such that for all $k$,
\[\|A^{(k)} - A\|_{\op}\le L_{1}, \quad \frac{\|(A^{(k)} - A)U\|_{\op}}{\lambda_{\min}^{*}}\le \lb \kappa^{*}L_{2} + L_{3}\rb \mnorm{U}.\]
\end{enumerate}
Similarly we define $\Lambda^{(k)}$ as the diagonal matrix given by the $(s + 1)$-th to the $(s + r)$-th largest eigenvalues and $U^{(k)}\in \R^{n\times r}$ as a matrix of eigenvectors corresponding to $\Lambda^{(k)}$ i.e. 
\[A^{(k)}U^{(k)} = U^{(k)}\Lambda^{(k)}.\]
Further let
\[H^{(k)} = (U^{(k)})^{T}U^{*}.\]
The following proposition provides a simple yet important property of eigen-separation. 
\begin{proposition}\label{prop:separation}
  For any Hermitian matrices $B_{1}, B_{2}\in \R^{n\times n}$,
\[\sep_{s+1, s+r}(B_{1}, B_{2})\ge \max \{\sep_{s+1, s+r}(B_{1}), \sep_{s+1, s+r}(B_{2})\} - \max_{i\in [s + 1, s + r]}|\lambda_{i}(B_{1}) - \lambda_{i}(B_{2})|.\]
\end{proposition}
\begin{proof}
For any $i, j$,
\begin{align*}
  |\lambda_{i}(B_{1}) - \lambda_{j}(B_{2})|&\ge |\lambda_{i}(B_{2}) - \lambda_{j}(B_{2})| - |\lambda_{i}(B_{1}) - \lambda_{i}(B_{2})|.
\end{align*}
The proof is completed by considering all pairs of $i$ and $j$.
\end{proof}
Based on Proposition \ref{prop:separation}, we can derive the eigen-separation among $A$, $A^{*}$ and $A^{(k)}$.
\begin{lemma}\label{lem:separation}
Let $E$ be defined as in \eqref{eq:E}. Under condition \textbf{C}0, 
  \[\sep_{s+1, s+r}(A, A^{*})\ge \sep_{s+1, s+r}(A^{*}) - \maxnorm{\Lambda - \Lambda^{*}},\]
and for any $k$,
\[\sep_{s+1, s+r}(A^{(k)}, A)\ge \sep_{s+1, s+r}(A^{*}) - L_{1} - 2\maxnorm{\Lambda - \Lambda^{*}}.\]
\end{lemma}
\begin{proof}
The first part is a direct result of Proposition \ref{prop:separation}. By definition, 
\[\|A^{(k)} - A\|_{\op}\le L_{1}.\]
The second part is then proved by noting that
\[\sep_{s+1, s+r}(A^{(k)}, A)\ge \sep_{s+1, s+r}(A) - \|A^{(k)} - A\|_{\op}\]
where the last inequality uses Weyl's inequality and
\[\sep_{s+1, s+r}(A) \ge \sep_{s+1, s+r}(A, A^{*}) - \maxnorm{\Lambda - \Lambda^{*}} \ge \sep_{s+1, s+r}(A^{*}) - 2\maxnorm{\Lambda - \Lambda^{*}}.\]
\end{proof}

Recall the definition of $\gap^{*}$ in \eqref{eq:gap} and let
\begin{equation}\label{eq:tdgap}
\tdgap = \frac{\gap^{*} - L_{1}}{2}.
\end{equation}
Note that the first term of $\tdgap$ is essentially the half eigen-gap if $0$ is an eigenvalue but not in $\Lambda^{*}$. Under assumption \textbf{A}4, $\tdgap$ has the same order as $\gap^{*}$. Throughout the rest of this section, we assume the following condition:
\begin{enumerate}[\textbf{C}1]
\item $\maxnorm{\Lambda - \Lambda^{*}}\le \tdgap / 2$ where $E$ is defined in \eqref{eq:E} and $\tdgap$ is defined in \eqref{eq:tdgap}.
\end{enumerate}

\begin{corollary}\label{cor:separation}
Under condition \textbf{C}0 and \textbf{C}1,
\[\min\{\sep_{s+1, s+r}(A, A^{*}), \sep_{s+1, s+r}(A^{(k)}, A)\}\ge \tdgap.\]
\end{corollary}

The above results on eigen-gaps allow us to apply Davis-Kahan Theorem \citep{davis1970rotation} to bound the discrepancy between the eigenspaces of $A$ and $A^{*}$. In particular, we use the following version of Davis-Kahan Theorem.

\begin{proposition}\label{prop:davis_kahan}\citep[Chap. V, Theorem 3.6][]{stewart1990matrix}
For any Hermitian matrix $B\in \R^{n\times n}$, $M\in \R^{r\times r}$ and any matrix $Z\in \R^{n\times r}$ with orthonormal columns, let $B$ have the spectral decomposition
\[
  \begin{bmatrix}
    U^{T} \\ \td{U}^{T}
  \end{bmatrix} B
  \begin{bmatrix}
    U & \td{U}
  \end{bmatrix}
 =
  \begin{bmatrix}
    \Lambda & 0\\
    0 & \td{\Lambda}
  \end{bmatrix}.\]
Assume that there exists some $\omega > 0$ and $a, b\in \R$, 
\[\eig(M)\subset [a, b], \quad \eig(\td{\Lambda})\subset \R \setminus [a - \omega, b + \omega],\]
where $\eig(\cdot)$ denote the set of all eigenvalues. Further let 
\[R = BZ - ZM\]
and $\Theta$ be the principal angle matrix between $U$ and $Z$. Then for any unitarily invariant norm $\|\cdot\|$, 
\[\|\sin \Theta\|\le \frac{\|R\|}{\omega}.\]
\end{proposition}

\begin{proposition}\label{prop:sintheta}\citep[Chap. I, Theorem 5.5][]{stewart1990matrix}
Let $\Theta$ be the principal angle (matrix) between $U\in \R^{n\times r}$ and $Z\in \R^{n\times r} $, then
\[\|UU^{T} - ZZ^{T}\|_{\op} = \|\sin \Theta\|_{\op}.\]  
\end{proposition}

\begin{lemma}\label{lem:davis_kahan}
Let $E$ be defined as in \eqref{eq:E}. Under condition \textbf{C}0 and \textbf{C}1, 
\begin{equation}\label{eq:davis_kahan_UUstar}  
\|UU^{T} - U^{*}(U^{*})^{T}\|_{\op}\le \frac{\|EU^{*}\|_{\op}}{\tdgap},
\end{equation}
and for any $k$,
\begin{equation}\label{eq:davis_kahan_UmU} 
\|U^{(k)}(U^{(k)})^{T} - UU^{T}\|_{\op}\le \frac{\lambda_{\min}^{*}(\kappa^{*}L_{2} + L_{3})}{\tdgap}\mnorm{U}.
\end{equation}
\end{lemma}

\begin{proof}
First let $B = A, Z = U^{*}, M = \Lambda^{*}$ in Proposition \ref{prop:davis_kahan}. Then by Proposition \ref{prop:davis_kahan} and Proposition \ref{prop:sintheta},
\[\|UU^{T} - U^{*}(U^{*})^{T}\|_{\op}\le \frac{\|AU^{*} - U^{*}\Lambda^{*}\|_{\op}}{\sep_{s+1, s+r}(A, A^{*})} = \frac{\|AU^{*} - A^{*}U^{*}\|_{\op}}{\sep_{s+1, s+r}(A, A^{*})} = \frac{\|E U^{*}\|_{\op}}{\sep_{s+1, s+r}(A, A^{*})}.\]
The proof of the first part is completed by Corollary \ref{cor:separation}.

~\\
\noindent For the second part, let $B = A^{(k)}, Z = U, M = \Lambda$. Then by Proposition \ref{prop:sintheta} and Proposition \ref{prop:sintheta},
\[\|U^{(k)}(U^{(k)})^{T} - UU^{T}\|_{\op}\le \frac{\|A^{(k)}U - U\Lambda\|_{\op}}{\sep_{s+1, s+r}(A^{(k)}, A)} = \frac{\|(A^{(k)} - A)U\|_{\op}}{\sep_{s+1, s+r}(A^{(k)}, A)}.\]
The proof is completed by condition \textbf{C}0 and Corollary \ref{cor:separation}.
\end{proof}

\subsection{Step I: a preliminary deterministic bound}\label{subapp:step1}
Throughout this subsection we assume that all eigenvalues are of the same sign, i.e. $\lambda_{s + 1}^{*} \lambda_{s + r}^{*} > 0$. In step V we deal with the general case. 

\begin{lemma}\label{lem:step1}
Assume that $\lambda_{s + 1}^{*}\lambda_{s + r}^{*} > 0$. Under condition \textbf{C}0 and \textbf{C}1,
\begin{align}\label{eq:deterministic1}
  &\mnorm{(U\sign(H) - AU^{*}(\Lambda^{*})^{-1})}\nonumber\\
\le &\beta \mnorm{U} + \frac{\mnorm{A^{*}(UH - U^{*})}}{\lambda_{\min}^{*}} + \frac{\max_{k}\|E_{k}^{T}(U^{(k)}H^{(k)} - U^{*})\|_{2}}{\lambda_{\min}^{*}},
\end{align}
where 
\begin{equation}
  \label{eq:beta}
  \beta \triangleq \frac{\|EU^{*}\|_{\op}^{2}}{\tdgap^{2}} + \frac{\|EU^{*}\|_{\op}}{\lambda_{\min}^{*}} + \frac{\mnorm{E}(\kappa^{*}L_{2}(\delta) + L_{3}(\delta))}{\tdgap}.
\end{equation}
\end{lemma}
\begin{proof}
Without loss of generality we assume that $\lambda_{s + 1}^{*}\ge \lambda_{s + r}^{*} > 0$. Otherwise we replace $A$ (resp. $A^{*}, \Lambda, \Lambda^{*}$) by $-A$ (resp. $-A^{*}, -\Lambda, -\Lambda^{*}$). 

Applying the triangle inequality, we have
  \begin{equation*}
    \|(U\sign(H) - AU^{*}(\Lambda^{*})^{-1})_{k}\|_{2} \le \underbrace{\|U_{k}^{T}(\sign(H) - H)\|_{2}}_{J_{1}} + \|(UH - AU^{*}(\Lambda^{*})^{-1})_{k}\|_{2}.
  \end{equation*}
The second term can be further bounded as follows:
  \begin{align}
    \lefteqn{(UH - AU^{*}(\Lambda^{*})^{-1})_{k}^{T}} \nonumber\\
& = (UH\Lambda^{*}(\Lambda^{*})^{-1} - U\Lambda H(\Lambda^{*})^{-1})_{k}^{T} + (U\Lambda H(\Lambda^{*})^{-1} - AU^{*}(\Lambda^{*})^{-1})_{k}^{T}\nonumber\\
& \stackrel{(i)}{=} U_{k}^{T}(H\Lambda^{*} - \Lambda H)(\Lambda^{*})^{-1} + (AUH - AU^{*})_{k}^{T}(\Lambda^{*})^{-1}\nonumber\\
& = \left\{U_{k}^{T}(H\Lambda^{*} - \Lambda H) + A_{k}^{T}(UH - U^{*})\right\}(\Lambda^{*})^{-1}\nonumber\\
& = \left\{U_{k}^{T}(H\Lambda^{*} - \Lambda H)  + E_{k}^{T}(UH - U^{*}) + (A_{k}^{*})^{T}(UH - U^{*})\right\}(\Lambda^{*})^{-1}\nonumber\\
& = \left\{U_{k}^{T}(H\Lambda^{*} - \Lambda H)  + E_{k}^{T}(UH - U^{(k)}H^{(k)}) + E_{k}^{T}(U^{(k)}H^{(k)} - U^{*}) + (A_{k}^{*})^{T}(UH - U^{*})\right\}(\Lambda^{*})^{-1}\label{eq:main_trick}
\end{align}
where (i) uses the fact that $U\Lambda = AU$. Applying the triangle inequality again we obtain that
\begin{align*}
  \|(UH - AU^{*}(\Lambda^{*})^{-1})_{k}\|_{2}& \le \frac{1}{\lambda_{\min}^{*}}\bigg\{\underbrace{\|U_{k}^{T}(H\Lambda^{*} - \Lambda H)\|_{2}}_{J_{2}} + \underbrace{\|E_{k}^{T}(UH - U^{(k)}H^{(k)})\|_{2}}_{J_{3}}\\
& \qquad \qquad  + \|E_{k}^{T}(U^{(k)}H^{(k)} - U^{*})\|_{2} + \|(A_{k}^{*})^{T}(UH - U^{*})\|_{2}\bigg\}.
\end{align*}
We will derive bounds for $J_{1}$, $J_{2}$ and $J_{3}$ separately in the rest of the proof. 

~\\
\noindent \textbf{Step 1: Bounding $J_{1}$.} Let $H$ have the singular value decomposition $H = \bar{U}(\cos \Theta)\bar{V}^{T}$. Then
\begin{align*}
  \|H - \sign(H)\|_{\op} & = \|\bar{U} (I - \cos \Theta)\bar{V}^{T}\|_{\op} \le \|I - \cos \Theta\|_{\op}.
\end{align*}
For any $\theta\le \pi / 2$, 
\[1 - \cos\theta \le 1 - \cos^{2}\theta = \sin^{2}\theta.\]
Thus,
\[\|I - \cos \Theta\|_{\op}\le \|\sin \Theta\|_{\op}^{2}.\]
By Proposition \ref{prop:sintheta}, 
\[\|\sin \Theta\|_{\op}^{2} = \|UU^{T} - U^{*}(U^{*})^{T}\|_{\op}^{2}.\]
Finally by Lemma \ref{lem:davis_kahan} we obtain that
\begin{equation}
  \label{eq:J1}
  J_{1}\le \|U_{k}\|_{2}\|H - \sign(H)\|_{\op}\le \frac{\|EU^{*}\|_{\op}^{2}}{\tdgap^{2}}\|U_{k}\|_{2} \le \frac{\|EU^{*}\|_{\op}^{2}}{\tdgap^{2}}\mnorm{U}.
\end{equation}

~\\
\noindent \textbf{Step 2: Bounding $J_{2}$.} By definition, $U^{T}A = (A^{T}U)^{T} = (AU)^{T} = (U\Lambda)^{T} = \Lambda U^{T}$ and $U^{*}\Lambda^{*} = A^{*}U^{*}$. Thus,
\[H\Lambda^{*} - \Lambda H = \Lambda U^{T}U^{*} - U^{T}U^{*}\Lambda^{*} = U^{T}AU^{*} - U^{T}A^{*}U^{*} = U^{T}EU^{*}.\]
Since $U$ and $U^{*}$ have orthonormal columns,
\[\|H\Lambda^{*} - \Lambda H\|_{\op} \le \|EU^{*}\|_{\op}.\]
Thus, 
\begin{equation}
  \label{eq:J2}
  J_{2}\le \|EU^{*}\|_{\op} \|U_{k}\|_{2}\le \|EU^{*}\|_{\op} \mnorm{U}.
\end{equation}

~\\
\noindent \textbf{Step 3: Bounding $J_{3}$.} Since $H^{(k)} = (U^{(k)})^{T}U^{*}$ and $U^{*}$ has orthonormal columns,
\begin{equation}
  \label{eq:UmU_H}
  \|U^{(k)}H^{(k)} - UH\|_{\op} = \|U^{(k)}(U^{(k)})^{T}U^{*} - UU^{T}U^{*}\|_{\op} \le \|U^{(k)}(U^{(k)})^{T} - UU^{T}\|_{\op}.
\end{equation}
By Lemma \ref{lem:davis_kahan},
\begin{equation}
  \label{eq:J3}
  J_{3}\le \frac{\|E_{k}\|_{2}\lambda_{\min}^{*}( \kappa^{*}L_{2} + L_{3})}{\tdgap}\|U_{k}\|_{2}\le \frac{\mnorm{E}\lambda_{\min}^{*}( \kappa^{*}L_{2} + L_{3})}{\tdgap}\mnorm{U}.
\end{equation}
The proof is then completed by combining \eqref{eq:J1}, \eqref{eq:J2} and \eqref{eq:J3}.
\end{proof}

\subsection{Step II: deterministic bound for $\|A^{*}(UH - U^{*})\|_{2}$ via Kato's integral}\label{subapp:step2}

\begin{lemma}\label{lem:step2}
Assume that $\lambda_{s + 1}^{*}\lambda_{s + r}^{*} > 0$. Under condition \textbf{C}1,
  \begin{itemize}
  \item It always holds that
    \begin{equation}\label{eq:case1}
      \mnorm{A^{*}(UH - U^{*})}\le \frac{\|EU^{*}\|_{\op}\mnorm{A^{*}}}{\tdgap}.
    \end{equation}
  \item If $A^{*}$ is positive semidefinite, then 
    \begin{equation}\label{eq:case2}
      \mnorm{A^{*}(UH - U^{*})}\le 3.61\frac{\|E\bar{U}^{*}\|_{\op}\sqrt{\lambda_{\max}^{*}\maxnorm{A^{*}}}}{\tdgap}.
    \end{equation}
  \item If $A^{*}$ is low-rank (with rank $K$) with $\bar{U}^{*}$ being defined in \eqref{eq:barU} in page \pageref{eq:barU}, then
    \begin{equation}
      \label{eq:case3}
      \mnorm{A^{*}(UH - U^{*})}\le 3.84\frac{\|E\bar{U}^{*}\|_{\op}\lambda_{\max}^{*}}{\tdgap}\mnorm{\bar{U}^{*}}.
    \end{equation}
  \end{itemize}
\end{lemma}

\begin{proof}
First we notice that 
\[\mnorm{A^{*}(UH - U^{*})} = \mnorm{A^{*}(UU^{T} - U^{*}(U^{*})^{T})U^{*}}\le \mnorm{A^{*}(UU^{T} - U^{*}(U^{*})^{T})}.\]
We derive bounds for $\mnorm{A^{*}(UU^{T} - U^{*}(U^{*})^{T})}$ in each case separately.

\noindent \textbf{Case 1: $A^{*}$ has no constraint}

By Lemma \ref{lem:davis_kahan}, 
\[\|UU^{T} - U^{*}(U^{*})^{T}\|_{\op}\le \frac{\|EU^{*}\|_{\op}}{\tdgap}.\]
Thus,
\[\mnorm{A^{*}(UU^{T} - U^{*}(U^{*})^{T})}\le \frac{\|UU^{T} - U^{*}(U^{*})^{T}\|_{\op}\mnorm{A^{*}}}{\tdgap}\le \frac{\|EU^{*}\|_{\op}\mnorm{A^{*}}}{\tdgap}.\]

\noindent \textbf{Case 2: $A^{*}$ is positive semidefinite}

Recall that $\bar{U}^{*}\bar{\Lambda}^{*}(\bar{U}^{*})^{T}$ is the singular value decomposition of $A^{*}$. Then
\[\mnorm{A^{*}(UU^{T} - U^{*}(U^{*})^{T})} \le \mnorm{\bar{U}^{*}(\bar{\Lambda}^{*})^{1/2}}\|(\bar{\Lambda}^{*})^{1/2}(\bar{U}^{*})^{T}\lb UU^{T} - U^{*}(U^{*})^{T}\rb\|_{\op}.\]
Let $Q = \bar{U}^{*}(\bar{\Lambda}^{*})^{1/2}$ with $Q_{i}$ being the $i$-th row. Then $A^{*} = QQ^{T}$ and hence 
\[\max_{i}\|Q_{i}\|_{2} = \max_{i}\sqrt{A_{ii}^{*}} = \sqrt{\maxnorm{A^{*}}}.\]
As a result,
\begin{equation}\label{eq:case2_pre}
  \mnorm{A^{*}(UU^{T} - U^{*}(U^{*})^{T})} \le \sqrt{\maxnorm{A^{*}}}\|(\bar{\Lambda}^{*})^{1/2}(\bar{U}^{*})^{T}\lb UU^{T} - U^{*}(U^{*})^{T}\rb\|_{\op}.
\end{equation}
Thus it is left to bound $\|(\bar{\Lambda}^{*})^{1/2}\bar{U}^{*}\lb UU^{T} - U^{*}(U^{*})^{T}\rb\|_{\op}$.

WLOG, we assume that all eigenvalues are positive. For convenience, write 
\begin{equation}\label{eq:abh}
  a = \lambda_{\min}^{*}, \quad b = \lambda_{\max}^{*}, \quad h = \frac{\tdgap}{2} = \frac{\gap^{*} - L_{1}}{4}.
\end{equation}
Further we write 
\begin{equation}\label{eq:abprime}
  a' = a - 2h, \quad b' = b + 2h.
\end{equation}
Note that 
\[h\le \frac{1}{4}\gap^{*} = \frac{1}{4}\min\{\sep_{s+1, s+r}(A^{*}), \lambda_{\min}^{*}\},\]
and thus all eigenvalues of $A^{*}$, as well as $0$, are at least $2h$ apart from $a'$ and $b'$. Fix any $\gamma > 0$. Let $\C$ be a positively oriented rectangular contour on the complex plane with corners $a'\pm \gamma\i$ and $b' \pm \gamma \i$. Then all eigenvalues in $\Lambda^{*}$ are inside $\C$ while all other eigenvalues are outside $\C$. By Assumption \textbf{C}1, $\C$ also separates the eigenvalues in $\Lambda$. The famous Kato's integral \citep{kato1949convergence} implies that
\[UU^{T} = \frac{1}{2\pi \i}\oint_{\C}(A - z I)^{-1}dz, \quad U^{*}(U^{*})^{T} = \frac{1}{2\pi \i}\oint_{\C}(A^{*} - z I)^{-1}dz.\]
Noting that $C^{-1} - B^{-1} = -B^{-1}(C - B)C^{-1}$, we have
\begin{equation}\label{eq:kato_generic}
  UU^{T} - U^{*}(U^{*})^{T} = -\frac{1}{2\pi\i}\oint_{\C}(A^{*} - z I)^{-1}E(A - z I)^{-1}dz.
\end{equation}
Let $A = \bar{U}\bar{\Lambda}\bar{U}^{T}$ be the SVD of $A$ where $\bar{U}\in \R^{n\times n}$ be an orthogonal matrix. Then 
\begin{align}
  &(\bar{\Lambda}^{*})^{1/2}\bar{U}^{*}\lb UU^{T} - U^{*}(U^{*})^{T}\rb = -\frac{1}{2\pi\i}(\bar{\Lambda}^{*})^{1/2}\bar{U}^{*}\oint_{\C}(A^{*} - z I)^{-1}E(A - z I)^{-1}dz\nonumber\\
= & -\frac{1}{2\pi\i}(\bar{\Lambda}^{*})^{1/2}(\bar{U}^{*})^{T}\oint_{\C}\bar{U}^{*}(\bar{\Lambda}^{*} - z I)^{-1}(\bar{U}^{*})^{T}E\bar{U}(\bar{\Lambda} - z I)^{-1}\bar{U}^{T}dz\nonumber\\
= & -\frac{1}{2\pi\i}\lb \oint_{\C}(\bar{\Lambda}^{*})^{1/2}(\bar{\Lambda}^{*} - z I)^{-1}(\bar{U}^{*})^{T}E\bar{U}(\bar{\Lambda} - z I)^{-1} dz\rb \bar{U}^{T}.\label{eq:case2_Kato}
\end{align}
Let $q(z)$ be the integrand and $\C_1, \C_2, \C_3, \C_4$ be the four edges of $\C$, i.e.
\[\C_{1} = \{a' + x\i: x\in [-\gamma, \gamma]\}, \C_{2} = \{y - \gamma \i: y\in [a', b']\},\] 
\[\C_{3} = \{b' + x\i: x\in [-\gamma, \gamma]\}, \C_4 = \{y + \gamma \i: y\in [a', b']\}.\]
Then 
\begin{align}
  &\|(\bar{\Lambda}^{*})^{1/2}\bar{U}^{*}\lb UU^{T} - U^{*}(U^{*})^{T}\rb\|_{\op}\nonumber\\
\le &\frac{1}{2\pi}\lb\left\|\oint_{\C_1}q(z)dz\right\|_{\op} + \left\|\oint_{\C_2}q(z)dz\right\|_{\op} + \left\|\oint_{\C_3}q(z)dz\right\|_{\op} + \left\|\oint_{\C_4}q(z)dz\right\|_{\op}\rb.  \label{eq:C1C2C3C4}
\end{align}
We note that Kato's integral is also deployed by \cite[][Lemma A.2]{oliveira2009concentration} and \cite[][Section 5]{mao2017estimating}. However, they directly bound the above quantity by
\[\max_{z\in \C} |q(z)| \times (\mbox{the perimeter of }\C).\]
This turns out to be loose. Instead we will bound each term in \eqref{eq:C1C2C3C4} separately.

We start from $\left\|\oint_{\C_1}q(z)dz\right\|_{\op}$. Since $\C_1$ is a vertical line with intercept $a'$, we have
\begin{align*}
  \left\|\oint_{\C_1}q(z)dz\right\|_{\op} = \left\|\int_{-\gamma}^{\gamma}q(a' + x\i)dx\right\|_{\op}\le \int_{-\gamma}^{\gamma}\|q(a' + x\i)\|_{\op}dx.
\end{align*}
Because $\|\cdot\|_{\op}$ is sub-multiplicative,
\begin{align}
  \|q(a' + x\i)\|_{\op}&\le \|(\bar{\Lambda}^{*})^{1/2}(\bar{\Lambda}^{*} - (a' + x\i) I)^{-1}\|_{\op}\|(\bar{U}^{*})^{T}E\bar{U}\|_{\op}\|(\bar{\Lambda} - (a' + x\i) I)^{-1}\|_{\op}\nonumber\\
& \le \|E\bar{U}^{*}\|_{\op}\max_{i\in [n]}\frac{1}{\sqrt{(\lambda_{i} - a')^{2} + x^{2}}}\max_{i\in [n]} \sqrt{\frac{|\lambda_{i}^{*}|}{(\lambda_{i}^{*} - a')^{2} + x^{2}}} \label{eq:doublemax1}
\end{align}
We emphasize that the maximum is taken over all eigenvalues instead of just the eigenvalues in $\Lambda$ and $\Lambda^{*}$. By Assumption \textbf{C}1,
\[|\lambda_{i} - \lambda_{i}^{*}|\le \maxnorm{\Lambda - \Lambda^{*}}\le \frac{\tdgap}{2} = h.\]
By construction, for any $i$,
\[|\lambda_{i}^{*} - a'| = |\lambda_{i}^{*} - a + 2h| \ge \min\{2h, \sep_{s+1, s+r}(A^{*}) - 2h\} = 2h.\]
By the triangle inequality we find that for any $i$,
\[|\lambda_{i} - a'|\ge h.\]
Therefore,
\begin{equation}
  \label{eq:doublemax2}
  \max_{i\in [n]}\frac{1}{\sqrt{(\lambda_{i} - a')^{2} + x^{2}}} \le \frac{1}{\sqrt{h^{2} + x^{2}}}.
\end{equation}
On the other hand, let 
\begin{equation}\label{eq:ratio_func}
g(z; x, a') = \frac{|z + a'|}{\sqrt{z^{2} + x^{2}}}.
\end{equation}
Then 
\begin{align}
  \max_{i\in [n]} \sqrt{\frac{|\lambda_{i}^{*}|}{(\lambda_{i}^{*} - a')^{2} + x^{2}}}&\le \max_{i\in [n]}\sqrt{g(\lambda_{i}^{*} - a'; x, a')}\max_{i\in [n]}\frac{1}{((\lambda_{i}^{*} - a')^{2} + x^{2})^{1/4}}\nonumber\\
&\le \max_{i\in [n]}\sqrt{g(\lambda_{i}^{*} - a'; x, a')}\frac{1}{(h^{2} + x^{2})^{1/4}}\nonumber\\
& \le \sup_{z: |z|\ge h} \sqrt{g(z; x, a')}\frac{1}{(h^{2} + x^{2})^{1/4}}\label{eq:doublemax3}.
\end{align}
where the last step uses the fact that $|\lambda_{i}^{*} - a'|\ge 2h\ge h$. Now we study the properties of $g(z; x, a')$. Note that when $z \not =  -a'$,
\begin{equation}
  \label{eq:gderiv}
  \frac{d}{dz}(\log g(z; x, a')) = \frac{1}{z + a'} - \frac{z}{z^{2} + x^{2}} = \frac{x^{2} - a'z}{(z^{2} + x^{2})(z + a')}.
\end{equation}
Since $a' = a - 2h\ge 2h > 0$, $g(z; x, a')$ is decreasing on $(-\infty, -a']$, increasing on $(-a', \frac{x^{2}}{a'}]$ and decreasing on $[\frac{x^{2}}{a'}, \infty)$. As a result, 
\begin{equation}\label{eq:gzxa}
 \sup_{z: |z|\ge h} g(z; x, a')\le \left\{
    \begin{array}{ll}
     \max(g(-\infty; x, a'), g(\pm h; x, a')\} & (\frac{x^{2}}{a'}\le h)\\
      g\lb\frac{x^{2}}{a'}; x, a'\rb & (\frac{x^{2}}{a'} > h)
    \end{array}
\right. = \left\{
    \begin{array}{ll}
     \frac{a' + h}{\sqrt{h^{2} + x^{2}}} & (\frac{x^{2}}{a'}\le h)\\
      \frac{\sqrt{a'^{2} + x^{2}}}{x} & (\frac{x^{2}}{a'} > h)
    \end{array}
\right..
\end{equation}
When $x^{2} / a' > h$, $\frac{\sqrt{a'^{2} + x^{2}}}{x}\le \sqrt{\frac{a' + h}{h}}$. Therefore
\begin{equation}
  \label{eq:gbound}
  \sup_{z: |z|\ge h} g(z; x, a')\le \frac{a' + h}{\sqrt{h^{2} + x^{2}}}I(|x|\le \sqrt{a'h}) + \sqrt{\frac{a' + h}{h}}I(|x| > \sqrt{a'h}).
\end{equation}
Putting \eqref{eq:doublemax1}, \eqref{eq:doublemax2}, \eqref{eq:doublemax3} and \eqref{eq:gbound} together, we obtain that
\begin{align}
  &\|q(a' + x\i)\|_{\op}\nonumber\\
\le &\|E\bar{U}^{*}\|_{\op}\frac{\sqrt{a' + h}}{h^{2} + x^{2}}I(|x|\le \sqrt{a'h}) + \|E\bar{U}^{*}\|_{\op}\lb\frac{a' + h}{h}\rb^{1/4}\frac{1}{(h^{2} + x^{2})^{3/4}}I(|x| > \sqrt{a'h}).  \label{eq:qbound}
\end{align}
As a consequence,
\begin{align}
   & \left\|\oint_{\C_1}q(z)dz\right\|_{\op}\le \int_{-\infty}^{\infty}\|q(a' + x\i)\|_{\op}dx\nonumber\\
\le & \|E\bar{U}^{*}\|_{\op}\sqrt{a' + h}\int_{|x| \le \sqrt{a'h}}\frac{dx}{h^{2} + x^{2}} + \|E\bar{U}^{*}\|_{\op}\lb\frac{a' + h}{h}\rb^{1/4}\int_{|x| > \sqrt{a'h}}\frac{dx}{(h^{2} + x^{2})^{3/4}}\nonumber\\
= & \frac{2\sqrt{a' + h}\|E\bar{U}^{*}\|_{\op}}{h}\int_{0}^{\sqrt{a' / h}}\frac{dy}{1 + y^{2}} + \frac{2(a' + h)^{1/4}\|E\bar{U}^{*}\|_{\op}}{h^{3/4}}\int_{\sqrt{a' / h}}^{\infty}\frac{dy}{(1 + y^{2})^{3/4}}\nonumber\\
\le & \frac{2\sqrt{a' + h}\|E\bar{U}^{*}\|_{\op}}{h}\int_{0}^{\infty}\frac{dy}{1 + y^{2}} + \frac{2(a' + h)^{1/4}\|E\bar{U}^{*}\|_{\op}}{h^{3/4}}\int_{\sqrt{a' / h}}^{\infty}\frac{dy}{y^{3/2}}\nonumber\\
= & \frac{\|E\bar{U}^{*}\|_{\op}}{h}\lb \pi\sqrt{a' + h} + \frac{4(a' + h)^{1/4}\sqrt{h}}{a'^{1/4}}\rb.\nonumber
\end{align}
Similarly, 
\begin{equation*}
  \left\|\oint_{\C_3}q(z)dz\right\|_{\op}\le \frac{\|E\bar{U}^{*}\|_{\op}}{h}\lb \pi\sqrt{b' + h} + \frac{4(b' + h)^{1/4}\sqrt{h}}{b'^{1/4}}\rb.
\end{equation*}
Since $x\mapsto (x + h)^{1/2}$ is concave and also $a' < a \le b < b'$, we have
\begin{equation}\label{eq:sqrt_concavity}
  \sqrt{a' + h} + \sqrt{b' + h}\le \sqrt{a + h} + \sqrt{b + h} \le 2\sqrt{b + h}.
\end{equation}
Since $h \le a / 4\le b / 4$, we have $h\le a' / 2$, $h \le b' / 6$. Then
\begin{align}
  \left\|\oint_{\C_1}q(z)dz\right\|_{\op} + \left\|\oint_{\C_3}q(z)dz\right\|_{\op}&\le \frac{\|E\bar{U}^{*}\|_{\op}}{h}\lb 2\pi\sqrt{b + h} + \frac{4(a' + h)^{1/4}\sqrt{h}}{a'^{1/4}} + \frac{4(b' + h)^{1/4}\sqrt{h}}{b'^{1/4}}\rb\nonumber\\
& \le \frac{\|E\bar{U}^{*}\|_{\op}}{h}\lb 2\pi\sqrt{\frac{5}{4}}\sqrt{b} + 4\lb\frac{3}{2}\rb^{1/4}\sqrt{h} + 4\lb\frac{7}{6}\rb^{1/4}\sqrt{h}\rb\nonumber\\
& \le \frac{\|E\bar{U}^{*}\|_{\op}}{h}\lb 2\pi\sqrt{\frac{5}{4}}\sqrt{b} + 2\lb\frac{3}{2}\rb^{1/4}\sqrt{b} + 2\lb\frac{7}{6}\rb^{1/4}\sqrt{b}\rb\nonumber\\
& \le \frac{11.32\|E\bar{U}^{*}\|_{\op}\sqrt{b}}{h}.  \label{eq:C1C3}
\end{align}

Note that \eqref{eq:C1C3} is independent of $\gamma$. For the integral on $\C_2$, we use the crude bound that
\[\left\|\oint_{\C_{2}}q(z)dz\right\|_{\op} \le |\C_2|\max_{z\in \C_2}\|q(z)\|_{\op} = (b' - a')\max_{z\in \C_2}\|q(z)\|_{\op}.\]
By \eqref{eq:doublemax1}, for any $y\in \R$,
\[\|q(y + \gamma \i)\|_{\op}\le \frac{\|E\bar{U}^{*}\|_{\op}\max_{i}\sqrt{\lambda_{i}^{*}}}{\gamma^{2}}.\]
Letting $\gamma \rightarrow \infty$,
\begin{equation}
  \label{eq:C2C4}
  \left\|\oint_{\C_{2}}q(z)dz\right\|_{\op} + \left\|\oint_{\C_{4}}q(z)dz\right\|_{\op} \rightarrow 0.
\end{equation}
Putting \eqref{eq:C1C2C3C4}, \eqref{eq:C1C3} and \eqref{eq:C2C4} together and recalling that $\tdgap = 2h$, we conclude that
\begin{align*}
  & \|(\bar{\Lambda}^{*})^{1/2}\bar{U}^{*}\lb UU^{T} - U^{*}(U^{*})^{T}\rb\|_{\op} \le \frac{11.32\|E\bar{U}^{*}\|_{\op}\sqrt{b}}{2\pi h} \le 3.61\frac{\|E\bar{U}^{*}\|_{\op}\sqrt{\lambda_{\max}^{*}}}{\tdgap}.
\end{align*}
The proof is then completed by \eqref{eq:case2_pre}.

\noindent \textbf{Case 3: $A^{*}$ is low rank}

Note that 
  \begin{equation}
    \label{eq:case3_pre}
  \mnorm{A^{*}(UU^{T} - U^{*}(U^{*})^{T})} \le \mnorm{\bar{U}^{*}}\|\bar{\Lambda}^{*}(\bar{U}^{*})^{T}\lb UU^{T} - U^{*}(U^{*})^{T}\rb\|_{\op}.  
  \end{equation}
Thus it is left to bound $\|\bar{\Lambda}^{*}(\bar{U}^{*})^{T}\lb UU^{T} - U^{*}(U^{*})^{T}\rb\|_{\op}$.

Similar to \eqref{eq:case2_Kato}, we deduce that
\begin{equation}
  \bar{\Lambda}^{*}\bar{U}^{*}\lb UU^{T} - U^{*}(U^{*})^{T}\rb = -\frac{1}{2\pi\i}\lb \oint_{\C}\bar{\Lambda}^{*}(\bar{\Lambda}^{*} - z I)^{-1}\bar{U}^{*}E\bar{U}^{T}(\bar{\Lambda} - z I)^{-1} dz\rb \bar{U}.\label{eq:case3_Kato}
\end{equation}
Let $\td{q}(z)$ be the integrand and $a', b',\C_1, \C_2, \C_3, \C_4$ and $g(z; x, a')$ be defined as in \eqref{eq:ratio_func} in Case 2. Throughout this part we assume that $\gamma \ge h$. Similar to \eqref{eq:doublemax1} and by \eqref{eq:gbound},
\begin{align}
  &\|\td{q}(a' + x\i)\|_{\op} \le \|E\bar{U}^{*}\|_{\op}\max_{i\in [n]}\frac{1}{\sqrt{(\lambda_{i} - a')^{2} + x^{2}}}\max_{i\in [n]} \frac{|\lambda_{i}^{*}|}{\sqrt{(\lambda_{i}^{*} - a')^{2} + x^{2}}}\nonumber\\
& \le \|E\bar{U}^{*}\|_{\op}\frac{1}{\sqrt{h^{2} + x^{2}}}\sup_{z: |z|\ge h} g(z; x, a')\nonumber\\
& \le \|E\bar{U}^{*}\|_{\op}\frac{a' + h}{h^{2} + x^{2}}I(|x|\le \sqrt{a'h}) + \|E\bar{U}^{*}\|_{\op}\sqrt{\frac{a' + h}{h (h^{2} + x^{2})}}I(|x| > \sqrt{a'h}). \label{eq:doublemax1_case4}
\end{align}
As a consequence, 
\begin{align}
   & \left\|\oint_{\C_1}\td{q}(z)dz\right\|_{\op}\le \int_{-\gamma}^{\gamma}\|\td{q}(a' + x\i)\|_{\op}dx\nonumber\\
\le & \|E\bar{U}^{*}\|_{\op}(a' + h)\int_{|x| \le \sqrt{a'h}}\frac{dx}{h^{2} + x^{2}} + \|E\bar{U}^{*}\|_{\op}\sqrt{\frac{a' + h}{h}}\int_{\gamma \ge |x| > \sqrt{a'h}}\frac{dx}{\sqrt{h^{2} + x^{2}}}\nonumber\\
= & \frac{2 (a' + h)\|E\bar{U}^{*}\|_{\op}}{h}\int_{0}^{\sqrt{a' / h}}\frac{dy}{1 + y^{2}} + \frac{2\sqrt{a' + h}\|E\bar{U}^{*}\|_{\op}}{\sqrt{h}}\int_{\sqrt{a' / h}}^{\gamma / h}\frac{dy}{\sqrt{1 + y^{2}}}\nonumber\\
\le & \frac{2 (a' + h)\|E\bar{U}^{*}\|_{\op}}{h}\int_{0}^{\infty}\frac{dy}{1 + y^{2}} + \frac{2\sqrt{a' + h}\|E\bar{U}^{*}\|_{\op}}{\sqrt{h}}\int_{1}^{\gamma / h}\frac{dy}{\sqrt{1 + y^{2}}}\nonumber\\
= & \frac{\|E\bar{U}^{*}\|_{\op}}{h}\lb \pi (a' + h) + 2\sqrt{a' + h}\sqrt{h}\log (y + \sqrt{1 + y^{2}})\bigg|_{1}^{\gamma / h}\rb\nonumber\\
= & \frac{\|E\bar{U}^{*}\|_{\op}}{h}\lb \pi (a' + h) + 2\sqrt{a' + h}\sqrt{h}\lb\log (\gamma / h + \sqrt{1 + \gamma^{2} / h^{2}}) - \log (1 + \sqrt{2})\rb\rb\nonumber.
\end{align}
Since $\gamma \ge h$, we have
\[\log (\gamma / h + \sqrt{1 + \gamma^{2} / h^{2}})\le \log (\gamma / h + \sqrt{2\gamma^{2} / h^{2}}) = \log\lb\frac{\gamma}{h}\rb + \log (1 + \sqrt{2}).\]
Thus we have
\[\left\|\oint_{\C_1}\td{q}(z)dz\right\|_{\op}\le \int_{-\gamma}^{\gamma}\|\td{q}(a' + x\i)\|_{\op}dx\le \frac{\|E\bar{U}^{*}\|_{\op}}{h}\lb \pi (a' + h) + 2\sqrt{a' + h}\sqrt{h}\log \lb\frac{\gamma}{h}\rb\rb.\]
Similarly, 
\[\left\|\oint_{\C_3}\td{q}(z)dz\right\|_{\op}\le \frac{\|E\bar{U}^{*}\|_{\op}}{h}\lb \pi (b' + h) + 2\sqrt{b' + h}\sqrt{h}\log \lb\frac{\gamma}{h}\rb\rb.\]
We recall \eqref{eq:sqrt_concavity} that 
\[\sqrt{a' + h} + \sqrt{b' + h}\le \sqrt{a + h} + \sqrt{b + h} \le 2\sqrt{b + h}.\]
Also recalling that $h\le a / 4 \le b / 4$, we have that
\begin{align}
  \left\|\oint_{\C_1}\td{q}(z)dz\right\|_{\op} + \left\|\oint_{\C_3}\td{q}(z)dz\right\|_{\op}&\le \frac{\|E\bar{U}^{*}\|_{\op}}{h}\lb \pi(a' + b' + 2h) + 4\sqrt{b + h}\sqrt{h}\log \lb\frac{\gamma}{h}\rb\rb\nonumber\\
& \le \frac{\|E\bar{U}^{*}\|_{\op}}{h}\lb 2\pi(b + h) + 4\sqrt{b + h}\sqrt{h}\log \lb\frac{\gamma}{h}\rb\rb\nonumber\\
& \le \frac{\|E\bar{U}^{*}\|_{\op}}{h}\lb \frac{5\pi}{2}b + 4\sqrt{\frac{5}{4}}\sqrt{hb}\log \lb\frac{\gamma}{h}\rb\rb.  \label{eq:C1C3_case3}
\end{align}
To bound the integrals on $\C_2$ and $\C_4$, we use the same strategy as in Case 2,
\begin{align}
  \left\|\oint_{\C_{2}}\td{q}(z)dz\right\|_{\op} + \left\|\oint_{\C_{4}}\td{q}(z)dz\right\|_{\op}&\le 2(b' - a')\max_{z\in \C_2\cup \C_4}\|\td{q}(z)\|_{\op}\nonumber\\
&\le  2b \max_{w\in [a', b']}\|\td{q}(w + \gamma \i)\|_{\op},\label{eq:C2C4_case3_1}
\end{align}
where the last step uses the fact that $b' - a' = b - a + 4h \le b$ and $\td{q}(w + \gamma\i) = \td{q}(w - \gamma\i)$. Then
\begin{align}
    \|\td{q}(w + \gamma\i)\|_{\op} &\le \|E\bar{U}^{*}\|_{\op}\max_{i\in [n]}\frac{1}{\sqrt{(\lambda_{i} - w)^{2} + \gamma^{2}}}\max_{i\in [n]} \frac{|\lambda_{i}^{*}|}{\sqrt{(\lambda_{i}^{*} - w)^{2} + \gamma^{2}}}\nonumber\\
& \le \frac{\|E\bar{U}^{*}\|_{\op}}{\gamma}\max_{i\in [n]} \frac{|\lambda_{i}^{*}|}{\sqrt{(\lambda_{i}^{*} - w)^{2} + \gamma^{2}}}\nonumber \\
& \le \frac{\|E\bar{U}^{*}\|_{\op}}{\gamma}\max_{i\in [n]}\sup_{w\in [a', b']} \frac{|\lambda_{i}^{*}|}{\sqrt{(\lambda_{i}^{*} - w)^{2} + \gamma^{2}}}\label{eq:doublemax5}.
\end{align}
To bound the last term we distinguish two cases:
\begin{itemize}
\item If $\lambda_{i}^{*} \in [a', b']$, then
\[\sup_{w\in [a', b']} \frac{|\lambda_{i}^{*}|}{\sqrt{(\lambda_{i}^{*} - w)^{2} + \gamma^{2}}} = \frac{|\lambda_{i}^{*}|}{\gamma}\le \frac{b'}{\gamma}.\]
\item If $\lambda_{i}^{*} \in (-\infty, a')$, then by \eqref{eq:gzxa},
\[\sup_{w\in [a', b']} \frac{|\lambda_{i}^{*}|}{\sqrt{(\lambda_{i}^{*} - w)^{2} + \gamma^{2}}} = \frac{|\lambda_{i}^{*}|}{\sqrt{(\lambda_{i}^{*} - a')^{2} + \gamma^{2}}} \le g(|\lambda_{i}^{*}| - a'; \gamma, a')\le \sup_{z}g(z; \gamma, a')\le \frac{\sqrt{a'^{2} + \gamma^{2}}}{\gamma}.\]
\item When $\lambda_{i}^{*}\in (b', \infty)$, using a similar argument as above, we obtain that
\[\sup_{w\in [a', b']} \frac{|\lambda_{i}^{*}|}{\sqrt{(\lambda_{i}^{*} - w)^{2} + \gamma^{2}}} = \frac{|\lambda_{i}^{*}|}{\sqrt{(\lambda_{i}^{*} - b')^{2} + \gamma^{2}}}\le \frac{\sqrt{b'^{2} + \gamma^{2}}}{\gamma}.\]
\end{itemize}
In summary, 
\begin{equation}
  \label{eq:doublemax6}
  \max_{i\in [n]}\sup_{w\in [a', b']} \frac{|\lambda_{i}^{*}|}{\sqrt{(\lambda_{i}^{*} - w)^{2} + \gamma^{2}}}\le \frac{\sqrt{b'^{2} + \gamma^{2}}}{\gamma}.
\end{equation}
Putting \eqref{eq:C2C4_case3_1}, \eqref{eq:doublemax5} and \eqref{eq:doublemax6} together, we have that
\begin{equation}
  \left\|\oint_{\C_{2}}\td{q}(z)dz\right\|_{\op} + \left\|\oint_{\C_{4}}\td{q}(z)dz\right\|_{\op} \le  \frac{2b\|E\bar{U}^{*}\|_{\op}\sqrt{b'^{2} + \gamma^{2}}}{\gamma^{2}} ,\label{eq:C2C4_case3}
\end{equation}
Then \eqref{eq:case3_Kato}, \eqref{eq:C1C3_case3} and \eqref{eq:C2C4_case3} together yield
\[\mnorm{\bar{\Lambda}^{*}\bar{U}^{*}\lb UU^{T} - U^{*}(U^{*})^{T}\rb}\le \frac{b\|E\bar{U}^{*}\|_{\op}}{2\pi h} \lb \frac{5\pi}{2} + 4\sqrt{\frac{5}{4}}\sqrt{\frac{h}{b}}\log\lb\frac{\gamma}{h}\rb\rb + \frac{b\|E\bar{U}^{*}\|_{\op}\sqrt{b'^{2} + \gamma^{2}}}{\pi\gamma^{2}}.\]
Let $\gamma = b$. Then
\begin{equation}\label{eq:setgammaj}
  \mnorm{\bar{\Lambda}^{*}\bar{U}^{*}\lb UU^{T} - U^{*}(U^{*})^{T}\rb}\le \frac{b\|E\bar{U}^{*}\|_{\op}}{2\pi h} \lb \frac{5\pi}{2} + 8\sqrt{\frac{5}{4}}\sqrt{\frac{h}{b}}\log\lb\sqrt{\frac{b}{h}}\rb + 2\sqrt{\frac{b^{'2} + b^{2}}{b^{2}}}\frac{h}{b}\rb.
\end{equation}
Let $m(x) = x / \exp(x)$. Note that $\frac{d}{dx}(\log m(x)) = \frac{1}{x} - 1$. Thus $m(x)$ reaches its maximum at $x = 1$. Then 
 \[\sqrt{\frac{h}{b}}\log\lb\sqrt{\frac{b}{h}}\rb = m\lb \log \sqrt{\frac{b}{h}}\rb\le m(1) = \frac{1}{e}.\]
On the other hand,
\[\frac{b^{'2} + b^{2}}{b^{2}} = \lb\frac{b + 2h}{b}\rb^{2} + 1\le \frac{9}{4} + 1 = \frac{13}{4}.\]
Thus, \eqref{eq:setgammaj} implies that
\[\|\bar{\Lambda}^{*}\bar{U}^{*}\lb UU^{T} - U^{*}(U^{*})^{T}\rb\|_{\op}\le \frac{b\|E\bar{U}^{*}\|_{\op}}{2\pi h} \lb \frac{5\pi}{2} + \frac{8}{e}\sqrt{\frac{5}{4}}  + \frac{1}{2}\sqrt{\frac{13}{4}}\rb\le \frac{1.918b\|E\bar{U}^{*}\|_{\op}}{h}\le \frac{3.84b\|E\bar{U}^{*}\|_{\op}}{\tdgap},\]
where the last inequality uses the fact that $h = \tdgap / 2$. The proof is then completed by \eqref{eq:case3_pre}.
\end{proof}

\subsection{Step III: stochastic bound for $\|E_{k}^{T}(U^{(k)}H^{(k)} - U^{*})\|_{2}$}
From this subsection we will derive the stochastic bound by taking the randomness of $A$ into account. We assume \textbf{A}1 and \textbf{A}3 hold. In particular, we choose $A^{(1)}, \ldots A^{(n)}$ that satisfy \textbf{A}1 with the subset $S = [r]$, i.e. 
\[d_{TV}(\P_{(A_{k}, A^{(k)})}, \P_{A_{k}}\times \P_{A^{(k)}})\le \delta / n,\]
and it holds with probability at least $1 - \delta$ that
\[\|A^{(k)} - A\|_{\op}\le L_{1}(\delta), \quad \frac{\|(A^{(k)} - A)U\|_{\op}}{\lambda_{\min}^{*}}\le \lb \kappa^{*}L_{2}(\delta) + L_{3}(\delta)\rb \mnorm{U},\]
simultaneously for all $k$. In addition, we re-define $\tdgap$ as follows by setting $L_{1} = L_{1}(\delta)$,
\begin{equation}
  \label{eq:tdgap_delta}
  \tdgap(\delta) = \frac{\gap^{*} - L_{1}(\delta)}{2}.
\end{equation}
We start from a concentration bound. 
\begin{lemma}\label{lem:EmW}
  Given any $\delta \in (0, 1)$ and $W^{(k)}\in \R^{n\times r}$ that only depends on $A^{(k)}$. Then under assumptions \textbf{A}1 and \textbf{A}3, it holds simultaneously for all $k$ that
\[\|E_{k}^{T}W^{(k)}\|_{2}\le b_{\infty}(\delta)\mnorm{W^{(k)}} + b_{2}(\delta)\|W^{(k)}\|_{\op}\]
with probability at least $1 - 2\delta$ where $b_{\infty}(\delta), b_{2}(\delta)$ are defined in assumption \textbf{A}3.
\end{lemma}
\begin{proof}
Using the representation of total variation distance, there exists a coupling $\hat{E}_{k}$ of $E_{k}$ for each $k$ such that
\[\hat{E}_{k}\mbox{ is independent of }A^{(k)}, \quad \P(E_{k} \not = \hat{E}_{k})\le \delta / n.\]
The lemma follows if we can prove that
\begin{equation*}
  \|\hat{E}_{k}^{T}W^{(k)}\|_{2}\le b_{\infty}(\delta)\mnorm{W^{(k)}} + b_{2}(\delta)\|W^{(k)}\|_{\op}
\end{equation*}
holds simultaneously for all $k$ with probability at least $1 - \delta$. Since $\hat{E}_{k}$ is indpendent of $W^{(k)}$, the above inequality is guaranteed by assumption \textbf{A}3.
\end{proof}

\begin{lemma}\label{lem:step3}
Assume that $\lambda_{s + r}^{*}\lambda_{s + 1}^{*} > 0$. Fix any $\delta \in (0, 1)$. Under assumptions \textbf{A}1 and \textbf{A}3, it holds simultaneously for all $k$ with probability at least $1 - 3\delta$,
\begin{align*}
  \lefteqn{\max_{k}\|E_{k}^{T}(U^{(k)}H^{(k)} - U^{*})\|_{2} \le b_{\infty}(\delta)\mnorm{U\sign(H) - AU^{*}(\Lambda^{*})^{-1}} + \frac{b_{\infty}(\delta)\mnorm{EU^{*}}}{\lambda_{\min}^{*}}} \\
& + \frac{b_{2}(\delta)\|EU^{*}\|_{\op}}{\tdgap(\delta)} + \lb\frac{\lambda_{\min}^{*}(b_{\infty}(\delta) + b_{2}(\delta))(\kappa^{*}L_{2}(\delta) + L_{3}(\delta))}{\tdgap(\delta)} + \frac{b_{\infty}(\delta)\|EU^{*}\|_{\op}^{2}}{\tdgap(\delta)^{2}}\rb\mnorm{U}.
\end{align*}
\end{lemma}
\begin{proof}
For notational convenience, we will suppress the notation $(\delta)$ for all quantities that involve it. Let $W^{(k)} = U^{(k)}H^{(k)} - U^{*}$ and let $\event_{1}$ denote the event that 
\[\|E_{k}^{T}W^{(k)}\|_{2}\le b_{\infty}\mnorm{W^{(k)}} + b_{2}\|W^{(k)}\|_{\op} \mbox{ simultaneously for all }k,\]
and $\event_{2}$ denote the event that
\[\|A^{(k)} - A\|_{\op}\le L_{1}, \quad \frac{\|(A^{(k)} - A)U\|_{\op}}{\lambda_{\min}^{*}}\le \lb \kappa^{*}L_{2} + L_{3}\rb \mnorm{U}, \mbox{ simultaneously for all }k.\]
Then Lemma \ref{lem:EmW} and assumption \textbf{A}1 implies that
\[\P(\event_{1})\ge 1 - 2\delta, \quad \P(\event_{2})\ge 1 - \delta.\]
A simple union bound implies that
\[P(\event)\ge 1 - 3\delta, \quad \mbox{where }\event = \event_{1}\cap \event_{2}.\]
Throughout the rest of the proof we restrict the attention into $\event$. On $\event$, for all $k$,
\begin{equation}
  \label{eq:step3_1}
  \|E_{k}^{T}(U^{(k)}H^{(k)} - U^{*})\|_{2}\le b_{\infty}\mnorm{U^{(k)}H^{(k)} - U^{*}} + b_{2}\|U^{(k)}H^{(k)} - U^{*}\|_{\op}.
\end{equation}
First we bound $\mnorm{U^{(k)}H^{(k)} - U^{*}}$. Applying the triangle inequality, 
\begin{equation}
  \label{eq:step3_term1_1}
  \mnorm{U^{(k)}H^{(k)} - U^{*}}\le \mnorm{U^{(k)}H^{(k)} - UH} + \mnorm{UH - U^{*}}.
\end{equation}
Note that $\mnorm{B}\le \|B\|_{\op}$ for any Hermitian matrix $B$ and thus
\begin{align}
  &\mnorm{U^{(k)}H^{(k)} - UH} \le \|U^{(k)}H^{(k)} - UH\|_{\op}\nonumber\\
= & \|(U^{(k)}(U^{(k)})^{T} - UU^{T})U^{*}\|_{\op}\le \|U^{(k)}(U^{(k)})^{T} - UU^{T}\|_{\op}.\label{eq:step3_term1_2}
\end{align}
On the other hand,
\begin{align}
  \lefteqn{\mnorm{UH - U^{*}}}\nonumber\\
&\le \mnorm{UH - AU^{*}(\Lambda^{*})^{-1}} + \mnorm{AU^{*}(\Lambda^{*})^{-1} - U^{*}}\nonumber\\
& = \mnorm{UH - AU^{*}(\Lambda^{*})^{-1}} + \mnorm{AU^{*}(\Lambda^{*})^{-1} - A^{*}U^{*}(\Lambda^{*})^{-1}}\nonumber\\
& = \mnorm{UH - AU^{*}(\Lambda^{*})^{-1}} + \mnorm{EU^{*}(\Lambda^{*})^{-1}}\nonumber\\
& \le \mnorm{UH - AU^{*}(\Lambda^{*})^{-1}} + \frac{\mnorm{EU^{*}}}{\lambda_{\min}^{*}}\nonumber\\
& \le \mnorm{U\sign(H) - AU^{*}(\Lambda^{*})^{-1}} + \mnorm{U(H - \sign(H))} + \frac{\mnorm{EU^{*}}}{\lambda_{\min}^{*}}\nonumber\\
& \le \mnorm{U\sign(H) - AU^{*}(\Lambda^{*})^{-1}} + \frac{\|EU^{*}\|_{\op}^{2}}{\tdgap^{2}}\mnorm{U} + \frac{\mnorm{EU^{*}}}{\lambda_{\min}^{*}}\label{eq:step3_term1_3},
\end{align}
where the last line uses \eqref{eq:J1} in page \pageref{eq:J1}. Putting \eqref{eq:step3_term1_1}--\eqref{eq:step3_term1_3} together, we obtain that
\begin{align}
  &\mnorm{U^{(k)}H^{(k)} - U^{*}} \le \|U^{(k)}(U^{(k)})^{T} - UU^{T}\|_{\op}\nonumber\\
&\,\, + \mnorm{U\sign(H) - AU^{*}(\Lambda^{*})^{-1}} + \frac{\|EU^{*}\|_{\op}^{2}}{\tdgap^{2}}\mnorm{U} + \frac{\mnorm{EU^{*}}}{\lambda_{\min}^{*}}.\label{eq:step3_term1}
\end{align}
Next we bound $\|U^{(k)}H^{(k)} - U^{*}\|_{\op}$. Applying the triangle inequality, 
\begin{equation*}
  \|U^{(k)}H^{(k)} - U^{*}\|_{\op}\le \|U^{(k)}H^{(k)} - UH\|_{\op} + \|UH - U^{*}\|_{\op}.
\end{equation*}
It has been proved in \eqref{eq:step3_term1_2} that 
\[\|U^{(k)}H^{(k)} - UH\|_{\op}\le \|U^{(k)}(U^{(k)})^{T} - UU^{T}\|_{\op}.\]
Similarly, 
\[\|UH - U^{*}\|_{\op}\le \|UU^{T} - U^{*}(U^{*})^{T}\|_{\op}.\]
Thus, 
\begin{equation}
  \label{eq:step3_term2}
  \|U^{(k)}H^{(k)} - U^{*}\|_{\op}\le \|U^{(k)}(U^{(k)})^{T} - UU^{T}\|_{\op} + \|UU^{T} - U^{*}(U^{*})^{T}\|_{\op}.
\end{equation}
Putting \eqref{eq:step3_1}, \eqref{eq:step3_term1} and \eqref{eq:step3_term2} together, we conclude that 
\begin{align*}
  \lefteqn{\|E_{k}^{T}(U^{(k)}H^{(k)} - U^{*})\|_{2}}\\
&\le b_{\infty}\lb \mnorm{U\sign(H) - AU^{*}(\Lambda^{*})^{-1}} + \frac{\|EU^{*}\|_{\op}^{2}}{\tdgap^{2}}\mnorm{U} + \frac{\mnorm{EU^{*}}}{\lambda_{\min}^{*}}\rb\\
& \,\, + (b_{\infty} + b_{2})\|U^{(k)}(U^{(k)})^{T} - UU^{T}\|_{\op} + b_{2}\|UU^{T} - U^{*}(U^{*})^{T}\|_{\op}.
\end{align*}
The proof is then completed by Lemma \ref{lem:davis_kahan}.
\end{proof}

\subsection{Step IV: summarizing Step I -- Step III}
Putting Lemma \ref{lem:step1} -- \ref{lem:step3} together, we arrive at our first bound.
\begin{lemma}\label{lem:step4}
Assume that $\lambda_{s + 1}^{*}\lambda_{s + r}^{*} > 0$. Under assumptions \textbf{A}1 - \textbf{A}3 and
\[\tdgap(\delta)\ge 2\max\{E_{+}(\delta), \lambda_{-}(\delta)\},\]
it holds with probability at least $1 - 4\delta$ that
\begin{align*}
\lefteqn{\lb 1 - \frac{(\kappa^{*}L_{2}(\delta) + L_{3}(\delta) + 1)\eta(\delta) + E_{+}(\delta)}{\tdgap(\delta)}\rb\mnorm{U\sign(H) - AU^{*}(\Lambda^{*})^{-1}}}\\ 
 &\le \frac{(\kappa^{*}L_{2}(\delta) + L_{3}(\delta) + 1)\eta(\delta) + E_{+}(\delta)}{\tdgap(\delta)}\lb\mnorm{U^{*}} + \frac{\mnorm{EU^{*}}}{\lambda_{\min}^{*}}\rb + \frac{1}{\tdgap(\delta)}\lb \frac{E_{+}(\delta)b_{2}(\delta)}{\lambda_{\min}^{*}} + \xi(\delta)\rb,
\end{align*}
where $\tdgap(\delta)$ is defined in \eqref{eq:tdgap}, 
\begin{align}
  \xi(\delta) & = \min\bigg\{E_{+}(\delta)\xi_{1}, \,\, 3.61\bar{E}_{+}(\delta)\sqrt{\kappa^{*}}\xi_{2}, \,\, 3.84\bar{E}_{+}(\delta)\kappa^{*}\xi_{3}\bigg\}\nonumber
\end{align}
\end{lemma}
\begin{proof}
Without loss of generality we assume that $\lambda_{s + 1}^{*}\ge \lambda_{s + r}^{*} > 0$. Otherwise we replace $A$ (resp. $A^{*}, \Lambda, \Lambda^{*}$) by $-A$ (resp. $-A^{*}, -\Lambda, -\Lambda^{*}$). Let $\event$ be the event in Lemma \ref{lem:step3} and $\event'$ be the event in assumption \textbf{A}2. Then
\[P(\td{\event})\ge 1 - 4\delta, \quad \mbox{where }\td{\event} = \event\cap \event'.\]
Throughout the rest of the proof we restrict the attention onto $\td{\event}$. For notational convenience, we will suppress the notation $(\delta)$ for all quantities that involve it.

Since $\tdgap\ge 2\lambda_{-}$, condition \textbf{C}1 is satisfied on event $\td{\event}$. By Lemma \ref{lem:step2}, on $\td{\event}$,
\[\frac{\|A^{*}(UH - U^{*})\|_{2}}{\lambda_{\min}^{*}}\le \frac{\xi}{\tdgap},\]
By Lemma \ref{lem:step1} and Lemma \ref{lem:step3}, on $\td{\event}$,
\begin{align}
    &\|(U\sign(H) - AU^{*}(\Lambda^{*})^{-1})_{k}\|_{2} \le \frac{b_{\infty}}{\lambda_{\min}^{*}}\mnorm{U\sign(H) - AU^{*}(\Lambda^{*})^{-1}} + \frac{1}{\tdgap}\lb \frac{E_{+}b_{2}}{\lambda_{\min}^{*}} + \xi\rb\nonumber\\
& \,\, + \frac{b_{\infty}\mnorm{EU^{*}}}{\lambda_{\min}^{*2}} + \lb\beta + \frac{(b_{\infty} + b_{2})(\kappa^{*}L_{2} + L_{3})}{\tdgap} + \frac{b_{\infty}E_{+}^{2}}{\lambda_{\min}^{*}\tdgap^{2}}\rb\mnorm{U} .\label{eq:step4_1}
\end{align}
Recalling the definition of $\beta$ in \eqref{eq:beta}, on event $\td{\event}$,
\[\beta \le \frac{E_{+}^{2}}{\tdgap^{2}} + \frac{E_{+}}{\lambda_{\min}^{*}} + \frac{E_{\infty}(\kappa^{*}L_{2} + L_{3})}{\tdgap}.\]
Then \eqref{eq:step4_1} implies that
\begin{align}
  \lefteqn{\lb 1 - \frac{b_{\infty}}{\lambda_{\min}^{*}}\rb \mnorm{U\sign(H) - AU^{*}(\Lambda^{*})^{-1}}}\nonumber\\
& \le \td{\beta}\mnorm{U} + \frac{b_{\infty}}{\lambda_{\min}^{*}}\frac{\mnorm{EU^{*}}}{\lambda_{\min}^{*}} + \frac{1}{\tdgap}\lb \frac{E_{+}b_{2}}{\lambda_{\min}^{*}} + \xi\rb,\label{eq:step4_2}
\end{align}
where
\begin{equation}
  \label{eq:tdbeta}
  \td{\beta} = \beta + \frac{(b_{\infty} + b_{2})(\kappa^{*}L_{2} + L_{3})}{\tdgap} + \frac{b_{\infty}E_{+}^{2}}{\lambda_{\min}^{*}\tdgap^{2}}.
\end{equation}
On the other hand, since $\sign(H)$ is orthogonal,
\begin{align}
  & \mnorm{U} = \mnorm{U\sign(H)}  \le \mnorm{U\sign(H) - AU^{*}(\Lambda^{*})^{-1}} + \mnorm{AU^{*}(\Lambda^{*})^{-1}}\nonumber\\
\le & \mnorm{U\sign(H) - AU^{*}(\Lambda^{*})^{-1}} + \mnorm{A^{*}U^{*}(\Lambda^{*})^{-1}} + \frac{\mnorm{EU^{*}}}{\lambda_{\min}^{*}}\nonumber\\
= & \mnorm{U\sign(H) - AU^{*}(\Lambda^{*})^{-1}} + \mnorm{U^{*}} + \frac{\mnorm{EU^{*}}}{\lambda_{\min}^{*}}\label{eq:step4_3}.
\end{align}
Combining \eqref{eq:step4_3} with \eqref{eq:step4_2}, we obtain that
\begin{align}
  \lefteqn{\lb 1 - \td{\beta} - \frac{b_{\infty}}{\lambda_{\min}^{*}}\rb \mnorm{U\sign(H) - AU^{*}(\Lambda^{*})^{-1}}}\nonumber\\
& \le \td{\beta}\mnorm{U^{*}} + \lb\td{\beta} + \frac{b_{\infty}}{\lambda_{\min}^{*}}\rb\frac{\mnorm{EU^{*}}}{\lambda_{\min}^{*}} + \frac{1}{\tdgap}\lb \frac{E_{+}b_{2}}{\lambda_{\min}^{*}} + \xi \rb\nonumber\\
& \le \lb\td{\beta} + \frac{b_{\infty}}{\lambda_{\min}^{*}}\rb\lb\mnorm{U^{*}} + \frac{\mnorm{EU^{*}}}{\lambda_{\min}^{*}}\rb + \frac{1}{\tdgap}\lb \frac{E_{+}b_{2}}{\lambda_{\min}^{*}} + \xi \rb.\label{eq:step4_4}  
\end{align}
By definition of $\beta$ in \eqref{eq:beta}, that of $\td{\beta}$ in \eqref{eq:tdbeta} and that of $\eta$ in \eqref{eq:eta},
\begin{align}
\td{\beta} + \frac{b_{\infty}}{\lambda_{\min}^{*}}& = \lb 1 + \frac{b_{\infty}}{\lambda_{\min}^{*}}\rb\frac{E_{+}^{2}}{\tdgap^{2}} + \frac{b_{\infty} + E_{+}}{\lambda_{\min}^{*}} + \frac{(E_{\infty} + b_{\infty} + b_{2})(\kappa^{*}L_{2} + L_{3})}{\tdgap}\nonumber\\
 & = \lb 1 + \frac{E_{+}^{2}}{\tdgap^{2}}\rb\frac{b_{\infty}}{\lambda_{\min}^{*}} + \frac{E_{+}^{2}}{\tdgap^{2}} + \frac{E_{+}}{\lambda_{\min}^{*}} + \frac{\eta(\kappa^{*}L_{2} + L_{3})}{\tdgap}.\label{eq:tdbeta_bound}
\end{align}
Since $\tdgap\ge 2 E_{+}$,
\[\lb 1 + \frac{E_{+}^{2}}{\tdgap^{2}}\rb\frac{b_{\infty}}{\lambda_{\min}^{*}}\le \frac{5\eta}{4\lambda_{\min}^{*}}\le \frac{5\eta}{8\tdgap}\le \frac{\eta}{\tdgap}.\]
Similarly, 
\[\frac{E_{+}^{2}}{\tdgap^{2}} + \frac{E_{+}}{\lambda_{\min}^{*}}\le \frac{E_{+}}{2\tdgap} + \frac{E_{+}}{2\tdgap} = \frac{E_{+}}{\tdgap}.\]
Therefore, \eqref{eq:tdbeta_bound} implies that 
\begin{equation}
  \label{eq:tdbeta_bound2}
  \td{\beta} + \frac{b_{\infty}}{\lambda_{\min}^{*}}\le \frac{(\kappa^{*}L_{2} + L_{3} + 1)\eta + E_{+}}{\tdgap}.
\end{equation}
The proof is then completed by \eqref{eq:step4_4}.
\end{proof}

If $(\kappa^{*}L_{2} + L_{3} + 1)\eta + E_{+} < \tdgap$, we can use a self-bounding argument to derive the bound for $\mnorm{U\sign(H) - AU^{*}(\Lambda^{*})^{-1}}$. In particular, this is true if the following stronger version of \textbf{A}4 is satisfied:
\begin{enumerate}[$\td{\mathbf{A}}$1]
\setcounter{enumi}{3}
\item $\gap^{*}\ge 4\lb (\kappa^{*}L_{2}(\delta) + L_{3}(\delta) + 1)\eta(\delta) + E_{+}(\delta) + L_{1}(\delta) + \lambda_{-}(\delta)\rb.$
\end{enumerate}
Note that $\td{\mathbf{A}}$4' is stronger than \textbf{A}4 since $\kappa^{*}\ge \bar{\kappa}^{*}$. Then $\td{\mathbf{A}}$4' implies that
\begin{align*}
\tdgap(\delta)\ge 2\lb (\kappa^{*}L_{2}(\delta) + L_{3}(\delta) + 1)\eta(\delta) + E_{+}(\delta) + L_{1}(\delta) + \lambda_{-}(\delta)\rb.
\end{align*}
Since $ \eta(\delta)\ge E_{+}(\delta)$, 
\[\tdgap(\delta)\ge 2\max\{E_{+}(\delta), \lambda_{-}(\delta)\}.\]
Thus under $\td{\mathbf{A}}$4', the assumptions of Lemma \ref{lem:step4} are satisfied. On the other hand, 
\begin{align*}
\frac{(\kappa^{*}L_{2}(\delta) + L_{3}(\delta) + 1)\eta(\delta) + E_{+}(\delta)}{\tdgap(\delta)}\le \frac{1}{2},
\end{align*}
and 
\[\tdgap(\delta) \ge 4L_{1}(\delta)\Longrightarrow\tdgap(\delta) = \frac{1}{2}(\gap^{*} - L_{1}(\delta))\ge \frac{3}{8}\gap^{*}.\]
Therefore, Lemma \ref{lem:step4} implies the following result.
\begin{lemma}\label{lem:bound1}
Assume that $\lambda_{s + r}^{*}\lambda_{s + 1}^{*} > 0$. Then under conditions \textbf{A}1 - \textbf{A}3 and $\td{\mathbf{A}}$4,
  \begin{align}
    &\mnorm{U\sign(H) - AU^{*}(\Lambda^{*})^{-1}}\nonumber\\
\le &\frac{16}{3\gap^{*}}\left\{\lb(\kappa^{*}L_{2}(\delta) + L_{3}(\delta) + 1)\eta(\delta) + E_{+}(\delta)\rb\lb\mnorm{U^{*}} + \frac{\mnorm{EU^{*}}}{\lambda_{\min}^{*}}\rb + \lb\frac{E_{+}(\delta)b_{2}(\delta)}{\lambda_{\min}^{*}} + \xi(\delta)\rb\right\},\nonumber
  \end{align}
with probability at least $1 - 4\delta$, where the quantities are defined in Lemma \ref{lem:step4}.
\end{lemma}

\subsection{Step V: removing the dependence on the condition number via eigen-partition}\label{subapp:step5}
The bound in Lemma \ref{lem:bound1} involves the condition number $\kappa^{*}$, which can be ineffective in ill-conditioned cases. Fortunately, we can remove this dependence by appropriately partitioning the columns $U$ and $U^{*}$ into blocks and applying Lemma \ref{lem:bound1} separately on each block. This idea is also proposed in \cite{mao2017estimating}. Specifically, given a partition $(S_{1}, \ldots, S_{B})$ of $\{s + r, s + r - 1, \ldots, s + 1\}$ with each $S_{j}$ being a contiguous block, let $\Lambda_{j}^{*}$ (resp. $\Lambda_{j}$) be the diagonal matrix formed by the eigenvalues $(\lambda_{i}^{*}: i\in S_{j})$ (resp. $(\lambda_{i}: i\in S_{j})$) and $U_{j}^{*}$ (resp. $U_{j}$) be the eigenvectors corresponding to $\Lambda_{j}^{*}$ (resp. $\Lambda_{j}$). 

Let
\begin{equation}\label{eq:gapj}
\kappa_{j}^{*} = \kappa(\Lambda_{j}^{*}), \quad \lambda_{\min, j}^{*} = \lambda_{\min}(\Lambda_{j}^{*}), \quad \sep_{j}(A^{*}) = \sep_{S_{j}}(A^{*}),\quad   \gap_{j}^{*} \triangleq \min\{\sep_{j}(A^{*}), \lambda_{\min, j}^{*}\}.
\end{equation}
Further let $\xi_{j}(\delta), \xi_{1j}, \xi_{2j}, \xi_{3j}$ be the counterpart in Lemma \ref{lem:bound1} for $j$-th block. Note that $\eta(\delta)$ does not depend on $j$. Then under assumptions \textbf{A}1 - \textbf{A}3 and assume $\td{\mathbf{A}}$4 holds for each block, by Lemma \ref{lem:bound1}, it holds simultaneously for all blocks with probability at least $1 - 4B\delta$ that
\begin{align}
  \lefteqn{\mnorm{U_{j}\sign(H_{j}) - AU_{j}^{*}(\Lambda_{j}^{*})^{-1}}}\nonumber\\
& \le \frac{16}{3\gap_{j}^{*}}\left\{\lb\{\kappa_{j}^{*}L_{2}(\delta) + L_{3}(\delta) + 1\} \eta(\delta) + E_{+}(\delta)\rb\lb\mnorm{U_{j}^{*}} + \frac{\mnorm{EU_{j}^{*}}}{\lambda_{\min, j}^{*}}\rb + \lb\frac{E_{+}(\delta)b_{2}(\delta)}{\lambda_{\min, j}^{*}} + \xi_{j}(\delta)\rb\right\}\nonumber\\
& \le \frac{16}{3\gap_{j}^{*}}\left\{\lb\{\kappa_{j}^{*}L_{2}(\delta) + L_{3}(\delta) + 1\} \eta(\delta) + E_{+}(\delta)\rb\lb\mnorm{U^{*}} + \frac{\mnorm{EU^{*}}}{\lambda_{\min, j}^{*}}\rb + \lb\frac{E_{+}(\delta)b_{2}(\delta)}{\lambda_{\min, j}^{*}} + \xi_{j}(\delta)\rb\right\},
\label{eq:step5_1}
\end{align}
where the last inequality uses the fact that $U_{j}^{*}$ (resp. $EU_{j}^{*}$) is a sub-block of $U^{*}$ (resp. $EU^{*}$) and thus has a smaller norm. A sufficient condition for $\td{\mathbf{A}}$4 to hold on each block is 
\begin{equation}
  \label{eq:sufficient_A4}
  \gap_{j}^{*}\ge \gap^{*}.
\end{equation}

Let $O_{j} = \sign(H_{j})$ and $O = \diag(O_{1}, \ldots, O_{B})$. By definition,
\[UO - AU^{*}(\Lambda^{*})^{-1} = (U_{1}O_{1} -  AU_{1}^{*}(\Lambda^{*}_{1})^{-1}\, \vdots\,  U_{2}O_{2} - AU_{2}^{*}(\Lambda^{*}_{2})^{-1}\, \vdots \, \cdots \, \vdots \, U_{B}O_{B} - AU_{B}^{*}(\Lambda^{*}_{B})^{-1}).\]
Thus by \eqref{eq:step5_1},
\begin{align}
  &\mnorm{UO - AU^{*}(\Lambda^{*})^{-1}} \le \sqrt{\sum_{j=1}^{B}\mnorm{U_{j}O_{j} - AU_{j}^{*}(\Lambda_{j}^{*})^{-1}}^{2}} \le \sum_{j=1}^{B}\mnorm{U_{j}O_{j} - AU_{j}^{*}(\Lambda_{j}^{*})^{-1}}\nonumber\\
\le & \sum_{j=1}^{B}\frac{16}{3\gap_{j}^{*}}\left\{\lb\{\kappa_{j}^{*}L_{2}(\delta) + L_{3}(\delta) + 1\} \eta(\delta) + E_{+}(\delta)\rb\lb\mnorm{U^{*}} + \frac{\mnorm{EU^{*}}}{\lambda_{\min, j}^{*}}\rb + \frac{E_{+}(\delta)b_{2}(\delta)}{\lambda_{\min, j}^{*}} + \xi_{j}(\delta)\right\}\label{eq:step5_later}
\end{align}
Since $O^{T}O = \diag(O_{j}^{T}O_{j}) = I$, $O\in \O^{r}$. Therefore,
\begin{align}
  \lefteqn{\frac{3}{16}d_{\ttinf}(U, AU^{*}(\Lambda^{*})^{-1})}\nonumber\\
& \le \sum_{j=1}^{B}\frac{1}{\gap_{j}^{*}}\left\{\lb\{\kappa_{j}^{*}L_{2}(\delta) + L_{3}(\delta) + 1\} \eta(\delta) + E_{+}(\delta)\rb\lb\mnorm{U^{*}} + \frac{\mnorm{EU^{*}}}{\lambda_{\min, j}^{*}}\rb + \frac{E_{+}(\delta) b_{2}(\delta)}{\lambda_{\min, j}^{*}} + \xi_{j}\right\}\nonumber\\
& = L_{2}(\delta)\eta(\delta)\lb\sum_{j=1}^{B}\frac{\kappa_{j}^{*}}{\gap_{j}^{*}}\rb\mnorm{U^{*}} + ((L_{3}(\delta) + 1)\eta(\delta) + E_{+}(\delta))\lb\sum_{j=1}^{B}\frac{1}{\gap_{j}^{*}}\rb\mnorm{U^{*}} \nonumber\\
& \,\, + L_{2}(\delta)\eta(\delta)\lb\sum_{j=1}^{B}\frac{\kappa_{j}^{*}}{\gap_{j}^{*}\lambda_{\min, j}^{*}}\rb\mnorm{EU^{*}} + ((L_{3}(\delta) + 1)\eta(\delta) + E_{+}(\delta))\lb\sum_{j=1}^{B}\frac{1}{\gap_{j}^{*}\lambda_{\min, j}^{*}}\rb\mnorm{EU^{*}}\nonumber\\
& \,\, + E_{+}(\delta)b_{2}(\delta)\lb\sum_{j=1}^{B} \frac{1}{\gap_{j}^{*}\lambda_{\min, j}^{*}}\rb + \lb\sum_{j=1}^{B} \frac{\xi_{j}(\delta)}{\gap_{j}^{*}}\rb.\label{eq:step5_2}
\end{align}
By definition of $\xi_{j}$, we deduce that
\begin{align}
    &\lefteqn{\sum_{j=1}^{B} \frac{\xi_{j}(\delta)}{\gap_{j}^{*}} \le \min\bigg\{\mnorm{A^{*}}\lb\sum_{j=1}^{B}\frac{1}{\gap_{j}^{*}\lambda_{\min, j}^{*}}\rb,} \nonumber\\
& \qquad 3.61\frac{\sqrt{\maxnorm{A^{*}}}}{I(A^{*}\mbox{ is psd})}\lb\sum_{j=1}^{B}\frac{\sqrt{\kappa_{j}^{*}}}{\gap_{j}^{*}\sqrt{\lambda_{\min, j}^{*}}}\rb, 3.84\mnorm{\bar{U}^{*}}\lb\sum_{j=1}^{B}\frac{\kappa_{j}^{*}}{\gap_{j}^{*}}\rb \bigg\} .\label{eq:step5_3}
\end{align}

The final task is to find a desirable partition of eigenvalues. In particular, we propose a generic partition that is a modification of the one in Definition 5.1 of \cite{mao2017estimating} which yields a better pre-conditioning. 

\textbf{Warm-up: eigen-partition for positive eigenvalues}

Since the description is rather technical, we start from a simple case where all eigenvalues in $\Lambda^{*}$ are strictly positive. 

\begin{definition}[a pre-conditioned eigen-partition for positive eigenvalues]\label{def:partition}
Assume that $\lambda_{s + r}^{*} > 0$. Let 
\[g_{t}^{*} = \lambda_{t}^{*} - \max\{\lambda_{t + 1}^{*}, 0\}, \quad t = s + r, s + r - 1, \ldots, s + 1 \quad \mbox{and }\lambda_{0}^{*} = \infty.\]
Let $t_{0} = s + r$ and define $t_{1}, t_{2}, \ldots$ recursively as
\[t_{\ell} = \max\{s < t < t_{\ell-1}: g_{t}^{*} > 2g_{t_{\ell-1}}^{*}, \lambda_{t}^{*} > 2\lambda_{t_{\ell - 1}}^{*}\}.\]
Let $B = \min\{\ell: t_{\ell}\mbox{ does not exist}\}$ and let $t_{B} = s$. Finally we define the partition as 
\[S_{j} = \{t_{j-1}, t_{j-1} - 1, \ldots t_{j} + 1\}, \quad j = 1, 2, \ldots, B.\]
\end{definition}

We use the following example to illustrate these quantities. 
\begin{example}\label{example:separation}
  Let $s = 0, r = 10$. Table \ref{tab:example_separation} gives the values of $\lambda_{t}^{*}$, $g_{t}^{*}$ and $t_{\ell}$. 
  \begin{table}[h]
    \centering
    \begin{tabular}{c|cccccccccccc}
      \toprule
      index $k$ & $11$ & $10$ & $9$ & $8$ & $7$ & $6$ & $5$ & $4$ & $3$ & $2$ & $1$ & $0$\\
      \midrule 
      $\lambda^{*}_{t}$ & $-1$ & $1$ & $3$ & $5$ & $8$ & $15$ & $16$ & $23$ & $25$ & $40$ & $55$ & \\
      $g_{t}$ & & $1$ & $2$ & $2$ & $3$ & $7$ & $1$ & $7$ & $2$ & $15$ & $15$ &\\
      $t_{\ell}$ & & $t_{0}$ & & & $t_{1}$ & & & $t_{2}$ & & & $t_{3}$ & $t_{4}$\\
      \bottomrule
    \end{tabular}
    \caption{An illustrating example of eigen-separation (Example \ref{example:separation}).}\label{tab:example_separation}
  \end{table}
This setting gives four blocks: $S_{1} = \{10, 9, 8\}, S_{2} = \{7, 6, 5\}, S_{3} = \{4, 3, 2\}, S_{4} = \{1\}$. 
\end{example}

Roughly speaking, the first step guarantees that each block has sufficient eigen-gap with other eigenvalues and the second step guarantees the condition number within each block and the number of blocks are both small. The following lemma gives the property of the partition.

\begin{lemma}\label{lem:partition}
 Assume that $\lambda_{s + r}^{*} > 0$. Let $S_{1}, \ldots, S_{B}$ be the partition generated in Definition \ref{def:partition}. 
Then 
\[B \le \min\{r, 1 + \log_{2}\kappa^{*}\}, \quad \kappa_{j}^{*}\le 2|S_{j}|, \quad \gap_{j}^{*}\ge \gap^{*}, \quad \lambda_{\min, j}^{*}\ge \lambda_{\min}^{*},\]
and for any $\gamma, \gamma' > 0$,
\[\sum_{j=1}^{B}\frac{1}{\gap_{j}^{*\gamma}\lambda_{\min, j}^{*\gamma'}}\le \frac{H(\gamma, \gamma')}{\gap^{*\gamma}\lambda_{\min}^{*\gamma'}},\]
where
\begin{equation}
  \label{eq:Hfunc}
  H(\gamma, \gamma') =
  \left\{\begin{array}{ll}
     \frac{1}{a} + \frac{\gamma}{\gamma + \gamma'}\lb \frac{a\gamma'}{\gamma + \gamma'}\rb^{\gamma' / \gamma},  & (\gamma' > 0)\\
     \frac{1}{a} + 1 & (\gamma' = 0)
   \end{array}\right., \quad a = 1 - 2^{-(\gamma + \gamma')}
\end{equation}
\end{lemma}
\begin{proof}
We use the notation in Definition \ref{def:partition}. To prove the first result, note that
\[\kappa^{*} = \frac{\lambda_{s+1}^{*}}{\lambda_{s+r}^{*}} = \frac{\lambda_{t_{B} + 1}^{*}}{\lambda_{t_{0}}^{*}} \ge \frac{\lambda_{t_{B-1}}^{*}}{\lambda_{t_{0}}^{*}} =\prod_{j=1}^{B-1}\frac{\lambda_{t_{j}}^{*}}{\lambda_{t_{j-1}}^{*}} > 2^{B - 1}.\]
This implies that $B \le 1 + \log_{2}\kappa^{*}$. Since all blocks are non-empty, we also have $B\le r$.

To prove the second result, let 
\[t_{j}' = \min\{t_{j} < t < t_{j - 1}: g_{t} > 2g_{t_{j - 1}}, \lambda_{t}^{*} \le 2 \lambda_{t_{j - 1}}^{*}\}.\]
In other words, $t_{j}'$ is the point that is closest to $t_{j}$ in $j$-th block such that the eigengap is sufficiently large but the corresponding eigenvalue is small. Note that $t_{j}'$ may not exist. In Example \ref{example:separation}, it is easy to see that $t_{1}'$ does not exist, $t_{2}' = 6$ and $t_{3}' = 2$. We distinguish three cases:
\begin{itemize}
\item If $t_{j}'$ does not exist, by definition of $t_{j}$ in Definition \ref{def:partition}, we know that
\[g_{t} \le 2g_{t_{j - 1}}, \quad \forall\,\, t \in (t_{j}, t_{j - 1}).\]
Then 
\[\kappa_{j}^{*} = \frac{\lambda_{t_{j} + 1}^{*}}{\lambda_{t_{j - 1}}^{*}}\le \frac{\lambda_{t_{j} + 1}^{*} - \max\{\lambda_{t_{j - 1} + 1}^{*}, 0\}}{\lambda_{t_{j - 1}}^{*} - \max\{\lambda_{t_{j - 1} + 1}^{*}, 0\}} = \frac{\sum_{i = t_{j} + 1}^{t_{j - 1}}g_{i}}{g_{t_{j - 1}}}\le 2(t_{j - 1} - t_{j}) = 2|S_{j}|.\]
\item if $t_{j}' = t_{j} + 1$, then
  \[\kappa_{j}^{*} = \frac{\lambda_{t_{j} + 1}^{*}}{\lambda_{t_{j-1}}^{*}} = \frac{\lambda_{t_{j}'}^{*}}{\lambda_{t_{j-1}}^{*}}\le 2 \le 2 |S_{j}|.\]
\item if $t_{j}' > t_{j} + 1$, by definition of $t_{j}'$, 
\[\lambda_{t_{j}'}^{*}\le 2\lambda_{t_{j-1}}^{*}, \,\, \mbox{and }g_{t} \le 2g_{t_{j - 1}} < g_{t_{j}'}, \quad \forall\,\, t \in (t_{j}', t_{j - 1}).\]
Using a similar argument as in the above case we can show that
\[\frac{\lambda_{t_{j} + 1}^{*}}{\lambda_{t_{j}'}^{*}}\le \frac{\lambda_{t_{j} + 1}^{*} - \max\{\lambda_{t_{j}' + 1}^{*}, 0\}}{\lambda_{t_{j - 1}}^{*} - \max\{\lambda_{t_{j}' + 1}^{*}, 0\}} = \frac{\sum_{i = t_{j} + 1}^{t_{j}'}g_{i}}{g_{t_{j'}}}\le t_{j}' - t_{j} \le |S_{j}|.\]
Therefore,
\[\kappa_{j}^{*} = \frac{\lambda_{t_{j} + 1}^{*}}{\lambda_{t_{j - 1}}^{*}} = \frac{\lambda_{t_{j} + 1}^{*}}{\lambda_{t_{j}'}^{*}}\frac{\lambda_{t_{j}'}^{*}}{\lambda_{t_{j - 1}}^{*}}\le 2|S_{j}|.\]
\end{itemize}

To prove the last three results, note that $t_{0} = s + r$ and $t_{B} = s$,
\[\sep_{j}(A^{*}) = \min\left\{\lambda_{t_{j-1}}^{*} - \lambda_{t_{j-1} + 1}^{*}, \lambda_{t_{j}}^{*} - \lambda_{t_{j} + 1}^{*}\right\} = \min\{g_{t_{j-1}}, g_{t_{j}}\},\]
and  
\[\sep_{s+1, s+r}(A^{*}) = \min\left\{\lambda_{t_{0}}^{*} - \lambda_{t_{0} + 1}^{*}, \lambda_{t_{B}}^{*} - \lambda_{t_{B} + 1}^{*}\right\} = \min\{g_{t_{0}}, g_{t_{B}}\}.\]
We distinguish two cases:
\begin{itemize}
\item If $j < B$, the definition of $t_{j}$ guarantees that 
\[\sep_{j}(A^{*}) = g_{t_{j-1}} > 2g_{t_{j - 2}} = \sep_{j-1}(A^{*}).\]
Also noticing that $\lambda_{\min, j}^{*} = \lambda^{*}_{t_{j-1}}\ge 2\lambda^{*}_{t_{j-2}} = 2\lambda_{\min}(\Lambda_{j-1}^{*})$, we have
\begin{align*}
  &\gap_{j}^{*} \ge 2\gap_{j-1}^{*}.
\end{align*}
\item If $j = B$, then
\[\sep_{j}(A^{*}) = \min\{g_{t_{B-1}}, g_{t_{B}}\}\ge \min\{g_{t_{0}}, g_{t_{B}}\} = \sep_{s+1, s+r}(A^{*}),\]
and $\lambda_{\min}(\Lambda_{B}^{*})\ge \lambda_{\min}^{*}$. Thus,
\[\gap_{B}^{*}\ge \gap^{*}.\]
\end{itemize}
In summary, 
\[\lambda_{\min, j}^{*}\ge 2^{j-1}\lambda_{\min}^{*} \quad (j = 1, \ldots, B)\]
and
\[\gap_{j}^{*}\ge 2^{j-1}\gap^{*} \quad (j = 1, \ldots, B - 1), \quad \gap_{B}^{*}\ge \gap^{*}.\]
As a result,
\begin{align*}
  \sum_{j=1}^{B}\frac{1}{\gap_{j}^{*\gamma}\lambda_{\min, j}^{*\gamma'}}&\le \frac{1}{\gap^{*\gamma}\lambda_{\min}^{*\gamma'}}\lb 2^{-(B - 1)\gamma'} + \sum_{j=1}^{B-1}2^{-(j-1)(\gamma + \gamma')}\rb \\
& = \frac{1}{\gap^{*\gamma}\lambda_{\min}^{*\gamma'}}\lb 2^{-(B - 1)\gamma'} + \frac{1 - 2^{-(B - 1)(\gamma + \gamma')}}{1 - 2^{-(\gamma + \gamma')}}\rb.
\end{align*}
Let $x = 2^{-(B - 1)\gamma'}$ and $a = 1 - 2^{-(\gamma + \gamma')}$. If $\gamma' = 0$, then
\[2^{-(B - 1)\gamma'} + \frac{1 - 2^{-(B - 1)(\gamma + \gamma')}}{1 - 2^{-\gamma}} = 1 + \frac{1 - 2^{-(B - 1)\gamma}}{a}\le 1 + \frac{1}{a} = H(\gamma, 0).\]
If $\gamma' > 0$, then 
\[2^{-(B - 1)\gamma'} + \frac{1 - 2^{-(B - 1)(\gamma + \gamma')}}{1 - 2^{-(\gamma + \gamma')}} = x + \frac{1 - x^{(\gamma + \gamma') / \gamma'}}{a}\triangleq h(x; \gamma, \gamma').\]
Note that
\[\frac{d}{dx}h(x; \gamma, \gamma') = 1 - \frac{\gamma + \gamma'}{a\gamma'}x^{\gamma / \gamma'}\]
and thus $h(x; \gamma, \gamma')$ reaches its maximum at $x_{*} = (a \gamma' / (\gamma + \gamma'))^{\gamma' / \gamma}$. Then
\[h(x; \gamma, \gamma') \le h(x_{*}; \gamma, \gamma') = \frac{1}{a} + x_{*}\lb 1 - \frac{x_{*}^{\gamma / \gamma'}}{a}\rb = \frac{1}{a} + \frac{\gamma x_{*}}{\gamma + \gamma'} = H(\gamma; \gamma').\]
\end{proof}

\textbf{Eigen-partition in general cases}

Suppose $\Lambda^{*}$ contains both positive and negative eigenvalues, then we can first split them into the positive and the negative blocks and partition each according to Definition \ref{def:partition}. 

\begin{definition}[a pre-conditioned eigen-partition in general cases]\label{def:partition_general}
~  \begin{enumerate}
  \item If $\lambda_{s + r}^{*} > 0$, define the partition $S_{1}, \ldots, S_{B}$ by Definition \ref{def:partition};
  \item If $\lambda_{s + 1}^{*} < 0$, define the partition $S_{1}, \ldots, S_{B}$ on $(-\lambda^{*}_{s+r}, \ldots, -\lambda^{*}_{s+1})$ by Definition \ref{def:partition};
  \item If $\lambda_{s+r}^{*} < 0 < \lambda_{s+1}^{*}$, let $b$ be the integer such that $\lambda_{s+b}^{*} < 0 < \lambda_{s+b+1}^{*}$. Define $S_{1}^{+}, \ldots, S_{B^{+}}^{+}$ on $(\lambda^{*}_{s+b+1}, \ldots, \lambda_{s+1}^{*})$ and $S_{1}^{-}, \ldots, S_{B^{-}}^{-}$ on $(\lambda^{*}_{s+r}, \ldots, \lambda_{s+b}^{*})$ by Definition \ref{def:partition}. Finally re-index the subsets as $S_{1}, \ldots, S_{B}$ with $B = B^{+} + B^{-}$ with any ordering.
  \end{enumerate}
\end{definition}

It is straightforward to derive the following counterpart result of Lemma \ref{lem:partition}.
\begin{lemma}\label{lem:partition_general}
  Let $S_{1}, \ldots, S_{B}$ be the partition generated in Definition \ref{def:partition_general}. Then 
\[B \le \min\{r, 2 + 2\log_{2}\kappa^{*}\}, \quad \kappa_{j}^{*}\le 2|S_{j}|,\quad \gap_{j}^{*}\ge \gap^{*}, \quad \lambda_{\min, j}^{*}\ge \lambda_{\min}^{*},\]
and for any $\gamma, \gamma' > 0$,
\[\sum_{j=1}^{B}\frac{1}{\gap_{j}^{*\gamma}\lambda_{\min, j}^{*\gamma'}}\le \frac{2H(\gamma, \gamma')}{\gap^{*\gamma}\lambda_{\min}^{*\gamma'}},\]
where $H(\gamma, \gamma')$ is defined in \eqref{eq:Hfunc}.
\end{lemma}

\textbf{Removing the dependence on the condition number viaeigen-partition}

Let $S_{1}, \ldots, S_{B}$ be the partition generated in Definition \ref{def:partition_general}. By Lemma \ref{lem:partition_general}, 
\begin{align*}
  \sum_{j=1}^{B}\frac{\kappa_{j}^{*}}{\gap_{j}^{*}} & \le \sum_{j=1}^{B}\frac{2|S_{j}|}{\gap^{*}} = \frac{2r}{\gap^{*}}\\
\sum_{j=1}^{B}\frac{1}{\gap_{j}^{*}} & \le \frac{2H(1, 0)}{\gap^{*}} = \frac{6}{\gap^{*}}\\
\sum_{j=1}^{B}\frac{\kappa_{j}^{*}}{\gap_{j}^{*}\lambda_{\min, j}^{*}} & \le \sum_{j=1}^{B}\frac{2|S_{j}|}{\gap^{*}\lambda_{\min}^{*}}\le \frac{2r}{\gap^{*}\lambda_{\min}^{*}}\\
\sum_{j=1}^{B}\frac{1}{\gap_{j}^{*}\lambda_{\min, j}^{*}} & \le \frac{2H(1, 1)}{\gap^{*}\lambda_{\min}^{*}} \le \frac{3.06}{\gap^{*}\lambda_{\min}^{*}}\\
\sum_{j=1}^{B}\frac{\sqrt{\kappa_{j}^{*}}}{\gap_{j}^{*}\sqrt{\lambda_{\min, j}^{*}}} & \le \frac{\sqrt{2r}2H(1, 0.5)}{\gap^{*}\sqrt{\lambda_{\min}^{*}}} = \sqrt{2r}\frac{3.72}{\gap^{*}\sqrt{\lambda_{\min}^{*}}}
\end{align*}
To apply \eqref{eq:step5_2}, we still need \textbf{A}4 holds for each block. By Lemma \ref{lem:partition_general}, $\kappa_{j}^{*}\le 2r$ for all $j$, thus it is sufficient to assume the following stronger version of \textbf{A}4:
\begin{enumerate}[$\td{\td{\mathbf{A}}}$1]
\setcounter{enumi}{3}
\item $\gap^{*}\ge 4\bigg(\{2rL_{2}(\delta) + L_{3}(\delta) + 1\}\eta(\delta) + E_{+}(\delta) + L_{1}(\delta) + \lambda_{-}(\delta)\bigg)$.
\end{enumerate}
Combining the bounds with \eqref{eq:step5_1}, \eqref{eq:step5_2} and \eqref{eq:step5_3}, we reach the following result.
\begin{lemma}\label{lem:bound2}
Under assumptions \textbf{A}1 - \textbf{A}3 and $\td{\td{\mathbf{A}}}$4,
  \begin{align}
    &d_{\ttinf}(U, AU^{*}(\Lambda^{*})^{-1})\le \frac{C}{\gap^{*}}\bigg\{\lb\{2rL_{2}(\delta) + L_{3}(\delta) + 1\}\eta(\delta) + E_{+}(\delta)\rb\lb\mnorm{U^{*}} + \frac{\mnorm{EU^{*}}}{\lambda_{\min}^{*}}\rb\nonumber\\
&\qquad   + \frac{E_{+}(\delta)b_{2}(\delta)}{\lambda_{\min}^{*}} + \min\left\{E_{+}(\delta)\xi_{1},\,\, \sqrt{2r}\bar{E}_{+}(\delta)\xi_{2}, \,\, 2r\bar{E}_{+}(\delta)\xi_{3}\right\}\bigg\},\nonumber
  \end{align}
with probability at least $1 - 4\min\{r, 2 + 2\log_{2}\kappa^{*}\}\delta$, where $C$ is a universal constant (that can be chosen as $72$).
 \end{lemma}

\subsection{Final step}\label{subapp:step6}

When $\kappa^{*} <\!\! < r$, Lemma \ref{lem:bound1} yields better results than Lemma \ref{lem:bound2}. However, it has an extra assumption that $\lambda_{s + r}^{*}\lambda_{s+1}^{*} > 0$. Fortunately, this condition can be removed by partitioning the eigenvalues into 2 blocks, with all positive and negative eigenvalues in $\Lambda^{*}$, respsectively. Using the same argument as \eqref{eq:step5_1} and noting that \textbf{A}4 holds for both blocks, we can prove the following result.

\begin{lemma}\label{lem:bound3}
Under assumptions \textbf{A}1 - \textbf{A}3 and $\td{\mathbf{A}}$4,
  \begin{align}
    &d_{\ttinf}(U, AU^{*}(\Lambda^{*})^{-1})\le \frac{C}{\gap^{*}}\bigg\{\lb\{\kappa^{*}L_{2}(\delta) + L_{3}(\delta) + 1\}\eta(\delta) + E_{+}(\delta)\rb\lb\mnorm{U^{*}} + \frac{\mnorm{EU^{*}}}{\lambda_{\min}^{*}}\rb\nonumber\\
&\qquad   + \frac{E_{+}(\delta)b_{2}(\delta)}{\lambda_{\min}^{*}} + \min\left\{E_{+}(\delta)\xi_{1},\,\, \sqrt{\kappa^{*}}\bar{E}_{+}(\delta)\xi_{2}, \,\, \kappa^{*}\bar{E}_{+}(\delta)\xi_{3}\right\}\bigg\},\nonumber
  \end{align}
with probability at least $1 - 8\delta$, where $C$ is a universal constant (that can be chosen as $41$).
\end{lemma}

Finally, if $\kappa^{*} > 2r$, \textbf{A}4 is equivalent to $\td{\td{\mathbf{A}}}$4 and we can apply Lemma \ref{lem:bound3}; otherwise, \textbf{A}4 is equivalent to $\td{\mathbf{A}}$4 and we can apply Lemma \ref{lem:bound2}. Theorem \ref{thm:generic_bound} is then proved by noticing that
\[\min\left\{1 - 4\min\{r, 2 + 2\log_{2}\kappa^{*}\}\delta, 1 - 8\delta\right\}\ge 1 - B(r)\delta,\]

\section{Proof of Other Results in Section \ref{sec:main}}\label{app:main}
\begin{proof}[\textbf{Proof of Proposition \ref{prop:A1}}]
Assume $S = [r]$ without loss of generality. Let $\event$ denote the event that 
\[\maxnorm{\Lambda - \Lambda^{*}}\le \lambda_{-}(\delta), \quad \|EU^{*}\|_{\op}\le E_{+}(\delta), \quad \mnorm{E}\le E_{\infty}(\delta).\]
Then by definition,
\[\P(\event)\ge 1 - \delta.\]
We prove each case separately. 
  \begin{enumerate}[(a)]
  \item Let $A^{(k)}$ be define as
\[[A^{(k)}]_{ij} = A_{ij}I(i\not = k, j \not = k).\]
Then since $A_{ij}$'s are independent, $A^{(k)}$ is independent of $A_{k}$. It is left to prove the deterministic inequalities on the event $\event$. First, 
\[\|A^{(k)} - A\|_{\op}\le \|A^{(k)} - A\|_{F} \le \sqrt{\sum_{j=1}^{n}A_{jk}^{2} + A_{kj}^{2}} = \sqrt{2}\|A_{k}\|_{2}\le \sqrt{2}\mnorm{A}.\]
Note that on the event $\event$,
\begin{equation}
  \label{eq:mnormAEop}
  \mnorm{A}\le \mnorm{A^{*}} + \mnorm{E}\le \mnorm{A^{*}} + E_{\infty}(\delta).
\end{equation}
Thus, $L_{1}(\delta)$ can be chosen as $\sqrt{2}(\mnorm{A^{*}} + E_{\infty}(\delta))$. On the other hand, let $A'_{ki} = A_{ki}I(i\not = k)$. Then
\begin{align*}
  \|(A^{(k)} - A)U\|_{\op} &\le \|A_{k}^{T}U\|_{2} + \|A'_{k}U_{k}^{T}\|_{\op} = \|(AU)_{k}^{T}\|_{2} + \|A'_{k}\|_{2}\|U_{k}\|_{2}\\
& = \|(U\Lambda)_{k}\|_{2} + \|A'_{k}\|_{2}\|U_{k}\|_{2} = \|U_{k}^{T}\Lambda\|_{2} + \|A'_{k}\|_{2}\|U_{k}\|_{2}\\
& \le (\lambda_{\max}(\Lambda) + \|A_{k}\|_{2})\|U_{k}\|_{2}.
\end{align*}
By definition,
\[|\lambda_{\max}(\Lambda) - \lambda_{\max}^{*}|\le \lambda_{-}(\delta).\]
Then by \eqref{eq:mnormAEop}, we conclude that on the event $\event$, 
\[\|(A^{(k)} - A)U\|_{\op} \le \lb\lambda_{\max}^{*} + E_{\infty}(\delta) + \lambda_{-}(\delta) + \mnorm{A^{*}}\rb\mnorm{U}\]
Thus, $L_{2}(\delta) = 1$ and $L_{3}(\delta) = \frac{E_{\infty}(\delta) + \lambda_{-}(\delta) + \mnorm{A^{*}}}{\lambda_{\min}^{*}}$.
\item This is a generalized version of part (a) and the proof strategy is almost the same. Let
  \begin{equation}
    \label{eq:AkNk}
    [A^{(k)}]_{ij} = A_{ij}I(i\not\in \N_{k}, j \not\in \N_{k}).
  \end{equation}
Then $A^{(k)}$ is independent of $A_{k}$. It is left to prove the deterministic inequalities on the event $\event$. First, 
\[\|A^{(k)} - A\|_{\op}\le \|A^{(k)} - A\|_{F} \le \sqrt{\sum_{i\in \N_{k}}\sum_{j=1}^{n}(A_{ji}^{2} + A_{ij}^{2})} = \sqrt{2\sum_{i\in \N_{k}}\|A_{i}\|_{2}^{2}}\le \sqrt{2|\N_{k}|}\mnorm{A}.\]
Since $|\N_{k}|\le m$, $L_{1}(\delta)$ can be taken as $\sqrt{2m}(\mnorm{A^{*}} + E_{+}(\delta))$ by \eqref{eq:mnormAEop}. On the other hand, let $\td{A}_{ij} = A_{ij}I(j\not\in \N_{k})$ for $i \in \N_{k}$. Then on event $\event$, 
\begin{align*}
  \|(A^{(k)} - A)U\|_{\op} &\le \sum_{i\in \N_{k}}\lb\|A_{i}^{T}U\|_{2} + \|\td{A}_{i}U_{i}^{T}\|_{\op}\rb = \sum_{i\in \N_{k}}\lb\|(A U)_{i}^{T}\|_{2} + \|\td{A}_{i}\|_{2}\|U_{i}\|_{2}\rb\\
& = \sum_{i\in \N_{k}}\lb\|(U\Lambda)_{i}\|_{2} + \|\td{A}_{i}\|_{2}\|U_{i}\|_{2}\rb = \sum_{i\in \N_{k}}\|U_{i}^{T}\Lambda\|_{2} + \|\td{A}_{i}\|_{2}\|U_{i}\|_{2}\\
& \le \sum_{i\in \N_{k}}(\lambda_{\max}(\Lambda) + \|A_{i}\|_{2})\|U_{i}\|_{2}\\ 
& \le |\N_{k}|(\lambda_{\max}^{*} + E_{\infty}(\delta) + \lambda_{-}(\delta) + \mnorm{A^{*}})\mnorm{U}.
\end{align*}
Since $|\N_{k}|\le m$, we can take $L_{2}(\delta) = m$ and $L_{3}(\delta) = \frac{m(E_{\infty}(\delta) + \lambda_{-}(\delta) + \mnorm{A^{*}})}{\lambda_{\min}^{*}}$.
  \end{enumerate}
\end{proof}

\begin{proof}[\textbf{Proof of Proposition \ref{prop:vector2matrix}}]
Let $\S^{r-1}$ be the $r$-dimensional unit sphere and $\M(\eps)$ be a minimal $\eps$-net of $\S^{r-1}$, i.e. $\forall \zeta\in \S^{r - 1}$, there exists $\zeta'\in \M(\eps)$ such that $\|\zeta - \zeta'\|_{2}\le \eps$. It is well-known that
    \begin{equation}\label{eq:Meps}
      |\M(\eps)|\le \lb 1 + \frac{2}{\eps}\rb^{r}.
    \end{equation}
Then for any vector $x\in \R^{r}$, 
\begin{align*}
  \|x\|_{2} &= \sup_{\zeta\in \S^{r-1}}\zeta^{T}x = \sup_{\zeta \in \S^{r-1}, \zeta'\in \M(\eps), \|\zeta - \zeta'\|_{2}\le \eps}(\zeta^{'T}x + (\zeta - \zeta')^{T}x)\\
&\le \max_{\zeta'\in \M(\eps)}(\zeta')^{T}x + \|x\|_{2}\eps.
\end{align*}
This implies that 
\begin{equation}
  \label{eq:eps_net}
  \|x\|_{2}\le \frac{1}{1 - \eps}\max_{\zeta\in \M(\eps)}x^{T}\zeta.
\end{equation}
Applying \eqref{eq:eps_net} with $x = E_{k}^{T}W$, we have
\[\|E_{k}^{T}W\|_{2}\le \frac{1}{1 - \eps}\max_{\zeta\in \M(\eps)}E_{k}^{T}(W\zeta).\]
Let $\delta' = \delta / (1 + 2 / \eps)^{r}n$, $a'_{\infty} = a_{\infty}(\delta') / (1 - \eps)$ and $a'_{2} = a_{2}(\delta') / (1 - \eps)$. Further, \eqref{eq:Meps} implies that $\delta' \le \delta / |\M(\eps)|n$ and \textbf{A}1 implies that for each given $\zeta \in \M(\eps)$,
\[E_{k}^{T}(W\zeta)\le (1 - \eps)\lb a'_{\infty}\|W\zeta\|_{\infty} + a'_{2}\|W\zeta\|_{2}\rb\le (1 - \eps)\lb a'_{\infty}\mnorm{W} + a'_{2}\|W\|_{\op}\rb.\]
with probability $1 - \delta'$. Applying the union bound implies that 
\[\|E_{k}^{T}W\|_{2} \le \frac{1}{1 - \eps}\max_{\zeta\in \M(\eps)}E_{k}^{T}(W\zeta)\le a'_{\infty}\mnorm{W} + a'_{2}\|W\|_{\op}\]
holds simultaneously for all $\zeta\in \M(\eps)$ with probability at least $1 - |\M(\eps)|\delta' \ge 1 - \delta / n$. The proof is then completed by \eqref{eq:eps_net} and taking $\eps = 0.5$.
\end{proof}

\begin{proof}[\textbf{Proof of Theorem \ref{thm:generic_bound2}}]
By the triangle inequality, 
\begin{equation}
  \label{eq:triangleq1}
  d_{\ttinf}(U, U^{*})\le d_{\ttinf}(U, AU^{*}(\Lambda^{*})^{-1}) + d_{\ttinf}(AU^{*}(\Lambda^{*})^{-1}, U^{*}).
\end{equation}
By definition,
\begin{equation}
  \label{eq:triangleq2}
  d_{\ttinf}(AU^{*}(\Lambda^{*})^{-1}, U^{*})\le \mnorm{AU^{*}(\Lambda^{*})^{-1} - U^{*}} = \mnorm{EU^{*}(\Lambda^{*})^{-1}}\le \frac{\mnorm{EU^{*}}}{\lambda_{\min}^{*}}.
\end{equation}
By Theorem \ref{thm:generic_bound},
  \begin{align}
    d_{\ttinf}(U, AU^{*}(\Lambda^{*})^{-1})&\le \frac{\mnorm{EU^{*}}}{\lambda_{\min}^{*}} + \frac{C}{\gap^{*}}\bigg\{\{\bar{\kappa}^{*}L_{2}(\delta) + L_{3}(\delta) + 1\}\eta(\delta)\lb\mnorm{U^{*}} + \frac{\mnorm{EU^{*}}}{\lambda_{\min}^{*}}\rb\nonumber\\
& \qquad  + \frac{E_{+}(\delta)b_{2}(\delta)}{\lambda_{\min}^{*}} + \min\left\{E_{+}(\delta)\xi_{1}, \bar{E}_{+}(\delta)\sqrt{\bar{\kappa}^{*}}\xi_{2}, \bar{E}_{+}(\delta)\bar{\kappa}^{*}\xi_{3}\right\}\bigg\},\nonumber
  \end{align}
By assumption \textbf{A}4, the coefficient of $\mnorm{EU^{*}} / \lambda_{\min}^{*}$ can be bounded by
\[1 + \frac{C\{\bar{\kappa}^{*}L_{2}(\delta) + L_{3}(\delta) + 1\}\eta(\delta)}{\gap^{*}}\le 1 + \frac{C}{4}.\]
When $C$ is chosen as $72$, $1 + C / 4\le 72$. The proof is then completed. 
\end{proof}

\begin{proof}[\textbf{Proof of Theorem \ref{thm:generic_bound_trick}}]
We only need to modify the six steps in the proof of Theorem \ref{thm:generic_bound} in Appendix \ref{app:generic_bound}. In particular, Step I and Step III need to be substantially modified while all other steps remain almost the same. Let $\event_{0}$ be the event given by \textbf{A}'0 - \textbf{A}'2, i.e. 
\[\frac{\min_{j\in [s + 1, s + r]}|\Lambda_{jj}^{*}|}{\min_{j\in [s + 1, s + r], k\in [n]}|\Lambda_{jj}^{*} - \Sigma_{kk}|}\le \Theta(\delta),\]
\[\|A^{(k)} - A\|_{\op}\le L_{1}(\delta), \quad \frac{\|(A^{(k)} - A)U\|_{\op}}{\lambda_{\min}^{*}}\le \lb \kappa^{*}L_{2}(\delta) + L_{3}(\delta)\rb \mnorm{U},\]
where $A^{(k)}$ satisfies the total variation condition in \textbf{A}'1 and 
\[\maxnorm{\Lambda - \Lambda^{*}}\le \lambda_{-}(\delta), \quad \|E U^{*}\|_{\op}\le E_{+}(\delta), \quad \mnorm{\Ep}\le \Ep_{\infty}(\delta).\]
Then 
\begin{equation}
  \label{eq:event0_trick}
\P(\event_{0})\ge 1 - 3\delta.
\end{equation}
Throughout the proof we will restrict the attention into $\event_{0}$ and suppress the notation $(\delta)$ for all quantities that involve it.  

\noindent \textbf{Step I}: assume \textbf{C}1 hold as in Appendix \ref{subapp:preparation}. Again we start by assuming that all eigenvalues are of the same sign, i.e. $\lambda_{s + 1}^{*} \lambda_{s + r}^{*} > 0$. In step V we deal with the general case. 

Recalling that
\[\Ap = A - \Sigma, \quad \Ap^{*} = \E \Ap, \quad \Ep = \Ap - \Ap^{*},\] 
we have
\begin{align}
    &(UH - U^{*})_{k}^{T} = (UH - AU^{*}(\Lambda^{*})^{-1})_{k}^{T} + (EU^{*}(\Lambda^{*})^{-1})_{k}^{T} \nonumber\\
& = \left\{U_{k}^{T}(H\Lambda^{*} - \Lambda H)  + A_{k}^{T}(UH - U^{*})\right\}(\Lambda^{*})^{-1} + E_{k}^{T}U^{*}(\Lambda^{*})^{-1}\nonumber\\
& = \left\{U_{k}^{T}(H\Lambda^{*} - \Lambda H)  + \td{A}_{k}^{T}(UH - U^{*})\right\}(\Lambda^{*})^{-1} + \Sigma_{kk}(UH - U^{*})_{k}^{T}(\Lambda^{*})^{-1} + E_{k}^{T}U^{*}(\Lambda^{*})^{-1}\nonumber\\
& = \left\{U_{k}^{T}(H\Lambda^{*} - \Lambda H)  + \Ep_{k}^{T}(UH - U^{(k)}H^{(k)}) + \Ep_{k}^{T}(U^{(k)}H^{(k)} - U^{*}) + (\Ap_{k}^{*})^{T}(UH - U^{*})\right\}(\Lambda^{*})^{-1}\nonumber\\
& \quad + \Sigma_{kk}(UH - U^{*})_{k}^{T}(\Lambda^{*})^{-1} + E_{k}^{T}U^{*}(\Lambda^{*})^{-1}.\nonumber
\end{align}
Rearranging the second last term to the left handed side and multiplying both sides by $\Lambda^{*}$ and recalling that $V_{k}^{T} = E_{k}^{T}U^{*}(\Lambda^{*} - \Sigma_{kk}I)^{-1}$, we obtain that
\begin{align*}
  \lefteqn{(UH - U^{*} - \V)_{k}^{T}(\Lambda^{*} - \Sigma_{kk}I)} \\
& = U_{k}^{T}(H\Lambda^{*} - \Lambda H)  + \Ep_{k}^{T}(UH - U^{(k)}H^{(k)}) + \Ep_{k}^{T}(U^{(k)}H^{(k)} - U^{*}) + (\Ap_{k}^{*})^{T}(UH - U^{*})
\end{align*}
By the triangle inequality and the definition of $\Theta$, on event $\event$ we obtain that
\begin{align}
  \|(UH - U^{*} - \V)_{k}\|_{2} &\le \frac{\Theta}{\lambda_{\min}^{*}}\bigg\{\|U_{k}^{T}(H\Lambda^{*} - \Lambda H)\|_{2} + \|\Ep_{k}^{T}(UH - U^{(k)}H^{(k)})\|_{2}\nonumber\\
& \qquad \qquad  + \|\Ep_{k}^{T}(U^{(k)}H^{(k)} - U^{*})\|_{2} + \|(\Ap_{k}^{*})^{T}(UH - U^{*})\|_{2}\bigg\}.\nonumber
\end{align}
By \eqref{eq:J2} and \eqref{eq:J3} in page \pageref{eq:J3}, on event $\event_{0}$ defined at the beginning of the proof,
\begin{align}
  \|(UH - U^{*} - \V)_{k}\|_{2} & \le \Theta\lb\frac{E_{+}}{\lambda_{\min}^{*}} + \frac{\Ep_{\infty}(\kappa^{*}L_{2} + L_{3})}{\tdgap}\rb\mnorm{U} \nonumber\\
& \quad  + \frac{\Theta}{\lambda_{\min}^{*}}\left\{\|\Ep_{k}^{T}(U^{(k)}H^{(k)} - U^{*})\|_{2} + \|(\Ap_{k}^{*})^{T}(UH - U^{*})\|_{2}\right\},\nonumber
\end{align}
where 
\begin{equation*}
  \tdgap = \frac{1}{2}(\gap^{*} - L_{1}).
\end{equation*}
By \eqref{eq:J1},
\begin{align}
  \lefteqn{ \|(U\sign(H) - U^{*} - \V)_{k}\|_{2} \le \|(UH - U^{*} - \V)_{k}\|_{2} + \mnorm{U(\sign(H) - H)}}\nonumber\\
&\le \beta\mnorm{U}  + \frac{\Theta}{\lambda_{\min}^{*}}\left\{\|\Ep_{k}^{T}(U^{(k)}H^{(k)} - U^{*})\|_{2} + \mnorm{\Ap^{*}(UH - U^{*})}\right\}\label{eq:Step1_trick},
\end{align}
where 
\begin{equation}
  \label{eq:beta'}
    \beta \triangleq \frac{E_{+}^{2}}{\tdgap^{2}} + \frac{\Theta E_{+}}{\lambda_{\min}^{*}} + \frac{\Theta \Ep_{\infty}(\kappa^{*}L_{2} + L_{3})}{\tdgap}.
\end{equation}

\noindent \textbf{Step II}: since $\Ap^{*}$ may have very different eigenvalues and eigenvectors from $A^{*}$, the Kato's integral cannot be directly applied here. For this reason, we only consider the bound \eqref{eq:case1}. The same proof shows that
\begin{equation}
  \label{eq:Step2_trick}
  \mnorm{\Ap^{*}(UH - U^{*})}\le \frac{E_{+}\mnorm{\Ap^{*}}}{\tdgap}.
\end{equation}

\noindent \textbf{Step III}: assuming $\lambda_{s + 1}^{*}\lambda_{s + r}^{*} > 0$. We can follow the proof of Lemma \ref{lem:step3} to derive a bound for $\max_{k}\|\Ep_{k}^{T}(U^{(k)}H^{(k)} - U^{*})\|_{2}$. since $\Ep_{k}$ is a function of $A_{k}$, 
\[d_{TV}(\Ep_{k}, A^{(k)})\le \delta / n,\]
we can still apply Lemma \ref{lem:EmW} with $W^{(k)} = U^{(k)}H^{(k)} - U^{*}$. Let $\event_{1}$ denote the event that 
\[\|\Ep_{k}^{T}W^{(k)}\|_{2}\le \bp_{\infty}\mnorm{W^{(k)}} + \bp_{2}\|W^{(k)}\|_{\op} \mbox{ simultaneously for all }k.\]
Then Lemma \ref{lem:EmW} implies that
\begin{equation}
  \label{eq:event1_trick}
  \P(\event_{1})\ge 1 - 2\delta.
\end{equation}
A simple union bound implies that
\begin{equation}
  \label{eq:event_trick}
  \P(\event)\ge 1 - 5\delta, \quad \mbox{where }\event = \event_{0}\cap \event_{1}.
\end{equation}
Throughout the rest of the proof we will restrict the attention into $\event$. By \eqref{eq:step3_term1_1} and \eqref{eq:step3_term1_2} in page \pageref{eq:step3_term1_2}, we have
\begin{equation*}
  \mnorm{W^{(k)}}\le \|U^{(k)}(U^{(k)})^{T} - UU^{T}\|_{\op} + \mnorm{UH - U^{*}}.
\end{equation*}
By \eqref{eq:step3_term2},
\begin{equation*}
  \|W^{(k)}\|_{\op}\le \|U^{(k)}(U^{(k)})^{T} - UU^{T}\|_{\op} + \|UU^{T} - U^{*}(U^{*})^{T}\|_{\op}.
\end{equation*}
Putting pieces together, we know that on event $\event$, 
\begin{align*}
\lefteqn{  \|\Ep_{k}^{T}(U^{(k)}H^{(k)} - U^{*})\|_{2}}\\
&\le  (\bp_{\infty} + \bp_{2})\|U^{(k)}(U^{(k)})^{T} - UU^{T}\|_{\op} + \bp_{2}\|UU^{T} - U^{*}(U^{*})^{T}\|_{\op} + \bp_{\infty}\mnorm{UH - U^{*}}\\
& \le (\bp_{\infty} + \bp_{2})\|U^{(k)}(U^{(k)})^{T} - UU^{T}\|_{\op} + \bp_{2}\|UU^{T} - U^{*}(U^{*})^{T}\|_{\op} \\
& \qquad + \bp_{\infty}\mnorm{U\sign(H) - U^{*}} + \bp_{\infty}\mnorm{U(\sign(H) - H)}
\end{align*}
By Lemma \ref{lem:davis_kahan}, \eqref{eq:J1} and Corollary \ref{cor:separation},
\begin{align}
  \lefteqn{\|\Ep_{k}^{T}(U^{(k)}H^{(k)} - U^{*})\|_{2}\le \bp_{\infty}\mnorm{U\sign(H) - U^{*}} }\nonumber\\
& + \lb\frac{\lambda_{\min}^{*}(\bp_{\infty} + \bp_{2})(\kappa^{*}L_{2} + L_{3})}{\tdgap} + \frac{\bp_{\infty}E_{+}^{2}}{\tdgap^{2}}\rb\mnorm{U} + \frac{\bp_{2}E_{+}}{\tdgap}\label{eq:Step3_trick}.
\end{align}

\noindent \textbf{Step IV}: assuming $\lambda_{s + 1}^{*}\lambda_{s + r}^{*} > 0$. Putting \eqref{eq:Step1_trick}, \eqref{eq:Step2_trick} and \eqref{eq:Step3_trick} together, we obtain that
\begin{align}
\lefteqn{ \lb 1 - \frac{\Theta \bp_{\infty}}{\lambda_{\min}^{*}}\rb \| (U\sign(H) - U^{*} - \V)_{k}\|_{2}} \nonumber\\
&\le \td{\beta}\mnorm{U}  + \frac{\Theta}{\lambda_{\min}^{*}}\lb\frac{\bp_{2}E_{+}}{\tdgap} + \frac{E_{+}\mnorm{\Ap^{*}}}{\tdgap}\rb.\label{eq:Step4_trick_1}
\end{align}
where
\begin{equation}
  \label{eq:tdbeta'}
  \td{\beta} = \beta + \frac{\Theta(\bp_{\infty} + \bp_{2})(\kappa^{*}L_{2} + L_{3})}{\tdgap} + \frac{\Theta \bp_{\infty}E_{+}^{2}}{\lambda_{\min}^{*}\tdgap^{2}}
\end{equation}
and $\beta$ is defined in \eqref{eq:beta'}. On the other hand, since $\sign(H)$ is orthogonal,
\begin{align*}
  \mnorm{U} = \mnorm{U\sign(H)}\le \mnorm{U\sign(H) - U^{*} - \V} + \mnorm{U^{*}} + \mnorm{\V}.
\end{align*}
Plugging this into \eqref{eq:Step4_trick_1}, we have
\begin{align}
 \lefteqn{\lb 1 - \td{\beta} - \frac{\Theta \bp_{\infty}}{\lambda_{\min}^{*}}\rb \|(U\sign(H) - U^{*} - V)_{k}\|_{2}}\nonumber\\ 
& \le \td{\beta}\lb\mnorm{U^{*}} + \mnorm{\V}\rb  + \frac{\Theta(\bp_{2} + \mnorm{\Ap^{*}})E_{+}}{\lambda_{\min}^{*}\tdgap}.\label{eq:Step4_trick_2}
\end{align}
By definition, 
\begin{align}
\lefteqn{\td{\beta} + \frac{\Theta \bp_{\infty}}{\lambda_{\min}^{*}}}\nonumber\\
 & = \lb 1 + \frac{\Theta \bp_{\infty}}{\lambda_{\min}^{*}}\rb\frac{E_{+}^{2}}{\tdgap^{2}} + \frac{\Theta(\bp_{\infty} + E_{+})}{\lambda_{\min}^{*}} + \frac{\Theta(\Ep_{\infty} + \bp_{\infty} + \bp_{2})(\kappa^{*}L_{2} + L_{3})}{\tdgap}\nonumber\\
 & = \frac{E_{+}^{2}}{\tdgap^{2}} + \lb 1 + \frac{E_{+}^{2}}{\tdgap^{2}}\rb\frac{\Theta \bp_{\infty}}{\lambda_{\min}^{*}} + \frac{\Theta(\etap(\kappa^{*}L_{2} + L_{3}) + E_{+})}{\tdgap}.\label{eq:tdbetak}
\end{align}
Similar to Step IV in Appendix \ref{app:generic_bound}, we start from a stronger version of assumption \textbf{A}4:
\begin{enumerate}[$\td{\mathbf{A}}$'1]
\setcounter{enumi}{3}
\item $\gap^{*}\ge 4\bigg(\Theta (\{\kappa^{*}L_{2} + L_{3} + 1\}\etap + E_{+}) + L_{1} + \lambda_{-} + E_{+}\bigg)$.
\end{enumerate}
Under assumption $\td{\mathbf{A}}$'4, \textbf{C}1 holds and 
\[\tdgap\ge 2E_{+}.\]
As a result, 
\[\lb 1 + \frac{E_{+}^{2}}{\tdgap^{2}}\rb\frac{b_{\infty}}{\lambda_{\min}^{*}}\le \frac{5\etap}{4\lambda_{\min}^{*}}\le \frac{5\etap}{8\tdgap}\le \frac{\etap}{\tdgap}.\]
By \eqref{eq:tdbetak}, 
\[\td{\beta} + \frac{\Theta \bp_{\infty}}{\lambda_{\min}^{*}}\le \frac{E_{+}^{2}}{\tdgap^{2}} + \frac{\Theta(\{\kappa^{*}L_{2} + L_{3} + 1\}\etap + E_{+})}{\tdgap}.\]
On the other hand, assumption $\td{\mathbf{A}}$'4 implies that
\[\tdgap\ge 2 \Theta(\{\kappa^{*}L_{2} + L_{3} + 1\}\etap + E_{+}).\]
Then 
\[\td{\beta} + \frac{\Theta \bp_{\infty}}{\lambda_{\min}^{*}}\le \frac{1}{4} + \frac{1}{2} = \frac{3}{4}.\]
By \eqref{eq:Step4_trick_2}, we deduce that
\begin{align}
\lefteqn{ \|(U\sign(H) - U^{*} - \V)_{k}\|_{2}} \nonumber\\
&\le 4\lb \frac{E_{+}^{2}}{\tdgap^{2}} + \frac{\Theta(\{\kappa^{*}L_{2} + L_{3} + 1\}\etap + E_{+})}{\tdgap}\rb\lb\mnorm{U^{*}} + \mnorm{\V}\rb + \frac{4\Theta(\bp_{2} + \mnorm{\Ap^{*}})E_{+}}{\lambda_{\min}^{*}\tdgap}\nonumber\\
&\le 29\lb \frac{E_{+}^{2}}{(\gap^{*})^{2}} + \frac{\Theta(\{\kappa^{*}L_{2} + L_{3} + 1\}\etap + E_{+})}{\gap^{*}}\rb\lb\mnorm{U^{*}} + \mnorm{\V}\rb + \frac{11\Theta(\bp_{2} + \mnorm{\Ap^{*}})E_{+}}{\lambda_{\min}^{*}\gap^{*}},\label{eq:Step4_trick_3}
\end{align}
where the last inequality uses the fact that 
\[\tdgap = \frac{1}{2}(\gap^{*} - L_{1})\ge \frac{3}{8}\gap^{*}.\]
On the other hand, 
\begin{equation}\label{eq:mnormV}
\mnorm{\V} = \max_{k}\|\V_{k}\|_{2}\le \frac{\Theta\mnorm{E U^{*}}}{\lambda_{\min}^{*}}.  
\end{equation}
By \eqref{eq:Step4_trick_3} we obtain that
\begin{align}
\lefteqn{ \mnorm{U\sign(H) - U^{*} - \V}} \nonumber\\
&\le 29\lb \frac{E_{+}^{2}}{(\gap^{*})^{2}} + \frac{\Theta(\{\kappa^{*}L_{2} + L_{3} + 1\}\etap + E_{+})}{\gap^{*}}\rb\lb\mnorm{U^{*}} + \frac{\Theta\mnorm{E U^{*}}}{\lambda_{\min}^{*}}\rb + \frac{11\Theta(\bp_{2} + \mnorm{\Ap^{*}})E_{+}}{\lambda_{\min}^{*}\gap^{*}}.\label{eq:Step4_trick}
\end{align}
Recall that this is true on $\event$, which has probability at least $1 - 5\delta$ according to \eqref{eq:event_trick}.

\noindent \textbf{Step V}: let $S_{1}, \ldots, S_{B}$ be the partition given by Lemma \ref{lem:partition_general}. As in Appendix \ref{subapp:step5}, let 
\begin{equation}\label{eq:gapj}
\sep_{j}(A^{*}) = \sep_{S_{j}}(A^{*}),\quad   \gap_{j}^{*} \triangleq \min\{\sep_{j}(A^{*}), \lambda_{\min, j}^{*}\}.
\end{equation}
Then with probability at least $1 - 5B\delta$, it holds simultaneously for all blocks that 
\begin{align}
\lefteqn{ \mnorm{U_{j}\sign(H_{j}) - U_{j}^{*} - \V_{j}}} \nonumber\\
&\le 29\lb \frac{E_{+}^{2}}{(\gap_{j}^{*})^{2}} + \frac{\Theta(\{\kappa_{j}^{*} L_{2} + L_{3} + 1\}\etap + E_{+})}{\gap_{j}^{*}}\rb\lb\mnorm{U_{j}^{*}} + \frac{\Theta\mnorm{EU_{j}^{*}}}{\lambda_{\min, j}^{*}}\rb + \frac{11\Theta(\bp_{2} + \mnorm{\Ap^{*}})E_{+}}{\lambda_{\min, j}^{*}\gap_{j}^{*}}\nonumber\\
&\le 29\lb \frac{E_{+}^{2}}{(\gap_{j}^{*})^{2}} + \frac{\Theta(\{\kappa_{j}^{*} L_{2} + L_{3} + 1\}\etap + E_{+})}{\gap_{j}^{*}}\rb\lb\mnorm{U^{*}} + \frac{\Theta\mnorm{E U^{*}}}{\lambda_{\min, j}^{*}}\rb + \frac{11\Theta(\bp_{2} + \mnorm{\Ap^{*}})E_{+}}{\lambda_{\min, j}^{*}\gap_{j}^{*}},\nonumber
\end{align}
where the last inequality uses the fact that $U_{j}^{*}$ (resp. $EU_{j}^{*}$) is a sub-block of $U^{*}$ (resp. $EU^{*}$) ,and thus has a smaller norm. Recalling \eqref{eq:step5_later} in page \pageref{eq:step5_later}, we have
\begin{align}
 d_{\ttinf}(U, U^{*} + \V)&\le 29\sum_{j=1}^{B}\lb \frac{E_{+}^{2}}{(\gap_{j}^{*})^{2}} + \frac{\Theta(\{\kappa_{j}^{*} L_{2} + L_{3} + 1\}\etap + E_{+})}{\gap_{j}^{*}}\rb\lb\mnorm{U^{*}} + \frac{\Theta\mnorm{E U^{*}}}{\lambda_{\min, j}^{*}}\rb\nonumber\\
&\quad  + 11\sum_{j=1}^{B}\frac{\Theta(\bp_{2} + \mnorm{\Ap^{*}})E_{+}}{\lambda_{\min, j}^{*}\gap_{j}^{*}}.\label{eq:Step5_trick_1}
\end{align}
By Lemma \ref{lem:partition_general}, 
\begin{align*}
  \sum_{j=1}^{B}\frac{1}{(\gap_{j}^{*})^{2}}& \le \frac{2H(2, 0)}{(\gap^{*})^{2}} = \frac{14}{3(\gap^{*})^{2}}\\
  \sum_{j=1}^{B}\frac{\kappa_{j}^{*}}{\gap_{j}^{*}} & \le \sum_{j=1}^{B}\frac{2|S_{j}|}{\gap^{*}} = \frac{2r}{\gap^{*}}\\
\sum_{j=1}^{B}\frac{1}{(\gap_{j}^{*})^{2}\lambda_{\min, j}^{*}} & \le \frac{2H(2, 1)}{(\gap^{*})^{2}\lambda_{\min}^{*}} \le \frac{3.01}{\gap^{*}\lambda_{\min}^{*}}\\
\sum_{j=1}^{B}\frac{\kappa_{j}^{*}}{\gap_{j}^{*}\lambda_{\min, j}^{*}} & \le \sum_{j=1}^{B}\frac{2|S_{j}|}{\gap^{*}\lambda_{\min}^{*}}\le \frac{2r}{\gap^{*}\lambda_{\min}^{*}}\\
\sum_{j=1}^{B}\frac{1}{\gap_{j}^{*}\lambda_{\min, j}^{*}} & \le \frac{2H(1, 1)}{\gap^{*}\lambda_{\min}^{*}} \le \frac{3.06}{\gap^{*}\lambda_{\min}^{*}}.
\end{align*}
To apply \eqref{eq:Step5_trick_1}, we still need \textbf{A}4 holds for each block. By Lemma \ref{lem:partition_general}, $\kappa_{j}^{*}\le 2r$ for all $j$, thus it is sufficient to assume the following stronger version of \textbf{A}4:
\begin{enumerate}[$\td{\td{\mathbf{A}}}$'1]
\setcounter{enumi}{3}
\item $\gap^{*}\ge 4\lb\Theta (\{2rL_{2} + L_{3} + 1\}\etap + E_{+}) + L_{1} + \lambda_{-} + E_{+}\rb$.
\end{enumerate}
Then under assumptions \textbf{A}'1 - \textbf{A}'3 and $\td{\td{\mathbf{A}}}$'4,
\begin{align}
\lefteqn{ d_{\ttinf}(U, U^{*} + \V)} \nonumber\\
& \le C\left\{\lb \frac{E_{+}^{2}}{(\gap^{*})^{2}} + \frac{\Theta(\{2rL_{2} + L_{3} + 1\}\etap + E_{+})}{\gap^{*}}\rb\lb\mnorm{U^{*}} + \frac{\Theta\mnorm{E U^{*}}}{\lambda_{\min}^{*}}\rb + \frac{\Theta(\bp_{2} + \mnorm{\Ap^{*}})E_{+}}{\lambda_{\min}^{*}\gap^{*}}\right\},\label{eq:Step5_trick}
\end{align}n
with probability at least $1 - 5B\delta$, where $C$ is a universal constant that can be chosen as $136$. 

\noindent \textbf{Final step}: when $\kappa^{*} <\!\! < r$, we should use \eqref{eq:Step4_trick} instead of \eqref{eq:Step5_trick}. We can split the eigenvalues into 2 blocks, with all positive and negative eigenvalues in $\Lambda^{*}$, respsectively. Then similar to Appendix \ref{subapp:step6}, we have
\begin{align}
d_{\ttinf}(U, U^{*} + \V)&\le C\bigg\{\lb \frac{E_{+}^{2}}{(\gap^{*})^{2}} + \frac{\Theta(\{\kappa^{*}L_{2} + L_{3} + 1\}\etap + E_{+})}{\gap^{*}}\rb\lb\mnorm{U^{*}} + \frac{\Theta\mnorm{E U^{*}}}{\lambda_{\min}^{*}}\rb\nonumber\\
& \qquad\quad   + \frac{\Theta(\bp_{2} + \mnorm{\Ap^{*}})E_{+}}{\lambda_{\min}^{*}\gap^{*}}\bigg\},\label{eq:Step5_trick}
\end{align}
with probability at least $1 - 10\delta$, where $C$ is a universal constant that can be chosen as $58$. 

The proof of Theorem \ref{thm:generic_bound_trick} is then completed by considering two cases $\kappa^{*} > 2r$ and $\kappa^{*} \le 2r$ separately as in Appendix \ref{subapp:step6}.
\end{proof}

\begin{proof}[\textbf{Proof of Theorem \ref{thm:generic_bound2_trick}}]
  By \eqref{eq:mnormV}, 
\[\mnorm{V}\preceq \frac{\Theta \mnorm{EU^{*}}}{\lambda_{\min}^{*}}.\]
The proof is then completed by assumption \textbf{A}'4.
\end{proof}

\section{Proofs of Results in Section \ref{sec:binarybound}}\label{app:binarybound}

\subsection{Proofs for Section \ref{subsec:binary_wigner}}
The proofs heavily exploit concentration inequalities for binary random variables derived in Appendix \ref{app:concentration}. 

\begin{proof}[\textbf{Proof of Lemma \ref{lem:A3_binary}}]
 Setting
\[2\gamma = (\log (1 / \delta))^{-(1 - \alpha)}\]
 in Lemma \ref{lem:bernoulli2} yields that the condition of Proposition \ref{prop:vector2matrix} holds with
\[a_{\infty}(\delta) = \frac{2\log (1 / \delta)}{F^{-1}(2\gamma\log (1 / \delta))}, \quad a_{2}(\delta) = a_{\infty}(\delta)\sqrt{\frac{p^{*}}{2(\log (1 / \delta))^{1 - \alpha}}}.\]
By Lemma \ref{lem:Finv}, $F^{-1}(x) \ge \log x / 2$ and thus
\[a_{\infty}(\delta)\le \frac{4\log (1 / \delta)}{\alpha \log \log (1 / \delta)}, \quad a_{2}(\delta)\le \frac{\sqrt{8p^{*}(\log (1 / \delta))^{1 + \alpha}}}{\alpha \log\log (1 / \delta)}.\]
By Proposition \ref{prop:vector2matrix}, 
\[b_{\infty}(\delta) = 2a_{\infty}\lb\frac{\delta}{5^{r}n}\rb, \quad b_{2}(\delta) = 2a_{2}\lb\frac{\delta}{5^{r}n}\rb.\]
The proof is completed by the fact that $x\mapsto \log x / \log \log x$ is increasing in $x$ and 
\[\log(5^{r}n / \delta) \preceq R(\delta).\]
\end{proof}

To prove Lemma \ref{lem:Eop_binary}, we need the following concentration inequality.
\begin{proposition}\label{prop:Eop}\citep[][Remark 4.12]{latala2018dimension}
  There exists a universal constant $C$ such that for any $\eps \in [0, 1]$ and $t \ge 0$,
\[\P\lb\|E\|_{\op}\ge 2(1 + \eps) \max_{i}\sqrt{\sum_{j}\E [E_{ij}^{2}]} + t\rb\le n\exp\left\{-\frac{\eps t^{2}}{C}\right\}.\]
\end{proposition}

\begin{proof}[\textbf{Proof of Lemma \ref{lem:Eop_binary}}]
For our purpose, we let $\eps = 1$, $t = \sqrt{C\log(n / \delta)}$, then
\[\|E\|_{\op}\le 2(1 + \eps) \max_{i}\sqrt{\sum_{j}\E [E_{ij}^{2}]} + t = 4 \max_{i}\sqrt{\sum_{j}\E [E_{ij}^{2}]} + \sqrt{C\lb\log (n / \delta)\rb},\]
with probability at least $1 - \delta$. Since $\E [E_{ij}^{2}] = p_{ij}(1 - p_{ij})\le p_{ij}$, we have
\[E_{2}(\delta)\preceq \sqrt{n\bar{p}^{*}} + \sqrt{\log (n / \delta)}.\]
The result is proved by \eqref{eq:cond_E2}.
\end{proof}

\begin{proof}[\textbf{Proof of Lemma \ref{lem:mnormEU_binary}}]
Note that $F^{-1}(e) = 1$. Setting
\[2\gamma = \frac{e}{\log (1 / \delta)}.\]
in Lemma \ref{lem:bernoulli2} yields that the condition of Proposition \ref{prop:vector2matrix} holds with
\[a_{\infty}(\delta) = \frac{2\log (1 / \delta)}{F^{-1}(e)} = 2\log (1 / \delta), \quad a_{2}(\delta) = a_{\infty}(\delta)\sqrt{\frac{ep^{*}}{2\log (1 / \delta)}} = \sqrt{2ep^{*}\log (1 / \delta)}.\]
Note that they are different from the ones in Lemma \ref{lem:A3_binary}. By Proposition \ref{prop:vector2matrix},
\[\mnorm{EU^{*}}\le 2a_{\infty}(\delta / 5^{r}n)\mnorm{U^{*}}  + 2a_{2}(\delta / 5^{r}n).\]
where we use the fact that $\|U^{*}\|_{\op} = 1$. Then
\[log(5^{r}n / \delta) = \log (n / \delta) + (\log 5)r \preceq R(\delta),\]
The proof is then completed.
\end{proof}

\begin{proof}[\textbf{Proof of Theorem \ref{thm:generic_binary}}]
Let $\event$ be the intersection of the events in Theorem \ref{thm:generic_bound} and Lemma \ref{lem:mnormEU_binary}. Then a union bound implies that
\[P(\event) \ge 1 - (B(r) + 1)\delta.\]
Throughout the rest of the proof we restrict the attention on $\event$. For notational convenience we will suppress the notation $(\delta)$ for all quantities that involve it. 

\noindent By Lemma \ref{lem:Eop_binary}
\begin{equation}\label{eq:etag}
  \eta\preceq \sqrt{n\bar{p}^{*}} + \sqrt{\log(n / \delta)} + \frac{R}{\alpha \log R} + \frac{\sqrt{R^{1 + \alpha}p^{*}}}{\alpha \log R}\preceq g.
\end{equation}
By part (a) of Proposition \ref{prop:A1} and Lemma \ref{lem:Eop_binary},
\begin{equation}
  \label{eq:L1_binary}
  L_{1} \preceq \mnorm{A^{*}} + E_{\infty} \preceq  g,
\end{equation}
where we use the fact that
\[\mnorm{A^{*}} = \max_{i}\sqrt{\sum_{j=1}^{n}p_{ij}^{2}}\le \sqrt{n\bar{p}^{*}p^{*}}\le \sqrt{n\bar{p}^{*}}.\]
In addition, 
\begin{equation}
  \label{eq:L2L3_binary}
  L_{2} = 1, \quad L_{3}\le \frac{\mnorm{A^{*}} + E_{\infty} + \lambda_{-}}{\lambda_{\min}^{*}}\preceq \frac{g}{\gap^{*}}\preceq 1.
\end{equation}
First we verify that assumption \textbf{A}4 holds in this case. By \eqref{eq:A4_binary} in page \pageref{eq:A4_binary}, \eqref{eq:etag} and \eqref{eq:L2L3_binary}, 
\[\gap^{*}\ge C\bar{\kappa}^{*}g \ge CC'(\bar{\kappa}^{*}L_{2} + L_{3} + 1)\eta\]
where $C'$ is a universal constant. In addition, by \eqref{eq:A4_binary}, \eqref{eq:L1_binary} and \eqref{eq:L2L3_binary}, 
\[\gap^{*}\ge C\bar{\kappa}^{*}g \ge CC''(E_{+} + L_{1} + \lambda_{-}).\]
By taking $C = 4/C' + 4/C''$, we prove that
\[\gap^{*}\ge 4\lb\sigma + L_{1} + \lambda_{-}\rb.\]
This validates assumption \textbf{A}4. Since assumptions \textbf{A}1 - \textbf{A}4 hold, by Theorem \ref{thm:generic_bound}, we obtain that 
  \begin{align*}
    \lefteqn{d_{\ttinf}(U, AU^{*}(\Lambda^{*})^{-1})}\\
& \stackrel{(i)}{\preceq} \frac{1}{\gap^{*}}\bigg\{(\bar{\kappa}^{*}\eta + E_{+})\lb\mnorm{U^{*}} + \frac{\mnorm{EU^{*}}}{\lambda_{\min}^{*}}\rb  + \lb\frac{E_{+}b_{2}}{\lambda_{\min}^{*}} + \min\{E_{+}\xi_{1}, \bar{E}_{+}\sqrt{\bar{\kappa}^{*}}\xi_{2}, \bar{E}_{+}\bar{\kappa}^{*}\xi_{3}\}\rb\bigg\}\\
& \stackrel{(ii)}{\preceq} \frac{1}{\gap^{*}}\bigg\{\bar{\kappa}^{*}g\lb\mnorm{U^{*}} + \frac{R}{\lambda_{\min}^{*}}\mnorm{U^{*}} + \frac{\sqrt{Rp^{*}}}{\lambda_{\min}^{*}}\rb\\
& \qquad  + (\sqrt{n\bar{p}^{*}} + \sqrt{\log (n / \delta)})\lb\frac{\sqrt{R^{1 + \alpha}p^{*}}}{\alpha (\log R)\lambda_{\min}^{*}} + \min\{\xi_{1}, \sqrt{\bar{\kappa}^{*}}\xi_{2}, \bar{\kappa}^{*}\xi_{3}\}\rb\bigg\}\\
& \preceq \frac{1}{\gap^{*}}\left\{\bar{\kappa}^{*}g\lb 1 + \frac{R}{\lambda_{\min}^{*}}\rb\mnorm{U}^{*} + \frac{\sqrt{Rp^{*}}}{\lambda_{\min}^{*}}\lb \bar{\kappa}^{*}g + \frac{(\sqrt{n\bar{p}^{*}} + \sqrt{\log (n / \delta)})\sqrt{R^{\alpha}}}{\alpha \log R}\rb\right\}\\
& \qquad + \frac{\sqrt{n\bar{p}^{*}} + \sqrt{\log (n / \delta)}}{\gap^{*}}\min\{\xi_{1}, \sqrt{\bar{\kappa}^{*}}\xi_{2}, \bar{\kappa}^{*}\xi_{3}\}\\
& \stackrel{(iii)}{\preceq} \frac{1}{\gap^{*}}\left\{\bar{\kappa}^{*}g\lb 1 + \frac{R}{\lambda_{\min}^{*}}\rb\mnorm{U}^{*} + \frac{\sqrt{Rp^{*}}}{\lambda_{\min}^{*}}\lb \bar{\kappa}^{*}g + \frac{\sqrt{n\bar{p}^{*}R^{\alpha}}}{\alpha \log R}\rb\right\}\\
& \qquad + \frac{\sqrt{n\bar{p}^{*}} + \sqrt{\log (n / \delta)}}{\gap^{*}}\min\{\xi_{1}, \sqrt{\bar{\kappa}^{*}}\xi_{2}, \bar{\kappa}^{*}\xi_{3}\},
  \end{align*}
where (i) uses the fact that $\bar{\kappa}^{*}L_{2} + L_{3} + 1 \preceq \bar{\kappa}^{*}$, (ii) uses Lemma \ref{lem:Eop_binary} and \eqref{eq:etag}, and (iii) uses the fact that
\[\frac{\sqrt{\log (n / \delta)R^{\alpha}}}{\alpha \log R} \preceq \frac{R}{\alpha \log R}\preceq g \preceq \bar{\kappa}^{*}g.\]
The proof of this inequality is completed by plugging in the definition of $\xi_{1}, \xi_{2}$ and $\xi_{3}$.

\noindent By Lemma \ref{lem:mnormEU_binary} we obtain that
  \begin{align*}
    &d_{\ttinf}(U, U^{*})\preceq \left\{\frac{\bar{\kappa}^{*}g}{\gap^{*}}\lb 1 + \frac{R}{\lambda_{\min}^{*}}\rb + \frac{R}{\lambda_{\min}^{*}}\right\}\mnorm{U^{*}} + \frac{\sqrt{Rp^{*}}}{\lambda_{\min}^{*}}\lb \frac{\bar{\kappa}^{*}g}{\gap^{*}} + \frac{\sqrt{n\bar{p}^{*}R^{\alpha}}}{\alpha \gap^{*}\log R} + 1\rb\\
& \,\, + \frac{\sqrt{n\bar{p}^{*}} + \sqrt{\log (n / \delta)}}{\gap^{*}}\min\left\{\frac{\mnorm{A^{*}}}{\lambda_{\min}^{*}}, \frac{\sqrt{\bar{\kappa}^{*}p^{*}}}{\sqrt{\lambda_{\min}^{*}}I(A^{*}\mbox{ is psd})}, \bar{\kappa}^{*}\mnorm{\bar{U}^{*}}\right\}.
  \end{align*}
By \eqref{eq:A4_binary}, $\bar{\kappa}^{*}g / \gap^{*}\preceq 1$ and thus the above bound can be simplified as the one in Theorem \ref{thm:generic_binary}.
\end{proof}

\begin{proof}[\textbf{Proof of Corollary \ref{cor:full_recovery_binary}}]
  In this case, we only keep the third term $\bar{\kappa}^{*}\mnorm{\bar{U}^{*}}$ in the minimum. Then 
\[(\sqrt{n\bar{p}^{*}} + \sqrt{\log (n / \delta)})\bar{\kappa}^{*}\mnorm{\bar{U}^{*}} \preceq (\sqrt{n\bar{p}^{*}} + \sqrt{R})\bar{\kappa}^{*}\mnorm{\bar{U}^{*}}\preceq \bar{\kappa}^{*}g\mnorm{\bar{U}^{*}}\preceq \bar{\kappa}^{*}g\mnorm{U^{*}}.\]
This can be incorporated into the first term $\bar{\kappa}^{*}g\lb 1 + \frac{R}{\lambda_{\min}^{*}}\rb\mnorm{U^{*}}$. The proof is then completed.
\end{proof}

\begin{proof}[\textbf{Proof of Corollary \ref{cor:typical_binary}}]
 In this case, we only keep the first term $\frac{\mnorm{A^{*}}}{\lambda_{\min}^{*}}$ in the minimum. Then 
\[(\sqrt{n\bar{p}^{*}} + \sqrt{\log (n / \delta)})\frac{\mnorm{A^{*}}}{\lambda_{\min}^{*}}\preceq g\frac{\sqrt{n}p^{*}}{np^{*} / \sqrt{n}\mnorm{U^{*}}} = g\mnorm{U^{*}} \preceq \bar{\kappa}^{*}g\mnorm{U^{*}}.\]
This term can also be incorporated into the first term $\bar{\kappa}^{*}g\lb 1 + \frac{R}{\lambda_{\min}^{*}}\rb\mnorm{U^{*}}$. Thus, we obtain that
  \begin{align*}
    d_{\ttinf}(U, AU^{*}(\Lambda^{*})^{-1})&\preceq \frac{1}{\gap^{*}}\bigg\{\bar{\kappa}^{*}g\lb 1 + \frac{R}{\lambda_{\min}^{*}}\rb\mnorm{U^{*}} + \frac{\sqrt{Rp^{*}}}{\lambda_{\min}^{*}}\lb \bar{\kappa}^{*}g + \frac{\sqrt{n\bar{p}^{*}R^{\alpha}}}{\alpha \log R}\rb\bigg\}.
  \end{align*}
Note that $\sqrt{n}\mnorm{U^{*}}\succeq \sqrt{r}\succeq 1$. By \eqref{eq:typical_minLambda_binary}, 
\[np^{*} \preceq \lambda_{\min}^{*}\sqrt{n}\mnorm{U^{*}}\preceq \lambda_{\min}^{*}(\sqrt{n}\mnorm{U^{*}})^{2}.\]
As a result,
\begin{align*}
  \frac{\sqrt{Rp^{*}}}{\lambda_{\min}^{*}} &= \frac{1}{\sqrt{n}}\sqrt{\frac{R}{\lambda_{\min}^{*}}}\sqrt{\frac{np^{*}}{\lambda_{\min}^{*}}}\preceq \sqrt{\frac{R}{\lambda_{\min}^{*}}}\mnorm{U^{*}}\\
& \preceq \lb 1 + \frac{R}{\lambda_{\min}^{*}}\rb\mnorm{U^{*}}.
\end{align*}
On the other hand, 
\begin{align*}
  \frac{\sqrt{Rp^{*}}}{\lambda_{\min}^{*}}\frac{\sqrt{n\bar{p}^{*}R^{\alpha}}}{\alpha \log R} &\preceq \frac{1}{\sqrt{n}}\frac{np^{*}}{\lambda_{\min}^{*}} \frac{\sqrt{R^{1 + \alpha}}}{\alpha \log R}\\
& \preceq \mnorm{U^{*}} \frac{\sqrt{R^{1 + \alpha}}}{\alpha \log R} \preceq \mnorm{U^{*}}g
\end{align*}
where the last inequality uses the fact that $\alpha < 1$. Therefore, we conclude that 
  \begin{align*}
    d_{\ttinf}(U, AU^{*}(\Lambda^{*})^{-1})&\preceq \frac{\bar{\kappa}^{*}g}{\gap^{*}}\lb 1 + \frac{R}{\lambda_{\min}^{*}}\rb\mnorm{U^{*}}.
  \end{align*}
Finally, by the triangle inequality and Lemma \ref{lem:mnormEU_binary}, 
\begin{align*}
  d_{\ttinf}(U, U^{*}) &\preceq d_{\ttinf}(U, AU^{*}(\Lambda^{*})^{-1}) + \frac{\mnorm{EU^{*}}}{\lambda_{\min}^{*}}\\
& \preceq \left\{\frac{\bar{\kappa}^{*}g}{\gap^{*}}\lb 1 + \frac{R}{\lambda_{\min}^{*}}\rb + \frac{R}{\lambda_{\min}^{*}}\right\}\mnorm{U^{*}} + \frac{\sqrt{Rp^{*}}}{\lambda_{\min}^{*}}\\
& \preceq \lb\frac{\bar{\kappa}^{*}g}{\gap^{*}} + \frac{R}{\lambda_{\min}^{*}}\rb\mnorm{U^{*}} + \frac{\sqrt{Rp^{*}}}{\lambda_{\min}^{*}}.
\end{align*}
\end{proof}

\subsection{Proofs for Section \ref{subsec:binary_laplacian}}
Note that $\Ep = \Lp - \E \Lp = - (\Ap - \E \Ap)$ where $\Ap$ is a binary matrix with independent entries. Thus Lemma \ref{lem:A3_laplacian} is a direct consequence of Lemma \ref{lem:A3_binary} and the bound for $\Ep_{\infty}(\delta)$ in Lemma \ref{lem:Eop_laplacian} is a direct consequence of Lemma \ref{lem:Eop_binary}. For other results we need the following lemma.
\begin{lemma}\label{lem:E2Ep2}
For any $\delta \in (0, 1)$, it holds with probability $1 - \delta$ that
\[  \max_{k}|\L_{kk} - \L_{kk}^{*}|\le 4M(\delta).\]
Moreover, 
\[E_{2}(\delta)\preceq M(\delta), \quad \Ep_{2}(\delta)\preceq \sqrt{n\bar{p}^{*}} + \sqrt{\log (n / \delta)}.\]
\end{lemma}
\begin{proof}
By Lemma \ref{lem:bernoulli2} with $w = \one_{n}$ and $\gamma = e / 2\log(1 / \delta')$, it holds with probability $1 - \delta'$ that
\[\L_{kk} - \L_{kk}^{*}\le 2\log (1 / \delta')(1 + \sqrt{\gamma n\bar{p}^{*}})\le 2\log (1 / \delta') + \sqrt{2e n\bar{p}^{*}\log (1 / \delta')}.\]
Similarly, with probability $1 - \delta'$, 
\[\L_{kk}^{*} - \L_{kk} \le 2\log (1 / \delta') + \sqrt{2e n\bar{p}^{*}\log (1 / \delta')}.\]
Letting $\delta' = \delta / 2n$ and applying the union bound, we obtain that
\begin{equation}
  \label{eq:maxLkk}
  \max_{k}|\L_{kk} - \L_{kk}^{*}|\le 2\log (2n / \delta) + \sqrt{2e n\bar{p}^{*}\log (2n / \delta)}\le 4\log (n / \delta) + \sqrt{4e n\bar{p}^{*}\log(n / \delta)} \le 4M(\delta).
\end{equation}
The result on $\Ep_{2}(\delta)$ can be obtained from Lemma \ref{lem:Eop_binary}. By Weyl's inequality,
\[\|E\|_{\op}\le \|\Ep\|_{\op} + \max_{k}|\L_{kk} - \L_{kk}^{*}|.\]  
Thus, 
\[E_{2}(\delta)\preceq M(\delta) + \sqrt{n\bar{p}^{*}} + \sqrt{\log (n / \delta)}\preceq M(\delta).\]
\end{proof}

\begin{proof}[\textbf{Proof of Lemma \ref{lem:Eop_laplacian}}]
Since $\lambda_{-}(\delta), E_{+}(\delta)\preceq E_{2}(\delta)$. This is a direct consequence of Lemma \ref{lem:E2Ep2}.
\end{proof}

\begin{proof}[\textbf{Proof of Lemma \ref{lem:mnormEU_laplacian}}]
By Lemma \ref{lem:mnormEU_binary}, 
\[\mnorm{\Ep U^{*}}\preceq R(\delta)\mnorm{U^{*}} + \sqrt{R(\delta)p^{*}}.\]
By the triangle inequality and Lemma \ref{lem:E2Ep2},
\begin{align*}
  \mnorm{EU^{*}}&\le \mnorm{\Ep U^{*}} + \mnorm{(E - \Ep)U^{*}} \le \mnorm{\Ep U^{*}} + \max_{k}|\L_{kk} - \L_{kk}^{*}| \mnorm{U^{*}}\\
& \preceq \mnorm{\Ep U^{*}} + M(\delta)\mnorm{U^{*}} \preceq (M(\delta) + R(\delta))\mnorm{U^{*}} + \sqrt{R(\delta)p^{*}}.
\end{align*}
\end{proof}

\begin{proof}[\textbf{Proof of Lemma \ref{lem:L1L2L3_laplacian}}]
Let $\A^{(k)}$ be defined as in \eqref{eq:AkNk} in page \pageref{eq:AkNk} and $\D^{(k)} = \diag(\A^{(k)}\one_{n})$. Construct $\L^{(k)}$ as
\[\L^{(k)} = \D^{(k)} - \A^{(k)} + \sum_{i\in \N_{k}}\L_{ii}^{*}e_{i}e_{i}^{T},\]
where $e_{i}$ is the $i$-th canonical basis in $\R^{n}$. Then $\L^{(k)}$ is independent of $\L_{k}$ because $\L^{(k)}$ only depends on $(\A_{ij})_{i, j\not\in \N_{k}}$, which are independent of $\A_{k}$, and $\L_{k}$ is a function of $\A_{k}$. Let $\event(\delta)$ denote the event that $\|E\|_{\op}\le E_{2}(\delta / 2), \quad \|\Ep\|_{\op}\le \Ep_{2}(\delta / 2)$. Then $\P(\event(\delta))\ge 1 - \delta$ and it is left to prove the deterministic inequalities on the event $\event(\delta)$. First,
\begin{align*}
  &\|\L^{(k)} - \L\|_{\op}^{2} \le \|\L^{(k)} - \L\|_{F}^{2}\\
 = & \sum_{i\in \N_{k}, j\not = i}\{(\L_{ij}^{(k)} - \L_{ij})^{2} + (\L_{ji}^{(k)} - \L_{ji})^{2}\} + \sum_{i\in \N_{k}}(\L_{ii}^{(k)} - \L_{ii})^{2} + \sum_{i\in \N_{k}^{c}}(\L_{ii}^{(k)} - \L_{ii})^{2}\\
= & \sum_{i\in \N_{k}, j\not = i}\{\A_{ij}^{2} + \A_{ji}^{2}\} + \sum_{i\in \N_{k}}\lb \L_{ii} - \L_{ii}^{*}\rb^{2} + \sum_{i\in \N_{k}^{c}}\lb \sum_{j\in \N_{k}}\A_{ij}\rb^{2}\\
\stackrel{(i)}{=} & \sum_{i\in \N_{k}, j\not = i}(\A_{ij}^{2} + \A_{ji}^{2}) + \sum_{i\in \N_{k}} E_{ii}^{2} + |\N_{k}|\sum_{i\in \N_{k}^{c}, j\in \N_{k}}\A_{ij}^{2}\\
\stackrel{(ii)}{=} & 2\sum_{i\in \N_{k}, j\not = i}\A_{ij}^{2} + \sum_{i\in \N_{k}}E_{ii}^{2} + |\N_{k}|\sum_{i\in \N_{k}, j\in \N_{k}^{c}}\A_{ij}^{2}\\
\le &(m + 2)\sum_{i\in \N_{k}}\|\A_{i}\|_{2}^{2} + \sum_{i\in \N_{k}}E_{ii}^{2}\\
\stackrel{(iii)}{\le} & (m + 1)^{2}\mnorm{\A}^{2} + \sum_{i\in \N_{k}}E_{ii}^{2}\\
 \le & (m + 1)^{2}\mnorm{\A}^{2} + m\|E\|_{\op}^{2}
\end{align*}
where (i) uses the Cauchy-Schwarz inequality, (ii) uses the symmetry of $\A$ and (iii) uses the fact that $(m + 2)m\le (m + 1)^{2}$ and that $|E_{ii}|\le \|E\|_{\op}$. As a result,
\begin{align*}
  &\|\L^{(k)} - \L\|_{\op}\le (m + 1)\mnorm{\A} + \sqrt{m}\|E\|_{\op}\\
\le &(m + 1)\mnorm{\A^{*}} + (m + 1)\mnorm{\A - \A^{*}} + \sqrt{m}\|E\|_{\op}\\
= & (m + 1)\mnorm{\A^{*}} + (m + 1)\mnorm{\Ep} + \sqrt{m}\|E\|_{\op}\\
 \preceq & m\mnorm{\A^{*}} + m\Ep_{2}(\delta / 2) + \sqrt{m}E_{2}(\delta / 2)\\
\preceq & \sqrt{m}M(\delta) + m(\sqrt{n\bar{p}^{*}} + \sqrt{\log (n / \delta)}),
\end{align*}
where the last line uses Lemma \ref{lem:E2Ep2}.

~\\
\noindent On the other hand, we can derive a decomposition of $\L^{(k)} - \L$. Let $\td{\L}_{1}, \td{\L}_{2}, \td{\L}_{3}\in \R^{n}$ be three matrices with
\[\td{\L}_{1, ij} = \L_{ij}I(i\in \N_{k}, j\in [n]), \,\,\, \td{\L}_{2, ij} = \L_{ij}I(i\in \N_{k}^{c}, j\in \N_{k}), \,\,\, \td{\L}_{3} = \sum_{i\in \N_{k}}\L_{ii}^{*}e_{i}e_{i}^{T} + \sum_{i\in \N_{k}^{c}}\lb\sum_{j\in \N_{k}}\A_{ij}\rb e_{i}e_{i}^{T}.\]
By definition,
\[\L - \L^{(k)} = \td{\L}_{1} + \td{\L}_{2} - \td{\L}_{3}.\]
Then 
\begin{align*}
  &\|(\L^{(k)} - \L)U\|_{\op} \le \|\td{\L}_{1}U\|_{\op} + \|\td{\L}_{2}U\|_{\op} + \|\td{\L}_{3}U\|_{\op}\\ 
\le & \sum_{i\in \N_{k}}\lb\|\L_{i}^{T}U\|_{2} + \|\td{\L}_{2i}U_{i}^{T}\|_{\op}\rb + \left\|\sum_{i\in \N_{k}}\L_{ii}^{*}e_{i}U_{i}^{T}\right\|_{\op} + \left\|\sum_{i\in \N_{k}^{c}}\lb\sum_{j\in \N_{k}}\A_{ij}\rb e_{i}U_{i}^{T}\right\|_{\op}\\
\le & \sum_{i\in \N_{k}}\lb\|(\L U)_{i}^{T}\|_{2} + \|\td{\L}_{2i}\|_{2}\|U_{i}\|_{2}\rb + \sum_{i\in \N_{k}}\L_{ii}^{*}\|U_{i}\|_{2} + \sum_{i\in \N_{k}^{c}}\lb\sum_{j\in \N_{k}}\A_{ij}\rb \|U_{i}\|_{2}\\
  \le & \sum_{i\in \N_{k}}\lb\|(U\Lambda)_{i}\|_{2} + \|\td{\L}_{2i}\|_{2}\|U_{i}\|_{2}\rb + \left\{m\max_{i\in \N_{k}}\L_{ii}^{*} + \sum_{j\in \N_{k}}\sum_{i\in \N_{k}^{c}}\A_{ij}\right\}\mnorm{U}\\
  \le & \sum_{i\in \N_{k}}\lb\|(U\Lambda)_{i}\|_{2} + \|\td{\L}_{2i}\|_{2}\|U_{i}\|_{2}\rb + \left\{m\max_{i\in \N_{k}}\L_{ii}^{*} + \sum_{j\in \N_{k}}\sum_{i=1}^{n}\A_{ij}\right\}\mnorm{U}\\
  \le & \sum_{i\in \N_{k}}\lb\|(U\Lambda)_{i}\|_{2} + \|\td{\L}_{2i}\|_{2}\|U_{i}\|_{2}\rb + \left\{m\max_{i\in \N_{k}}\L_{ii}^{*} + m\max_{i\in \N_{k}}\sum_{i=1}^{n}\A_{ij}^{2}\right\}\mnorm{U}\\
  \le & \sum_{i\in \N_{k}}\lb\|(U\Lambda)_{i}\|_{2} + \|\td{\L}_{2i}\|_{2}\|U_{i}\|_{2}\rb + m\left\{\max_{i\in \N_{k}}\L_{ii}^{*} + \mnorm{\A}^{2}\right\}\mnorm{U}\\
= & \sum_{i\in \N_{k}}\|U_{i}^{T}\Lambda\|_{2} + \|\td{\L}_{2i}\|_{2}\|U_{i}\|_{2} + m\left\{\max_{i\in \N_{k}}\L_{ii}^{*} + \mnorm{\A}^{2}\right\}\mnorm{U}\\
\le &\sum_{i\in \N_{k}}(\lambda_{\max}(\Lambda) + \|\A_{i}\|_{2})\|U_{i}\|_{2} + m\left\{\max_{i\in \N_{k}}\L_{ii}^{*} + \mnorm{\A}^{2}\right\}\mnorm{U}\\
\le & \lb m\lambda_{\max}(\Lambda) + m\mnorm{\A} + m\max_{i}\L_{ii}^{*} + m\mnorm{\A}^{2}\rb\mnorm{U}.
\end{align*}
By triangle inequality,
\[\mnorm{\A}\le \mnorm{\A^{*}} + \mnorm{\Ep}\le \mnorm{\A^{*}} + \|\Ep\|_{\op}.\]
By Weyl's inequality,
\[|\lambda_{\max}(\Lambda) - \lambda_{\max}^{*}|\le \|E\|_{\op}.\]
As a result,
\[\|(\L^{(k)} - \L)U\|_{\op}\le m\lb\lambda_{\max}^{*} + \|E\|_{\op} + \mnorm{\A^{*}} + 2\mnorm{\A^{*}}^{2} + \|\Ep\|_{\op} + 2\|\Ep\|_{\op}^{2} + \max_{i}\L_{ii}^{*}\rb\mnorm{U}.\]
Thus on event $\event(\delta)$, by Lemma \ref{lem:E2Ep2}, we have $L_{2}(\delta) = m$ and $L_{3}(\delta) \preceq \frac{m(n\bar{p}^{*} + \log (n / \delta))}{\lambda_{\min}^{*}}$.   
\end{proof}

\begin{proof}[\textbf{Proof of Theorem \ref{thm:generic_binary_laplacian}}]
For notational convenience we will suppress the notation $(\delta)$ for all quantities that involve it. First we prove that assumption \textbf{A}'4 is satisfied. By Lemma \ref{lem:A3_laplacian} and Lemma \ref{lem:Eop_laplacian}, 
\[\td{\eta}\preceq \sqrt{n\bar{p}^{*}} + \sqrt{\log (n / \delta)} + \frac{R}{\alpha \log R}\preceq g.\]
By Lemma \ref{lem:L1L2L3_laplacian} and the fact that $\lambda_{\min}^{*}\ge \gap^{*}\succeq M \succeq \log (n / \delta)$,
\[\bar{\kappa}^{*}L_{2} + L_{3} + 1\preceq \bar{\kappa}^{*} + \frac{n\bar{p}^{*} + \log (n / \delta)}{\lambda_{\min}^{*}} \preceq \bar{\kappa}'.\]
By Lemma \ref{lem:Eop_laplacian}, Lemma \ref{lem:L1L2L3_laplacian} and the fact that $\mnorm{\Ap^{*}}\le \sqrt{n\bar{p}^{*}p^{*}}$, 
\[L_{1} + \lambda_{-} + E_{+}\preceq M.\]
Putting pieces together, we have 
\begin{equation}
  \label{eq:Thetasigmap}
  \Theta\sigmap \preceq \Theta(\bar{\kappa}' g + M), \quad \Theta \sigmap  + L_{1}  + \lambda_{-} + E_{+}\preceq \Theta\bar{\kappa}' g + (\Theta + 1)M.
\end{equation}
Thus, when $C$ is large enough, assumption \textbf{A}'4 is satisfied.

~\\
\noindent Next we prove the bound for $d_{\ttinf}(U, U^{*} + \V)$. We bound each term in Theorem \ref{thm:generic_bound_trick} separately. By Lemma \ref{lem:Eop_laplacian} and \eqref{eq:Thetasigmap},
\begin{equation}
  \label{eq:term1_binary_laplacian}
  \frac{E_{+}^{2}}{(\gap^{*})^{2}} + \frac{\Theta\sigmap}{\gap^{*}}\preceq \frac{M^{2}}{(\gap^{*})^{2}} + \frac{\Theta(\bar{\kappa}' g + M)}{\gap^{*}},
\end{equation}
By Lemma \ref{lem:mnormEU_laplacian}, 
\begin{align}
  \mnorm{U^{*}} + \frac{\Theta\mnorm{E U^{*}}}{\lambda_{\min}^{*}}&\preceq \lb 1 + \frac{\Theta(M + r)}{\lambda_{\min}^{*}}\rb\mnorm{U^{*}} + \frac{\Theta\sqrt{Rp^{*}}}{\lambda_{\min}^{*}}\nonumber\\
& \preceq \lb 1 + \frac{\Theta r}{\lambda_{\min}^{*}}\rb\mnorm{U^{*}} + \frac{\Theta\sqrt{Rp^{*}}}{\lambda_{\min}^{*}}, \label{eq:term2_binary_laplacian}
\end{align}
where the last line uses the condition that $\lambda_{\min}^{*}\ge \gap^{*}\succeq \Theta M$. By Lemma \ref{lem:A3_laplacian} and Lemma \ref{lem:Eop_laplacian},
\begin{equation}\label{eq:term3_binary_laplacian}
  \frac{\Theta(\bp_{2} + \mnorm{\Ap^{*}})E_{+}}{\gap^{*}\lambda_{\min}^{*}}\preceq \frac{\Theta M\sqrt{p^{*}}}{\gap^{*}\lambda_{\min}^{*}}\lb \sqrt{n\bar{p}^{*}} + \frac{\sqrt{R^{1 + \alpha}}}{\alpha \log R}\rb.
\end{equation}
By \eqref{eq:term1_binary_laplacian}, \eqref{eq:term2_binary_laplacian}, \eqref{eq:term3_binary_laplacian} and Theorem \ref{thm:generic_bound_trick}, we have
\begin{align*}
  \lefteqn{ d_{\ttinf}(U, U^{*} + \V)}\\
& \preceq \lb \frac{M^{2}}{(\gap^{*})^{2}} + \frac{\Theta(\bar{\kappa}'g + M)}{\gap^{*}}\rb\left\{\lb 1 + \frac{\Theta r}{\lambda_{\min}^{*}}\rb\mnorm{U^{*}} + \frac{\Theta\sqrt{Rp^{*}}}{\lambda_{\min}^{*}}\right\} + \frac{\Theta M\sqrt{p^{*}}}{\gap^{*}\lambda_{\min}^{*}}\lb \sqrt{n\bar{p}^{*}} + \frac{\sqrt{R^{1 + \alpha}}}{\alpha \log R}\rb.
\end{align*}
This completes the proof of the first bound. For the second one, we apply Theorem \ref{thm:generic_bound2_trick}. By Lemma \ref{lem:mnormEU_laplacian},
\[\frac{\Theta\mnorm{\Ep U^{*}}}{\lambda_{\min}^{*}}\preceq \frac{\Theta (M + r)}{\lambda_{\min}^{*}}\mnorm{U^{*}} + \frac{\Theta\sqrt{Rp^{*}}}{\lambda_{\min}^{*}}.\]
By the triangle inequality, 
\begin{align*}
  d_{\ttinf}(U, U^{*}) & \preceq \left\{\lb \frac{M^{2}}{(\gap^{*})^{2}} + \frac{\Theta(\bar{\kappa}'g + M)}{\gap^{*}}\rb\lb 1 + \frac{\Theta r}{\lambda_{\min}^{*}}\rb + \frac{\Theta(M + r)}{\lambda_{\min}^{*}}\right\}\mnorm{U^{*}}\\
& \quad + \left\{\lb \frac{M^{2}}{(\gap^{*})^{2}} + \frac{\Theta(\bar{\kappa}'g + M)}{\gap^{*}}\rb + 1\right\}\frac{\Theta\sqrt{Rp^{*}}}{\lambda_{\min}^{*}}\\
& \quad + \frac{\Theta M\sqrt{p^{*}}}{\gap^{*}\lambda_{\min}^{*}}\lb \sqrt{n\bar{p}^{*}} + \frac{\sqrt{R^{1 + \alpha}}}{\alpha \log R}\rb.
\end{align*}
The proof is then completed by assumption \textbf{A}'4 that
\[\frac{M^{2}}{(\gap^{*})^{2}} + \frac{\Theta(\bar{\kappa}'g + M)}{\gap^{*}}\preceq 1\Longrightarrow \lb\frac{M^{2}}{(\gap^{*})^{2}} + \frac{\Theta(\bar{\kappa}'g + M)}{\gap^{*}}\rb \frac{\Theta r}{\lambda_{\min}^{*}}\preceq \frac{\Theta(M + r)}{\lambda_{\min}^{*}}.\]
\end{proof}

\begin{proof}[\textbf{Proof of Lemma \ref{lem:Theta1}}]
By \eqref{eq:maxLkk}, with probability $1 - \delta$,
\[\max_{k}|\L_{kk} - \L^{*}_{kk}|\le 4M(\delta)\le \frac{4}{5}\min_{j\in [s+1, s+r], k\in [n]}|\Lambda^{*}_{jj} - \L^{*}_{kk}|.\]
As a result,
\[\Theta(\delta)\le \frac{\min_{j\in [s+1, s+r]}|\Lambda^{*}_{jj}|}{\min_{j\in [s+1, s+r], k\in [n]}|\Lambda^{*}_{jj} - \L^{*}_{kk}| - 4M(\delta)}\le 5\Theta^{*}.\]
\end{proof}

\section{Other Proofs}\label{app:other}
\subsection{Proofs in Section \ref{sec:spectral_norm}}\label{subapp:other_spectral_norm}
\begin{proof}[\textbf{Proof of Lemma \ref{lem:markov}}]
  By Markov inequality, for any $k \le k_{0}$,
\[\P(|Z - \E Z| \ge t)\le t^{-k}\E |Z - \E Z|^{k} = \lb \frac{\sigma\sqrt{k}}{t}\rb^{k} = \exp\left\{-k\log \lb \frac{t}{\sigma}\rb + \frac{k\log k}{2}\right\}\triangleq \exp\{q(k; t)\}.\]
Note that $q'(k; t) = (\log k + 1) / 2 - \log (t / \sigma)$ and
\[\log [(q')^{-1}(0; t)] = 2\log (t / \sigma) - 1.\]
When $\sqrt{2e}\le t / \sigma\le \sqrt{k_{0}e}$, $2\le (q')^{-1}(0; t)\le k_{0}$. Thus if we let $k = (q')^{-1}(0; t) = t^{2} / e\sigma^{2}$, then 
\[\P(|Z - \E Z| \ge t)\le \exp\{-k / 2\} = \exp\left\{-\frac{t^{2}}{2e\sigma^{2}}\right\}.\]
When $t / \sigma < \sqrt{2e}$, 
\[\P(|Z - \E Z| \ge t)\le 1 \le \exp\left\{1-\frac{t^{2}}{2e\sigma^{2}}\right\}.\]
\end{proof}

\begin{proof}[\textbf{Proof of Lemma \ref{lem:deltastar}}]
  Assume that $C_{0} > \max\{64C^{2}, 12C, e\}$  and $n\ge 16$. Let $h_{1}(y) = y - \log y, h_{2}(y) = \log y / \log\log y$. It is easy to see that $h_{1}$ is increasing on $[1, \infty)$ and $h_{1}(y)\ge 1$. Since $h_{2}(y) = \exp\{h_{1}(\log\log y)\}$, $h_{2}$ is increasing on $[e^{e}, \infty)$. Since $n \ge 16 > e^{e}$,  $g(\delta)$ is decreasing. Noting that $\bar{\kappa}^{*} = 1$ and $\gap^{*} = (n - 1)p$, it is left to prove that 
\[g(\delta^{*}) = \sqrt{np} + \frac{\log (n / \delta^{*})}{\log \log (n / \delta^{*})}\le \gap^{*} / C.\]
Since $p\ge C_{0}\log n / (n \log\log n)$, we have $np > C_{0}h_{2}(n) > C_{0}$. As a result, 
\[\sqrt{np} \le \frac{np}{\sqrt{C_{0}}} \le \frac{2(n - 1)p}{\sqrt{C_{0}}}\le \frac{2\gap^{*}}{\sqrt{C_{0}}}\le \frac{\gap^{*}}{4C}.\]
Thus it is left to show that 
\[h_{2}(n / \delta^{*}) = \frac{\log (n / \delta^{*})}{\log \log (n / \delta^{*})}\le \frac{3\gap^{*}}{4C}.\]
By definition, 
\[\log(n / \delta^{*}) = \frac{np\log (np)}{2C} + \log n.\]
By \eqref{eq:C0}, 
\begin{align*}
  np\log (np)&\ge C_{0}\frac{\log n}{\log\log n}\lb \log C_{0} + \log\log n - \log\log\log n\rb \ge C_{0}(\log n)\lb 1- \exp\{-h_{1}(\log \log \log n)\}\rb\\
& \ge C_{0}(\log n)\lb 1- \exp\{-1\}\rb \ge \frac{C_{0}}{2}\log n > 6C\log n.
\end{align*}
As a result,
\[\log(n / \delta^{*})\le \frac{2}{3C}np\log (np)\]
On the other hand, recalling that $np > C_{0} > e$, 
\[\log\log(n / \delta^{*})\ge \log\log (1 / \delta^{*}) = \log (n p) + \log\log (n p) > \log (np).\]
The proof is then completed by 
\[h_{2}(n / \delta^{*}) < \frac{2np}{3C} = \frac{2n}{3(n-1)C}\gap^{*} \le \frac{32}{45C}\gap^{*}\le \frac{3}{4C}\gap^{*}.\]
\end{proof}

\begin{proof}[\textbf{Proof of Theorem \ref{thm:inhomogeneous}}]
We follow the same pipeline as the proof for \ER graphs. First it is easy to see by modifying \eqref{eq:V+W} that 
\[V_{+} \le \|u_{1}\|_{\infty}^{4}\td{W}, \quad \mbox{where }\td{W} = 4\sum_{i < j} (p_{ij} + (1 - 2p_{ij})A_{ij}).\]
In addition, $V_{+}\le 2$ almost surely. Similarly we have $M \le 2\|u_{1}\|_{\infty}^{2}\le 2$. 

Next, we derive the minimal $\delta$ for Corollary \ref{cor:typical_binary} to work. Since $\bar{\kappa}^{*} = 1$ and $C_{0}$ is sufficiently large, it is easy to see that $\delta^{*}$, as in Lemma \ref{lem:deltastar}, can be chosen as follows
\[\delta^{*} = \exp\left\{-\frac{\gap^{*}\log \gap^{*}}{2C}\right\}.\]
By Corollary \ref{cor:typical_binary} with $\alpha = 0.5$ and $\delta \ge \delta^{*}$,
\begin{align*}
  d_{\ttinf}(u_{1}, u_{1}^{*})&\preceq \lb\frac{\sqrt{n\bar{p}^{*}} + \log (n / \delta) / \log\log (n / \delta)}{\gap^{*}} + \frac{\log (n / \delta)}{|\lambda_{1}^{*}|}\rb\frac{\zeta}{\sqrt{n}} + \frac{\sqrt{\log (n / \delta) p^{*}}}{\lambda_{1}^{*}}\\
& \stackrel{(i)}{\preceq} \frac{\sqrt{n\bar{p}^{*}} + \log (n / \delta)}{\gap^{*}}\frac{\zeta}{\sqrt{n}} + \frac{\sqrt{\log (n / \delta) p^{*}}}{\lambda_{1}^{*}}\\
& \stackrel{(ii)}{\preceq} \frac{\sqrt{n\bar{p}^{*}} + \log (n / \delta)}{\gap^{*}}\frac{\zeta}{\sqrt{n}} + \sqrt{\frac{\log (n / \delta)}{np^{*}}}\frac{\zeta}{\sqrt{n}}\\
& \stackrel{(iii)}{\preceq} \lb\frac{\sqrt{n\bar{p}^{*}} + \log (n / \delta)}{\gap^{*}} + \sqrt{\frac{\log(n / \delta)}{\gap^{*}}}\rb\frac{\zeta}{\sqrt{n}}\\
& \stackrel{(iii)}{\preceq} \lb 1 + \frac{\sqrt{n\bar{p}^{*}} + \log (n / \delta)}{\gap^{*}}\rb\frac{\zeta}{\sqrt{n}},
\end{align*}
where (i) uses the fact that $|\lambda_{1}^{*}| \ge \gap^{*}$, (ii) uses the assumption that $\lambda_{1}^{*}\succeq np^{*} / \zeta$, (iii) applies the inequality that $\gap^{*}\le |\lambda_{1}^{*}|\preceq np^{*}$ and (iv) applies the simple inequality that $2\sqrt{y}\le y + 1$. By the triangle inequality and noting that $\sqrt{n}\|u_{1}^{*}\| = \zeta$, there exists a universal constant $C_{1}$ such that for each $\delta \ge \delta^{*}$, 
\begin{equation}
  \label{eq:u1delta_general}
  \sqrt{n}\|u_{1}\|_{\infty}\le C_{1}\lb 1 + \frac{\sqrt{n\bar{p}^{*}} + \log (n / \delta)}{\gap^{*}}\rb\zeta = C_{1}\lb 1 + \frac{\sqrt{n\bar{p}^{*}} + \log n}{\gap^{*}}\rb\zeta + C_{1}\frac{\zeta}{\gap^{*}}\log\lb\frac{1}{\delta}\rb,
\end{equation}
with probability $1 - \delta$. Denote by $B_{u}$ the RHS of \eqref{eq:u1delta_general} with $\delta = \delta^{*}$ and by $\event_{1}$ the event that $\sqrt{n}\|u_{1}\|_{\infty}\le B_{u}$. Then 
\[\P(\event_{1}) \ge 1 - \delta^{*} = 1 - \exp\left\{-\frac{\gap^{*}\log \gap^{*}}{2C}\right\}.\]
Let $\event_{2}$ denote the event that $\td{W}\le 10n^{2}\bar{p}$. Using the same argument above \eqref{eq:event2} in the proof of Lemma \ref{lem:V+moment}, we can show that 
\[\P(\event_{2})\ge 1 - \exp\{-n^{2}\bar{p} / 3\}.\]
Let $\event = \event_{1}\cap \event_{2}$. Then
\[\P(\event^{c})\preceq \exp\left\{-\frac{\gap^{*}\log \gap^{*}\wedge n^{2}\bar{p}}{2C\vee 3}\right\}.\] 
Using the same argument as in the proof of Lemma \ref{lem:V+moment}, it is easy to derive the following analogue of \eqref{eq:sqrtnu1infty}:
\[\lb\E (\sqrt{n}\|u_{1}\|_{\infty})^{2k}I_{\event}\rb^{1/k}\preceq \zeta^{2}\left\{1 + \lb\frac{k\vee (\sqrt{n\bar{p}^{*}} + \log n)}{\gap^{*}}\rb^{2}\right\}.\]
Thus, 
\begin{align*}
  \lb\E V_{+}^{k / 2}\rb^{1 / k} &= \lb \E V_{+}^{k / 2}I_{\event} + \E V_{+}^{k / 2}I_{\event^{c}}\rb^{1 / k} \le \lb \E V_{+}^{k / 2}I_{\event}\rb^{1 / k} + \lb \E V_{+}^{k / 2}I_{\event^{c}}\rb^{1 / k}\\
& \preceq \sqrt{\bar{p}}\E \lb\left[(\sqrt{n}\|u_{1}\|_{\infty})^{2k}I_{\event}\right]\rb^{1/k} + \exp\left\{-\frac{\gap^{*}\log \gap^{*}\wedge n^{2}\bar{p}}{(2C\vee 3)k}\right\}.
\end{align*}
Using the same argument for \eqref{eq:sqrtnu1infty} in the proof of Lemma \ref{lem:V+moment}, we can show that 
\[\lb\E (\sqrt{n}\|u_{1}\|_{\infty})^{2k}I_{\event}\rb^{1/k} \preceq \left\{1 + \lb\frac{(k \vee (\sqrt{n\bar{p}^{*}} + \log n))}{\gap^{*}}\rb^{2}\right\}\zeta^{2}.\]
As a result, 
\begin{equation}
  \label{eq:V+moment_general}
  \lb\E [V_{+}^{k/2}]\rb^{1/k}\preceq \bar{p}\left\{1 + \lb\frac{(k \vee (\sqrt{n\bar{p}^{*}} + \log n))}{\gap^{*}}\rb^{2}\right\}\zeta^{2} + \exp\left\{-\frac{\gap^{*}\log \gap^{*}\wedge n^{2}\bar{p}}{(2C\vee 3)k}\right\}.
\end{equation}
Let $k = 2$ and by Efron-Stein inequality, we obtain that 
\[\Var(\|A\|_{\op})\le \E [V_{+}]\preceq \bar{p}\left\{1 + \lb\frac{\sqrt{n\bar{p}^{*}} + \log n}{\gap^{*}}\rb^{4}\right\}\zeta^{4} + \exp\left\{-\frac{\gap^{*}\log \gap^{*}\wedge n^{2}\bar{p}}{2C\vee 3}\right\}.\]
On the other hand, recalling that $M \le 2\|u_{1}\|_{\infty}^{2}\le $, similarly to \eqref{eq:V+moment_general} we have
\begin{equation}
  \label{eq:Mmoment_general}
  \lb \E [M^{k}]\rb^{1/k}\preceq \frac{1}{n}\left\{1 + \lb\frac{(k \vee (\sqrt{n\bar{p}^{*}} + \log n))}{\gap^{*}}\rb^{2}\right\}\zeta^{2} + \exp\left\{-\frac{\gap^{*}\log \gap^{*}\wedge n^{2}\bar{p}}{(2C\vee 3)k}\right\}.
\end{equation}
Then by Proposition \ref{prop:boucheron}, we obtain that
\begin{equation}
  \label{eq:Aopmoment_general}
  \lb \E |\|A\|_{\op} - \E \|A\|_{\op}|^{k}\rb^{1/k}\preceq \lb\sqrt{k\bar{p}} + \frac{k}{n}\rb\left\{1 + \lb\frac{\sqrt{n\bar{p}^{*}} + \log n}{\gap^{*}}\rb^{2}\right\}\zeta^{2} + k\exp\left\{-\frac{\gap^{*}\log \gap^{*}\wedge n^{2}\bar{p}}{(2C\vee 3)k}\right\}.
\end{equation}
Let 
\[k_{0} = \frac{1}{2C\vee 3}\min\left\{\gap^{*}, \frac{\gap^{*}\log \gap^{*}}{\log (1 / \bar{p}\zeta^{2})}, \frac{n^{2}\bar{p}}{\log (n^{2}\bar{p})}, \frac{n^{2}\bar{p}}{\log (1 / \bar{p}\zeta^{2})}\right\}.\]
Consider any $k\le k_{0}$. Since $n^{2}\bar{p}\succeq 1$, $k_{0}\preceq n^{2}\bar{p}$. Thus, 
\[\sqrt{k\bar{p}} + \frac{k}{n}\preceq \sqrt{k\bar{p}}.\]
For the second term,
\begin{align*}
  & k\exp\left\{-\frac{\gap^{*}\log \gap^{*}\wedge n^{2}\bar{p}}{(2C\vee 3)k}\right\} = \sqrt{k\bar{p}}\zeta \exp\left\{-\frac{\gap^{*}\log \gap^{*}\wedge n^{2}\bar{p}}{(2C\vee 3)k} + \frac{1}{2}\log k + \frac{1}{2}\log \lb\frac{1}{\bar{p}\zeta^{2}}\rb\right\}.
\end{align*}
Since $k\le \frac{\gap^{*}\log \gap^{*}\wedge n^{2}\bar{p}}{(2C\vee 3)\log (1 / \bar{p}\zeta^{2})}$, 
\[\frac{1}{2}\log \lb\frac{1}{\bar{p}\zeta^{2}}\rb\le \frac{\gap^{*}\log \gap^{*}\wedge n^{2}\bar{p}}{2(2C\vee 3)k}.\]
Since $k\le \frac{\gap^{*}\wedge n^{2}\bar{p} / \log(n^{2}\bar{p})}{2C\vee 3}$,
\[k\log k\le \frac{\gap^{*}\log \gap^{*}}{2C\vee 3},\]
and 
\[k\log k\le \frac{1}{2C\vee 3}\frac{n^{2}\bar{p}\lb \log (n^{2}\bar{p} / (2C\vee 3)) - \log\log (n^{2}\bar{p} / (2C\vee 3))\rb}{\log (n^{2}\bar{p} / (2C\vee 3))}\le \frac{n^{2}\bar{p}}{2C\vee 3}\]
where we use the condition that $n^{2}\bar{p}$ is sufficiently large. Thus, 
\[\frac{1}{2}\log k\le \frac{\gap^{*}\log \gap^{*}\wedge n^{2}\bar{p}}{2(2C\vee 3)k}.\]
\[\sqrt{k\bar{p}}\zeta \exp\left\{-\frac{\gap^{*}\log \gap^{*}\wedge n^{2}\bar{p}}{(2C\vee 3)k} + \frac{1}{2}\log k + \frac{1}{2}\log \lb\frac{1}{\bar{p}\zeta^{2}}\rb\right\}\le \sqrt{k\bar{p}}\zeta.\]
Putting pieces together into \eqref{eq:Aopmoment_general}, we conclude that 
\[\lb \E |\|A\|_{\op} - \E \|A\|_{\op}|^{k}\rb^{1/k}\preceq \sqrt{k\bar{p}}\left\{1 + \lb\frac{\sqrt{n\bar{p}^{*}} + \log n}{\gap^{*}}\rb^{2}\right\}\zeta^{2}.\]
Finally, the proof is completed by Lemma \ref{lem:markov}. 
\end{proof}

\subsection{Proofs in Section \ref{sec:exact_recovery}}\label{subapp:other_exact_recovery}

\begin{proof}[\textbf{Proof of Lemma \ref{lem:Kmedian}}]
  First we prove that $\hat{v}_{r}$'s defined in \eqref{eq:hatv} are distinct. By definition, 
  \begin{align*}
    \frac{1}{n}\sum_{i=1}^{n}\min_{r\in [K]}\|U_{i} - \hat{v}_{r}\|_{2}&\le \frac{1}{n}\sum_{i=1}^{n}\min_{r\in [K]}\|U_{i} - v_{r}^{*}\|_{2} = \frac{1}{n}\sum_{i=1}^{n}\min_{j\in [n]}\|U_{i} - \td{U}_{j}^{*}\|_{2}\\
& \le \frac{1}{n}\sum_{i=1}^{n}\|U_{i} - \td{U}_{i}^{*}\|_{2}\le \max_{i}\|U_{i} - \td{U}_{i}^{*}\|_{2}.
  \end{align*}
By the triangle inequality, 
\[\frac{1}{n}\sum_{i=1}^{n}\min_{r\in [K]}\|U_{i} - \hat{v}_{r}\|_{2}\ge \frac{1}{n}\sum_{i=1}^{n}\min_{r\in [K]}\|\td{U}_{i}^{*} - \hat{v}_{r}\|_{2} - \max_{i}\|U_{i} - \td{U}_{i}^{*}\|_{2}.\]
The above two inequalities imply that
\begin{equation}
  \label{eq:Kmedian1}
  \frac{1}{n}\sum_{i=1}^{n}\min_{r\in [K]}\|\td{U}_{i}^{*} - \hat{v}_{r}\|_{2} \le 2\max_{i}\|U_{i} - \td{U}_{i}^{*}\|_{2}.
\end{equation}
For each $i$, let $r_{i} = \argmin_{r\in [K]}\|\td{U}_{i}^{*} - \hat{v}_{r}\|_{2}$. Since $\td{U}_{i}^{*} = \td{U}_{i'}^{*}$ if $c_{i} = c_{i'}$, it must be true that $r_{i} = r_{i'}$. Write $r_{i}$ as $r_{s}$ for $i\in \C_{s}$. Then
\[\frac{1}{n}\sum_{i=1}^{n}\min_{r\in [K]}\|\td{U}_{i}^{*} - \hat{v}_{r}\|_{2} = \sum_{s = 1}^{K} \pi_{s}\|v_{s}^{*} - \hat{v}_{r_{s}}\|_{2}.\]
For any $s$, \eqref{eq:Kmedian1} implies that
\begin{equation}
  \label{eq:Kmedian2}
  \pi_{s}\|v_{s}^{*} - \hat{v}_{r_{s}}\|_{2}\le 2\max_{i}\|U_{i} - \td{U}_{i}^{*}\|_{2}\Longrightarrow \|v_{s}^{*} - \hat{v}_{r_{s}}\|_{2}\le \frac{2}{\min_{r\in [K]} \pi_{r}}\max_{i}\|U_{i} - \td{U}_{i}^{*}\|_{2}.
\end{equation}
By the triangle inequality, 
\begin{align*}
  \|\hat{v}_{r_{s}} - \hat{v}_{r_{s'}}\|_{2}&\ge \|v_{s}^{*} - v_{s'}^{*}\|_{2} - \|v_{s}^{*} - \hat{v}_{r_{s}}\|_{2} - \|v_{s'}^{*} - \hat{v}_{r_{s'}}\|_{2}\\
&\ge \|v_{s}^{*} - v_{s'}^{*}\|_{2} - \frac{4}{\min_{r\in [K]} \pi_{r}}\max_{i}\|U_{i} - \td{U}_{i}^{*}\|_{2}\\
& \ge \frac{1}{3}\|v_{s}^{*} - v_{s'}^{*}\|_{2} > 0.
\end{align*}
This proves that $\hat{v}_{r_{s}}$'s are distinct and hence $r_{s}\not = r_{s'}$ for $s \not = s'$. Since $\{r_{1}, \ldots, r_{K}\} = \{1, \ldots, K\}$, the former must be a permutation of the latter. Without loss of generality we assume that $r_{s} = s$. It is left to prove that for any $i\in \C_{s}$,
\begin{equation}
  \label{eq:Kmedian3}
  \argmin_{r\in [K]}\|U_{i} - \hat{v}_{r}\|_{2} = s.
\end{equation}
By the triangle inequality and \eqref{eq:Kmedian2},
\begin{align*}
  \|U_{i} - \hat{v}_{s}\|_{2}&\le \|U_{i} - \td{U}_{i}^{*}\|_{2} + \|v_{s}^{*} - \hat{v}_{s}\|_{2}\le \max_{i}\|U_{i} - \td{U}_{i}^{*}\|_{2} + \frac{2}{\min_{r\in [K]} \pi_{r}}\max_{i}\|U_{i} - \td{U}_{i}^{*}\|_{2}\\
&  < \frac{3}{\min_{r\in [K]} \pi_{r}}\max_{i}\|U_{i} - \td{U}_{i}^{*}\|_{2}.
\end{align*}
On the other hand, for any $s' \not = s$,
\begin{align*}
  \|U_{i} - \hat{v}_{s'}\|_{2}&\ge -\|U_{i} - \td{U}_{i}^{*}\|_{2} + \|v_{s}^{*} - \hat{v}_{s'}\|_{2}\ge -\|U_{i} - \td{U}_{i}^{*}\|_{2} + \|v_{s}^{*} - v_{s'}^{*}\|_{2} - \|v_{s'}^{*} - \hat{v}_{s'}\|_{2}\\
& \ge \|v_{s}^{*} - v_{s'}^{*}\|_{2} - \frac{3}{\min_{r\in [K]} \pi_{r}}\max_{i}\|U_{i} - \td{U}_{i}^{*}\|_{2}\\
& \ge \frac{3}{\min_{r\in [K]} \pi_{r}}\max_{i}\|U_{i} - \td{U}_{i}^{*}\|_{2}.
\end{align*}
This proves \eqref{eq:Kmedian3} and hence completes the proof.
\end{proof}

\begin{proof}[\textbf{Proof of Lemma \ref{lem:eigen_laplacian}}]
  Let $V^{*} =
  \begin{bmatrix}
    U^{*} & \td{U}^{*}
  \end{bmatrix}
$. Then 
\begin{align*}
  V^{*}(V^{*})^{T} &= U^{*}(U^{*})^{T} + \td{U}^{*}(\td{U}^{*})^{T} = QVV^{T}Q^{T} +   \begin{bmatrix}
    Q_{1}Q_{1}^{T} & 0 & \ldots & 0\\
    0 & Q_{2}Q_{2}^{T} & \ldots & 0\\
    \vdots & \vdots & \ddots & \vdots\\
    0 & 0 & \ldots & Q_{K} Q_{K}^{T}
    \end{bmatrix}\\
& = QQ^{T} + \begin{bmatrix}
    I_{n_{1}} - \one_{n_{1}}\one_{n_{1}}^{T} / n_{1} & 0 & \ldots & 0\\
    0 &    I_{n_{2}} - \one_{n_{2}}\one_{n_{2}}^{T} / n_{2} & \ldots & 0\\
    \vdots & \vdots & \ddots & \vdots\\
    0 & 0 & \ldots &    I_{n_{K}} - \one_{n_{K}}\one_{n_{K}}^{T} / n_{K}
    \end{bmatrix} = I_{n}.
\end{align*}
Thus $V^{*}$ is orthogonal. Then it is left to prove that 
\begin{equation}
  \label{eq:laplacian_eigen}
  \L^{*}U^{*} = U^{*}\Lambda^{*}, \quad \L^{*}\td{U}^{*} = \td{U}^{*}\td{\Lambda}^{*}.
\end{equation}
The first equation is equivalent to
\[\L^{*}QV = n\rho_{n}QV\Sigma  \Longleftrightarrow \L^{*}Q = n\rho_{n}Q\td{\L}\]
Note that 
\[\td{\D}^{*} =   \begin{bmatrix}
    d^{*}_{1}I_{n_{1}} & 0 & \ldots & 0\\
    0 & d^{*}_{2}I_{n_{2}} & \ldots & 0\\
    \vdots & \vdots & \ddots & \vdots\\
    0 & 0 & \ldots & d^{*}_{K}I_{n_{K}} 
  \end{bmatrix}
\]
where 
\[d_{i}^{*} = \sum_{j=1}^{n}n_{j}(\rho_{n}B_{0, ij}) = n\rho_{n} \sum_{j=1}^{n}\pi_{j}B_{0, ij} = n\rho_{n}\td{d}_{i}.\]
Since $\L^{*}$ does not depend on the diagonal elements of $\A^{*}$, we have $\L^{*} = \td{\D}^{*} - \Ap^{*}$ where $\Ap^{*}$ is defined in \eqref{eq:AstarQM} and $\td{\D}^{*} = \diag(\Ap^{*}\one_{n})$. Then 
\[\td{\D}^{*}Q = \begin{bmatrix}
    \frac{d^{*}_{1}}{\sqrt{n_{1}}}\one_{n_{1}} & 0 & \ldots & 0\\
    0 & \frac{d^{*}_{2}}{\sqrt{n_{2}}}\one_{n_{2}} & \ldots & 0\\
    \vdots & \vdots & \ddots & \vdots\\
    0 & 0 & \ldots & \frac{d^{*}_{K}}{\sqrt{n_{K}}}\one_{n_{K}} 
  \end{bmatrix} = n\rho_{n}Q\td{\D}.\]
On the other hand,
\[\Ap^{*}Q = n\rho_{n}Q(R B_{0} R)Q^{T}Q = n\rho_{n}Q(R B_{0} R).\]
As a result, 
\[\L^{*}Q = \td{\D}^{*}Q - \Ap^{*}Q = n\rho_{n}Q\td{\L}.\]
This proves the first equation of \eqref{eq:laplacian_eigen}. To prove the second one, notice that 
\[Q^{T}\td{U}^{*} = 0 \Longrightarrow \Ap^{*}\td{U}^{*} = 0.\]
Thus,
\[\L^{*} \td{U}^{*} = \td{\D}^{*}\td{U}^{*} = \begin{bmatrix}
    d^{*}_{1}Q_{1} & 0 & \ldots & 0\\
    0 & d^{*}_{2}Q_{2} & \ldots & 0\\
    \vdots & \vdots & \ddots & \vdots\\
    0 & 0 & \ldots & d^{*}_{K}Q_{K}
  \end{bmatrix} = \td{U}^{*}\td{\Lambda}^{*}.\]
This proves the second equation of \eqref{eq:laplacian_eigen} and thus completes the proof.
\end{proof}

\begin{proof}[\textbf{Proof of Theorem \ref{thm:laplacian_spectral_clustering}}]
First we note that $\L^{*}$ does not depend on the diagonal elements of $\A^{*}$. Thus we can pretend $\A^{*} = \td{\A}^{*}$ without loss of generality. Next we note that the smallest eigenvalue of $\L^{*}$ is $0$ with an eigenvector $\one_{n}$. Since it is a constant for all units, the output of $K$-medians is not affected if it is removed. Thus, we can take $\Lambda^{*}\in \R^{(K - 1)\times (K - 1)}$ as the diagonal matrix of the second to the $K$-th smallest eigenvalues of $\L^{*}$ and $U^{*}\in \R^{n\times (K - 1)}$ as the corresponding eigenvector matrix. 

Let $\lambda_{(2)}(\cdot)$ denote the second smallest eigenvalue and
\[\beta = \frac{1}{2}\min\{\lambda_{(2)}(\td{L}_{0}), \lambda_{\min}(\td{\D}_{0}) - \lambda_{\max}(\td{L}_{0})\}.\]
Define $R, \td{\D}$ and $\td{\L}$ as in Lemma \ref{lem:eigen_laplacian}. Then
\[\td{L}\rightarrow \td{L}_{0},\quad \td{\D}\rightarrow \td{\D}_{0}.\]
Thus there exists a constant $n_{0}$ that only depends on $B_{0}$ and $\pi_{r}$'s such that 
\begin{equation}
  \label{eq:laplacian1}
  \min\{\lambda_{(2)}(\td{L}), \lambda_{\min}(\td{\D}) - \lambda_{\max}(\td{L})\} > \beta.
\end{equation}
By Lemma \ref{lem:eigen_laplacian},
\begin{equation}
  \label{eq:laplacian2}
  \lambda_{\min}^{*}\ge \gap^{*}\ge n\rho_{n}\beta.
\end{equation}
Furthermore, the matrix $U^{*}$ in this proof differs from the one in Lemma \ref{lem:eigen_laplacian} by just a column of $\one_{n}$. By Lemma \ref{lem:eigen_laplacian},
\[
  \begin{bmatrix}
    \one_{n} & U^{*}
  \end{bmatrix} = QV.
\]
It is easy to see that $U_{i}^{*} = \nu_{s}^{*}$ if $i\in \C_{s}$ and thus,
\[\|\nu_{s}^{*} - \nu_{s'}^{*}\|_{2} = \left\|
  \begin{bmatrix}
    1 \\ \nu_{s}^{*}
  \end{bmatrix}
  -
  \begin{bmatrix}
    1 \\ \nu_{s'}^{*}
  \end{bmatrix}\right\|_{2} = \left\|\frac{V_{s}}{\sqrt{n_{s}}} - \frac{V_{s'}}{n_{s'}}\right\|_{2} = \sqrt{\frac{1}{n_{s}} + \frac{1}{n_{s'}}}\ge \frac{1}{\min_{s\in [K]}\sqrt{\pi_{s}}}\frac{1}{\sqrt{n}}.\]
Moreover,
\[\mnorm{U^{*}}\le \mnorm{QV} \le \frac{1}{\min_{s\in [K]}\sqrt{n_{s}}}\le \frac{1}{\min_{s\in [K]}\sqrt{\pi_{s}}}\frac{1}{\sqrt{n}}.\]
By Lemma \ref{lem:Kmedian}, it is left to prove 
\begin{equation}
  \label{eq:laplacian_goal}
  d_{\ttinf}(U, U^{*}) \le \frac{\min_{s\in [K]}\sqrt{\pi_{s}}}{6\sqrt{n}} \triangleq \frac{c_{1}}{\sqrt{n}}.
\end{equation}

Set 
\[\delta = n^{-q}, \quad \alpha = 1 / \log R(\delta)\] 
in Theorem \ref{thm:generic_binary_laplacian}. Note that this choice of $\alpha$ implies that
\[\frac{R(\delta)}{\alpha \log R(\delta)} = R(\alpha), \quad R(\delta)^{1 + \alpha} = R(\delta)\exp\{\alpha \log R(\delta)\} = eR(\delta).\]
Then $\bar{\kappa}^{*}\le 2(K - 1)\preceq 1$, $p^{*}\preceq \rho_{n}$, 
\[\bar{\kappa}' \preceq 1 + \frac{n\rho_{n}}{n\rho_{n}\beta}\preceq 1, \quad R(\delta)\preceq \log n, \quad g(\delta) \preceq \sqrt{n\rho_{n}} + \log n, \quad M(\delta)\preceq \sqrt{n\rho_{n}\log n}.\]
By \eqref{eq:laplacian2}, when $c$ in the condition \eqref{eq:laplacian_cond} and $n$ are sufficiently large, 
\[\gap^{*} \ge C(\bar{\kappa}'g(\delta) + M(\delta)),\]
where $C$ is the universal constant in \eqref{eq:A4_binary}. On the other hand, consider $\Theta^{*}$ in Lemma \ref{lem:Theta1}. By definition, 
\[\Lambda_{jj}^{*} = n\rho_{n}\Sigma_{jj}, \quad \L_{kk}^{*} = n\rho_{n}\td{\D}_{kk}\]
where $\Sigma$ and $\td{\D}$ are defined in Lemma \ref{lem:eigen_laplacian}. Then
\[\frac{|\Lambda^{*}_{jj}|}{|\Lambda^{*}_{jj} - \L^{*}_{kk}|} = \frac{\Sigma_{jj}}{|\Sigma_{jj} - \td{\D}_{kk}|}.\]
By \eqref{eq:laplacian1}, 
\[\td{\D}_{kk} \ge \lambda_{\min}(\td{\D}) > \lambda_{\max}(\td{L}) + \beta = \lambda_{\max}(\Sigma) + \beta\ge \Sigma_{jj} + \beta,\]
and 
\[\Sigma_{jj}\le \lambda_{\max}(\td{L})\rightarrow \lambda_{\max}(\td{L}_{0}).\]
As a result,
\[\Theta^{*} = \frac{\min_{j\in [s+1, s+r]}|\Lambda^{*}_{jj}|}{\min_{j\in [s+1, s+r], k\in [n]}|\Lambda^{*}_{jj} - \L^{*}_{kk}|}\le \frac{\lambda_{\max}(\td{L}_{0})}{\beta} \preceq 1,\]
and 
\[\min_{j\in [s+1, s+r], k\in [n]}|\Lambda^{*}_{jj} - \L^{*}_{kk}|\ge n\rho_{n}\beta \ge 5 M(\delta)\]
When $c$ in the condition \eqref{eq:laplacian_cond} and $n$ are sufficiently large. By Lemma \ref{lem:Theta1}, $\Theta(\delta)\preceq 1$. 

In summary, both conditions of Theorem \ref{thm:generic_binary_laplacian} are satisfied. Then by Theorem \ref{thm:generic_binary_laplacian}, we have
\begin{align*}
  d_{\ttinf}(U, U^{*})& \preceq \lb \frac{n\rho_{n}\log n}{(n\rho_{n})^{2}} + \frac{\sqrt{n\rho_{n}} + \sqrt{\log n} + \sqrt{n\rho_{n}\log n}}{n\rho_{n}} + \frac{1}{n\rho_{n}}\rb \frac{1}{\sqrt{n}} + \frac{\sqrt{(\log n) \rho_{n}}}{n\rho_{n}}\\
& \quad  + \frac{\sqrt{n\rho_{n}\log n}\sqrt{\rho_{n}}}{(n\rho_{n})^{2}}\lb\sqrt{n\rho_{n}} + \sqrt{\log n}\rb\\
& \preceq \sqrt{\frac{\log n}{n\rho_{n}}}\frac{1}{\sqrt{n}}.
\end{align*}
Equivalently, there exists a constant $c_{2}$ that only depends on $B_{0}$, $q$ and $\pi_{r}$'s such that
\[d_{\ttinf}(U, U^{*})\le \sqrt{\frac{\log n}{n\rho_{n}}}\frac{c_{2}}{\sqrt{n}}.\]
By condition \eqref{eq:laplacian_cond},
\[d_{\ttinf}(U, U^{*})\le \frac{c_{2}}{\sqrt{c}}\frac{1}{\sqrt{n}}.\]
Therefore, \eqref{eq:laplacian_goal} follows if $c > c_{2}^{2} / c_{1}^{2}$. The proof is then completed.
\end{proof}

\begin{proof}[\textbf{Proof of Theorem \ref{thm:grow_K} part (2)}]
First we note that $\L^{*}$ does not depend on the diagonal elements of $\A^{*}$. Thus we can pretend $\A^{*} = \td{\A}^{*}$.

In this case, $\D^{*} = m\rho_{n}(a + (K - 1)b)I_{n}$. Thus, $\L^{*}$ and $\A^{*}$ have the same eigen-structure except that the eigenvalues of $\L^{*}$ are equal to $m\rho_{n}(a + (K - 1)b)$ minus those of $\A^{*}$. 
Similar to the proof of Theorem \ref{thm:laplacian_spectral_clustering}, we can ignore the first eigenvector of $\L$ in the analysis and focus on the second to the $K$-th eigenvectors. Equivalently, $U^{*}$ is taken as $U_{2}^{*}$ in part (1) and 
\[\Lambda^{*} = m\rho_{n}(a + (K - 1)b)I_{n} - m\rho_{n}(a - b)I_{n} = m\rho_{n}Kb I_{n} = n\rho_{n}b I_{n}.\]
As a result, 
\begin{equation}
  \label{eq:grow_K_laplacian1}
  \mnorm{U^{*}}\le \frac{1}{\sqrt{m}}, \quad \lambda_{\min}^{*} = n\rho_{n}b, \quad \gap^{*} = m\rho_{n}\min\{Kb, a - b\}, \quad \bar{\kappa}^{*} = 1.
\end{equation}
Using the same argument as \eqref{eq:grow_K_goal}, it is left to show that
\begin{equation}
  \label{eq:grow_K_goal_laplacian}
  d_{\ttinf}(U, U^{*})\le \frac{\sqrt{2}}{6K\sqrt{m}}.
\end{equation}
Set $\delta = n^{-q}$ and $\alpha = 1 / \log R(\delta)$ in Theorem \ref{thm:generic_binary_laplacian}. Then $p^{*}\preceq \rho_{n}$,
\[R(\delta)\preceq \log n + K, \quad g(\delta)\preceq \sqrt{n\rho_{n}} + \log n + K, \quad M(\delta)\preceq \sqrt{n\rho_{n}\log n}, \quad \bar{\kappa}' \preceq 1 + \frac{n\rho_{n}}{n\rho_{n}b}\preceq 1.\]
On the other hand, 
\[\Lambda_{jj}^{*} = m\rho_{n}(a - b), \quad \L_{kk}^{*} = \D_{kk}^{*} = m\rho_{n}(a + (K - 1)b).\]
Let $\Theta^{*}$ be defined in Lemma \ref{lem:Theta1}. Then
\[\Theta^{*} = \frac{a - b}{Kb}\preceq 1,\]
and for sufficiently large $n$ and $c$,
\[|\Lambda_{jj}^{*} - \L_{kk}^{*}| = m\rho_{n}Kb = n\rho_{n}b \ge 5 M(\delta).\]
By Lemma \ref{lem:Theta1}, we have
\begin{equation}
  \label{eq:grow_K_laplacian2}
  \Theta \preceq 1 / K.
\end{equation}
Since $n\rho_{n} > cK^{3}\log n$, for sufficiently large $n$ and $c$,
\[\gap^{*} = m\rho_{n}\min\{Kb, a - b\} = n\rho_{n}\frac{\min\{Kb, a - b\}}{K}\ge C(\Theta(\delta)\bar{\kappa}'g(\delta) + (\Theta(\delta) + 1)M(\delta)),\]
where $C$ is the universal constant in \eqref{eq:A4_binary}.

Thus both conditions of Theorem \ref{thm:generic_binary_laplacian} are satisfied. By \eqref{eq:grow_K_laplacian1}, \eqref{eq:grow_K_laplacian2} and Theorem \ref{thm:generic_binary_laplacian},
\begin{align*}
  d_{\ttinf}(U, U^{*}) & \preceq \lb\frac{n\rho_{n}\log n}{(m\rho_{n})^{2}} + \frac{\sqrt{n\rho_{n}} + \log n + K + \sqrt{n\rho_{n}\log n}}{Km\rho_{n}} + \frac{1}{n\rho_{n}}\rb\frac{1}{\sqrt{m}} + \frac{\sqrt{(\log n) \rho_{n}}}{Kn\rho_{n}}\\
& \quad + \frac{\sqrt{n\rho_{n}\log n}\sqrt{\rho_{n}}}{K(m\rho_{m})(n\rho_{n})}\lb\sqrt{n\rho_{n}} + \sqrt{\log n}\rb\\
& \preceq \lb\frac{K^{2}\log n}{n\rho_{n}} + \sqrt{\frac{\log n}{n\rho_{n}}}\rb\frac{1}{\sqrt{m}}.
\end{align*}
It is straightforward to show that each term is bounded by $1 / 36K$ for sufficiently large $c$. This proves \eqref{eq:grow_K_goal_laplacian} and hence the theorem. 
\end{proof}

\subsection{Proofs in Section \ref{sec:hierarchical}}\label{subapp:other_hierarchical}
\begin{proof}[\textbf{Proof of Lemma \ref{lem:prob_Zi}}]
First we prove part (1). Let $m = n / K$. By definition, 
\[Z_{i} \stackrel{d}{=} \sum_{i=1}^{m-1}X_{i0} + \sum_{j=1}^{d-1}\sum_{i=1}^{m2^{j-1}}X_{ij} - \sum_{i=1}^{m2^{d-1}}X_{id}\]
where $X_{ij}\stackrel{i.i.d.}{\sim} \mathrm{Ber}(p_{j})$. Then for any $\nu > 0$, 
\begin{align*}
  \log \E [e^{-\nu Z_{i}}] &= (m - 1)\log \lb p_{0}e^{-\nu} + 1 - p_{0}\rb + m\sum_{j=1}^{d}2^{j-1}\log \lb p_{j}e^{-\nu} + 1 - p_{j}\rb + m2^{d - 1}\log \lb p_{d}e^{\nu} + 1 - p_{d}\rb\\
& \le (m - 1)p_{0}\lb e^{-\nu} - 1\rb + m\sum_{j=1}^{d}2^{j-1}p_{j}\lb e^{-\nu} - 1\rb + m2^{d - 1}p_{d}(e^{\nu} - 1)\\
& = \lb (m - 1)p_{0} + m\sum_{j=1}^{d}2^{j-1}p_{j}\rb(e^{-\nu} - 1) + m2^{d - 1}(e^{\nu} - 1)\\
& = \frac{\lambda_{1}^{*} + \lambda_{2}^{*}}{2} \lb e^{-\nu} - 1\rb  + \frac{\lambda_{1}^{*} - \lambda_{2}^{*}}{2}\lb e^{\nu} - 1\rb,
\end{align*}
where the last line uses Proposition \ref{prop:eigen_BTSBM}.Note that $\lambda_{2}^{*} > 0$. Let 
\[\nu = \frac{1}{2}\log \frac{\lambda_{1}^{*} + \lambda_{2}^{*}}{\lambda_{1}^{*} - \lambda_{2}^{*}}.\]
Then $\nu > 0$ and 
\[\log \E [e^{-\nu Z_{i}}] = -\frac{1}{2}\lb\sqrt{\lambda_{1}^{*} + \lambda_{2}^{*}} - \sqrt{\lambda_{1}^{*} - \lambda_{2}^{*}}\rb^{2}.\]
By Markov's inequality,
\[\log \P(Z_{i}\le t) = \log \P\lb e^{-\nu Z_{i}}\ge e^{-\nu t}\rb\le \frac{t}{2}\log \lb\frac{\lambda_{1}^{*} + \lambda_{2}^{*}}{\lambda_{1}^{*} - \lambda_{2}^{*}}\rb  -\frac{1}{2}\lb\sqrt{\lambda_{1}^{*} + \lambda_{2}^{*}} - \sqrt{\lambda_{1}^{*} - \lambda_{2}^{*}}\rb^{2}.\]
This proves the part (1). For part (2), as in part (1), for any $\nu > 0$, 
\[\log \E [e^{\nu Z_{i}}]  = \frac{\lambda_{1}^{*} + \lambda_{2}^{*}}{2} \lb e^{\nu} - 1\rb  + \frac{\lambda_{1}^{*} - \lambda_{2}^{*}}{2}\lb e^{-\nu} - 1\rb.\]
Note that $\lambda_{2}^{*} < 0$. Let 
\[\nu = \frac{1}{2}\log \frac{\lambda_{1}^{*} - \lambda_{2}^{*}}{\lambda_{1}^{*} + \lambda_{2}^{*}}.\]
Then $\nu > 0$ and 
\[\log \E [e^{\nu Z_{i}}] = -\frac{1}{2}\lb\sqrt{\lambda_{1}^{*} + \lambda_{2}^{*}} - \sqrt{\lambda_{1}^{*} - \lambda_{2}^{*}}\rb^{2}.\]
By Markov's inequality,
\[\log \P(Z_{i}\ge -t) = \log \P\lb e^{\nu Z_{i}}\ge e^{-\nu t}\rb\le \frac{t}{2}\log \lb\frac{\lambda_{1}^{*} - \lambda_{2}^{*}}{\lambda_{1}^{*} + \lambda_{2}^{*}}\rb  -\frac{1}{2}\lb\sqrt{\lambda_{1}^{*} + \lambda_{2}^{*}} - \sqrt{\lambda_{1}^{*} - \lambda_{2}^{*}}\rb^{2}.\]
This completes the proof.
\end{proof}

\begin{proof}[\textbf{Proof of Theorem \ref{thm:first_split}}]
  Throughout the proof we use the notation of Theorem \ref{thm:generic_binary}. Then by Proposition \ref{prop:eigen_BTSBM}, 
\[n\bar{p}^{*} = (m - 1)p_{0} + m\sum_{i=1}^{d}2^{i-1}p_{i} = \lambda_{1}^{*}, \quad \bar{\kappa}^{*} = 1.\]
Since $\delta \ge n^{-q}$, 
\[\log n \le R(\delta)\le (q + 1)\log n\Longrightarrow g(\delta) \le \frac{(q + 1)\log n}{\alpha \log \log n}.\]
Then \eqref{eq:first_split_cond} implies that 
\[\gap^{*} > C\bar{\kappa}^{*}g(\delta).\]
Thus the condition of Theorem \ref{thm:generic_binary} is satisfied. By Theorem \ref{thm:generic_binary}, with probability $1 - \delta$.
\begin{align*}
  \bigg\|u_{2} - \frac{Au_{2}^{*}}{\lambda_{2}^{*}}\bigg\|_{\infty} & \preceq \frac{1}{\gap^{*}}\bigg\{\lb\sqrt{\lambda_{1}^{*}} + \frac{\log n}{\alpha\log \log n}\rb\lb 1 + \frac{\log n}{|\lambda_{2}^{*}|}\rb \mnorm{u_{2}^{*}} \\
& \qquad + \frac{\sqrt{(\log n)p^{*}}}{|\lambda_{2}^{*}|}\lb \sqrt{\lambda_{1}^{*}} + \frac{\log n}{\alpha\log \log n} + \frac{\sqrt{\lambda_{1}^{*}(\log n)^{\alpha}}}{\alpha \log \log n}\rb\\
& \qquad + \lb\sqrt{|\lambda_{2}^{*}|} + \sqrt{\log n}\rb \min\left\{\frac{\sqrt{\lambda_{1}^{*}p^{*}}}{|\lambda_{2}^{*}|}, \frac{\sqrt{p^{*}}}{\sqrt{|\lambda_{2}^{*}|}I(A^{*}\mbox{ is psd})}, \sqrt{\frac{K}{n}}\right\}\bigg\}\\
& \stackrel{(i)}{\preceq} \frac{1}{\sqrt{n}\gap^{*}}\bigg\{\lb\sqrt{\lambda_{1}^{*}} + \frac{\log n}{\alpha\log \log n}\rb\lb 1 + \frac{\log n}{|\lambda_{2}^{*}|}\rb + \frac{\sqrt{(\log n)np^{*}}}{|\lambda_{2}^{*}|} \frac{\log n + \sqrt{\lambda_{1}^{*}(\log n)^{\alpha}}}{\alpha\log \log n}\\
& \qquad + \sqrt{\lambda_{2}^{*}}\min\left\{\frac{\sqrt{\lambda_{1}^{*}np^{*}}}{|\lambda_{2}^{*}|}, \frac{\sqrt{np^{*}}}{\sqrt{|\lambda_{2}^{*}|}I(A^{*}\mbox{ is psd})}, \sqrt{K}\right\}\bigg\}\\
& \preceq \frac{\xi_{n1} + \xi_{n2}}{\sqrt{n}\gap^{*}},
\end{align*}
where (i) uses the fact that
\[\sqrt{\log n}\min\left\{\frac{\sqrt{\lambda_{1}^{*}np^{*}}}{|\lambda_{2}^{*}|}, \frac{\sqrt{np^{*}}}{\sqrt{|\lambda_{2}^{*}|}I(A^{*}\mbox{ is psd})}, \sqrt{K}\right\}\preceq \sqrt{\log n}\frac{\sqrt{\lambda_{1}^{*}np^{*}}}{|\lambda_{2}^{*}|} = \frac{\sqrt{(\log n)np^{*}}}{|\lambda_{2}^{*}|}\sqrt{\lambda_{1}^{*}}\]
and
\[\sqrt{\lambda_{1}^{*}} + \frac{\sqrt{\lambda_{1}^{*}(\log n)^{\alpha}}}{\alpha \log \log n}\preceq \frac{\sqrt{\lambda_{1}^{*}(\log n)^{\alpha}}}{\alpha \log \log n}.\]
Equivalently, there exists a universal constant $C'$ such that
\begin{equation}
  \label{eq:first_split_1}
  \sqrt{n}\bigg\|u_{2} - \frac{Au_{2}^{*}}{\lambda_{2}^{*}}\bigg\|_{\infty} \le C'\frac{\xi_{n1} + \xi_{n2}}{\gap^{*}}
\end{equation}
with probability $1 - \delta$. 

On the other hand, by Lemma \ref{lem:prob_Zi}, in the assortative case we have
\[\log\P\lb Z_{i}\le C'\frac{\xi_{n1} + \xi_{n2}}{\gap^{*}}\lambda_{2}^{*}\rb \le C'\frac{\lambda_{2}^{*}}{\gap^{*}}\bigg|\log \lb\frac{\lambda_{1}^{*} + \lambda_{2}^{*}}{\lambda_{1}^{*} - \lambda_{2}^{*}}\rb\bigg|(\xi_{n1} + \xi_{n2})  -\frac{1}{2}\lb\sqrt{\lambda_{1}^{*} + \lambda_{2}^{*}} - \sqrt{\lambda_{1}^{*} - \lambda_{2}^{*}}\rb^{2}.\]
Under condition \eqref{eq:first_split_cond}, 
\[\P\lb Z_{i}\le C'\frac{\xi_{n1} + \xi_{n2}}{\gap^{*}}\lambda_{2}^{*}\rb\le \exp\left\{-\log n - \log \lb\frac{1}{\delta}\rb\right\}\le \frac{\delta}{n}.\]
A simple union bound then implies that
\[\P\lb \min_{i\in [n]}Z_{i}\le C'\frac{\xi_{n1} + \xi_{n2}}{\gap^{*}}\lambda_{2}^{*}\rb\le \delta.\]
Finally, by \eqref{eq:first_split_1}, 
\[\P\lb \min_{i\in [n]}Z_{i} \le \lb \sqrt{n}\bigg\|u_{2} - \frac{Au_{2}^{*}}{\lambda_{2}^{*}}\bigg\|_{\infty}\rb \lambda_{2}^{*}\rb \le 2\delta.\]
The proof for assortative BTSBM is then completed by \eqref{eq:minmaxZi}. Similarly we can prove it for dis-assortative BTSBM. 
\end{proof}

\begin{proof}[\textbf{Proof of Theorem \ref{thm:partial_exact_recovery}}]
It is left to show that for any node at $r$-th layer ($r \le \ell$), its first split can be exactly recovered with probability $1 - o(1)$ as $n$ tends to infinity. Assume $r = \ell$ without loss of generality. Note that this node corresponds to a BTSBM with size $n' = n / 2^{\ell - 1}$ and parameters $(a_{0}, a_{1}, \ldots, a_{d - \ell + 1})$. Throughout the rest of the proof, all symbols (e.g. $\lambda_{1}^{*}, \lambda_{2}^{*}, \gap^{*}$) are defined for this sub-model.

Let $\eps \in (0, 1)$ be any constant such that
\[\frac{1}{2^{d - \ell + 1}}\lb\sqrt{\bar{a}_{d}} - \sqrt{a_{d}}\rb^{2} > 1 + 3\eps.\]
By Proposition \ref{prop:eigen_BTSBM} and definition of $\bar{a}_{\ell}$,
\[\lambda_{1}^{*} = m\rho_{n}\lb 2^{\ell - 1}\bar{a}_{\ell} + 2^{\ell - 1}a_{\ell}\rb - \rho_{n}a_{0} = \log n\lb\frac{\bar{a}_{\ell} + a_{\ell}}{2^{d - \ell + 1}} - \frac{a_{0}}{n}\rb,\]
and 
\[\lambda_{2}^{*} = m\rho_{n}\lb 2^{\ell - 1}\bar{a}_{\ell} - 2^{\ell - 1}a_{\ell}\rb - \rho_{n}a_{0} = \log n\lb\frac{\bar{a}_{\ell} - a_{\ell}}{2^{d - \ell + 1}} - \frac{a_{0}}{n}\rb.\]
As a result,
\begin{align*}
  &\frac{1}{2}\lb\sqrt{\lambda_{1}^{*} + \lambda_{2}^{*}} - \sqrt{\lambda_{1}^{*} - \lambda_{2}^{*}}\rb^{2}  = \frac{1}{2^{d - \ell + 1}}\lb \sqrt{\bar{a}_{\ell} - \frac{2^{d - \ell + 1}a_{0}}{n}} - \sqrt{a_{\ell}}\rb^{2}\\
= & \frac{\log n}{2^{d - \ell + 1}}\lb \sqrt{\bar{a}_{\ell}} - \sqrt{a_{\ell}} - \frac{2^{d - \ell + 1}a_{0}}{n\lb \sqrt{\bar{a}_{\ell} - \frac{2^{d - \ell + 1}a_{0}}{n}} + \sqrt{\bar{a}_{\ell}}\rb}\rb^{2}.
\end{align*}
Then for sufficiently large $n$, 
\begin{equation}
  \label{eq:lambda1pmlambda2}
  \frac{1}{2}\lb\sqrt{\lambda_{1}^{*} + \lambda_{2}^{*}} - \sqrt{\lambda_{1}^{*} - \lambda_{2}^{*}}\rb^{2}\ge \frac{1 + 2\eps}{1 + 3\eps}\frac{\log n}{2^{d - \ell + 1}}\lb \sqrt{\bar{a}_{\ell}} - \sqrt{a_{\ell}}\rb^{2}\ge (1 + 2\eps)\log n.
\end{equation}
Since $(a_{0}, \ldots, a_{\ell})$ and $K$ are all constants, $n\sim n'$ and $\log n \succeq \lambda_{1}^{*} \succeq \lambda_{2}^{*} \succeq \log n$ and $\gap^{*}\succeq \log n$. Let $\alpha = 0.5$ in Theorem \ref{thm:first_split}. Then for sufficiently large $n$ 
, the condition \ref{eq:first_split_cond} is satisfied since the RHS is 
\[\sqrt{\lambda_{1}^{*}} + \frac{\log n'}{\alpha \log\log n'} \preceq \sqrt{\log n} + \frac{\log n}{\log\log n} = o(\log n).\]
In addition, it is easy to see that 
\begin{align*}
  \xi_{n1}\preceq \frac{\log n'}{\log \log n'} = o(\log n), \quad \xi_{n2}\preceq \sqrt{\log n'} = o(\log n),
\end{align*}
and
\[\frac{\lambda_{2}^{*}}{\gap^{*}} \preceq 1, \quad \bigg|\log \lb\frac{\lambda_{1}^{*} + \lambda_{2}^{*}}{\lambda_{1}^{*} - \lambda_{2}^{*}}\rb\bigg|\preceq 1.\]
Thus, for sufficiently large $n$ 
, 
\[C'\frac{\lambda_{2}^{*}}{\gap^{*}}\bigg|\log \lb\frac{\lambda_{1}^{*} + \lambda_{2}^{*}}{\lambda_{1}^{*} - \lambda_{2}^{*}}\rb\bigg|(\xi_{n1} + \xi_{n2})\le \eps\log n.\]
Combined with \eqref{eq:lambda1pmlambda2}, we have
\begin{align*}
  \frac{1}{2}\lb\sqrt{\lambda_{1}^{*} + \lambda_{2}^{*}} - \sqrt{\lambda_{1}^{*} - \lambda_{2}^{*}}\rb^{2} - \log n - C'\frac{\lambda_{2}^{*}}{\gap^{*}}\bigg|\log \lb\frac{\lambda_{1}^{*} + \lambda_{2}^{*}}{\lambda_{1}^{*} - \lambda_{2}^{*}}\rb\bigg|(\xi_{n1} + \xi_{n2})\ge \eps \log n.
\end{align*}
Let $\delta = n^{-\eps}$ in Theorem \ref{thm:first_split}. Then the first split is exactly recovered with probability $1 - 2n^{-\eps} = 1 - o(1)$. This completes the proof.
\end{proof}

\begin{proposition}[\textbf{Theorem 1 of \cite{abbe2015community}}]\label{prop:necessary_cond_exact_recovery}
For a general SBM with connection probability matrix $B = B_{0}\frac{\log n}{n}$ where $B_{0}\in \R^{K\times K}$ is a fixed matrix. Further let $\Pi = \diag(\pi_{1}, \ldots, \pi_{K})$. Then exact recovery is achievable iff
\[\min_{i, j\in [K]}D_{+}((\Pi B)_{i}, (\Pi B)_{j})\ge 1,\]
where $D_{+}: \R^{K}\times \R^{K}\rightarrow \R$ with
\[D_{+}(\theta, \psi) = \max_{t\in [0, 1]}\sum_{i=1}^{K}\lb t\theta_{i} + (1 - t)\psi_{i} - \theta_{i}^{t}\psi_{i}^{1 - t}\rb.\]
\end{proposition}

\begin{proof}[\textbf{Proof of Lemma \ref{lem:necessary_cond_BTSBM}}]
Using the notation of Proposition \ref{prop:necessary_cond_exact_recovery}, we have $\Pi = (1 / K)I_{K}$. Thus, 
\[D_{+}((\Pi B)_{i}, (\Pi B)_{j}) = \frac{1}{K}D_{+}(B_{i}, B_{j}).\]
We prove that
\begin{equation}
  \label{eq:D+BiBj}
  \min_{i\not = j}D_{+}(B_{i}, B_{j}) = (\sqrt{a_{0}} - \sqrt{a_{1}})^{2}.
\end{equation}
When $i = 1$ and $j = 2$, only the first two entries differ and thus,
\begin{align*}
  D_{+}(B_{1}, B_{2}) &= \max_{t\in [0, 1]} (ta_{0} + (1 - t)a_{1} - a_{0}^{t}a_{1}^{1-t}) + (ta_{1} + (1 - t)a_{0} - a_{1}^{t}a_{0}^{1-t})\\
& = a_{0} + a_{1} - \min_{t\in [0, 1]}(a_{0}^{t}a_{1}^{1-t} + a_{1}^{t}a_{0}^{1-t})\\
& = a_{0} + a_{1} - 2\sqrt{a_{0}a_{1}} = (\sqrt{a_{0}} - \sqrt{a_{1}})^{2},
\end{align*}
where the second last line uses the convexity and the symmetry of $t\mapsto a_{0}^{t}a_{1}^{1-t} + a_{1}^{t}a_{0}^{1-t}$. It is left to prove that for any $i \not = j$,
\[D_{+}(B_{i}, B_{j}) \ge (\sqrt{a_{0}} - \sqrt{a_{1}})^{2}.\]
By definition, $B_{i, i} = B_{j, j} = a_{0}$ and $B_{i, j} = B_{j, i} = a_{k}$ for some $k\not = 0$. Ignoring all other entries,
\[D_{+}(B_{i}, B_{j})\ge \max_{t\in [0, 1]} (ta_{0} + (1 - t)a_{k} - a_{0}^{t}a_{k}^{1-t}) + (ta_{k} + (1 - t)a_{0} - a_{k}^{t}a_{0}^{1-t}).\]
Using the same argument as above, we have
\[D_{+}(B_{i}, B_{j})\ge (\sqrt{a_{0}} - \sqrt{a_{k}})^{2}\ge (\sqrt{a_{0}} - \sqrt{a_{1}})^{2}.\]
Thus \eqref{eq:D+BiBj} is proved.

If $|\sqrt{a_{0}} - \sqrt{a_{1}}| < \sqrt{K}$, then 
\[\min_{i, j\in [K]}D_{+}((\Pi B)_{i}, (\Pi B)_{j}) = \frac{(\sqrt{a_{0}} - \sqrt{a}_{1})^{2}}{K} < 1.\]
By Proposition \ref{prop:necessary_cond_exact_recovery}, it is impossible to achieve exact recovery.
\end{proof}

\section{Comparison With Other Bounds on Binary Random Matrices}\label{app:comparison}

\subsection{Comparison with \cite{abbe2017entrywise}} 

The assumptions they required are the following:
\begin{enumerate}[\textbf{B}1]
\setcounter{enumi}{2}
\item Suppose $\phi(x)$ is continuous and non-decreasing on $\R_{+}$ with $\phi(0) = 0$ and $\phi(x) / x$ being non-increasing. For any $\delta \in (0, 1)$ and matrix $W\in \R^{n\times r}$, it holds with probability at least $1 - \delta$ simultaneously for all $k\in [n]$ that
  \begin{align*}
    \|E_{k}^{T}W\|_{2} &\le \gap^{*}\mnorm{W}\phi\lb\frac{\|W\|_{\mathrm{F}}}{\sqrt{n}\mnorm{W}}\rb;
  \end{align*}
\item $\gap^{*}\ge \gamma^{-1}\max\{\mnorm{A^{*}}, \|E\|_{\op}\}$ with probability $1 - \delta$ for any $\gamma > 0$ such that
  \begin{equation}
    \label{eq:kappagamma}
    \kappa^{*} \le \frac{1}{32\max\{\gamma, \phi(\gamma)\}}.
  \end{equation}
\end{enumerate}
Under their assumption \textbf{B}3, as shown in their proof of Lemma 6 (equation (59)) in Section A.2, 
\[\|E_{k}^{T}W\|_{2}\preceq \phi(\gamma)\lb \mnorm{W} + \frac{\|W\|_{\mathrm{F}}}{\sqrt{n\gamma}}\rb\]
for any $\gamma > 0$. This corresponds to our assumption \textbf{A}3 with
\[b_{\infty}(\delta) = \phi(\gamma), \quad b_{2}(\delta) = 
\frac{\phi(\gamma)}{\sqrt{n\gamma}}.\]

There are two points that are worth made here for clarity. Firstly, they treated $\delta$ as a constant and hence did not explicitly specify the dependence on $\delta$. This essentially sets $\delta = O(1 / n)$ as in our assumption \textbf{A}3. Specifying the dependence on $\delta$ yields a tighter moment bound for $d_{\ttinf}(U, U^{*})$, which is useful in our first example on concentration of spectral norm of random graphs (Section \ref{sec:spectral_norm}). Secondly, the $\phi$ function may implicitly depend on $r$. For instance, in the binary case, our Proposition \ref{prop:vector2matrix} shows that $b_{\infty}(\delta)$ scales linearly with $r$. \cite{abbe2017entrywise} did not specify the dependence on $r$ because they either consider the Gaussian case where $b_{\infty}(\delta) = 0$ or the general case with $r = O(1)$. So in our comparison we will also keep these settings. For simplicity we only consider the regime
\begin{equation}
  \label{eq:comparison_regime}
  np^{*}\ge C\log n,
\end{equation}
for some universal constant $C > 0$.

\noindent \textbf{Comparison of assumptions}

First we compare the assumptions of \cite{abbe2017entrywise} and our theory. As stated in their Section 1.2, for binary matrices with independent entries and $r = O(1)$,
\begin{equation}\label{eq:phi}
  \phi(x) \preceq \frac{np^{*}}{\gap^{*}\max\{1, \log(1 / x)\}}
\end{equation}
Then the condition \eqref{eq:kappagamma} on $\gamma$ reads as
\[\kappa^{*} \max\left\{\gamma, \frac{np^{*}}{\gap^{*}\max\{1, \log(1 / \gamma)\}}\right\}\le \frac{1}{32}.\]
The first term implies that
\[\gamma \le (32\kappa^{*})^{-1}\preceq 1 / 32\Longrightarrow \max\{1, \log(1 / \gamma)\}\ge \log(1 / \gamma).\]
The second term then implies that
\[\gap^{*}\ge \frac{32\kappa^{*}}{\log(1 / \gamma)}np^{*}.\]
Putting the pieces together, \textbf{B}4 implies that
\begin{equation}
  \label{eq:abbeB3B4_1}
  \gap^{*}\succeq \max\left\{\frac{\mnorm{A^{*}}}{\gamma}, \frac{\|E\|_{\op}}{\gamma}, \frac{\kappa^{*} np^{*}}{\log(1 / \gamma)}\right\}.
\end{equation}
This implies that
\[\gap^{*}\succeq\min_{\gamma}\max\left\{\frac{\|E\|_{\op}}{\gamma}, \frac{\kappa^{*} np^{*}}{\log(1 / \gamma)}\right\}.\]
The first part is decreasing in $\gamma$ while the second part is increasing $\gamma$. Thus, the minimum is achieved at $\gamma^{*}$ such that two terms are equal, i.e.
\[\frac{1}{\gamma^{*}}\log \lb\frac{1}{\gamma^{*}}\rb = \frac{\kappa^{*} np^{*}}{\|E\|_{\op}}.\]
Under the regime \eqref{eq:comparison_regime}, $1\preceq \|E\|_{\op}\preceq \sqrt{np^{*}}$ and thus $np^{*} / \|E\|_{\op}\succeq \sqrt{np^{*}}$. As a result,
\[\frac{1}{\gamma^{*}}\log \lb\frac{1}{\gamma^{*}}\rb \sim \frac{\kappa^{*} np^{*}}{\|E\|_{\op}} \Longleftrightarrow \frac{1}{\gamma^{*}} \sim \frac{\kappa^{*} np^{*} / \|E\|_{\op}}{\log \kappa^{*} + \log (np^{*})}.\]
Therefore, their assumptions \textbf{B}3 and \textbf{B}4 hold only if 
\begin{equation}
  \label{eq:abbeB3B4}
  \gap^{*}\succeq\max\left\{\frac{\|E\|_{\op}}{\gamma^{*}}, \frac{\kappa^{*} np^{*}}{\log(1 / \gamma^{*})}\right\}\succeq \frac{\kappa^{*} np^{*}}{\log \kappa^{*} + \log (np^{*})}.
\end{equation}
By contrast, our theory (condition \eqref{eq:A4_binary} in Theorem \ref{thm:generic_binary}) only requires 
\begin{equation}
  \label{eq:our_cond}
  \gap^{*}\succeq \bar{\kappa^{*}}^{*}g(\delta) = g(\delta) = \sqrt{np^{*}} + \frac{\log n}{\log \log n}.
\end{equation}
This is always less stringent than \eqref{eq:abbeB3B4}. The only case it is equivalent to \eqref{eq:abbeB3B4} is when $np^{*} \sim \log n$ and $\kappa^{*} \preceq 1$. When $np^{*}\ge (\log n)^{2}$, our condition is $\kappa^{*} \sqrt{np^{*}}$ better than \eqref{eq:abbeB3B4}. 

Furthermore, note that
\begin{equation}
  \label{eq:gap_upper}
  \gap^{*} \le \lambda_{\min}^{*}\le \lambda_{\max}^{*}\le np^{*}.
\end{equation}
By \eqref{eq:gap_upper}, their condition \eqref{eq:abbeB3B4} can hold only if
\[\kappa^{*} \preceq \log (np^{*}).\]
When $np^{*} \preceq (\log n)^{b}$ for some $b > 0$ as typically studied for random graphs, \eqref{eq:abbeB3B4} only permits well-conditioned case with $\kappa^{*} \preceq \log \log n$. Even in the case of dense graphs where $np^{*}$ grows polynomially, the condition number should be no larger than $\log n$. By contrast, our theory allows the condition number to be arbitrarily large.  

On the other hand, even in the well-conditioned case $\kappa^{*} \preceq 1$, \eqref{eq:abbeB3B4} can only hold if $\gap^{*}\succeq np^{*} / \log (np^{*})$. Therefore, the minimal eigen-gap for which their theory works is only $\log\log n$ smaller than the upper bound \eqref{eq:gap_upper} when $np^{*}$ grows poly-logarithmically and is $\log n$ smaller than the upper bound when $np^{*}$ grows polynomially. By contrast, our theory allows the eigen-gap to be much smaller than the upper bound.

\noindent \textbf{Comparison of bounds}

By \eqref{eq:phi} and \eqref{eq:abbeB3B4}.
\[\frac{\mnorm{A^{*}}}{\gap^{*}} \preceq \frac{\sqrt{n}p^{*}}{\gap^{*}}\preceq \frac{\phi(1)}{\sqrt{n}}\preceq \phi(1)\mnorm{U^{*}}.\]
In their Theorem 2.1, they showed 
\begin{align}
d_{\ttinf}(U, AU^{*}(\Lambda^{*})^{-1}) &\preceq \kappa^{*} (\kappa^{*} + \phi(1))\lb \gamma + \phi(\gamma)\rb\mnorm{U^{*}} + \gamma \frac{\mnorm{A^{*}}}{\gap^{*}}\nonumber\\
& \preceq \kappa^{*}(\kappa^{*} + \phi(1))\lb \gamma + \phi(\gamma)\rb\mnorm{U^{*}};\label{eq:abbe_bound1}\\
d_{\ttinf}(U, U^{*}) &\preceq \lb \phi(1) + \kappa^{*}(\kappa^{*} + \phi(1))(\gamma + \phi(\gamma))\rb\mnorm{U^{*}} + \gamma \frac{\mnorm{A^{*}}}{\gap^{*}}\nonumber\\
& \preceq \lb \frac{np^{*}}{\gap^{*}} + \kappa^{*}(\kappa^{*} + \phi(1))\lb \gamma + \phi(\gamma)\rb\rb\mnorm{U^{*}}.\label{eq:abbe_bound2}
\end{align}
By \eqref{eq:abbeB3B4_1} and \eqref{eq:gap_upper},
\[\gamma \succeq \frac{\|E\|_{\op}}{\gap^{*}}\succeq \frac{1}{np^{*}}.\]
As a result,
\[\gamma + \phi(\gamma) \succeq \phi(\gamma)\succeq \frac{np^{*}}{\gap^{*}\log (np^{*})}.\]
Thus, their bounds \eqref{eq:abbe_bound1} and \eqref{eq:abbe_bound2} are at least 
\begin{equation}
  \label{eq:abbe_lower_upper}
  \frac{np^{*}\kappa^{*2}}{\gap^{*}\log (np^{*})}\mnorm{U^{*}}\quad \mbox{and}\quad \frac{np^{*}}{\gap^{*}}\lb 1 + \frac{\kappa^{*2}}{\log (np^{*})}\rb \mnorm{U^{*}}.
\end{equation}
For simple comparison, we assume $\gap^{*}\sim np^{*}$. Then their bounds \eqref{eq:abbe_lower_upper} are at least
\begin{equation}
  \label{eq:abbe_lower_upper_simple}
  \frac{\kappa^{*2}}{\log (np^{*})}\mnorm{U^{*}}\quad \mbox{and}\quad \lb 1 + \frac{\kappa^{*2}}{\log (np^{*})}\rb \mnorm{U^{*}}.
\end{equation}
Turning to our bound. Since $\lambda_{\min}^{*}\ge \gap^{*} \sim np^{*}$, the condition of Corollary \ref{cor:typical_binary} is satisfied. By Corollary \ref{cor:typical_binary} and \eqref{eq:comparison_regime},
\begin{align}
  d_{\ttinf}(U, AU^{*}(\Lambda^{*})^{-1}) &\preceq \lb\frac{1}{\sqrt{np^{*}}} + \frac{\log n}{np^{*}\log \log n}\rb\mnorm{U^{*}}\label{eq:our_bound1};\\
  d_{\ttinf}(U, U^{*}) &\preceq \lb\frac{\sqrt{np^{*}} + \log n}{np^{*}}\rb\mnorm{U^{*}} + \sqrt{\frac{\log n}{np^{*}}}\frac{1}{\sqrt{n}}\preceq \sqrt{\frac{\log n}{np^{*}}}\mnorm{U^{*}}\label{eq:our_bound2}.
\end{align}
It is easy to see that both bounds dominate \eqref{eq:abbe_lower_upper_simple} except in the case $\kappa^{*} \preceq 1, np^{*}\sim \log n$ where two bounds are equivalent. 

More importantly, \eqref{eq:abbe_lower_upper_simple} does not improve in order when $np^{*}$ increases except when $np^{*}$ grows from $(\log n)^{b}$ to $n^{b}$ so that the bound grows from $\kappa^{*} / \log \log n\mnorm{U^{*}}$ to $\kappa^{*2} / \log n \mnorm{U^{*}}$. By contrast, our bounds improves constantly as $np^{*}$ grows in order. For instance, when $np^{*} >\!\! > (\log n)^{2}$, \eqref{eq:our_bound1} significantly improves upon \eqref{eq:abbe_lower_upper_simple}.

\subsection{Comparison with \cite{eldridge2017unperturbed}}

\noindent \textbf{Comparison of assumptions}

\cite{eldridge2017unperturbed} considered the case where  
\[r = 1, \quad np^{*} \succeq (\log n)^{2 + \eps},\] 
for some $\eps > 0$. They also implicitly assumed
\[\lambda_{\min}^{*}\ge \gap^{*}\succeq np^{*}, \quad \max_{i\in [s + 1, s + r]}\|U_{i}^{*}\|_{\infty} \preceq \frac{1}{\sqrt{n}}.\]
This is at least as strong as the condition of Corollary \ref{cor:typical_binary}, itself being the most special case of our general theory in Section \ref{sec:binarybound}. Although the bound on a single eigenvector can yield bounds for eigenspaces, it requires the multiplicity of each eigenvalue to be $1$ and sufficient eigen-gap for each eigenvalue. This cannot be applied to problems in Section \ref{subsec:grow_K}. 

\noindent \textbf{Comparison of bounds}

In this setting they proved that with high probability,
\[\|U - U^{*}\|_{\infty}\preceq \sqrt{\frac{(\log n)^{2 + \eta}}{np^{*}}}\|U^{*}\|_{\infty},\]
for any $\eta \in (0, \eps / 2)$. By contrast, as shown in \eqref{eq:our_bound2}, our bound implies that
\[\|U - U^{*}\|_{\infty}\preceq \sqrt{\frac{\log n}{np^{*}}}\|U^{*}\|_{\infty}.\]
Thus, our bound is at least $\sqrt{(\log n)^{1 + \eps}}$ better than their bound even in this special setting. In order for $\|U - U^{*}\|_{\infty} \preceq \|U^{*}\|_{\infty}$ as in most applications, our bound only requires $np^{*}\succeq \log n$ while their bound requires $np^{*}\succeq (\log n)^{2 + \eps}$. 

\subsection{Comparison with \cite{cape2019signal}}

\noindent \textbf{Comparison of assumptions}

\cite{cape2019signal} considers the full recovery for low-rank matrices, i.e. 
\[s = 0, \quad \lambda_{r+1}^{*} = \cdots = \lambda_{n}^{*} = 0.\]
They further assume that 
\[np^{*} \succeq (\log n)^{2 + \eps}, \quad r\preceq (\log n)^{2 + \eps}\]
for some $\eps > 0$ and 
\[\kappa^{*}\preceq 1,\quad \lambda_{\min}^{*} = \gap^{*}\succeq np^{*}.\]
This is a highly specialized setting and cannot be applied to problems in Section \ref{sec:spectral_norm}, Section \ref{subsec:grow_K} and Section \ref{sec:hierarchical}. 

\noindent \textbf{Comparison of bounds}

Under their assumptions, they proved that 
\[d_{\ttinf}(U, U^{*})\preceq \frac{\sqrt{r(\log n)^{2 + \eps}}}{\sqrt{np^{*}}}\mnorm{U^{*}}.\]
As shown in \eqref{eq:our_bound2}, our bound reads as 
\[d_{\ttinf}(U, U^{*})\preceq \sqrt{\frac{\log n}{np^{*}}}\mnorm{U^{*}}.\]
It is clear that our bound is $\sqrt{r(\log n)^{1 + \eps}}$ better than their bound. More importantly, our bound is not affected by the number of eigenvectors to recover and thus allow $r$ to be as large as $n$, in which case their bound is not informative. 

\subsection{Comparison with \cite{mao2017estimating}}

\noindent \textbf{Comparison of assumptions}

As in \cite{cape2019signal}, \cite{mao2017estimating} considers the full recovery problem for low-rank matrices. They assumed that 
\[np^{*}\succeq (\log n)^{2 + \eps}, \quad \lambda_{\min}^{*} = \gap^{*}\succeq \sqrt{np^{*}}(\log n)^{1 + \eps / 2}, \quad \max_{j\in [r]}\|U_{j}^{*}\|_{\infty}\preceq \sqrt{p^{*}},\]
for some $\eps > 0$. The assumption on $\|U_{j}^{*}\|_{\infty}$ forces the eigenvectors to be diffused and the matrix $A^{*}$ to have low coherence \citep{candes2009exact}. By contrast, we do not have any assumption on $U^{*}$. Moreover, our assumption on the eigen-gap is 
\[\gap^{*}\succeq \bar{\kappa}^{*}\lb \sqrt{np^{*}} + \frac{\log n}{\log \log n}\rb.\]
Under their regime $np^{*}\succeq (\log n)^{2 + \eps}$, this is weaker than their condition if $\min\{\kappa^{*}, r\}\preceq (\log n)^{1 + \eps / 2}$. However, their condition can be weaker in the ill-conditioned case with many eigenvectors to be recovered. 

\noindent \textbf{Comparison of bounds}

Under their assumptions, they proved that 
\begin{equation}
  \label{eq:mao2017}
  \mnorm{UU^{T} - U^{*}(U^{*})^{T}}\preceq \frac{\bar{\kappa}^{*}\sqrt{np^{*}}}{\gap^{*}}\lb \bar{\kappa}^{*} + (\log n)^{1 + \eps / 2}\rb\lb \sqrt{r}\max_{i\in [r]}\|U_{i}^{*}\|_{\infty}\rb.
\end{equation}
They show that the same bound holds for $d_{\ttinf}(U, U^{*})$. By contrast, our Corollary \ref{cor:full_recovery_binary} with $\alpha = 0.5$ implies that 
\begin{align*}
  d_{\ttinf}(U, U^{*})&\preceq \frac{1}{\gap^{*}}\left\{\lb \bar{\kappa}^{*}\sqrt{np^{*}} + \frac{\bar{\kappa}^{*}\log n}{\log \log n} + \log n\rb\mnorm{U^{*}} + \sqrt{(\log n) p^{*}}\lb 1 + \frac{\sqrt{np^{*}(\log n)^{1/2}}}{\gap^{*}\log n}\rb\right\}.
\end{align*}
Under their assumptions, it is not hard to see that the above bound simplifies as below:
\begin{align}\label{eq:mao2017_our}
  d_{\ttinf}(U, U^{*}) \preceq \frac{\bar{\kappa}^{*}\sqrt{np^{*}}}{\gap^{*}}\lb 1 + \frac{\sqrt{\log n}}{\sqrt{n}\mnorm{U^{*}}}\rb \mnorm{U^{*}}\preceq \frac{\bar{\kappa}^{*}\sqrt{np^{*}}}{\gap^{*}}\lb 1 + \sqrt{\frac{\log n}{r}}\rb\mnorm{U^{*}}.
\end{align}
Since $\mnorm{U^{*}}\le \sqrt{r}\max_{i\in [r]}\|U_{i}^{*}\|_{\infty}$, our bound is always better than \eqref{eq:mao2017}. If $r \preceq \log n$, our bound is $\sqrt{r (\log n)^{1 + \eps}}$ better than theirs; if $r \succeq \log n$, our bound is $\bar{\kappa}^{*} + (\log n)^{1 + \eps / 2}$ better than theirs.

\subsection{Comparison with other deterministic bounds}
In literature there are also several deterministic $\ell_{\ttinf}$ bounds that do not depend on the random structure of the matrices. Because of the generality, they are typical much weaker than those tailored for random matrices. Although it is unfair to compare two types of bounds, we discuss the comparison here for completeness. 

We are aware of three purely deterministic $\ell_{\ttinf}$ bounds by \cite{fan2018eigenvector}, \cite{cape2019two} and \cite{damle2019uniform}. The first two are derived for rectangular matrices and the last one is derived for symmetric matrices. When applied to symmetric matrices, all of the above works only consider top-$r$ recovery. 

\cite{fan2018eigenvector} assumes 
\begin{equation}\label{eq:fan2018_cond}
  \lambda_{\min}^{*} \succeq r^{3}(\sqrt{n}\mnorm{U^{*}})^{2}\|E\|_{\infty} + \|A^{*} - U^{*}\Lambda^{*}U^{T}\|_{\infty}.
\end{equation}
The second term is hard to bound in general except in the full recovery problem where the second term vanishes. In this case, \eqref{eq:fan2018_cond} can be simplified as
\begin{equation}
  \label{eq:fan2018_cond2}
  \lambda_{\min}^{*} \succeq r^{3}(\sqrt{n}\mnorm{U^{*}})^{2}\|E\|_{\infty}.
\end{equation}
Note that
\[\lambda_{\min}^{*}\le \lb\frac{\one_{n}}{\sqrt{n}}\rb^{T}A^{*}\lb\frac{\one_{n}}{\sqrt{n}}\rb = \frac{1}{n}\sum_{i,j=1}^{n}A_{ij}^{*} \le n\bar{p}^{*},\]
and Bernstein's inequality implies that 
\[\|E\|_{\infty}\succeq n\bar{p}^{*} + \log n.\]
As a result, \eqref{eq:fan2018_cond2} holds only if
\[n\bar{p}^{*}\succeq \log n, \quad r\preceq 1, \quad \sqrt{n}\mnorm{U^{*}}\preceq 1, \quad \gap^{*} = \lambda_{\min}^{*}\succeq n\bar{p}^{*}.\]
Under these conditions, they prove that 
\[d_{\ttinf}(U, U^{*})\preceq \frac{r^{5/2}(\sqrt{n}\mnorm{U^{*}})^{2}\|E\|_{\infty}}{\lambda_{\min}^{*}\sqrt{n}}\preceq \mnorm{U^{*}}.\]
This bound matches our bound \eqref{eq:our_bound2} only when $np^{*}\sim \log n$, but is $\sqrt{np^{*} / \log n}$ worse than ours when $np^{*}>\!\!> \log n$.

The condition was improved by \cite{cape2019two} into
\[\lambda_{\min}^{*}\succeq \|E\|_{\infty}.\]
This eliminates the constraint on $r$ and $\sqrt{n}\mnorm{U^{*}}$ but still requires $n\bar{p}^{*}\succeq \log n$. For full recovery problem, \cite{cape2019two} obtained essentially the same bound as below:
\begin{equation}
  \label{eq:cape2019two}
  d_{\ttinf}(U, U^{*})\preceq \frac{\|E\|_{\infty}}{\lambda_{\min}^{*}}\mnorm{U^{*}}\preceq \mnorm{U^{*}}.
\end{equation}

On the other hand, for top-$r$ problem, \cite{damle2019uniform} requires 
\[\min\lb \gap^{*}, \sep_{\ttinf, \td{U}^{*}}(\Lambda^{*}, A^{*} - U^{*}\Lambda^{*}(U^{*})^{T})\rb\succeq \sqrt{np^{*}},\]
where $(\td{U}^{*}, U^{*})$ forms an orthonormal basis in $\R^{n}$ and 
\[\sep_{\ttinf, W}(B, C) = \inf\{\|ZB - CZ\|_{\ttinf}: Z\in \mathrm{ran} W, \|Z\|_{\ttinf} = 1\}.\]
However this condition is hard to parse except in the full recovery problem for which it is shown that 
\[\min\{\gap^{*}, \sep_{\ttinf, \td{U}^{*}}(\Lambda^{*}, A^{*} - U^{*}\Lambda^{*}(U^{*})^{T}\} = \gap^{*}.\]
In this case, their condition reads as $\gap^{*}\succeq \sqrt{np^{*}}$. This is the weakest one among all aforementioned works. However, their bound is also the weakest for binary random matrices with independent entries, although it is tight for some deterministic matrices. Indeed, their bound is 
\[d_{\ttinf}(U, U^{*})\preceq \lb\frac{\|E\|_{\op}}{\gap^{*}}\rb^{2}\mnorm{U^{*}} + \frac{\mnorm{\td{U}^{*}E_{2, 1}}}{\gap^{*}} + \frac{\mnorm{\td{U}^{*}(\td{U}^{*})^{T}E}\|E\|_{\op}}{(\gap^{*})^{2}}.\]
The third term is large in general as $\td{U}^{*}$ includes all other eigenvectors include those corresponding to the zero eigenvalue. For instance consider an \ER graph with self-loop, i.e. $A^{*} = p^{*}\one_{n}\one_{n}^{T}$. In this case $U^{*} = \one_{n} / \sqrt{n}$ and thus
\[\td{U}^{*}(\td{U}^{*})^{T} = I_{n} - \frac{1}{n}\one_{n}\one_{n}^{T}.\]
As a result,
\[\mnorm{\td{U}^{*}(\td{U}^{*})^{T}E} = \mnorm{\lb I_{n} - \frac{1}{n}\one_{n}\one_{n}^{T}\rb E}\sim \mnorm{E} \sim \sqrt{np^{*}}.\]
Therefore, when $p^{*}\succeq \log n / n$, the third term alone is lower bounded by
\[\frac{\sqrt{np^{*}}\sqrt{np^{*}}}{(np^{*})^{2}} = \frac{1}{np^{*}}.\]
This is too loose compared to all aforementioned bounds.

\section{Concentration Inequalities for Binary Random Variables}\label{app:concentration}
\begin{lemma}\label{lem:bernoulli}
  Let $(X_{i})_{i=1}^{n}$ be independent Bernoulli variables with $\E X_{i} = p_{i}$. Given any vector $w\in \R^{n}$, let 
\[S_{n} = \sum_{i=1}^{n}w_{i}(X_{i} - p_{i}).\]
Then for any $\delta \in (0, 1)$, it holds with probability $1 - \delta$ that
\[S_{n}\le \frac{2\log (1 / \delta)}{F^{-1}(2\Omega\log (1 / \delta))}\|w\|_{\infty}.\]
where
\[\Omega = \frac{\|w\|_{\infty}^{2}}{\sum_{i=1}^{n}p_{i}w_{i}^{2}}, \quad F(x) = x^{2}e^{x}.\]
\end{lemma}
\begin{proof}
  Without loss of generality we assume $\|w\|_{\infty} = 1$. For any $\lambda > 0$ and $t > 0$, by Markov's inequality
\begin{equation}\label{eq:Markov}
  \log \P\lb S_{n}\ge t \rb \le -\lambda t + \log \E \left[e^{\lambda S_{n}}\right].
\end{equation}
By definition,
\begin{align*}
\log \E \left[e^{\lambda S_{n}}\right] &= \sum_{i=1}^{n}\lb \log(1 - p_{i} + p_{i}e^{\lambda w_{i}}) - \lambda w_{i}p_{i}\rb\\
& \stackrel{(i)}{\le} \sum_{i=1}^{n}p_{i}\lb e^{\lambda w_{i}} - \lambda w_{i} - 1\rb\\
& \stackrel{(ii)}{\le} \sum_{i=1}^{n}p_{i} \frac{(\lambda w_{i})^{2}}{2}e^{\lambda |w_{i}|}\\
& \stackrel{(iii)}{\le} \sum_{i=1}^{n}p_{i} \frac{(\lambda w_{i})^{2}}{2}e^{\lambda} = \frac{F(\lambda)}{2\Omega},
\end{align*}
where (i) uses the inequality that $\log (1 + x)\le x$ for all $x > -1$, (ii) uses the inequality that $e^{x} - x - 1\le \frac{x^{2}e^{|x|}}{2}$ and (iii) uses the fact that $|w_{i}|\le \|w\|_{\infty} = 1$. 

~\\
\noindent Fix any $\lambda$, let $t = \frac{F(\lambda)}{\lambda \Omega}$. Then \eqref{eq:Markov} implies that
\[\log \P\lb S_{n}\ge \frac{F(\lambda)}{\lambda\Omega} \rb \le -\frac{F(\lambda) }{2\Omega}.\]
Let $\lambda = F^{-1}\lb 2\Omega\log\lb 1 / \delta\rb\rb$. Then we obtain that with probability $1 - \delta$,
\[S_{n}\le \frac{2\log (1 / \delta)}{F^{-1}(2\Omega\log (1 / \delta))}.\]
\end{proof}

\begin{lemma}\label{lem:Finv}
$F^{-1}(x)$ is increasing and $F^{-1}(x) / \sqrt{x}$ is decreasing. For any $x_{0}\ge e$,
  \[F^{-1}(x) \ge \left\{
      \begin{array}{ll}
        \sqrt{x / x_{0}} &  \mbox{for any }x \le x_{0}\\
        \log x - 2\log \log x & \mbox{for any } x > e\\
        \log x / 2 & \mbox{for any } x > 0
      \end{array}
\right..\]
\end{lemma}
\begin{proof}
Notice that $F(\lambda)$ is increasing. This implies that $F^{-1}(x)$ is increasing. On the other hand, let $F^{-1}(x) / \sqrt{x} = v(x)$, then by definition
\[v(x)^{2}e^{\sqrt{x}v(x)} = 1.\]

This implies that $v(x)$ is decreasing. 
Since $F(\lambda)$ is increasing and $F(\log x_{0}) = x_{0}(\log x_{0})^{2}\ge x_{0}$, for any $x \le x_{0}$
\[F^{-1}(x)\le \log x_{0}.\] 
By definition, 
\[x = (F^{-1}(x))^{2} e^{F^{-1}(x)}\le (F^{-1}(x))^{2} x_{0}\Longrightarrow F^{-1}(x) \ge \sqrt{\frac{x}{x_{0}}}.\]
On the other hand, for any $x > e$,
\[F(\log x - 2 \log \log x) = (\log x - 2\log \log x)^{2}e^{\log x - 2\log \log x} = x \frac{(\log x - 2\log \log x)^{2}}{(\log x)^{2}}\le x.\]
Thus, 
\[F^{-1}(x) \ge \log x - 2\log \log x.\]
Finally, noting that for any $\lambda > 0$,
\[e^{\lambda} = \sum_{n\ge 0}\frac{\lambda^{n}}{n!} = \ge \lambda + \frac{\lambda^{2}}{2} + \frac{\lambda^{3}}{6} = \frac{\lambda^{2}}{2} + 2\sqrt{\lambda\frac{\lambda^{3}}{6}}\ge \lambda^{2},\]
we have
\[F(\lambda)\le e^{2\lambda}\Longrightarrow F^{-1}(x)\ge \frac{\log x}{2}.\]
\end{proof}

\begin{lemma}\label{lem:bernoulli2}
  Under the same setting of Lemma \ref{lem:bernoulli}, for any $\gamma > 0$, it holds with probability $1 - \delta$ that
  \begin{align*}
     S_{n}&\le \frac{2\log (1 / \delta)}{F^{-1}(2\gamma\log (1 / \delta))}\lb \|w\|_{\infty} + \sqrt{\gamma \sum_{i=1}^{n}p_{i}w_{i}^{2}}\rb\\
& \le \frac{2\log (1 / \delta)}{F^{-1}(2\gamma\log (1 / \delta))}\lb \|w\|_{\infty} + \min\{\sqrt{\gamma p^{*}}\|w\|_{2}, \sqrt{\gamma n\bar{p}}\|w\|_{\infty}\}\rb,
  \end{align*}
where
\begin{equation}
  \label{eq:pbarpstar}
  \bar{p} = \frac{1}{n}\sum_{i=1}^{n}p_{i}, \quad p^{*} = \max_{i}p_{i}.
\end{equation}
\end{lemma}
\begin{proof}
If $\Omega \ge \gamma$, since $F^{-1}(x)$ is increasing,
\[F^{-1}(2\Omega\log (1 / \delta)) \ge F^{-1}(2\gamma\log (1 / \delta));\]
If $\Omega < \gamma$, since $F^{-1}(x) / \sqrt{x}$ is decreasing,
\begin{align*}
  &\frac{F^{-1}(2\Omega\log (1 / \delta))}{\sqrt{2\Omega\log (1 / \delta)}}\ge \frac{F^{-1}(2\gamma\log (1 / \delta))}{\sqrt{2\gamma\log (1 / \delta)}}\\
\Longrightarrow& F^{-1}(2\Omega\log (1 / \delta))\ge F^{-1}(2\gamma\log (1 / \delta)) \sqrt{\frac{\Omega}{\gamma}}.
\end{align*}
By Lemma \ref{lem:bernoulli}, with probability $1 - \delta$,
\begin{align}
  S_{n}&\le \frac{2\log (1 / \delta)}{F^{-1}(2\gamma\log (1 / \delta))}\|w\|_{\infty}\lb 1 + \sqrt{\frac{\gamma}{\Omega}}\rb \nonumber\\
& \le \frac{2\log (1 / \delta)}{F^{-1}(2\gamma\log (1 / \delta))}\lb \|w\|_{\infty} + \sqrt{\gamma \sum_{i=1}^{n}p_{i}w_{i}^{2}}\rb. \nonumber
\end{align}
\end{proof}

\end{document}